\documentclass[a4paper,11pt,headsepline]{scrbook}

\usepackage[english]{babel}
\usepackage[T1]{fontenc}
\usepackage[utf8]{inputenc}
\usepackage[autostyle]{csquotes}
\usepackage{bera}

\usepackage[sort,numbers]{natbib}

\usepackage[final]{microtype}

\usepackage{setspace}
\usepackage[left=2.65cm,right=3.35cm,top=2.5cm,bottom=2cm,includeheadfoot]{geometry}
\usepackage{mathtools}
\mathtoolsset{showonlyrefs,showmanualtags}

\usepackage{graphicx}
\usepackage{color}
\usepackage{subfigure} 
\usepackage{caption}
\usepackage{pgf}
\usepackage{tikz}
\usepackage{stmaryrd}
\usetikzlibrary{arrows,shapes,automata}
\usetikzlibrary{shapes.symbols}
\usetikzlibrary{shapes.callouts}
\usetikzlibrary{positioning}
\usepackage{multirow}
\usepackage{multicol}
\usepackage{booktabs}
\usepackage{url}

\usepackage{amsmath}
\usepackage{amsfonts}
\usepackage{bbm}
\usepackage[mathscr]{eucal}
\usepackage{nicefrac}
\usepackage{dsfont}
\usepackage{amssymb}
\usepackage{amsthm}
\usepackage{leftidx}

\usepackage{ifthen}
\usepackage{calc}
\usepackage{nomencl}

\usepackage{scrpage2}

\theoremstyle{plain}
\newtheorem{thm}{Theorem}[section]
\newtheorem{lem}[thm]{Lemma}
\newtheorem{prop}[thm]{Proposition}
\newtheorem{cor}[thm]{Corollary}

\theoremstyle{definition}
\newtheorem{defn}[thm]{Definition}

\newtheorem{exmp}[thm]{Example}

\newtheoremstyle{remark}
{10pt \topsep}   
{7pt \topsep}   
{\normalfont}  
{0pt}       
{\bfseries} 
{.}         
{5pt} 
{}          

\theoremstyle{remark}
\newtheorem{rem}[thm]{Remark}

\newtheoremstyle{not}
{10pt \topsep}   
{5pt \topsep}   
{\normalfont}  
{0pt}       
{\bfseries} 
{: \vspace{0.05 cm}}         
{5pt} 
{}          

\theoremstyle{not}
\newtheorem*{notation}{Notation}

\newtheoremstyle{conv}
{10pt \topsep}   
{5pt \topsep}   
{\normalfont}  
{0pt}       
{\bfseries} 
{: \vspace{0.05 cm}}         
{5pt} 
{}          

\theoremstyle{conv}
\newtheorem*{conv}{Convention}

\automark[section]{chapter}
\ihead{}
\ohead{\headmark}
\ofoot{\pagemark}
\renewcommand{\O}[1]{\Omega_{#1}}

\newcommand{\Vad}[2]{\mathcal{V}^{ad}_{#1}(#2)}

\newcommand{\Oext}{\Omega^{ext}}

\newcommand{\se}{\sigma}
\newcommand{\ve}{\varepsilon}
\newcommand{\Div}{\mathrm{div}}
\newcommand{\tr}{\mathrm{tr}}

\newcommand{\Umg}[2]{B_{#1}(#2)}

\newcommand{\T}[1]{T_{#1}}

\newcommand{\rk}{\rbrace}
\newcommand{\lk}{\lbrace}
\newcommand{\lgk}{\left\lbrace}
\newcommand{\rgk}{\right\rbrace}

\newcommand{\Menge}[2]{\lgk #1 \middle| #2 \rgk}

\newcommand{\skp}[2]{\left\langle #1,#2 \right\rangle}

\newcommand{\R}[1]{\mathbb{R}^{#1}}
\newcommand{\N}[1]{\mathbb{N}^{#1}}
\newcommand{\C}[1]{C^{#1}}
\newcommand{\Cm}[2]{C^{#1}(#2)}		

\newcommand{\Norm}[2]{\left\Vert #1 \right\Vert_{#2}} 
\newcommand{\Hnorm}[2]{\left\vert #1 \right\vert_{#2}}

\newcommand{\dist}{\text{dist}}
\newcommand{\diam}{\text{diam}}
\newcommand{\supp}{\text{supp}}

\newcommand{\ind}[1]{\text{ind}(#1)}
\newcommand{\coker}[1]{\text{coker}(#1)}
\newcommand{\im}[1]{\text{im}(#1)}

\newcommand{\Varlambda}{{\scriptstyle \varLambda}}

\newcommand{\sbt}{\,\begin{picture}(-1,1)(-1,-3)\circle*{1.4}\end{picture}\ }

\makeindex 
\allowdisplaybreaks[1]

\begin{document}

\begin{titlepage}
	\vspace*{-7em}
\begin{center}
\textcolor{black!40}{\rule{\textwidth}{1pt}}\\ \vspace*{-1.25em}
\textcolor{black!40}{\rule{\textwidth}{1pt}}\\ \vspace*{-0.5ex}
\Huge \textbf{On Shape Calculus with Elliptic PDE Constraints in Classical Function Spaces}\\
\vspace*{-1.6ex}
\textcolor{black!40}{\rule{\textwidth}{1pt}}\\ \vspace*{-1.35em}
\textcolor{black!40}{\rule{\textwidth}{1pt}}
\\[1.5em]
\includegraphics[width=0.3\textwidth]{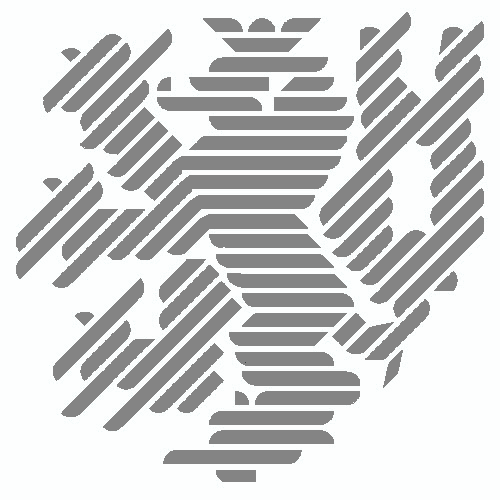} \\[1.5em]
\LARGE \textbf{DISSERTATION} \\[1em]
\large zur Erlangung des\\ akademischen Grades \\[1.5em] 
\textbf{Doktor der Naturwissenschaften} \\ \textbf{-- Dr. rer. nat. --}\\[1.5em]
Vorgelegt  and der Fakultät für Mathematik und Naturwissenschaften \\ der  Bergischen Universität Wuppertal  von \\[1.5em]
\textbf{Laura Bittner }\\[1.5em]
im Oktober 2018 \\[4em]

\begin{tabular}{l l} 
	\large \textbf{Betreut von:} & Prof. Dr. Hanno Gottschalk  \\
								 & Prof. Dr. Balint Farkas  
\end{tabular} 
\end{center} 
	
	
	
	
	\vfill 
	
\end{titlepage}

\renewcommand{\arraystretch}{1.4}

\frontmatter    



\section*{{\huge Abstract}}
\addcontentsline{toc}{section}{Abstract}
\thispagestyle{plain}
\noindent In this thesis we develop a functional analytic framework for shape optimization with elliptic partial differential equation (PDE) constraints in classical function spaces. Here in particular Hölder spaces have to be mentioned.
\\
\indent  This approach is motivated by shape optimization problems, which involve a special class of shape functionals, called \textit{reliability functionals}, and are subjected to linear elasticity constraints. The functionals we consider calculate the failure rate of a mechanically loaded device with respect to the component shape $ \Omega $. They reflect the physics of crack formation and depend highly nonlinear and non-quadratic on the stress field, i.e. on the first order derivatives of the state. Therefore, these objectives are ill-defined for $ H^1 $-solutions of the state equation and the shape derivatives are not defined for $ H^1 $-material derivatives.  Therefore, the resulting \textit{optimal reliability problems} can not be solved by the already existing methods of shape calculus and it becomes unavoidable to involve regularity theory for elliptic boundary value problems, Schauder estimates and classical PDE solutions.
\\
\indent
We develop a general concept on Banach and Hilbert spaces which is based on parameter depending variational equations and compact embeddings and which allows to transfer differentiability in lower Banach space topologies to higher ones.
We show, that this framework can particularly be applied to the variational formulation of the linear elasticity equation, given that the domain is transformed according to the speed method.
Once the existence of material and shape derivatives in Hölder spaces has been proved, we also show the existence of shape derivatives and derive so called adjoint equations. These equations allow to consider decent directions e.g. for iterative descent methods. However, these equations can not be derived strayightly since known approaches lead again to $ H^1 $-ill-defined equations. \\
\indent A crucial part of this work is the classification of the classical $ L^2 $ -shape gradient with respect to its regularity and thus with respect to its potential to sustain the domain regularity along a descent flow. In the scope of the presented concept, we proof what has numerically been observed for a long time: without regularization shape degeneration is predestined. We give an outlook on the existing regularization methods, illustrate their degree of smoothing and propose an approach that will hopefully prepare the ground to proof the existence of descent flows with respect to other metrics. 
\newpage 
\thispagestyle{plain}
$  $

\newpage
\section*{{\huge Zusammenfassung}}
\addcontentsline{toc}{section}{Zusammenfassung}
\thispagestyle{plain}

In dieser Arbeit wird ein funktionalanalytisches Konzept für Formoptimierung mit partiellen Differentialgleichungen (PDE) in klassischen Funktionenräumen, allen voran Hölder-Räumen, entwickelt. \\
\indent Dieser Ansatz ist motiviert durch Formoptimierungsprobleme mit speziellen Zielfunktionalen, die unter der Nebenbedingung der linearisierten Elastizitätsgleichung zu lösen sind. Diese sogenannten \textit{Zuverlässigkeits-Funktionale} berechnen die Versagenswahrscheinlichkeit von mechanisch belasteten Bauteilen in Abhängigkeit von deren Form. Sie berücksichtigen das pysikalische Werkstoffverhalten und hängen höchst nichtlinear und nicht quadratisch von den Spannungszuständen und damit von den Ableitungen erster Ordnung des Zustands ab.  
Daher sind diese Zielfunktionen ebenso wenig für Funktionen aus dem Sobolev-Raum $ H^1 $ definiert, wie auch die Formableitungen nicht für Materialableitungen in diesem Raum definiert sind. Somit werden diese Probleme nicht vom bestehenden Formkalkül abgedeckt, wodurch es unweigerlich notwendig wird, sich der Regularitätstheorie für elliptische Randwertprobleme und klassischen PDE-Lösungen zuzuwenden. 
\\
\indent Wir entwickeln ein allgemeines Konzept für Banach- und Hilberträume, das zulässt Differenzierbarkeit mittels kompakter Einbettungen von niedrigeren in höhere Topologien zu "transportieren". Dieses Konzept ermöglicht es die Existenz von Material- und lokalen Formableitungen in klassischen Funktionenräumen aber auch in Sobolev-Räumen höherer Ordnung zu zeigen. Außerdem wird es so möglich, die Existenz von Formableitungen für die vorgestellten Zuverlässigkeits-Funktionale zu beweisen.  Anschließend betrachten wir adjungierte Gleichungen, die zulassen Abstiegsrichtungen für iterative Minimierungsverfahren zu bestimmen. Allerdings können wir auch hier nicht gradlinig vorgehen, da bekannte Ansätze zu nicht lösbare Gleichungen führen. 
\\
\indent Zentral ist die Klassifizierung der klassischen Hadamard $ L^2$-Abstiegsrichtungen im Hinblick auf ihre Regularität und somit im Hinblick auf ihr Potential die Regularität der Startform entlang eines Flusses in Abstiegsrichtung zu erhalten. Wir zeigen im Rahmen des vorgestellten Konzepts, was schon lange in der Simulation beobachtet wird. Nämlich, dass die Degenerierung der Form vorprogrammiert ist, wenn die $ L^2 $-Abstiegsrichtung nicht regularisiert wird.
Wir geben einen Ausblick auf die existierenden Glättungsmethoden, illustrieren deren Glättungsgrad und schlagen einen möglichen Ansatz vor, der hoffentlich den Weg in Richtung der Existenz von Abstiegsflüssen ebnet. 
\newpage
\thispagestyle{plain}

$  $
\newpage

\section*{{\huge Danksagung}}
\addcontentsline{toc}{section}{Danksagung}
\thispagestyle{plain}
\newcommand*{\changefont}[3]{%
	\fontfamily{#1}\fontseries{#2}\fontshape{#3}\selectfont}

Das soll es jetzt also gewesen sein. Die letzte Formel ist geschrieben, die Arbeit formatiert und bereit zum Druck. Das Ende einer langen Zeit an der Bergischen Universität Wuppertal steht kurz bevor. Vor allem die Zeit des Promotionsstudiums war intensiv, lehrreich, manchmal schwieirig, spannend aber auch schön. Und so möchte ich mich an dieser Stelle bei allen bedanken, die währenddessen an meiner Seite waren und mich - in welcher Art und Weise auch immer - unterstützt haben.

\noindent Zuallererst möchte ich mich meinem Doktorvater Prof. Dr. Hanno Gottschalk danken, der mich ermutigte nach der Masterarbeit an der Universität zu bleiben. 
{\changefont{ppl}{m}{it}Dank dir, lieber Hanno, bin ich schließlich hier angekommen. Dein unerschütterlicher Optimismus hat mich oft zum schmunzeln gebracht und auch deine ebenso hilfsbereite, neugierige, aber auch kritische Art haben meine Zeit als Doktorandin bei dir sehr bereichert.
	Ohne deine Ermunterungen an meine eigenen Grenzen zu gehen, deine immerwährende Diskussionsbe\-reitschaft und deine Unterstützung wäre diese Thesis in dieser Form nicht möglich gewesen.} 

\noindent Mein Dank geht auch an Prof. Dr. Balint Farkas, der ebenfalls immer für fachliche Diskussionen offen war und der gerne die Zweitbetreuung meiner Dissertation übernahm. 
{\changefont{ppl}{m}{it} Danke dafür, lieber Balint, und auch für deine Bereitschaft meine Arbeit zu begutachten.}

\noindent Gleiches gilt auch für Herrn Prof. Dr. Volker Schulz, der ebenso bereit war ein Gutachten zu erstellen und bei dem ich mich auch für seine Einladung zu einem sehr gewinnbringenden fachlichen Austausch mit Dr. Kathrin Welker bedanken möchte. {\changefont{ppl}{m}{it} Vielen Dank dafür!}

\noindent Danke an Dr. Hannah Rittich, Dr. Stephan Schmidt, ohne die die Implementierung des LCF Funktionals in FEniCS wohl nicht zustande gekommen wäre, und auch an Dr. Kathrin Welker. {\changefont{ppl}{m}{it} Durch die Diskussion mit euch hat sich mir manches erst erschlossen. Dafür bin ich euch sehr dankbar und in deinem Fall, liebe Hannah, freue ich mich Dich nicht nur als fähige Mitmathematikerin, sondern vor allem als echte Freundin an meiner Seite zu wissen.} 

\noindent Gewinnbringend war es auch, mein Büro während des letzten Promotionsjahres mit Dr. Daniel Siemssen teilen zu dürfen. {\changefont{ppl}{m}{it} Deine Denkimpulse, deine LaTeX-Tipps und die Auflockerung durch deinen trockenen Humor werden mir fehlen.}

\noindent Danke vor allem auch an meine Eltern und Schwiegereltern, die mir das Studium erst ermöglicht und mich immer unterstützt haben - finanziell, aber auch mit Rat und Tat - auch wenn mein Schwiegervater diesen Moment leider nicht mehr miterleben durfte. Das Gleiche gilt für meine Schwester Nora, die immer ein offenes Ohr hat und zuallererst für meinen Mann Sebastian. {\changefont{ppl}{m}{it} Danke, dass Du in jeder, aber auch wirklich jeder Lebenslage so unerschütterlich an meiner Seite bist.} 
\enlargethispage{\baselineskip}
\begin{flushright}
	{\changefont{ppl}{m}{it} Rien ne se perd, rien ne se crée, tout se transforme.}\\
	\footnotesize{Antoine de Lavoisier}
\end{flushright}


\mainmatter
\tableofcontents
\thispagestyle{plain}



\setlength{\nomlabelwidth}{.27\hsize}
\setlength{\nomitemsep}{-0.1\parsep}
\renewcommand{\nomrefpage}[1]{Seite \textsl{#1}}
\renewcommand{\nomgroup}[1]
{\ifthenelse{\equal{#1}{N}}{\vspace{0.5cm}\item[\textbf{Norms, semi-norms, scalar products and distances}]}
	{\ifthenelse{\equal{#1}{F}}{\vspace{0.5cm} \item[\textbf{Function spaces}]}
		{\ifthenelse{\equal{#1}{A}}{\vspace{0.5cm}\item[\textbf{Sets and spaces}]}
			{\ifthenelse{\equal{#1}{D}}{\vspace{0.5cm}\item[\textbf{Differential operators}]}
				{\ifthenelse{\equal{#1}{V}}{\vspace{0.5cm}\item[\textbf{Linear elasticity}]}
					{\ifthenelse{\equal{#1}{C}}{\vspace{0.5cm}\item[\textbf{Linear algebra}]}
						{\ifthenelse{\equal{#1}{K}}{\vspace{0.5cm}\item[\textbf{Abbreviations}]}
							{\ifthenelse{\equal{#1}{B}}{\vspace{0.5cm}\item[\textbf{Measures and characteristic functions}]}
								{}
							} 
						} 
					} 
				}
			}
		}
	}
}

\vspace{-1cm}
\makenomenclature 
\renewcommand{\nomname}{Notation}
\markboth{\nomname}{\nomname}
\printnomenclature 


\nomenclature[A 0]{$ \N{} $, $ \N{}_{0} $}{natural numbers; $ \N{}_{0} = \N{} \cup \{0\}$}
\nomenclature[A 02]{$n,\, k$}{ elements of $ \N{}_0 $ ;}
\nomenclature[A 10]{$\R{n}$}{Euclidean space of dimension $ n \in \N{} $ ;}
\nomenclature[A 11]{$\mathbb{P}(M)$}{power set of a set $ M $;}
\nomenclature[A 111]{$ \mathbb{S}^{n-1}$}{$ (n-1) $-sphere in $ \R{n} $; $ \mathbb{S}^{n-1} =\{x \in \R{n} \vert \Norm{x}{} = 1\}$;}
\nomenclature[A 12 ]{$ \Umg{r}{z}$}{open ball with radius $ r $ with and center $ z $}
\nomenclature[A 13 ]
{$ \Sigma_{r} $}
{hemipsphere in $ \R{n} $, $ \Sigma_{r} =\lk x \in \R{n} \vert\,  \Norm{x}{}   \leq r ,\, x_n \geq 0 \rk $;}

\nomenclature[A 14]
{$ F_{r} $}
{ground of the hemisphere $ \Sigma_{r}, ~F_{r}=\lbrace x \in \R{n}_{+} \vert \, \Norm{x}{}\leq r,\,x_{n}=0 \rbrace$;}

\nomenclature[A 4]{$ \Omega $, $ \overline{\Omega} $, $ \overset{\circ}{\Omega}$, $ \Gamma =\partial \Omega$ }{domain in $ \R{n} $; closure, interior and boundary of $ \Omega $;}



\nomenclature[A 62]
{$\supp(f)$}
{support of the function $ f $}

\nomenclature[A 63]{$ \Omega' \Subset  \Omega $}{$ \Omega' $ is compactly contained in $ \Omega$ $ \Leftrightarrow \overline{\Omega'} \subset \Omega $ and $ \overline{\Omega'} $ is compact;}

\nomenclature[B 1]{$ \delta_{\omega} $}{Dirac Measure, $ \delta_{\omega}(\Omega)=1 $, if $ \omega \in \Omega $, otherwise $ \delta_{\omega}(\Omega)=0 $ }

\nomenclature[B 2]{$\delta_{ij}$}{Kronecker Delta: $ \delta_{ij}=1 $ if $ i=j $ and $ 0 $, if $ i \neq j $, $ i,j=1,\ldots ,n $}

\nomenclature[B 3]{$ \mathds{1}_{\Omega} $}{characteristic function on the set $ \Omega $ }

\nomenclature[B 4]{$ dx $}{Lebesgue Measure on $ \R{n} $}

\nomenclature[B 5]{$ dS $}{differential surface element}







\nomenclature[C]{$ x^{\top},~A^{\top} $}{transpose of $ x \in \R{n}$ and $ A \in \R{n \times n}$;}

\nomenclature[C]{$ A^{-1} $}{inverse of $ A \in \R{n \times n}$;}

\nomenclature[C]{$ \tr(A)$}{trace of $ A \in  \R{n \times n} $, $ \tr(A)=\sum_{i=1}^{n}a_{ii} $;}

\nomenclature[C]{$ \text{I} $}{identity matrix in $ \R{n} $ ;}

\nomenclature[C]{$A:B $}{$ A:B = \tr(A^{\top}B) $; Frobenius skalar product;}

\nomenclature[C]{$a \otimes b = ab^{\top} $}{$\vec{a} \vec{b}^{\top}:= (a_{i}b_{j})_{i,j=1,\ldots,m};$ tensor product of two vectors $ a,\, b \in \R{m}$;}


%

\nomenclature[D 3]
{$ \frac{\partial^{\vert \beta \vert}f}{\partial^{\beta}x}$}
{partial derivatives of $ f $ of degree $ \vert \beta \vert $, $ \beta \in \N{n}$ multi index, $\vert \beta \vert=\sum_{j=1}^{n}\beta_{i}$, $ \partial x^{\beta}=\partial x_{1}^{\beta_{1}}\dots x_{n}^{\beta_{n}};$}

\nomenclature[D 4 ]
{$ \frac{\partial f}{\partial \vec{n}} = Df\vec{n} $}
{derivative of $ f:\R{n} \to \R{m} $ in outward normal direction $ \vec{n}$;}

\nomenclature[D 40]
{$ \nabla f $}
{gradient $ f:\R{n}\to \R{} $;}

\nomenclature[D 41]
{$ Df $}
{Jacobian of $ f:\R{n} \to \R{m} $, $ Df=\nabla f^{\top} $ if $ f $ is scalar;}

\nomenclature[D 42]
{$\Delta =\frac{\partial^{2}}{\partial x_{1}^{2}}+ \dots \frac{\partial^{2}}{\partial x_{n}^{2}}$:}
{Laplace operator;}

\nomenclature[D 421]
{$ \Div(f)$}
{ divergence of $f:\R{n} \to \R{} $, $ \Div(f)=\tr(Df)=\sum_{j=1}^{n}\frac{\partial f}{\partial x_{j}}$;}

\nomenclature[D 422]{$ \Div(A)$}{column-wise divergence of $ A:\R{m} \to \R{n \times m} $, $  \Div(A)_{j}=\sum_{i=1}^{m} \hspace*{-1mm}\frac{\partial A_{ij}}{\partial x_{i}} $; }

\nomenclature[D 43]
{$ \nabla_{\Gamma} f =\nabla f - \langle \nabla f,\vec{n} \rangle\vec{n}  $}
{tangential gradient $ f:\Gamma\to \R{} $;}

\nomenclature[D 44]
{$ D_{\Gamma}f = Df - Df \vec{n}\vec{n}^{\top}
	$}
{tangential Jacobian on the boundary of $ f:\Gamma \to \R{m} $;}

\nomenclature[D 45]
{$ \Div_{\Gamma}f = \tr(D_{\Gamma}f)
	$}
{tangential divergence of $ f:\Gamma \to \R{m} $;}


\nomenclature[F 01]{$\mathcal{R}(X)$}{space of radon measures on $ (X,\mathscr{T}) $;}

\nomenclature[F 02]{$\mathcal{R}_c(X)$}{space of radon counting measures on $ (X,\mathscr{T}) $; }

\nomenclature[F 03]{$\mathcal{B}(X)$}{Borel $ \sigma $-Algebra on the topological space $ (X,\mathscr{T}) $;}

\nomenclature[F 04]
{$\mathcal{L}(X,Y)$}
{space of linear and continuous mappings from the Banach space $ X $ to the Banach space $ Y $;}

\nomenclature[F 05]
{$X' = \mathcal{L}(X,\R{})$}
{topological dual space of the Banach space $ X $;}

\nomenclature[F 1]
{$\Cm{}{\Omega}$}
{$ \Cm{0}{\Omega} $, space of continuous functions;}

\nomenclature[F 11]
{$\Cm{0,\phi}{\Omega}$}
{ H\"older-continuous functions;  $ f \in  \Cm{0}{\Omega} $ with $ [f]_{C^{0,\phi}(\Omega)}<\infty $,  $ 0<\phi \leq 1 $;}

\nomenclature[F 2]
{$C^{k}(\Omega)$}
{set of $ k $-times continuously differentiable functions $ f:\Omega\to \R{} $, $ k  \in \lk 0 \rk \cup \N{} \cup \lk \infty \rk $;}

\nomenclature[F 3]
{$\Cm{k,\phi}{\Omega}$}
{ functions in $ \Cm{k}{\Omega} $, with H\"older continuous $ k $-th derivative with exponent $ 0<\phi \leq 1 $, $ k\in \N{}_{0} $ (see \cite{GilbTrud});}

\nomenclature[F 31]
{$\Cm{k,0}{\Omega}$}
{$ \Cm{k}{\Omega} $;}

\nomenclature[F 32]
{$Diff^{k}(\Omega,\Omega')$}
{vector fields $ f\in C^{k}(\Omega,\Omega') $, $ f$ is bijective and $ f^{-1} \in C^{k}(\Omega',\Omega) $;}

\nomenclature[F 33]
{$C^{k}_{0}(\Omega)$}
{$ f \in C^{k}(\Omega) $ with $ \supp(f)\Subset \Omega $;}

\nomenclature[F 4]
{ $ L^{p}(\Omega) $}
{$ p $-Lebesgue integrable functions on $ \Omega $;  $ p\in [1,\infty)$, $ L^{p}(\Omega)=\lbrace f \mathrm{~measurable~}:~ \int_{\Omega} \vert f \vert^p < \infty \rbrace$}

\nomenclature[F 5]
{$ L^{\infty}(\Omega) $}
{essentially bounded functions on $ \Omega $;}

\nomenclature[F 6]
{$ W^{k,p}(\Omega) $}
{Sobolev space; $ L^{p} $-integrable functions, with weak derivatives up to order $ k $ in $ L^{p} $ (see \cite{AdamsFournier});}

\nomenclature[F 7]
{$ H^{k}(\Omega) $}
{$ W^{k,2}(\Omega) $;}

\nomenclature[F 8]
{$ W^{k-\nicefrac{1}{p},p}(\Omega) $}
{$= \mathbf{T_{\Gamma}}(W^{k,p}(\Omega))$; trace space of $ W^{k,p}(\Omega) $ functions (see \cite{Cia1988,AdamsFournier});}

\vspace*{2em}	

\noindent The vector valued versions of the spaces  $C(\Omega) $, $ C^{k}(\Omega) $, $ W^{k,p}(\Omega) $ ... are denoted by $ C(\Omega,\R{m}) $, $ C^{k}(\Omega,\R{m}) $, $ W^{k,p}(\Omega,\R{m}) $... The norms of these spaces result of the chosen norm on $ \R{m} $ (we choose the Euclidean $ 2 $-norm here).

\nomenclature[N 00]
{$ \Norm{.}{X}$}
{norm of the space $ X $;}  

\nomenclature[N 01]
{$ \Norm{.}{X'}$}
{operator norm, $ \Norm{F}{X^{\prime}} = \sup_{\Norm{x}{X} \leq 1} \vert F(x)\vert,\, F \in X'  $;}  	

\nomenclature[N 011]
{$ x_n \rightarrow x,\, n \to \infty $}
{(strong) convergence, $ x_n \underset{n \to \infty}{\rightarrow} x:= \lim_{n\to \infty} \Norm{x_n - x}{X} =0$;}

\nomenclature[N 012]
{$ x_n \rightharpoonup x,\, n \to \infty $}
{weak convergence, $ x_n \underset{n \to \infty}{\rightharpoonup} x:= \lim_{n\to \infty} l(x_n) = l(x)\, \forall l \in X' $;}

\nomenclature[N 02]
{$ \Norm{.}{} = \sqrt{\langle .,. \rangle} = \Norm{.}{\R{n}}$}
{norm in $ \R{n} $;} 

\nomenclature[N 03]
{$ \dist(x, \Omega)$}
{distance from $ x \in \R{n} $ to $ \Omega,\,\dist(x, \Omega)=\inf\lk \vert x-y \vert~ \vert ~y \in \Omega \rk$;}

\nomenclature[N 04]
{$\diam(\Omega)$}
{diameter of $\Omega $, $\diam(\Omega)=\sup \lk \vert x-y \vert~ \vert~ x,y \in \Omega \rk $ ;}

\nomenclature[N 03]
{$\Norm{f}{\infty,\Omega} = \Norm{f}{C(\Omega)}$}
{norm on $C(\Omega) $; $ \Norm{f}{C(\Omega)}=\sup_{x \in \Omega} \vert f(x )\vert $; abbreviaton  $ \Norm{f}{\infty} $; }

\nomenclature[N 04]{$\Norm{f}{C^k(\Omega)} $}
{norm on $ C^k(\Omega)$, $\Norm{f}{C^k(\Omega)}= \left(\sum_{\vert \alpha \vert \leq k}\Norm{\frac{\partial^{\vert \alpha \vert}f}{\partial x^{\alpha}}}{\infty}^p\right)^{1/p},\, p \in \N{}$;\\ usually $ p=1 $, here $ p=2 $;}

\nomenclature[N 05]{$[f]_{C^{0,\phi}(\Omega)} $}
{semi norm on $ C^{0,\phi}(\Omega) $, $[f]_{C^{0,\phi}(\Omega)} =   \sup_{x,y \in \Omega \atop  x \neq y} \frac{\vert  f(x) - f(y)  \vert }{\vert x-y\vert^\alpha } $;}

\nomenclature[N 06]{$\Norm{f}{C^{k,\phi}(\Omega)} $}
{norm on $ C^{k,\phi}(\Omega) $,\\ $ \displaystyle \Norm{f}{C^{k,\phi}(\Omega)}= \Big(\sum_{\vert \alpha \vert \leq k}\Norm{\tfrac{\partial^{\vert \alpha \vert}f}{\partial x^{\alpha}}}{\infty}^p + \sum_{\vert \alpha \vert =k } \left[\tfrac{\partial^{\vert \alpha \vert f}}{\partial x^{\alpha}}\right]_{0,\phi}^p \Big)^{1/p}  $,  $ p \in  \N{} $;\\ usually $ p=1 $, here $ p=2 $;}

\nomenclature[N 07]{$\Norm{f}{W^{k,p}(\Omega)} $}
{norm on $ W^{k,p}(\Omega) $, $\Norm{f}{W^{k,p}(\Omega)}= \left( \int_{\Omega}  \sum_{\vert \alpha \vert  \leq k} \vert  f(x) \vert^{p} \, dx \right)^{1/p}$ ;}


\chapter*{Introduction}
\thispagestyle{plain}
\addcontentsline{toc}{chapter}{Introduction}

17th April, 2018. The following message fills the news: "Southwest Airlines engine explodes in flight" \cite{NYTBoing2018}, "Material fatigue causes accident of a Boeing.", "According to initial reports of the National Transportation Safety Board (NTSB) blade number 13 out of 24 severed." \cite{NTVBoing2018}. No good news at all, but a good motivation for shape optimization and reliability optimization, in particular.  \\
\indent The goal of shape optimization is to obtain lower failure rates, less
material, more stability, or higher efficiency - all in all: more functionality. 
Shape optimization has many applications as there are for example airplaine wing desings with better airo-dynamics \cite{Schmidt2011airfoil}, lower failure rates for gas turbines \citep{GottschSchmitz}, bridges with more stability \cite{Bendsoe2003sigmund} or better image reconstruction like it appears e.g. in electrical impedance tomography \cite{Dambrine2008second}. The cost or objective functional under consideration depends on a shape and often also on the solution of a partial differential equation (PDE) which reflects the physical impacts and is called state equation. This solution itself is also shape dependent as the problem of the reentrant corner illustrates so impressively. Thus the solution $ u(\Omega) $ is coupled to the shape $ \Omega $ and therefore PDE constraint shape optimization problems can also be seen as a special class of optimal control problems \cite{Troltzsch2005optimale,Ulbrich2012constrained}.\\

\noindent \textbf{\textsf{Historical background and recent development}} \\[0.5\baselineskip]
\noindent In finite dimensional analysis, the necessary condition for a differentiable function $ f:\R{n} \to \R{} $ having a local minimum or maximum at a point $ x $ is $ \nabla f(x) =0$. Unfortunately, there is no straight forward way to define derivatives and gradients for functionals $ J $ that map a shape $ \Omega $, i.e. a subset of $ \R{2} $ or $ \R{3} $, to real values. 
\\
\indent In the beginning of shape optimization, more precisely in 1907 \cite{Hadamard1908memoire}, J. Hadamard presented an approach on how to obtain such a derivative of a shape functional $ J(\Omega) $. Therein, he considered normal perturbations of the boundary $ \Gamma = \partial\Omega $ of a smooth and bounded $ \Omega $.
It was only in 1975 when D. Chains published her famous paper on the existence of optimal shapes \cite{Chen75}\nocite{ExOptShape2003}. Four years later the approach of Hadamard was elaborated by Zolésio \cite{Zolesio1979identification} in the so-called “Hadamard structure theorem” which became central in shape calculus. This theorem first allowed to define descent directions, above all the $ L^2 $-shape gradient, for iterative optimization schemes and necessary optimality conditions. \\
\indent Nowadays, there are multiple approaches to define shape and topological derivatives: The speed or velocity method \cite{SokZol92}, the level set method \cite{Allaire2002level}, the homogenization method \cite{Allaire2012shape}, the perturbation of identity method \cite{DelfZol11} and many more. \\
It is even more complicated to reasonably define second order derivatives, then those of first order.  Since there is no intrinsic definition of distances on the power set of $ \R{n} $, even simple examples can be found \cite{Delfour1991anatomy} such that the second order shape derivative is non-symmetric. Understanding the set of shapes as a manifold overcame this problem, consider for example \cite{Schulz2014riemannian,Michor2003ShapeManifold}. Also other metrics than the $ L^2 $-metric are considered \cite{Schulz2016Steklov,Schulz2016metricsComparison,Schulz2015PdeShapeManifolds} and second order schemes like Newton, and Newton-like methods became relevant \cite{Schulz2015towardsLN,Delfour1991Velocity,Schmidt2018weak} in the numerical implementation. This accelerated the computational progress in shape optimization significantly and led further to the consideration of Lipschitz shapes \cite{Sturm2015shape,Welker2017optimization} and other shape spaces. Also, the Lagrangian method has a large impact on PDE constraint shape optimization. It recently allowed to automate the computation of shape derivatives in the finite element software \textsf{FEniCS} and the unified form Language (\textsf{UFL}) \cite{Wechsung2018automated,Schmidt2018weak}. It was first proposed by Céa \cite{Cea1986Lagrange} and later corrected and developed further  \cite{Ito2008variational,Sturm2013lagrange,Sturm2015shape,SturmLaurain2016distributed}. 
However, mesh degeneration remains a big problem in numerical optimization schemes and there are many approaches which aim to prevent this behavior \cite{Iglesias2017shape,Paganini2018higher,Schulz2016metricsComparison,Dokken2018shape}. Meanwhile, isogeometric analysis has found its way into shape optimization \cite{Fusseder2015isogeometric,Wall2008isogeometric} and allows to parameterize shapes with a high accuracy. Somewhat contrarily, shape optimization under uncertainties gains more and more popularity which is due to the aim of more realistic and robust models \cite{Conti2009shape,Harbrecht2015shape,GottschSchmitz,GBS_Ceramic2014} . This is, where we arrive at the motivation for this thesis.\\

\thispagestyle{plain}
\noindent \textbf{\textsf{Motivation}} \\[0.5\baselineskip]
\noindent To prevent fatigue fracture we need functionals which make predictions on the durability of the mechanic device which is represented by the shape $ \Omega \subset \R{3}$. 
The majority of shape optimization problems concerns functionals of the energy or tracking type. But these functionals are not a good choice in the context of reliability optimization for loaded systems since they do not reflect the material theoretic nature of fatigue \cite{ProbLCF2013}.
A question that arises immediately is "How can reliability be measured in a meaningful way?". A promising access to this problem is to combine a deterministic ansatz with a stochastic approach \cite{GottschSchmitz,GBS_Ceramic2014,ProbLCF2013}. 
\\
\indent Since cracks originate where the stress is accumulated, the deterministic capacity of a component, e.g. the number of load cycles to crack or the ultimate amount of tensile loading, is calculated based on the stress states acting on the material \cite{CMB,RO,Neuber1961}. These stress states $ \sigma(u) $ can be calculated by solving an elasticity equation on $ \Omega $. \\ 
The damage mechanism of low cycle fatigue (LCF) is best understood for poly-crystalline metal \cite{Werkstoffe,LCFalu2004}: Shear stresses are acting on the atomic layers of the material and initiates the transport of lattice defects to the surface. The resulting reentrant corners lead to local stress concentration and cracks originate at these corner tips.\\
In case of brittle media cracks start at the porosities which can be seen as initial flaws \cite{Ceramics2012}. However, crack initiation is also subjected to empirical scattering \cite{CerWeiScale} which has to be taken into account. Thus, the deterministic capacity is integrated into a probabilistic attempt based on Poisson Point Processes \cite{Wahrsch}. This leads to shape functionals which determine the risk of failure with respect to the device shape, the forces acting on it, i.e. the resulting stress states, the material and the physics of crack formation.\\ 
\thispagestyle{plain}
\indent Due to this complex and close-to-reality construction we have to accept that these functionals depend highly nonlinear and non-quadratic on the stress states, or in other words these \textit{reliability functionals} turn out to have challenging properties. In particular they are $ H^1 $-ill defined or strictly speaking: They are only defined for functions in Sobolev spaces $ W^{k,p}(\Omega,\R{3}) $ of higher order and with high values for $ p $ or for functions in spaces $ C^{l}(\overline{\Omega},\R{3}) $ or $ C^{l,\phi}(\overline{\Omega},\R{3})  $ of differentiable functions. Under suitable assumptions, it has already been shown that there exist optimal shapes \cite{GottschSchmitz,BittGottsch,GBS_Ceramic2014} for these problems, but providing the necessary shape calculus remained an open task and became subject of this thesis. \\

\thispagestyle{plain}
\noindent \textbf{\textsf{Outline and Contributions}}\\[0.5\baselineskip]
\noindent\textsf{Chapter 1:} In this chapter the stochastic foundations of reliability theory and point processes are introduced briefly. Moreover, two reliability functionals \cite{GottschSchmitz,GBS_Ceramic2014,BittGottsch} - called  LCF-reliability functional $ J^{\mathrm{lcf}} $ and ceramic-reliability functional $ J^{\mathrm{cer}} $ - and the associated shape reliability optimization problems under linear elasticity constraints are presented. It is illustrated, that these functionals are $ H^1 $-ill defined to motivate the approach of shape calculus in classical function spaces.
\\[0.5\baselineskip]
\noindent\textsf{Chapter 2:} This chapter provides the notation and the concepts that are mandatory to treat linear elliptic systems of partial differential equations. Since the classical solution theory of  linear elasticity has an important impact on the study of its shape derivatives in Hölder spaces, a compilation of the results on existence \cite{Cia1988}, regularity theory \cite{Geymonat} and Schauder estimates \cite{Agm59,Agm64}, is given here. 
Most of the provided results are well known, but scattered across the literature and thus hard to find. Therefore, they are complemented, where no suitable sources could be found. 
\\[0.5\baselineskip]
\noindent\textsf{Chapter 3: } The foundations of calculus in Banach spaces \cite{Cheney,Werner_Funkana}, that are needed to treat material derivatives with respect to Banach space topologies are summarized. In this work, the latter ones appear in particular in the shape of Hölder topologies \cite{GilbTrud}. 
\\[0.5\baselineskip]
\noindent\textsf{Chapter 4:}  This chapter gives an introduction to shape optimization and recalls
some basic material from shape calculus. Further, it provides the main results and arithmetic rules regarding (local) shape and material derivatives \cite{SokZol92,DelfZol11,ShapeOpt} in spaces of differentiable functions. Moreover, the existence of shape derivatives for general local objective functionals of the volume $ J_{vol}(\Omega) =\int_{\Omega}\mathcal{F}(x,u,\sigma(u))\,dx $ and surface type $ J_{sur}(\Omega) =\int_{\Gamma}\mathcal{F}(x,u,\sigma(u))\, dS $ are shown under differentiability assumptions that are 
suitable in the scope of this thesis. 
\\[0.5\baselineskip]
\noindent\textsf{Chapter 5:} A novel framework for sensitivity analysis of solutions of linear variational equalities with parameter dependence on Hilbert spaces $ b^{t}(u^{t},v)=l^{t}(v)~\forall v \in H $ is developed. Such equations appear e.g. when linear PDE on parameter dependent domains are considered. Two main results are shown here: 
\begin{itemize}
	\item  Under suitable assumptions on coercivity, continuity and differentiability of the linear and bilinear forms $ l^{t} $ and $ b^{t} $, the solutions $ u^{t} $ are continuously differentiable with respect to the parameter $ t $ in the strong Hilbert space topology. 
	\item  The second result shows how differentiability can be transferred to higher topologies in Banach spaces. 
\end{itemize}
\thispagestyle{plain}
\noindent \textsf{Chapter 6:} The outcomes of Chapter 5 prepare the ground on our mathematical way towards material derivatives w.r.t. classical function space topologies. In order to apply the regularity theory for PDE, introduced in \textsf{Chapter 2}, shapes are considered to be $ C^{k} $-domains. It is shown that the results of the previous sections can be applied to the PDE of linearized elasticity when the domain is perturbed according to the velocity method \cite{SokZol92}.
\begin{itemize}
	\item The existence of material derivatives for linear elasticity in Hölder space topologies is proved.
	\item Further it is shown, that the general local cost functionals $ J_{vol}(\Omega)  $ and $ J_{sur}(\Omega) $, introduced in Chapter 4, are shape differentiable under elasticity constraints. In particular, these results apply for the reliability functionals $ J^{\mathrm{lcf}} $ and $ J^{\mathrm{cer}} $ which were introduced in Chapter 1.
\end{itemize}

\noindent \textsf{Chapter 7:} Adjoint equations to $ J_{vol}(\Omega)  $ and $ J_{sur}(\Omega) $ w.r.t. linear elasticity are considered in order to derive Hadamard $ L^2 $-shape gradients for the mentioned reliability functionals. The adjoint states and the shape gradients are analyzed in view of their regularity and their potential to maintain the smoothness along a gradient flow.
\begin{itemize}
	\item A regularity theory for adjoint equations and $ L^2 $-shape gradients in Hölder spaces is presented. 
	\item It is shown that the Hadamard gradient is insufficient to pertain the shape regularity in the framework of $ C^{k} $-shapes. 
	\item Further, a brief investigation on how the regularity assumptions can be reduced without loosing the property of shape differentiability is given.
	With a glance on other descent directions \cite{Sturm2015shape,Sturm2013lagrange,Schulz2016metricsComparison} and a proposal on a further one, this thesis is concluded.
\end{itemize}


\chapter[Probabilistic Failure Models]{Probabilistic Failure Models for Devices under Load} \label{Reliability}

In the case of metal devices, repeated mechanical (and thermal) loading leads to a slow deterioration of the material - also known as fatigue. 
Also one time ultimate loading can lead to fracture of components made of brittle material, like ceramic.
In both cases, it is impossible to determine the mechanical capacity of the component exactly. Hence, it is more promising to set up probabilistic models for crack formation that involve material behavior and estimate failure probability.  The stochastic nature of crack formation has been widely studied in the materials science literature, see e.g. \cite{Werkstoffe}. For information on fatigue, fracture mechanics and their mathematical investigation we refer to Section \ref{subsec:LCF} below.
 
Here we briefly resume the basic definitions needed in the following sections.  For a more detailed introduction to stochastics we recommend for example the book \cite{Wahrsch}. 

\section{Hazard functions and point processes}
\begin{defn}[Density, Distribution,- Survival-, and Hazard Function]
	Let $ (\varSigma ,{\cal A},P) $ be a probability space and $ \mathcal{Z} :(\varSigma ,{\cal A},P)\to (\R{},\mathcal{B}(\R{}))$ be a random variable. 
\begin{itemize}
	\item[i)] The associated \textit{distribution function} is defined by 
	\begin{equation}
	F_{\mathcal{Z}}:\R{} \to [0,1],\,~~~~~ F_{\mathcal{Z}}(s):=P(\mathcal{Z} \leq s)
	\end{equation}
	and the survival function by $ S_{\mathcal{Z}}(s):=P(\mathcal{Z} > s) = 1- P(\mathcal{Z} \leq s) $. 
	\item[ii)] Suppose that $ F_{\mathcal{Z}} $ is differentiable. Then the \textit{probability density} associated with $ \mathcal{Z} $ is given by 
	\begin{equation}
	f_{\mathcal{Z}}:\R{} \to \R{+}_{0},\,~~~~~ f_{\mathcal{Z}}(s)  = \frac{d}{ds} F_{\mathcal{Z}}(s)\, .
	\end{equation}
	\item[iii)] The related \textit{hazard function} is defined by 
	\begin{align}
	h_{\mathcal{Z}}(s):=
	\begin{cases}
		\lim_{h \to 0} \frac{ P(\mathcal{Z} \in (s,s+h] \, | \, \mathcal{Z} >s)}{h}  & ,\, \text{if } S_{\mathcal{Z}}(s) >0 \\
		\infty &,\, \text{otherwise. } 
	\end{cases}
	\end{align}
	Here $ P(A \,|\, B):=\frac{ P(A \cap B)}{P(B)} $ for $ A,\, B \in \mathcal{A} $ denotes the conditional probability.
	\item[iv)] The \textit{cumlultive hazard function} is defined by
	\begin{align}
	H_{\mathcal{Z}}(s) := \int_{-\infty}^{s} h_{\mathcal{Z}}(\tau) \, d\tau.
	\end{align}
	\item[v)]  If $ f_{\mathcal{Z}} $ is continuous, then $$ h_{Z}(s)= \frac{f_{\mathcal{Z}}(s)}{1 - F_{\mathcal{Z}}(s)} $$ and $$ F_{\mathcal{Z}}(s)= 1 - e^{-H_{\mathcal{Z}}(s)}.$$
\end{itemize}
\end{defn}
Note that if $ F_{\mathcal{Z}} $ is continuous then $ P(\mathcal{Z} = s)=0 $ for any $ s \in \R{} $. In this situation e.g. $ P(\mathcal{Z} \leq s) = P(\mathcal{Z} < s) $.
Especially recall the following:
\begin{defn} 
	A random variable $\mathcal{Z}: (\varSigma,\mathcal{A},P) \to ([0,\infty],{\cal B}([0,\infty])$ is called \textit{Weibull distributed} with scale parameter $\eta>0$ and shape parameter $ m>0 $, written $ \mathcal{Z} \sim \mathrm{Wei}(\eta,m)$, if and only if the distribution function satisfies 
	\begin{align*}
	F_\mathcal{Z}(s)=
	\begin{cases}
	1-e^ {-\left(\frac{s}{\eta}\right)^ m} & \text{ if } s\geq 0\\
	0  & \text{ if } s<0.
	\end{cases}
	\end{align*} The associated hazard function is given by 
	\begin{align*}
	h_{\mathcal{Z}}(s) =
	\begin{cases}
	\frac{m}{\eta}\left(\frac{s}{\eta}\right)^{m-1} & \text{ if } s\geq 0\\
	0  & \text{ if } s<0.
	\end{cases}
	\end{align*}
\end{defn}

Assume that $ \mathcal{Z} $ is the random failure time of some device. The survival function determines the probability that the device survives beyond time $ s $ with probability $ S_{\mathcal{Z}}(s) $, whereas $ F_{\mathcal{Z}}(s) $ determines its failure probability until time $ s $. 

The hazard function is nether a density nor a probability function. Nevertheless, given
that the subject has survived until time $ s $, we can think of it as the
probability of failure in "the next moment" or in an infinitesimally short time period $ (s,s+h]$ . Thus, the hazard rate is a measure of risk:
the higher the values of the hazard function $( s_1, s_2 ]$, the higher the risk of failure in this interval.

\begin{defn}
	Let $ (X,\mathscr{T}) $ be a topological space and let $ \mathcal{B}(\mathscr{T}) $ denote the Borel $ \sigma $-algebra on $ (X,\mathscr{T}) $.
	\begin{itemize}
		\item[i)~] A measure $ \nu $ on $ (X,\mathcal{B}(\mathscr{T})) $ is called \textit{Radon measure}, if it is inner regular and locally finite\footnote{See also Appendix \ref{App: Sec:TopoMeasures} for the topological and measure theoretic foundations.}. By $ \mathcal{R}(X)  $ we denote the set of all Radon measures on $ X $.
		\item[ii)] Let $ \nu  $ be a Radon measure on $ X $. If $ \nu: \mathcal{B}(\mathscr{T}) \to \N{}_{0} \cup \{\infty\} $ then $ \nu $ is called \textit{Radon counting measure}. The space of Radon counting measures is denoted by $ \mathcal{R}_{c}(X) $.
	\end{itemize}
\end{defn}

For any continuous function $ h $ on $X$ with compact support i.e. $h\in C_0(X)$, and $ \nu \in \mathcal{R}(X) $ the integral $ \int_Xhd\nu $ is well defined and the mappings $\mathcal{R}(X) \, (\mathcal{R}_c(X))\ni \nu\to \int_Xh\,d\nu$, $h\in C_0(X)$  induce the weak-*-topology on the space of the Radon (counting) measures. 
The standard $ \sigma $-algebra on $ \mathcal{R}_{c}(X) $ is  generated by these mappings and denoted by $ \mathcal{N}(\mathcal{R}_{c}(X)) $. 
The associated Borel $\sigma$-algebra is denoted by $\mathcal{B}(\mathcal{R})$ ($\mathcal{B}(\mathcal{R}_c)$). 

\begin{defn} Let $(X,\mathscr{T})$ be a locally compact, second countable Hausdorff space equipped with its Borel $ \sigma $-algebra $ \mathcal{B}(\mathscr{T}) $ and let $(\varSigma,{\cal A},P)$ be a probability space. If $\gamma:(\varSigma ,{\cal A},P)\to (\mathcal{R}_{c}(X),\mathscr{N}(\mathcal{R}_{c}(X)),\, \varsigma \mapsto \gamma_{\varsigma}$ is a measurable mapping, then $\gamma$ is called a \textit{point process} (PP).
\end{defn}

\begin{defn}\label{Def: PP}	A point prosess $ \gamma $ on $(\varSigma,{\cal A},P)$
	\begin{itemize}
		\item[i)~~] is called \textit{simple}, if $ \gamma_{\varsigma} $ is simple for $ P $-almost all $ \varsigma \in \varSigma $, i.e. for $ P $-almost all $ \varsigma \in \varSigma $ and $\forall \, Y \in \mathcal{B}(\mathscr{T}) $ with $ \gamma_{\varsigma}(Y)< \infty  $, there is $ j \in \N{}  $ and $ y_{1}, \ldots, y_{j} \in Y ,\, y_i \neq y_j $ for $ i\neq j $ such that $ \gamma_{\varsigma}\vert_{Y} = \sum_{i=1}^{j}\delta_{y_{i}} $.
		\item[ii)~] is called \textit{non atomic}, if $ P(\gamma_{\varsigma}(\{x\})>0)=0~~ \forall x \in X ,\,\varsigma \in \Sigma$. 
 		\item[iii)]  has \textit{independent increments}, if the random variables $ \gamma_{(.)}(Y_{i}): (\varSigma,{\cal A},P) \to (\N{}_{0} \cup\{\infty\}, \mathbb{P}(\N{}_{0} \cup\{\infty\})) ,\, i \in\{1,\ldots,k\}  $ are stochastically independent for all disjoint sets $ Y_{1},...,Y_{k} \in \mathcal{B}(\mathscr{T}),\, \varsigma \in \Sigma $.
	\end{itemize}
\end{defn}

\begin{defn}\label{Def: PPP} A point process $ \gamma:(\varSigma ,{\cal A},P)\to (\mathcal{R}_{c}(X),\mathscr{N}(\mathcal{R}_{c}(X)) $  is a \textit{Poisson point process} (PPP), if for any $ \varsigma \in \varSigma $ there is a Radon measure $ \rho_{\varsigma}  \in  {\cal R}(X)$ such that $\gamma_{\varsigma}(Y)$ is Poisson distributed $\forall Y \in{\cal B}(\mathscr{T})$ with intensity $\rho_{\varsigma}(Y)$, i.e.
\begin{align}\label{Eq:PPP}
P(\gamma_{\varsigma}(Y)=n)=e^{-\rho_{\varsigma}(Y)}\,\frac{\rho_{\varsigma}(Y)^n}{n!}.
\end{align}
\end{defn}

As shown in  \cite{PPPWatanabe1964, Wahrsch} the point process $  \gamma $ is a Poisson point process if and only if it possesses the properties  i) -iii) of Definition \ref{Def: PP}.

\begin{rem}
	Instead of $ \gamma_{\varsigma} $ we usually write $ \gamma $ and suppose implicitly that $ \gamma = \gamma_{\varsigma}$ is a randomly generated counting measure $ \gamma: (X,\mathcal{B}(\mathscr{T})) \to (\N{}_{0}\cup\{\infty\}, \mathbb{P}(\N{}_{0})\cup\{\infty\}) $.
\end{rem}

\newpage
\section{Modeling mechanic loading} \label{Sec:MechLoad}
The PDE of linearized elasticity models the displacement that an elastic material undergoes under load. This system of partial differential equations will be introduced briefly in this section and discussed in detail in Chapter \ref{PDE_Systems}. For literature see e.g. \cite{Cia1988,Ciarlet1997_Band2,Ciarlet2000_Band3,HetEsl09,ErnGuerm04} or \cite{KnopsPayne1971}.

Let $\Omega\subseteq \R{3}$ be a bounded domain with Lipschitz boundary. It models a mechanic device like a blade of a turbine. Its  boundary is denoted by $\Gamma=\partial\Omega$ and its closure by $\overline{\Omega}$. We assume that the device consists of an isotropic material e.g. a forged or cast metal. 
Throughout this work, we assume that the boundary of $ \Omega $ is divided into two parts: An interior part $ \Gamma_{D} $, where the device is clamped, and an exterior boundary part $ \Gamma_N $ such that $ \Gamma_{D} \dot{ \cup} \Gamma_{N}=\Gamma $.

Under operation, the device $\Omega$ is loaded by a volume force $ f=f(\Omega):\Omega \to \R{3} $ like gravity or centrifugal loads. The vector field 
$ g_{N}=g(\Gamma_N):\Gamma_N \to \R{3} $ is a surface load e.g. caused by static gas pressure $ P_{N}=P(\Gamma_N):\Gamma_N \to \R{} $. Then $ g_{N}(x)=-P_{N}(x)\vec{n}(x),\, x \in \Gamma_N $ where $ \vec{n} $ denotes the outward normal vector field on $ \Gamma $.   
The related displacement is defined by the vector field $u=u(\Omega):\overline{ \Omega}\to \R{3}$ which can be derived as a solution of a linear elasticity problem which is given by a system of linear elliptic PDE of second order.       
According to \citep{ErnGuerm04,Cia1988} the \textit{disjoint displacement-traction problem of linear isotropic elasticity} is defined by 
\begin{align}\label{Reliability:Eq:LinEl} \tag{P1}
\left. 
\begin{array}{rcll}
-\Div( \se(u)) &=&f  &\text{ in } \Omega \\
u &=& 0  &\text{ on } \Gamma_{D}  \\
\se(u) \vec{n}&=&g_N &\text{ on }\Gamma_{N}
\end{array} 
\right.
\end{align}
on $ \Omega \subset \R{3} $ where
\begin{equation}\label{Reliability:Eq:Stress_Tensor} 
\se(u):=\lambda\Div(u)\mathrm{I}+\mu(Du+Du^{\top}) \text{ in } \Omega
\end{equation} 
is the stress tensor. By $ \mathrm{I} = \delta_{ij},\, i,j=1,2,3 $ we denote the identity matrix on $ \R{3} $.
It is also common to express the stress $ \sigma(u) $ via the strain-tensor $ \varepsilon(u):=\frac{1}{2}(Du+Du^{\top}) $. Then \eqref{Reliability:Eq:Stress_Tensor} reads $$ \se(u)=\lambda\,\tr(\varepsilon(u))\mathrm{I}+2\mu\,\varepsilon(u) \text{ in } \Omega.$$

In course of a load cycle, some heating and cooling might take place, such that $u_{\Omega}$ also depends on a temperature distribution $T(\Omega): \overline{ \Omega}\to \R{}$. Then $ u = u_{T}(\Omega) $ can be derived as a solution of a thermal elasticity problem. But since this is more or less a special case of \eqref{Reliability:Eq:LinEl}, we refer to \citep{BittGottsch} and \citep{HetEsl09} for further investigations. 

\section{Probabilistic models for mechanic failure mechanisms}
 
Because of technical or financial reasons it is not possible to inspect machines or buildings like engines, turbines and or bridges in an rhythm of several days, weeks or month in detail. The more important is a good prognosis regarding the lifetime and service interval scheduling \cite{BGMGSS,groger2016models}. 

A classical approach to estimate the lifetime of a component exposed to fatigue caused by slowly repeated cyclic loading (LCF) is the computation of the so called deterministic lifetime to crack initiation. But this concept alone is not sufficient since the component might also fail with a certain probability before or after that number of load cycles. Thus, it seems more favorable to combine the classical deterministic approach with a stochastic one. This approach was proposed in \cite{ProbLCF2013, DissSchmitz2014} and later examined in view of gas turbine design \cite{BittGottsch,ThermoLCF2013,LCFSchmid2014,GottschSchmitz}  and shape optimization \cite{MultiCrit_LCFAdj2018,GottschSaadi2018}. A survey can be found in \cite{BGMGSS}. We refer also to \cite{SPPCrack2018,ProbTurbMatDes1997} or \cite{EPNotch1985}. The probabilistic model is based however on the deterministic ansatz.

\subsection{Metal components under cyclic loading and LCF}\label{subsec:LCF}


It is well known \cite{Werkstoffe,PhysMatScieGottstein2004} that repeated loading of a mechanical component ultimately leads to failure, even if the maximal tensile strength of the material is much higher than the single loads. The resulting deterioration in the material is known as fatigue. LCF is a damage mechanism which is stress and surface driven and is best understood for poly-crystalline metal: Shear stress are acting on the atomic layers 
of the material and leads to the transport of lattice defects to the surface. After several load cycles these defects reach the surface of the component and form intrusions and extrusions, see Figure \ref{fig:force} (a). This leads to stress concentration at the tip of the intrusion and cracks originate there \cite{Werkstoffe}, see Figure \ref{fig:damage} (b).

\begin{multicols}{2}
	\begin{center}
		\includegraphics[scale=.155]{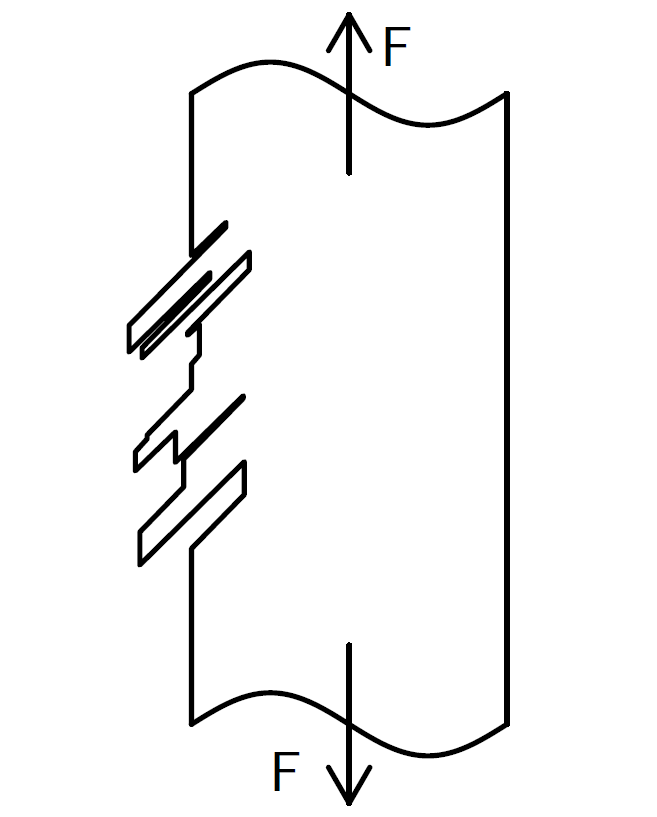}\\
		\captionof*{figure}{\footnotesize (a) Intrusions and extrusions at the \\ surface forming under cyclic appli- \\cation of the force $F$.}\label{fig:force}
	\end{center}
	\columnbreak 
	{ \centering
		\includegraphics[scale=1.17]{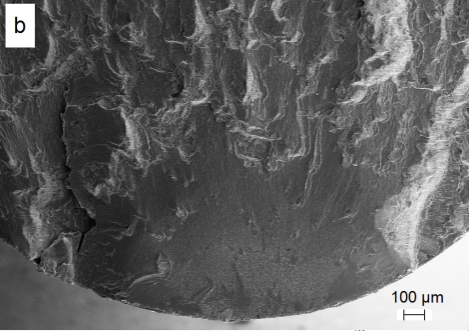}\\
		\captionof*{figure}{\footnotesize (b) Crack initation at the lower boundary of a specimen cracked during a cyclic life test for the Ni-based superalloy RENE80}\label{fig:damage}
	}
\end{multicols}

We derive the deterministic number of load-cycles to crack initiation $ N_{det} $ in the case of cyclic, purely mechanical loading as follows: 
Let  $\sigma=\sigma(u):\Omega\to \R{3\times 3}$ be the stress field associated with the displacement field $u=u(\Omega)$, i.e. $ \sigma(u) = \lambda \Div(u)\mathrm{I} + \mu(Du + Du^{\top}) $, $ \lambda,\, \mu >0 $. Here we suppress the  $\Omega$ dependence for notational simplicity.  

\begin{enumerate}
\item First calculate the trace free part $ \sigma'= TF(\sigma)=\sigma-\frac{1}{3}\tr(\sigma)\mathrm{I}$ of $ \sigma $. 
\item Then define the amplitude comparison stress as the von Mises stress associated with $ \sigma$, i.e. $ \sigma_v= VM(\sigma')=\sqrt{\frac{3}{2} \sigma': \sigma'}$ and define the amplitude stress $\sigma_a:=\nicefrac{\sigma_v}{2}$.
\item If $ \sigma$ is obtained from a linear elasticity equation \eqref{Reliability:Eq:LinEl}, convert the van Mises stress $ \sigma_v \in \R{+}_{0}$ to elastic-plastic amplitude stress $ \sigma^{el-pl}\in \R{+}_{0} $, e.g. by solving the Neuber equation \cite{Neuber1961,GlinkaNeuberMeth2000}
\begin{align}
	\label{Eq:SD}
	\sigma_v=SD(\sigma^{el-pl})=\sqrt{( \sigma^{el-pl})^2+ E \sigma^{el-pl}\left(\frac{\sigma^{el-pl}}{K}\right)^{1/\hat{n}}}.
\end{align}
Otherwise, i.e. if $\sigma$ is obtained from an elastoplastic problem, set $ \sigma^{el-pl}=\sigma_a$. In equation \eqref{Eq:SD} $E$ denotes the Young's modulus, $K$ the hardening coefficient and $\hat{n}$ the hardening exponent.	
	\item Afterwards, convert the elastic-plastic comparison stress amplitude to the elastic-plastic strain amplitude $\varepsilon^{el-pl} \in \R{+}_{0}$ via the Ramberg-Osgood relation \cite{RO}:
	\begin{equation}
	\label{Eq:RO}
	\varepsilon^{el-pl}=RO(\sigma^{el-pl})=\frac{ \sigma^{el-pl}}{E}+\left(\frac{\sigma^{el-pl}}{K}\right)^{1/\hat{n}}\, .
	\end{equation}
	\item Finally, solve the Coffin-Manson-Basquin equation \cite{CMB} for $N_{det} \in \R{+}_{0} \cup \{\infty\}$,
	\begin{equation}
	\label{Eq:CMB}
	\varepsilon^{el-pl}=CMB(N_{det})=\frac{\sigma_f'}{E}(2N_{det})^b+\varepsilon_f'(2N_{det})^c.
	\end{equation}
	Here $\sigma_f',\varepsilon_f'>0$ and $b,c<0$ are material constants.
\end{enumerate}
\vspace*{-1ex}
Note that models that include notch support factors \cite{Werkstoffe} also require derivatives of second order on $ u $ that enter into the definition of $N_{det}$. Examples for such models have been investigated e.g. in \cite{EPNotch1985,HertelVormwNotch2012} or \cite{NotchSizeLCF2018}. In \cite{LCFSchmid2014} also Schmidt factors are considered. Moreover, the van Mises stress $ \sigma_{v} $ can be replaced by the stress amplitude $ \sigma_a $, but then the constants have to be adapted. 

Based on that, we can define the mappings 
\begin{eqnarray*}
TF:~~ &~~~~\R{3 \times 3} \to \R{3 \times 3}, & M \mapsto M - \frac{1}{3}\tr(M)\mathrm{I}\, ,\\
VM:~~ &\R{3 \times 3} \to \R{},&M  \mapsto \sqrt{\frac{3}{2}M:M}\, , \\
SD:~~ &\R{+}_{0} \to \R{+}_{0},&x \mapsto \sqrt{x^2 + \frac{E}{K^{\nicefrac{1}{\hat{n}}}}\,x^{1+\nicefrac{1}{\hat{n}}}} \, ,\\
RO:~~ &\R{+}_{0} \to \R{+}_{0},&x \mapsto \frac{x}{E} + \left(\frac{x}{K}\right)^{\nicefrac{1}{\hat{n}}} \, , \\
CMB:~ &~~~~~~~~~~~\R{+}_{0} \to \R{+} \cup \{\infty\},& x \mapsto \frac{\sigma_{f}'}{E}(2x)^{b} + \epsilon_{f}'(2x)^{c}\, .
\end{eqnarray*}
The matrices without trace ($ M \in \ker(TF) $) are mapped to infinite life and thus we set
\begin{align}
N_{det}:\R{3 \times 3} \to \R{+} \cup \{\infty\},\, M \mapsto CMB^{-1}\circ RO \circ SD^{-1} \circ VM \circ TF(M)
\end{align}
and conclude with
\begin{align}
N_{det}(\sigma(u(x))):=CMB^{-1}\circ RO \circ SD^{-1} \circ VM \circ TF(\sigma(u(x))), ~~~ x \in \Gamma.
\end{align}

\subsubsection{Probabilistic models for fatigue under cyclic loading}

The LCF failure mechanism is purely surface and strain driven. To make the model accessible for more general failure mechanisms, we take all cracks into account that we can find after some number of load cycles $ s \in \N{} $ in the interior ($ x \in \overset{\circ}{\Omega} $) or the surface ($ x \in \Gamma $) of the component. Though, for technical reasons we assume that $ s $ is a positive real number. Thus we choose $\mathcal{C}=\mathbb{R}_+\times \overline{\Omega}$ to be the configuration space for crack initiation. 

Radon counting measures have the property to map measurable sets to natural numbers $\mathbb{N}_0\cup\{\infty\}$.
Thus, it is obvious to take such measures to "count" the number of cracks in the time-space set $ [0,s]\times \overline{\Omega} $.


\begin{defn}[Crack Initiation Process] \cite[Def. 2.2]{GottschSchmitz}
	Let $(\varSigma,{\cal A},P)$ be a probability space. Any point process $ \gamma(\Omega) :(\varSigma,{\cal A},P)\to (\mathcal{R}(\mathcal{C}),\mathcal{B}(\mathcal{R}_c))$   
	that satisfies Definition \ref{Def: PPP} i) and ii) is called a \textit{crack initiation process}.
\end{defn}

\begin{rem}
Actually, $ \gamma = \gamma_{(.)}(\Omega,.):(\varSigma,{\cal A},P)\to (\mathcal{R}(\mathcal{C}),\mathcal{B}(\mathcal{R}_c)),\, \varsigma \mapsto  \gamma_{\varsigma}(\Omega,.) $ generates a crack initiation process $ \gamma_{\varsigma}(\Omega,.)  $ which then counts the existing cracks in $ C_s = [0,s] \times \overline{ \Omega} $. Thus, strictly speaking,  $ \gamma(C_{s}) =  \gamma_{\varsigma}(\Omega,C_s) \in \N{}_{0} $. 
For notational simplicity we suppress the $ \Omega $  and $ \varsigma $ dependence if possible. 
\end{rem}

Assume that the component $ \overline{\Omega} $ is crack free at the beginning $ s=0 $.
For a crack initiation process $\gamma = \gamma(\Omega)$ we define the time of first failure associated with $\Omega$ by  $$\mathcal{T}(\Omega)=\mathcal{T}(\gamma(\Omega)) :=\inf\{s> 0:\gamma([s,\infty)\times \overline {\Omega})>0\} = \sup\{s  > 0 :\gamma([0,s]\times \overline {\Omega})=0 \}.$$
Since it is possible that the component is never destroyed, we interpret the failure time $ \mathcal{T} $ as random variable $\mathcal{T}:(\varSigma,{\cal A},P)$ $\to ([0,\infty],{\cal B}^{0,\infty})$. 

Simplicity and nonatomicness can be interpreted in the following way. 
First of all two cracks can not initiate at the same location and the same time (in that case they would be counted as one crack). Non atomicness is motivated by the fact that there should be no point where the probability that a crack originates exactly there is larger than zero. 

Moreover, we can assume that for small times $ s>0 $  cracks have not yet grown to a size where they influence the macroscopic stress field \cite{GottschSchmitz}. Also, due to the fact that we are interested in first failure times, we can assume that the various crack initiations are independent of each other. This means, that the crack initation process becomes a PPP see \cite{Wahrsch} and the section above. Consequently, we only have to model the intensity measure $\rho$ as a function of the stress on $\overline{\Omega}$.

With $\Omega\subseteq \R{3}$ we associate the displacement field $u=u(\Omega)$, which is obtained as a solution of $ \eqref{Reliability:Eq:LinEl} $. Here we assume that the solution is smooth enough that the integrals in the following section are well defined. We will discuss in Section \ref{Reliability:Sec:Reg_Req} which regularity is needed in the context of LCF or ceramic failure.

The expected number of cracks that are located in $ A \subset \overline{\Omega} $ up to time $ s $, i.e. $ \rho([0,s] \times A) = \mathbb{E}[\gamma([0,s] \times A)] $ \cite{GottschSaadi2018}, should depend on the local stress state $ \sigma(u) $ (\footnote{Derivatives of $ \sigma(u) $ are neglected, although \cite{NotchSizeLCF2018,NotchFracMech2016} show that at least second order derivatives of $ u $ are of interest in this context.}) on $ A $.
 Furthermore $ \rho([0,s] \times A) $ has to be monotonically increasing in time as cracks do not vanish once they originated. This leads to the flowing structure 
\begin{align*}
\rho(\Omega, [0,s] \times A)&=\int_{A \cap \overset{\circ}{\Omega} } \int_{0}^{s} \varrho_{vol}(\tau,x,u(x),\sigma(u(x)))\,dx \,d\tau\\
&+\int_{A \cap \Gamma}\int_{0}^{s}\varrho_{sur}(\tau,x,u(x),\sigma(u(x)))\,d\tau\,dS
\end{align*} 
or, to be more general,
\begin{equation}\label{Eq:locFailure}
\begin{split}
\rho(\Omega, C)=&\int_{C\cap (\R{+}\times \Omega)}\varrho_{vol}(\tau,x,u(x),\sigma(u(x)))\,d\tau\,dx\\
+ &\int_{C\cap (\R{+}\times \Gamma)}\varrho_{sur}(\tau,x,u(x),\sigma(u(x)))\,d\tau\,dS, ~~~~~~~~~  C \in \mathcal{B}(\mathcal C).
\end{split}
\end{equation}
Here $dS$ denotes the surface measure on $\Gamma$ and $\varrho_{vol/sur}$ are some nonnegative, integrable functions that reflect the physical behavior of crack formation. We will discuss this topic in detail in Section \ref{Subsec:Loc Weibull Modell} below. The function $\varrho_{vol}$ represents volume driven failure mechanisms. For example e-g-creep or ceramic failure shall be mentioned here. Contrarily, $\varrho_{sur}$ models surface driven crack formation processes like LCF for metal devices. 
Strictly speaking, $ \rho $ depends only on the stress $ \sigma(u) $. To keep the model as flexible as possible we also introduce the displacement $ u $ as a possible control variable.

\begin{lem}\cite{GottschSchmitz}
	\label{lem:Prob} 
	Let $\gamma=\gamma(\Omega,.)$ be the PPP associated with (\ref{Eq:locFailure}) and $\mathcal{T}=\mathcal{T}(\gamma)$ the associated first failure time. Let $C_s=[0,s]\times \overline{\Omega}$. Then the distribution function $F_\mathcal{T}$ satisfies
	\begin{equation}
	\label{Eq:CDFFFT}
	F_\mathcal{T}(s)=P(\mathcal{T}\leq s)=1-e^ {-\rho(\Omega,C_s)},~~s\in\R{}. 
	\end{equation}
\end{lem}

\begin{proof} Survival beyond time $ s>0 $ means that there was no crack until time $ s $ and thus $ \gamma(C_s)=0 $.
According to equation \eqref{Eq:PPP} the probability to have $ n \in \N{}_{0} $ cracks in $ C_{s} $ is given by
\begin{align}\label{Eq:PPPprob_n_cracks}
P(\gamma(C_s)=n)=e^ {-\rho(\Omega,C_s)} \frac{\rho(C_{s})^n}{n!}.
\end{align} 
Thus, the survival probability beyond $ s $ is given by $P(\mathcal{T}> s)=P(\gamma(C_s)=0)=e^ {-\rho(\Omega,C_s)}. $ This directly implies
\[ F_{\mathcal{T}}(s) = P(\mathcal{T}\leq s) =1 -P(\mathcal{T}> s)=1-e^ {-\rho(\Omega,C_s)}. \]
\end{proof}
Moreover $  F_{\mathcal{T}}(s) = 1 - e^{-H_{\mathcal{T}}(\Omega,s)} $ with $ H_{\mathcal{T}}(\Omega,s) =\int_{0}^{s} h_{\mathcal{T}}(\tau) d\tau $, consider for example  \cite{EscobarMeeker,Wahrsch}, since $ h_{\mathcal{T}}(s) = 0 $ for any $ s<0 $.
This means that $H_{\mathcal{T}}(\Omega,s) = \rho(\Omega,C_{s})$ is the cumulative hazard rate of the random variable $ \mathcal{T} $ and that any component $ \Omega $ induces its own probability distribution and its own failure time $ \mathcal{T}(\Omega) $. Thus, the question is how to compare these distributions.

\subsubsection{Optimal reliability under cyclic loading as a shape optimization problem}

Let $\Omega^{\ast},\,\Omega\subseteq \R{3}$ be two different designs and
$\mathcal{T} = \mathcal{T}(\Omega)$ and $\mathcal{T}^{\ast}=\mathcal{T}(\Omega^{\ast})$ their first failure times associated to some $\varrho_{vol/sur}$. Then we can define the following: 

\begin{defn}[Notions of Reliability] \cite{BittGottsch}
	\label{def:Reliability} 
	\begin{itemize}
		\item[i)] \textit{Reliability at fixed time:} Let $ s\in\R{}_+ $ be fixed. The design $\Omega^{\ast}$ is said to be more or equally reliable then $ \Omega $ at time $s$, if $F_{\mathcal{T}^{\ast}}(s)\leq F_{\mathcal{T}}(s)$.
		\item[ii)] \textit{Reliability in first order stochastic dominance:} The design $\Omega^{\ast}$ is more or equally reliable than $\Omega$ in first stochastic order, if i) holds at any time $s\in \R{}_+$.
		\item[iii)] \textit{Reliability in terms of instantaneous hazard:} Suppose that $\mathcal{T}^{\ast}$ and $\mathcal{T}$ are continuous random variables. Then $\Omega^{\ast}$ is more reliable than $\Omega$ in terms of instantaneous hazard, if $h_{\mathcal{T}^{\ast}}(s)\leq h_{\mathcal{T}u}(s)$ holds for any $ s\geq 0$. 
	\end{itemize}  
\end{defn} 

The three notions can be interpreted in an ascending order: Reliability at a fixed time means, that a design $ \Omega^{\ast} $ is only better than $ \Omega $ at one time spot, whereas reliability in first order stochastic dominance means that $ \Omega^{\ast} $ is always most reliable. 
The concept of reliability in terms of instantaneous hazard even strengthens this notion. Since $F_\mathcal{T}(s)=1-e^{-\int_0^s h_\mathcal{T}(\tau) d\tau}$  the estimate $h_\mathcal{T}(s)\leq h_{\mathcal{T}^{\ast}}(s)~\forall s\in\R{}_+$ implies $F_\mathcal{T}(s)\leq F_{\mathcal{T}^{\ast}}(s)$ $\forall s\in\R{}_+$. 
Therefore, being most reliable at each instant in time implies being most reliable at any time.

Further explanations and other perspectives on reliability optimization problems are given for example in \cite{BGMGSS,EscobarMeeker,GBS_Ceramic2014,Stat}. 
In this work, enhancing the reliability of a device means optimizing the set $ \Omega $ w.r.t. one of the notions introduced in Definition \ref{def:Reliability}.

Now we are ready to define a reliability optimization problem according to each of the notions of reliability.

\begin{defn}[Optimal Reliability Problem]
	\label{def:OptReliability}
	Let ${\cal O}\subset \mathbb{P}(\R{3})$ be some set of \textit{admissible domains} (shapes) $\Omega $. 
	Then, $\Omega^*\in{\cal O}$ solves the problem of optimal reliability on $ \mathcal{O} $ according to  Definition \ref{def:Reliability} (i), (ii) or (iii) if it is more or equally reliable than any other design $\Omega\in{\cal O}$ regarding (i), (ii) or (iii), respectively.
\end{defn}

\noindent Let $ \varrho $ be some non negative, integrable function, see \eqref{Eq:locFailure}, and define 
${\cal F}_{vol/sur}(s,\cdot):=\int_0^s \varrho_{vol/sur}(\tau,\cdot)\, d\tau$ and
\begin{align}
{\cal J}_{vol}(s,\Omega,u,\sigma)&:=\int_{\Omega}{\cal F}_{vol}(s,x,u,\sigma(u))\,dx\, , \\
{\cal J}_{sur}(s,\Omega,u,\sigma)&:=\int_{\Gamma}{\cal F}_{sur}(s,x,u,\sigma(u))\,dS\, .
\end{align}  
For $ C=C_s $ we immediately obtain 
\begin{equation}\label{Eq:ObjFunct}
\begin{split}
\rho(\Omega, C_s) 
&= \rho(\Omega,[0,s] \times \overline{\Omega})  \\
&=\int_{\Omega} \int_{0}^{s} \varrho_{vol}(\tau,x,u,\sigma(u))\,dx \,d\tau +\int_{\Gamma}\int_{0}^{s}\varrho_{sur}(\tau,x,u,\sigma(u))\,d\tau\,dS \\
&= \int_{\Omega} \mathcal{F}_{vol}(s,x,u,\sigma(u)) \, dx +  \int_{\Gamma} \mathcal{F}_{sur}(s,x,u,\sigma(u)) \, dS  \\[1ex]
&= {\cal J}_{vol}(s,\Omega,u,\sigma)+{\cal J}_{sur}(s,\Omega,u,\sigma)\\[1ex]
&:= {\cal J}(s,\Omega,u,\sigma).
\end{split}
\end{equation}

This indicates the following:

\begin{lem} \cite{BittGottsch}
	\label{lem:OpRelSO}
	Let the crack initiation process $\gamma=\gamma(\Omega)$ for some $\Omega\in{\cal O}$ be a PPP with intensity measure (\ref{Eq:locFailure}). Let $ u=u(\Omega) $, $ \sigma=\sigma(u) $ be the displacement and the stress field associated with $ \Omega $. A shape $\Omega^*\in{\cal O}$ together with $ u^{\ast}=u(\Omega^{\ast}) $ and $ \sigma^{\ast} =\sigma(u^{\ast})$ solves the optimal reliability problem
	\begin{itemize}
		\item[i)]   at fixed time $s\in \R{}_+$ if and only if 
		\begin{equation}
		\label{Eq:SO}
		{\cal J}(s,\Omega^{\ast},u^{\ast},\sigma^{\ast})\leq {\cal J}(s,\Omega,u,\sigma)~~\forall\, \Omega\in{\cal O}.
		\end{equation}
		\item[ii)] in first order stochastic dominance  if and only if $\,\Omega^*$ solves \eqref{Eq:SO} for all $s\in\R{}_+$.   
		\item[iii)] in terms of instantaneous hazard, if an only if
		\begin{equation}
		\label{Eq:OptHaz}
		\frac{d}{ds}{\cal J}(s,\Omega^{\ast},u^{\ast},\sigma^{\ast})\leq \frac{d}{ds}{\cal J}(s,\Omega,u,\sigma)~~\forall s\in\R{}_+,~\Omega\in{\cal O}.
		\end{equation} 
	\end{itemize}
\end{lem}  
\vspace*{-1em}
\begin{proof}
i) \& ii) According to Lemma \ref{lem:Prob} and equation \eqref{Eq:ObjFunct} we immediately receive 
\begin{align*}
{\cal J}(s,\Omega^{\ast},u^{\ast},\sigma^{\ast})\leq {\cal J}(s,\Omega,u,\sigma)
& \Leftrightarrow 1-e^{-\mathcal{J}(s,\Omega^*,u^{\ast},\sigma^{\ast})} \leq 1-e^{-{\cal J}(s,\Omega,u,\sigma)} \\
&\Leftrightarrow F_{\mathcal{T}^{\ast}}(s) \leq F_{\mathcal{T}}(s). 
\end{align*}
Moreover $F_\mathcal{T}(s)=1-e^{-\int_0^s h_\mathcal{T}(\tau) d\tau}=1-e^{-{\cal J}(s,\Omega,u,\sigma)}$ or equivalently $ {\cal J}(s,\Omega,u,\sigma) $ $= \int_0^s h_\mathcal{T}(\tau) d\tau $. This implies  $h_\mathcal{T}(s)= \frac{d}{ds} {\cal J}(s,\Omega,u,\sigma)$.
\end{proof}

Therefore, $ {\cal J}(s,\Omega,u,\sigma) $ can be interpreted from different perspectives: on one hand it is a cumulative hazard function since $ {\cal J}(s,\Omega,u,\sigma) = \int_0^s h_\mathcal{T}(\tau) d\tau = H(\Omega,s) $. From another other point of view, solving the optimal reliability problem means finding an optimal set $ \Omega $ such that the expected number of cracks that are located in $ \overline{\Omega} $ at time $ s $, i.e. $\mathbb{E}[\gamma(C_s)]   = \rho(\Omega,C_s) ={\cal J}(s,\Omega,u,\sigma)$, is minimized.

\subsubsection*{The local Weibull model for LCF} \label{Subsec:Loc Weibull Modell}

Clearly, the choice of $ \varrho_{vol/sur} $ has a huge impact on the notion of optimal reliability that is of interest. In reliability statistics the Weibull distribution turned out to be suitable for many applications \cite{Weib1939,EscobarMeeker} and is often used in engineering practice. 

\begin{defn}[Local Weibull Model] \label{ReliabilityFunctionals:Def:Loc_Weibull}\cite{GottschSchmitz}
	Let $m>0$ be a Weibull shape parameter and 
	\begin{equation}
	\label{Eq:locWei}
	\varrho_{vol/sur}(s,\cdot)=\frac{m}{N_{vol/sur}(\cdot)}\left(\frac{s}{N_{vol/sur}(\cdot)}\right)^{m-1}
	\end{equation}
	for functions $N_{vol/sur}(\cdot)=N_{vol/sur}(x,u,\sigma(u))$ with values in $[0,\infty]$. The associated crack initiation processes are called \textit{local Weibull models}. We use the convention $\frac{1}{\infty}=0$.
\end{defn} 

In the context of optimal reliability for mechanic components under cyclic loading the number $ N_{vol/sur} $ can be interpreted as the number of load cycles passed until crack formation. In this sense $ N_{vol}= \infty$ means that the failure mechanism is surface driven or vice versa. An example for the systematic derivation of such a functional $ N_{sur} $ will be presented in the next section.

\begin{lem}
	\label{lem:Wei}
	Let $\gamma^{w}=\gamma^{w}(\Omega)$ be the PPP from a local Weibull model. Then the first failure time  $\mathcal{T}^w=\mathcal{T}(\gamma^w)$ is Weibull distributed with parameters $\eta=\eta(\Omega)>0$ and $ m>1 $ given by
	\begin{align}
	\label{Eq:scaleWei}
	\eta&=\left[\int_\Omega \left(\frac{1}{N_{vol}(x,u,\sigma(u))}\right)^ m\,dx +\int_{\Gamma}\left(\frac{1}{N_{sur}(x,u,\sigma(u))}\right)^m \, dS\right]^{-\frac{1}{m}}.
	\end{align}
\end{lem} 
\begin{proof}
	We insert (\ref{Eq:locWei}) into \eqref{Eq:ObjFunct}. Then 
	\begin{align*}
	F_{\mathcal{T}}(s) &= 1-\exp\left(\int_{\Omega} \int_{0}^{s} \frac{m}{N_{vol}} \left(\frac{\tau}{N_{vol}}\right)^{m-1}\,d\tau\,dx +\int_{\Gamma}\int_{0}^{s}\frac{m}{N_{sur}} \left(\frac{\tau}{N_{sur}}\right)^{m-1}\,d\tau\,dS\right)\\
	&= 1-\exp\left(\int_{0}^{s} m\tau^{m-1} \, d\tau\, \left[ \int_{\Omega}  \left(\frac{1}{N_{vol}}\right)^{m}\,dx +\int_{\Gamma} \left(\frac{1}{N_{sur}}\right)^{m}\,dS \right] \right)\, .
	\end{align*} 
	 With $\eta$ given by (\ref{Eq:scaleWei}), we thus obtain $F_\mathcal{T}(s)=1-e^{-s^m\eta^{-m}} = 1-e^{-\left(\frac{s}{\eta}\right)^{m}}.$ 
\end{proof}

In the context of a local Weibull model the three different notions of the optimal reliability problem turn out to be equivalent:

\begin{prop} \cite{BittGottsch}
	\label{prop:OpRelFOSD}
	Let ${\cal O}$ be a set of admissible shapes $\gamma=\gamma(\Omega)$ be the crack initiation process associated with a local Weibull model and $m> 0$. Then, 
	\begin{itemize}
		\item[(i)] $\Omega^*\in{\cal O}$ is a solution to the optimal reliability problem \ref{def:OptReliability} (iii) if and only if it solves the optimal reliability problem \ref{def:OptReliability} (i) at a given time $t$.
		\item[(ii)]    $\Omega^*\in{\cal O}$ is a solution to the optimal reliability problem \ref{def:OptReliability} (i) if and only if 
		\begin{equation}
		\label{Eq:SOWei}
		J(\Omega^*,u^{\ast},\sigma^{\ast})\leq J(\Omega,u,\sigma)~~\forall \Omega\in{\cal O}
		\end{equation}
		where
		\begin{align*}
		\label{Eq:scaleWei2}
		 J(\Omega,u,\sigma)&=\int_\Omega \left(\frac{1}{N_{vol}(x,u,\sigma(u))}\right)^ m\,dx +\int_{\Gamma}\left(\frac{1}{N_{sur}(x,u,\sigma(u))}\right)^m \, dS = \eta(\Omega)^{-m}.
		\end{align*}
	\end{itemize} 
\end{prop}
\begin{proof}
	(i) Let $t\in\R{}_+$ be fixed. If $\Omega^ *$ solves the optimal reliability problem \ref{def:OptReliability} (iii) with respect to that time, we have $F_{\mathcal{T}(\Omega^*)}(t)\leq F_{\mathcal{T}(\Omega)}(t)$ $\forall\, \Omega\in{\cal O}$ and thus 
	$$
	1-e^{-t^m \eta(\Omega^*)^{-m}}\leq 1-e^{-t^m \eta(\Omega)^{-m}} ~\Leftrightarrow~ \eta(\Omega^*)\geq \eta(\Omega)~~\forall \Omega\in{\cal O}.$$
	But then the hazard rates fulfill for $m\geq 1$
	\begin{equation}
	h_{\mathcal{T}(\Omega^*)}(t)=\frac{m}{\eta(\Omega^*)}\left(\frac{t}{\eta(\Omega^*)}\right)^{m-1}\leq \frac{m}{\eta(\Omega)}\left(\frac{t}{\eta(\Omega)}\right)^{m-1}=h_{\mathcal{T}(\Omega)}(t)~,~~\forall t\in\R{}_+.
	\end{equation}
	(ii) Combine (i) for $t=1$, Lemma \ref{lem:Wei}, equation \eqref{Eq:locWei} and  Lemma \ref{lem:OpRelSO} (i).
\end{proof}

\subsubsection{Optimal reliability problem for metal components under cyclic loading}
Now we combine all the previous modeling steps to define the  \textit{optimal reliability problem for LCF}. We have already chosen $ \varrho_{vol/sur}  $  to be a local Weibull model, according to Definition \ref{ReliabilityFunctionals:Def:Loc_Weibull}. But we still have to define $ m $ and $ N_{vol/sur} $ such that they reflect the physical behavior of crack formation.

We set $ N_{vol} = \infty $ since LCF is a surface driven failure mechanism and define $ N_{sur}(x,u,\sigma):=N_{det}(\sigma(u(x))),\, x \in \Gamma $. Moreover $ 1.5 \leq m \leq 4 $, \cite{GottschSaadi2018}, are typical values.
\begin{defn}
The \textit{optimal reliability problem for LCF} is defined by the problem
\begin{align}
&\mathrm{ Solve } && \min_{\Omega \in \mathcal{O}} ~~~ J^{\mathrm{lcf}}(\Omega,u,\sigma(u))=\int_\Gamma \left(\frac{1}{N_{det}(\sigma(u(x)))}\right)^ m\,dS\\
& && s.t. ~~~~~ u \text{ solves } \eqref{Reliability:Eq:LinEl}\, .
\end{align}	
\end{defn} 

This problem is an example for a so called shape optimization problem, defined in Section \ref{Transf_shape_opt:Sec: Admissible Shapes} Definition \ref{Transf_shape_opt:Defn:shape_opt_probl}.
It already has been discussed in view of existence of optimal shapes in \cite{GottschSchmitz,BittGottsch}. For computational investigations we refer to \cite{DissSchmitz2014,GottschSaadi2018}

\subsubsection{A bend rod under simulated cyclic loading}\label{Reliability:Sec:Simulation}

For a numerical visualization of the behavior of this functional, we consider an example where one has an intuitive idea where the component should break. To receive comparable values to those calculated in \cite{GottschSaadi2018}, we chose a bend rod which is also $ 6 $ mm long, bend up to $ 3 $ mm and has a diameter of $ 1 $ mm. However we used a slightly different geometry (geometry $ \Omega_1 $), another type of mesh and Lagrange elements of second order (\textsf{CG} 2). The mesh consists of $ 6418 $ vertices and $ 25\,438 $ cells. 

\begin{multicols}{3}
	\begin{center}
		\includegraphics[scale=.084]{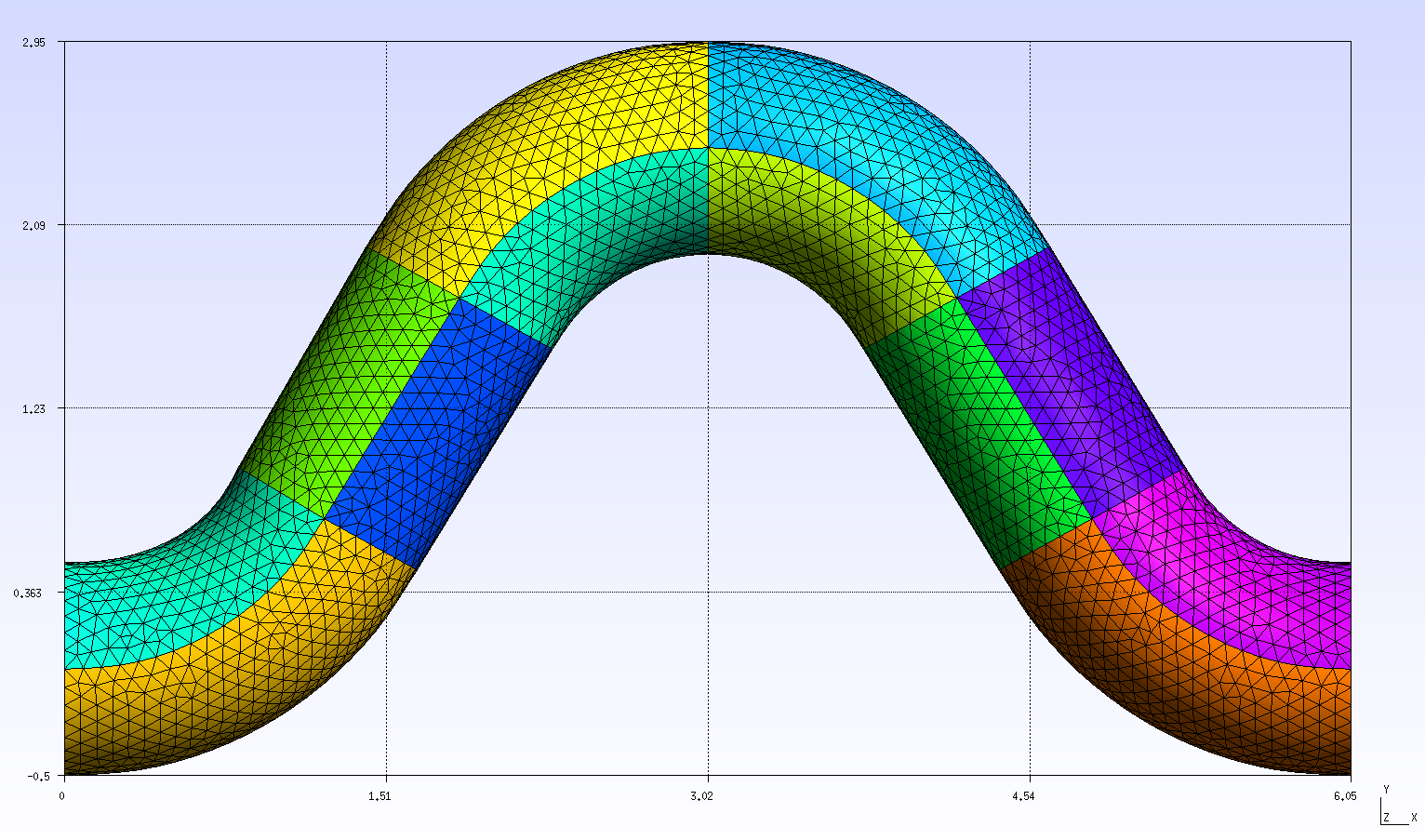}\\
		\captionof*{figure}{\footnotesize Geometry $ \Omega_0 $}\label{fig:mesh0}
	\end{center}
	\begin{center}
	\includegraphics[scale=.084]{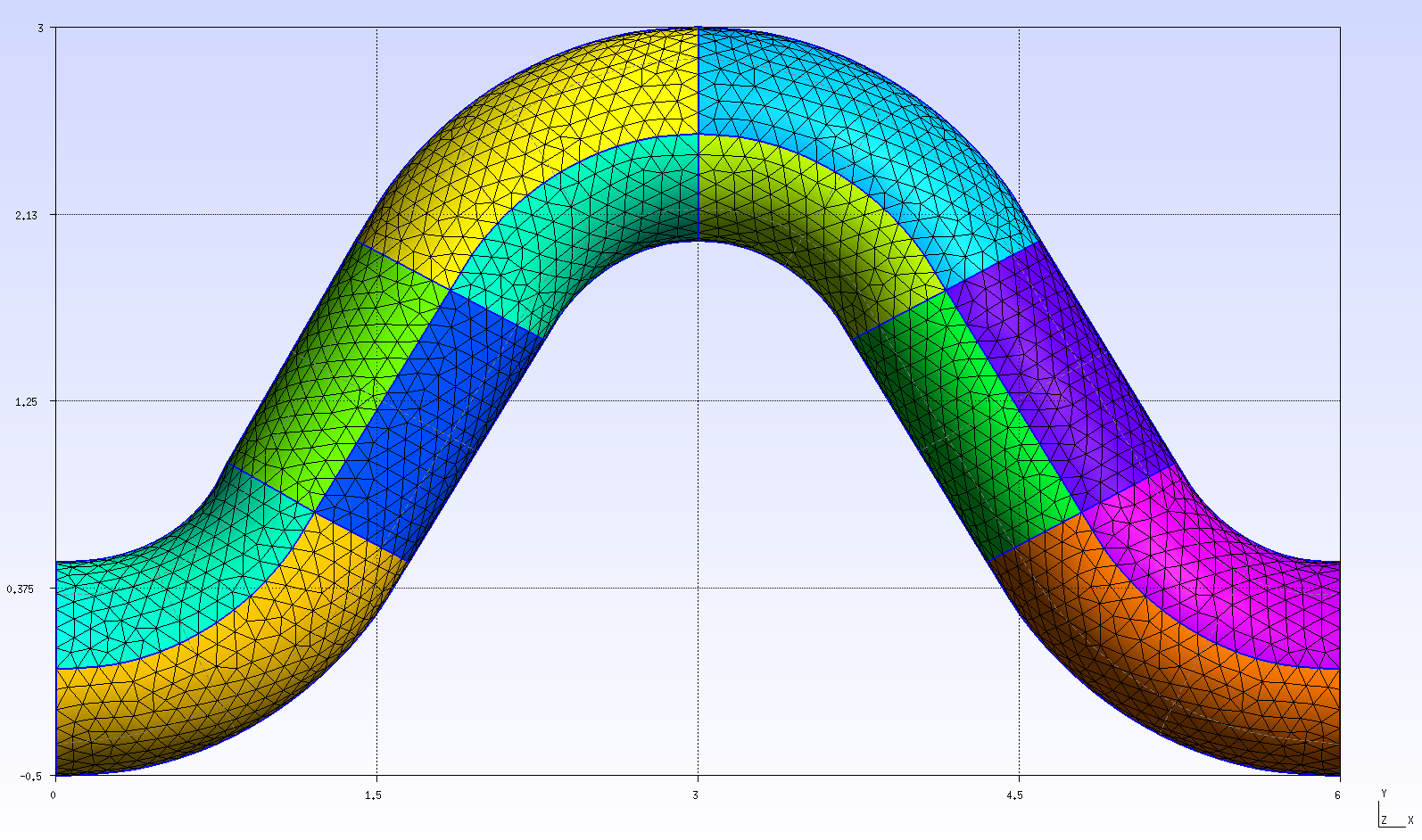}\\
	\captionof*{figure}{\footnotesize Geormetry $ \Omega_1 $}\label{fig:mesh1}
\end{center}
\begin{center}
	\includegraphics[scale=.084]{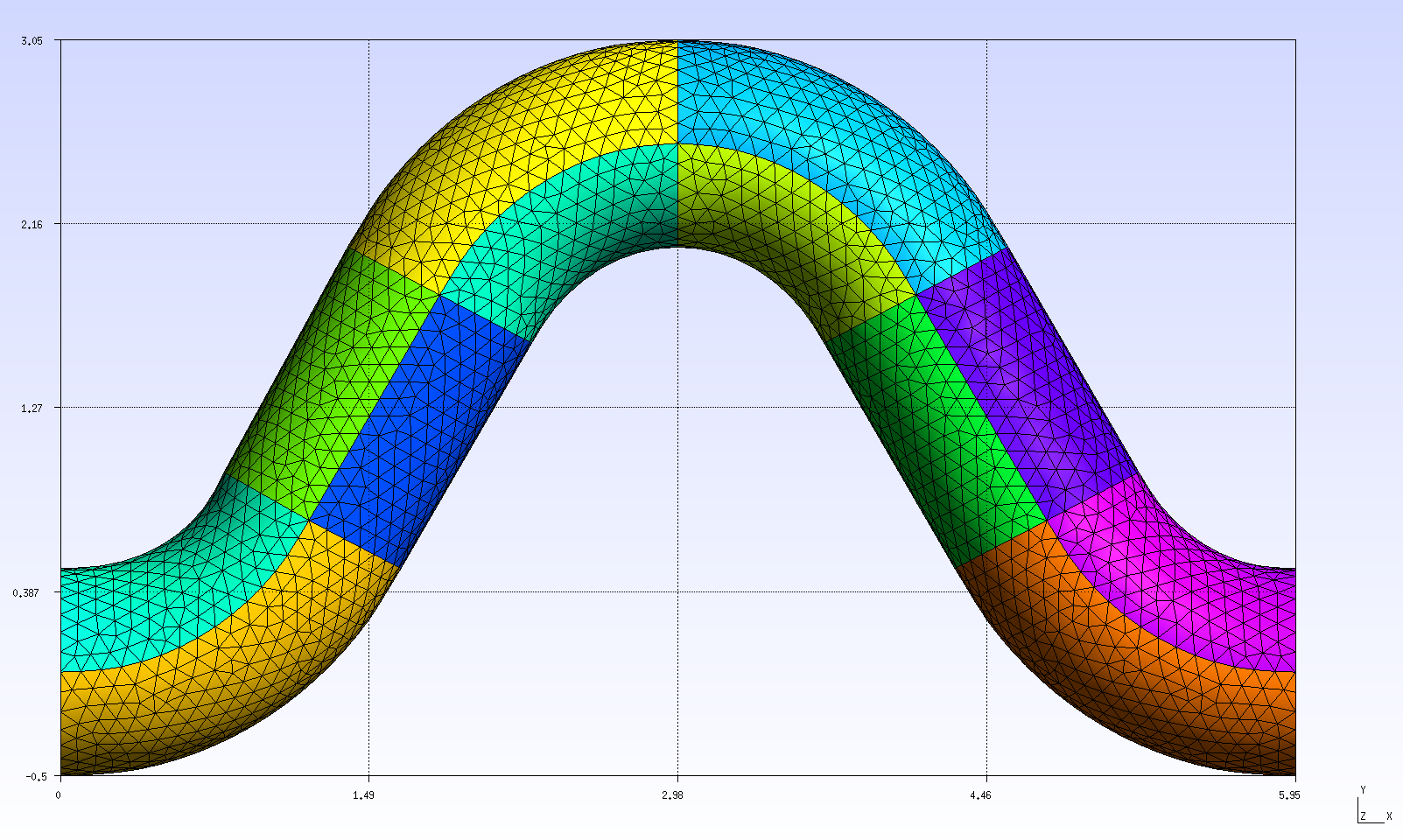}\\
	\captionof*{figure}{\footnotesize Geometry $ \Omega_2 $}\label{fig:mesh2}
\end{center}
\end{multicols}
\vspace*{-1.5em}
The model and the grid were created with the open source mesh generation software \textsf{Gmsh} and the computations were executed with \textsf{Python} and the \textsf{FEniCS} package\footnote{Software and documentation can be found on \url{https://fenicsproject.org/}, \url{http://gmsh.info/}}. The second geometry $ \Omega_2 $ is only slightly different: it is $ 5.95 $ mm long, bend up to $ 3.05 $ mm and has the same diameter and $ \Omega_0 $ is $ 2.95 $ mm high and $ 6.05 $ mm long.
We assume that the rod is fixed at the right end. This corresponds to a Dirichlet $ u=0 $ condition. On the left end we apply a tensile load of $ 19.6 \approx 1.0695 \times 18.4 ~[\mathrm{N}/\frac{\pi}{4}\mathrm{mm}^2] $ horizontal direction. We assume that the arc is made of an aluminium alloy with a specific density $ 2650 ~ [\text{kg}/\text{m}^3] $ and neglect air pressure and gravity. The Lamé coefficients for the material $  \lambda = \frac{E}{(1+\nu)(1-2\nu)} = 40.385 ~[\text{MPa}] $ and $ \mu = \frac{E}{2(1+\nu)} = 26.923~ [\text{MPa}]$ are calculated from the Young's modulus ($E = 70 000 ~[\text{MPa}]$) and
Poisson’s ratio $ \nu = 0.3 $. The Ramberg-Osgood parameters $ K = 443.9  ~[\text{MPa}] $ and $ \hat{n} = 0.064 $ are reported in \cite{LCFalu2004}.
The Coffin-Manson-Basquin constants are set to $ \sigma_f' = 2536 ~[\text{MPa} \times \text{mm}^{-\frac{2b}{m}}]$, $ \varepsilon_f'=0.254 ~[\text{mm}^{-\frac{2c}{m}}] $, $ b=-0.07 $, $ c=-0.593 $ and $ m=2 $. 
Note, that this are not the classical CMB-paramaters. The classical ones have to be adapted to the probabilistic LCF model as reported e.g. in \cite{DissSchmitz2014,ProbLCF2013,GottschSchmitz} or \cite{GottschSaadi2018}.

\begin{multicols}{2}
	\begin{center}
	\includegraphics[scale=0.25]{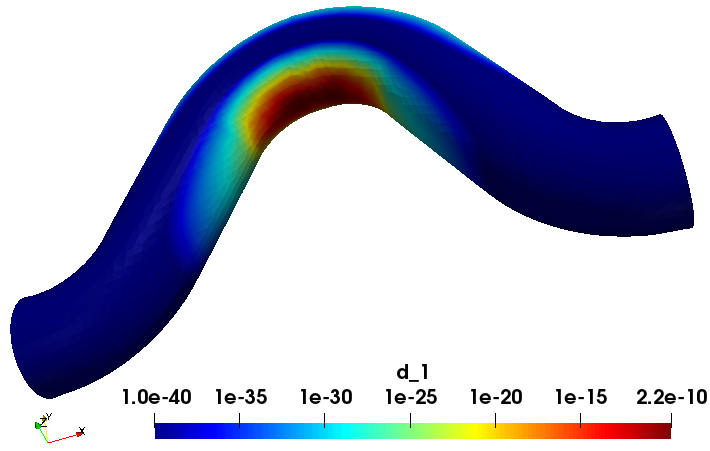}\\
	\captionof*{figure}{\footnotesize \textbf{Figure:} Plot of the crack initiation density $ d_1 $ in logarithmic scale on the surface of $ \Omega_1 $.}\label{fig:densityGeo1}
\end{center}
\begin{center}
	\includegraphics[scale=0.44]{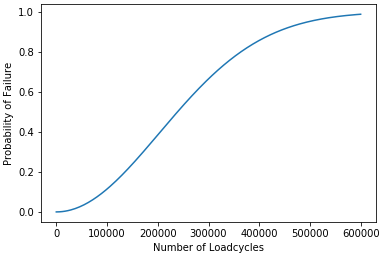} \\
	\captionof*{figure}{\footnotesize \textbf{Figure:} Probability distribution of $ \Omega_1 $.}\label{fig:probGeo1}
\end{center}
\end{multicols}

In the numerical experiments the rod behaves in the expected manner. The crack initiation density $d_1(x) =\frac{1}{N_{det}(\sigma(u_{\Omega_1}(x)))^m} ,\, x \in \partial\Omega_1 $ reaches its maximum at locally at the bottom side of the middle bending and is almost zero at most of the surface points. Moreover, the numerical simulation indicates that the LCF functional is very sensitive w.r.t. changes in the geometry. This can be observed on the size of $ \eta(\Omega) $:
The value $ \eta(\Omega) = \mathcal{J}(\Omega,u,\sigma(u))$ is the $ 0.63 $-quantile $ q_{0.63} $ of the Weibull distribution $ \mathrm{Wei}(\eta,m) $, i.e. $q_{p} := \inf\{s \in \R{+} \,\vert \, P(\mathcal{T}\leq s) \geq p\} $
the first time such that the probability of failure  is larger or equal than $ p \in[0,1] $. 
The figure on the right hand side shows that this probability is reached at about $ 300\,000 $ load cycles. Consider also the table below.
\begin{table}[h]
\begin{center}
	\begin{tabular}{|c |  l| l| l| l|} \hline
	& $ \max_{x \in \Omega} \Norm{u(x)}{} $& $ J(\Omega,u,\sigma) $ [cycles]$ ^{-m} $ & $ \eta(\Omega) $ [cycles] & $ q_{0.05}  $ [cycles] \\ \hline 
	$ \Omega_0 $ & $ \approx 0.1805 $ mm	& $ \approx 1.446456  \times 10^{-12} $ &$  \approx 780\,608 $& $ \approx 188\,311 $  \\\hline 
	$ \Omega_1 $ & $ \approx 0.1847 $ mm	& $ \approx 1.2138  \times 10^{-11} $ &$  \approx 287\,024 $& $ \approx 65\,005 $  \\\hline 
	$ \Omega_2 $	& $ \approx 0.1889 $ mm & $ \approx 4.35948 \times 10^{-11} $ & $ \approx 151\,454  $ & $ \approx 34\,140 $\\ \hline 
\end{tabular}
\caption{\footnotesize
		Comparison of the three geometries based on the value of the objective $ J $, the $ 0.63 $-quantile $ \eta(\Omega) $ and the $ 0.05 $-quantile $ q_{0.05} $.}
\end{center}
\end{table}

\subsection{Ceramic components under tensile loading}

We now investigate a mathematical model for sudden failure of components that are made of a brittle material like ceramic and follow \citep{GBS_Ceramic2014} closely. Therein, an objective functional is derived which measures the reliability of ceramic components under tensile load and examined in view of existence of optimal shapes under a uniform cone constraint \ref{PDE_Systems:Def:Cone}. Moreover, the article \citep{NumShapeCer2017} presents first numerical results on shape optimization to decrease the failure probability of ceramic components. 

The construction of the functional is again based on solutions to linear elasticity equation \eqref{Reliability:Eq:LinEl}, linear fracture mechanics and Weibull’s analysis of the stochastic nature of the ultimate strength of brittle material \citep{Weib1939}. 
General information regarding the material behavior and the properties of ceramics can be found for example in \citep{StatCer1974} or  \cite{Ceramics2012}.  
Note that in this application the probability of failure follows a Weibull distribution over the strength of external loads respectively the tension $ \sigma $, and not over the number of load cycles like it it the case in the LCF failure-probability model.  

A material like ceramic, contains porosities that are modeled in shape of a penny lying in a two dimensional plane with normal direction $ \mathfrak{n}  \in \mathbb{S}^{2}$. At these porosities cracks initiate. 
In linear fracture mechanics the three dimensional stress field $ \sigma $ measured at a point $ x $ close to the tip of the crack is modeled in the form
\[ \sigma(x) = \frac{1}{\sqrt{2 \pi r}} \left(K_{I}\tilde{\sigma}^{I} + K_{II}\tilde{\sigma}^{II} +K_{III}\tilde{\sigma}^{III}\right) + \text{regular terms}, \]
see e.g. \cite[Ch. 4]{Ceramics2012} where a detailed derivation of $ K_{I},\, K_{II},\, K_{III} $ is given.

\begin{multicols}{2}	
\begin{center}
	\includegraphics[scale=.29]{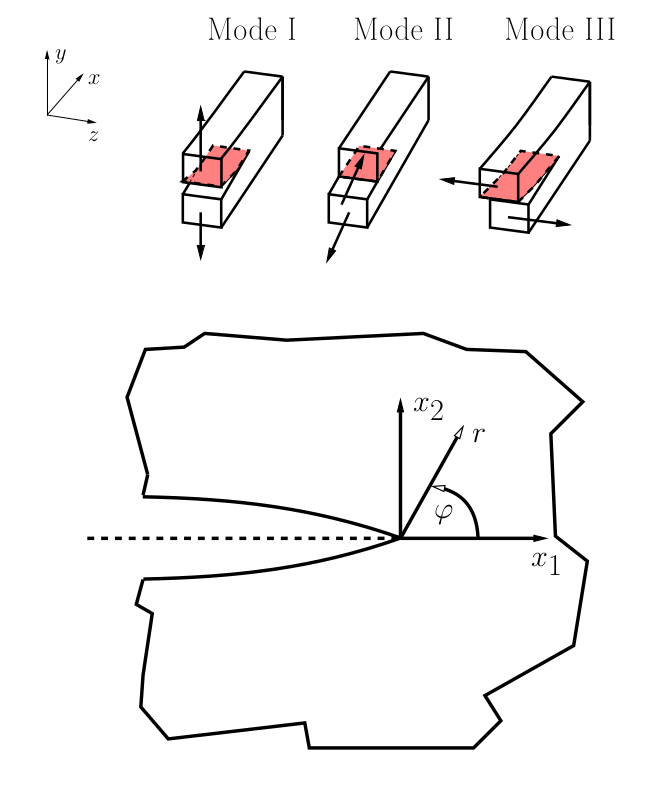}\\
	\captionof*{figure}{\footnotesize  Scatch of the opening modi (above) and the beginining crack with opening angle $ \varphi $ \\
	and cracklength $ r $ (below). Picture taken from \cite{GBS_Ceramic2014}. }\label{fig:ceramics}
\end{center}

The number $ r $ is the distance between the tip of the crack and $ \varphi $ is the angle between the crack plane and the point $ x $. The stress intensity factors $ K_{I},\, K_{II} $ and $ K_{III} $ depend on the mode and on the amount of loading.

We consider only loads in normal direction $ \mathfrak{n} $ of the crack plane since \cite{HegerZuvCer1993} indicates that shear stresses are negligible in this context.
\end{multicols}
These loads correspond to the first mode and thus $ K_{II},\, K_{III}  $ and also the "regular terms" become $ 0 $. The first stress intensity factor is obtained as 
 \[ K_{I}:=\frac{2}{\pi} \sigma_{n} \sqrt{\pi r} \]
 where $ r>0 $ is the radius of the penny shaped crack.

In the case of compressive loads, i.e. $ \sigma_{n}:= \mathfrak{n}^{\top} \sigma \mathfrak{n}\leq 0 $, no failure will occur even if $ r $ is large. Thus we only have to consider the case when $ K_{I} $ becomes larger than a critical value $ K_{I_c}>0 $ (between $ 3 $ and $ 16 $ [MPa $ \sqrt{\text{m}} $], see \cite{GBS_Ceramic2014}) and  $ \sigma_n>0 $. We therefore consider $ \sigma_n^{+}:=\max\{0,\sigma_n\} $ instead. 

In the case of brittle material we take any point in the device $ \Omega \subset \R{3} $ as a possible place for crack initiation into account. This crack is assumed to lie in two dimensional plane which is characterized by the normal $ \mathfrak{n} \in \mathbb{S}^2 $. Further the initial (penny shaped) flaw which has a diameter $ r \in \R{+}=(0,\infty) $. We thus define $$ \mathcal{M}:= \Omega \times \mathbb{S}^{2} \times \R{+} $$ to be the crack configuration space and endow it with the sigma algebra $ \mathcal{B}(\mathcal{M}) $. 

Since we can neither determine where the initial flaw lies in $ \Omega $, nor in which direction the normal $ \mathfrak{n} $ is directed, nor which radius the flaw has, we assume that these quantities are randomly distributed.
In this context, it is natural to assume crack initiation points are uniformly distributed in the volume since the material has the same structure everywhere. Thus we choose the Lebesgue measure  $  \mathds{1}_{\Omega} dx $  on $ \R{3} $ where $ \mathds{1}_{\Omega} $ is the indicator function of $ \Omega $. Also, it is obvious to assume that the directions of the crack plains are distributed according to the scaled surface measure $ \frac{1}{4\pi} d A_{\mathbb{S}^2} $ on $ \mathbb{S}^2 $. The only choice to be made concerns the measure  $ \nu = \nu(\Omega) $, that counts the random number of cracks in a given volume $ \Omega $. Therefore, we define
$$ \rho(\Omega) =\rho(\Omega):= \mathds{1}_{\Omega}\, dx \otimes \frac{1}{4\pi} d A_{\mathbb{S}^2} \otimes \nu(\Omega),$$
see Definition 3.1. in \cite{GBS_Ceramic2014}.

In the following, we construct the measure $ \nu = \nu(\Omega) $  on $ (\R{+},\mathcal{B}(\R{+})) $ such that it counts the number of cracks with radius $ (0,\infty) $ in the volume $ \Omega $ and such that it reflects the physical behavior of brittle material. We assume that $ \nu([c,\infty)) < \infty\, \forall c >0 $ to assure that only finitely many cracks can be contained in $ \Omega $. 


Let $ u =u(\Omega) \in H^1(\Omega,\R{3}) $ be the displacement field calculated as a solution of \eqref{Reliability:Eq:LinEl} and $ \sigma=\sigma(u) $ the associated stress field. As Lemma 3.1. in \cite{GBS_Ceramic2014} shows $$ A_{c}= A_{c}(\Omega,\sigma(u)):=\{(x,\mathfrak{n},r) \in \mathcal{M} : K_{I}(r,\sigma_n^{+}) = \frac{2}{\pi} \sigma_{n}^{+} \sqrt{\pi r} >K_{I_c} \}$$
is measurable and thus we can associate a Poisson point process $ \gamma=\gamma(A_c) $ with $ \rho $ on $ \mathcal{M} $, using the relation \ref{Eq:PPP}. Hence, the survival probability is given by
\[ P(\gamma(A_c) = 0) = e^{-\rho(A_c)}.\]
In this special situation, we can derive an explicit representation of $ \rho(A_c) $ as follows: 
\begin{align*}
 \rho(A_c) 
& = \frac{1}{4 \pi}\int_{\Omega}\int_{\mathbb{S}^2} \int_{0}^{\infty} \mathds{1}_{\left\lbrace \frac{2}{\pi} \sigma_{n}^{+}\sqrt{\pi r}>K_{I_c}\right\rbrace} \, d\nu(r)\, dS_{\mathbb{S}^{2}} \, dx \\
&= \frac{1}{4 \pi}\int_{\Omega}\int_{\mathbb{S}^2} \int_{0}^{\infty} \mathds{1}_{\left\lbrace r > \frac{\pi}{4} \left(\frac{K_{I_c}}{\sigma_{n}^{+}}\right)^2 \right\rbrace} \, d\nu(r)\, dS_{\mathbb{S}^{2}} \, dx\\
&= \frac{1}{4 \pi}\int_{\Omega}\int_{\mathbb{S}^2} \nu\left( r > \frac{\pi}{4} \left(\frac{K_{I_c}}{\sigma_{n}^{+}}\right)^2 \right) dS_{\mathbb{S}^{2}} \, dx.
\end{align*}
If we assume that $ \nu(r) = c  r^{-\beta} $ for some $ \beta > 1 $ we obtain 
\begin{align*}
\nu\left( r > \frac{\pi}{4} \left(\frac{K_{I_c}}{\sigma_{n}^{+}}\right)^2 \right) 
& = \int_{\frac{\pi}{4} \left(\frac{K_{I_c}}{\sigma_{n}^{+}}\right)^2 } ~~~ca^{-\beta} \, da 
= (\beta-1)\left(\frac{\pi}{4} K_{I_c}\right)^{2(1-\beta)} \left(\sigma_{n}^{+}\right)^{2(\beta-1)}.
\end{align*}
Finally, defining $ m:=2(\beta-1)>0 $ and  $ \sigma_{0}:=  \frac{\pi \, K_{I_c}}{4 (\beta-1)^{1/m}}  >0 $ leads to an expected number of cracks in the volume $ \Omega $ of
\begin{align*}
\mathbb{E}[\gamma(A_c)]  =\rho(A_c)  = \frac{1}{4 \pi} \int_{\Omega}\int_{S^2} \left(\frac{\sigma_n^{+}}{\sigma_0}\right)^{m} dS_{\mathbb{S}^{2}} \, dx.
\end{align*}

Again we can define a shape optimization problem, but this time for ceramic components under tensile load \cite{GBS_Ceramic2014}.

\begin{defn}
	The \textit{optimal reliability problem for ceramic components under tensile loading} is defined by 
	\begin{align}
	& \mathrm{ Solve } &&\min_{\Omega \in \mathcal{O}}~~~ J^{\mathrm{cer}}(\Omega,u,\sigma(u)):= \frac{1}{4 \pi} \int_{\Omega}\int_{\mathbb{S}^2} \left(\frac{\sigma(u(x))_n^{+}}{\sigma_0}\right)^{m} dS_{\mathbb{S}^{2}} \, dx\\
	& && s.t. ~~~~~ u \text{ solves } \eqref{Reliability:Eq:LinEl}\, .
	\end{align}	
\end{defn} 

\noindent For a computational investigation of this shape optimization problem we refer to \cite{NumShapeCer2017}. 

\subsection{Reliability functionals and regularity of PDE solutions} \label{Reliability:Sec:Reg_Req}

Now we investigate briefly how regular the solution $ u $ has to be such that the functionals are defined: \pagebreak

In case of the ceramic functional it is obvious that the solution $ u $ of \eqref{Reliability:Eq:LinEl} must at least be an element of $ W^{1,m}(\Omega,\R{3}) $ such that a derivative can be taken. A state-of-the-art ceramic has a Weibull module of $ m \geq 10 $ \cite{CerWeiScale} but typically $ m$ satisfies $  5 \leq m \leq 20 $, see \cite{Weib1939}. This means that the solution has to be an element of $ C^{0,\phi}(\Omega,\R{3}) $ for  $1\geq \phi \geq 1-\frac{3}{m} $ or higher. In the worst case of $ m = 20 $ the solution $ u $ must be at least an element of $ C^{0,17/20}(\Omega,\R{3}) $  

In case of the LCF reliability functional the situation is less clear. Since $ SD $ and $ CMB $ can not be inverted manually, we concentrate on the leading terms:\\
First of all. the Ramberg-Osgood equation \eqref{Eq:RO} and the CMB relation are satisfied and thus (if we neglect the linear term and the high cycle fatigue term $ \epsilon_{f}(2N_{det})^{c} $) we obtain $ \epsilon^{el-pl} \simeq (\sigma^{el-pl})^{1/\hat{n}} $ and also $ \epsilon^{el-pl} \simeq N_{det}^{b} $. Furthermore, the Neuber relation implies \[ (\sigma_{v})^{\frac{2}{1+\nicefrac{1}{\hat{n}}}} \simeq \sigma^{el-pl},  \]
if we only consider the leading term $ \sigma^{el-pl}(\frac{\sigma^{el-pl}}{K})^{1/\hat{n}} $ and ignore the quadratic term. This implies \[ N_{det} \simeq (\sigma^{el-pl})^{\frac{1}{b\hat{n}}} \Leftrightarrow \left(\frac{1}{N_{det}}\right)^{m} \simeq (\sigma^{el-pl})^{-\frac{m}{b\hat{n}}} =\sigma_v^{-\frac{2m}{b\hat{n}(1+1/\hat{n})}} =\sigma_v^{-\frac{2m}{b(\hat{n}+1)}} . \]
Moreover, $ \sigma_v \leq C = \sqrt{\sigma(u):\sigma(u)}$. Thus one obtains
$ \left(\frac{1}{N_{det}}\right)^{m} \leq C\, \tr(\sigma(u)^2)^{-\frac{m}{b(1+\hat{n})}} $
on the boundary $ \Gamma $ of $ \Omega $. We now insert the values for $ b \approx -0.07 $ and $ \hat{n} = 0.064 $ and end up with the exponent $ -\frac{m}{b(1+\hat{n})} = \frac{m}{0.07\cdot(1+0.064)} \approx m \cdot 13.43 $. This means that at least $ u \in W^{1,m\cdot14}(\Gamma,\R{3}) $ has to be claimed. Due to the trace we end up with $ u \in W^{2,m\cdot14}(\Omega,\R{3}) \hookrightarrow C^{1,\phi}(\overline{\Omega},\R{3}) $ for $ 0 \leq \phi\leq 1-\frac{3}{14m} \approx 1- \frac{1}{4.7 m}  $. Depending on the material, the module $ m $ lies between $ 1.5 $ and $ 4 $, what implies $ \frac{1}{7.05} \geq \frac{1}{4.7m} \geq \frac{1}{18.8}$. Thus, in the "worst" case $ \phi \geq 0.946 $ will be sufficient, what means that $ u $ has to be "nearly" two times continuously differentiable.

In both cases it is unavoidable to apply strong solution theory, at least in Sobolev spaces. Generally, even a regularity of $ W^{1,m},\, m>6 $ can not be reached with weak solution theory. Certainly a regularity of $ H^2 $ (see  \cite{Necas1967} and also Chapter \ref{PDE_Systems} and Section \ref{Sec:Linear_Elasticity:Sec:Regularity_Weak_Solutions}) can be reached with weak PDE theory, but since  $ H^{2}(\Omega) = W^{2,2}(\Omega)  \hookrightarrow W^{1,6}(\Omega) \nsupseteq W^{1,m}(\Omega),\, m>6$, solutions in $ H^2(\Omega) $ are principally not enough. 
If we consider the LCF functional the situation becomes even worse and $ W^{2,p} $ solutions with very high values for $ p $ or solutions in $ C^{1,\phi} $ with $ \phi $ close to $ 1 $ become mandatory. \\
New approaches actually consider failure models that involve notch support factors and thus second order derivatives, see \cite{NotchFracMech2016,HertelVormwNotch2012,NotchSizeLCF2018}.
Then solutions in $ W^{3,p}(\Omega,\R{3}) \hookrightarrow C^{2,\phi}(\Omega,\R{3})$ with $ p \gg 3 $ are needed. Thus it is obligatory to pass over to strong PDE solution theory and classical function spaces. This theory is provided in the next chapter.


\chapter[Systems of Elliptic Partial Differential Equations]{Systems of Elliptic Partial Differential Equations and Linearized Elasticity}
\label{PDE_Systems}

In the last paragraph of Chapter \ref{Reliability} we illustrated that, depending on the functional under consideration, solutions to the linear elasticity equation in classical function spaces, or at least in higher order Sobolev spaces, are obligatory. Thus we provide the necessary theory here.

In the first section of this chapter we introduce the notation and the concepts that are necessary to treat linear elliptic systems of partial differential equations. These concepts were mainly developed by S. Agmon, A. Douglis, L. Nirenberg \cite{Agm64},\cite{Necas1967} and also G. Geymonat \cite{Geymonat}. While \cite{Agm64} presents only pure regularity results, \cite{Geymonat} provides index theorems for the underlying differential operator and thus the foundation of regularity theory for elliptic PDEs.

Detailed information regarding the theory of scalar partial differential equations can also be found in \citep{Evans,GilbTrud,WLGleichung}, and \cite{McLean2000}  discusses also PDE systems. The paper \citep{Agm59} is especially concerned with Schauder estimates for scalar linear elliptic PDE of second order.  Theory of Sobolev spaces needed for the mathematical analysis of partial differential equations is provided in \cite{AdamsFournier}.

The following Sections \ref{Sec:Linear_Elasticity:Sec:Modeling Elastic Deformations} - \ref{Sec:Linear_Elasticity:Sec:Classical_solutions_and_Schauder} treat the existence and regularity of weak and strong solutions to the so called \textit{disjoint displacement traction problem of linear elasticity}. The books \citep{Cia1988,Ciarlet1997_Band2,Ciarlet2000_Band3,McLean2000} and \cite{KnopsPayne1971} provide a detailed discussion of the mathematical theory of three dimensional elasticity in $ H^1 $. The book \cite{Cia1988} considers also other Sobolev spaces - at least in the case of the pure traction problem (i.e. pure Dirichlet data).

Unfortunately, the results on strong solution theory are scattered across the literature and consequently they are very difficult to find. But, since classical solutions play a crucial role in the study of shape derivatives in Hölder spaces, we give a compilation of these results in Section \ref{Sec:Linear_Elasticity:Sec:Modeling Elastic Deformations}.

\section[Properties of domains]{Preliminaries: Properties of domains}

We start summarizing some definitions on properties of domains which are crucial for the regularity of PDE solutions.

\begin{defn}\label{PDE_Systems:Def:Cone} \cite{Fujii,Chen75}
	Let $ \Omega \subset \R{n} $ be a domain. $ \Omega $ satisfies a \textit{cone property} with angle $ \vartheta \in (0,\frac{\pi}{2}) $,  $ d>0 $ and hight $ h \in (0,\frac{d}{2}) $ if there exists a direction $ y_x  $ with $ \Norm{y_x}{} =1 $ for any $ x \in \partial\Omega $  such that the cone 
	$$ C_{x}=C_{x}(y_x,\vartheta,d):=\{x \in \R{n} \, | \, \Norm{x}{}<d,\, \langle x,y\rangle > \Norm{x}{} \cos(\vartheta)\} $$
	satisfies \[ p + C_{x} \subset \Omega~~~~~~\forall p \in U(x,h) \cap \Omega. \]
	The family of bounded and open subsets of $ \R{n}   $ that satisfy the cone property with $ \vartheta \in (0,\frac{\pi}{2}) $,  $ d>0 $ and $ h \in (0,\frac{d}{2}) $ is denoted by $ \mathcal{M}(\vartheta,d,h) $.
\end{defn}

\begin{defn}\cite{Fujii,Chen75}
	Let $ \mathcal{M} $ be a set of domains in $ \R{n} $. Then $ \mathcal{M}  $ is said to satisfy a \textit{uniform cone property} if there is $ \vartheta \in (0,\frac{\pi}{2}) $, $ d>0 $ and $ h \in (0,\frac{d}{2}) $ such that any $ \Omega \in \mathcal{M} $ is an element of $ \mathcal{M}(\vartheta,d,h) $.
\end{defn}

\begin{lem}\label{Ex_Shape_Deriv_LinEl:Lem: uniform cone lemma} \cite[Lemma 5.5]{GottschSchmitz}
	Let $ \mathcal{M} $ be a set of domains in $ \R{n} $ with a uniform cone property and let $ \Omega \in \mathcal{M} $. Then, for every $ \delta>0 $ there exists a constant $ C_{\delta}>0 $ such that for any $ v \in C^{1}(\Omega) $
	\begin{align*}
	\Norm{v}{C^{0}(\Omega)} \leq \delta\Norm{v}{C^{1}(\Omega)} + C_{\delta} \Norm{v}{L^1(\Omega)}.
	\end{align*}
	Particularly, the constant $ C_{\delta} $ can be chosen uniformly w.r.t. $ \mathcal{M} $.
\end{lem}

\begin{defn}\label{PDE_Systems:Def:Ck_phi_Boundary}\cite[Sec. 6.2]{GilbTrud} 
	Let $ \Omega \subset \R{n} $ be a domain with $ \Gamma \neq \emptyset $, $ k \in \N{}_0 $, $ 0\leq \phi \leq 1 $.  \begin{itemize}
		\item[a)] $\Omega $ is said to have a \textit{$ C^{k,\phi} $-boundary} or to be \textit{of class $ C^{k,\phi} $} if for any point $ x\in \Gamma=\partial\Omega $ there is a open ball $ U=B_{r}(x)\subset \R{n}$ $ $ with center $ x $ and radius $ r $ and a bijective mapping $ \Psi_{x}: U \to \tilde{U} \subset \R{n}$, such that 
		\begin{itemize}
			\item[i)] $ \Psi_{x}(U \cap \Omega) \subset \R{n}_{+} $, 
			\item[ii)] $ \Psi_{x}(U \cap \Gamma)  \subset \R{n-1} \times\{0\}$, 
			\item[iii)] $ \Psi_{x} \in C^{k,\phi}(U,\tilde{U})$, $  \Psi^{-1}_{x} \in C^{k,\phi}(\tilde{U},U) $.
		\end{itemize}
		\item[b)] A boundary part $ \Gamma' \subset \Gamma $ is called $ C^{k,\phi} $-boundary part if for any $ x \in \Gamma' $ there exists a ball $ U=U(x,r) $ with radius $ r>0 $ such that $ U\cap \Gamma \subset \Gamma' $  and a mapping $ \Psi_{x}: U \to \tilde{U} \subset \R{n}$   that satisfies i) - iii).
		\item[c)] If $ \Omega $ is a $ C^{k,\phi} $-domain and $ f:\Gamma' \to \R{m}$ is defined on a boundary part $ \Gamma' $ of $ \Gamma $, then $ f $ is 
		called a $ C^{k,\phi}(\Gamma,\R{m}) $-function (vector field) if 
		$ f \circ \Psi^{-1}_{x}\in C^{k,\phi}(\tilde{U} \cap \partial\R{n}_{+},\R{m}) $
		for any $ x \in \Gamma' $.
	\end{itemize}
\end{defn}

The cone condition (Definition \ref{PDE_Systems:Def:Cone}) is a rather weak condition since already any Domain with $ C^{0,1} $-boundary - also called Lipschitz boundary - possesses a cone property, see \cite{Fujii,AdamsFournier,Chen75} and vice versa. 

\begin{lem}\cite[Sec. 6.2, P. 89]{GilbTrud}
	A bounded 
	domain $ \Omega \subset \R{n} $ possesses a $C^{k,\phi}$-boundary if for every $ x \in \Gamma $ there is a neighborhood $ U_x\subset \Gamma $ in which $ \Gamma $ is the graph of a $ C^{k,\phi} $-transformation $ \Psi $ in $ (n-1) $ coordinates. The converse is true if $ k \geq 1 $.
\end{lem}

\begin{defn}[Hemisphere Transformations, Hemisphere Property]\label{PDE_Systems:Def:Hemisphere_Trafo}\cite{Agm64}
	Let $ \Omega \subset \R{n} $ be a domain with $ \Gamma \neq \emptyset $, $ k \in \N{}_0 $, $ 0\leq \phi \leq 1 $. Let $ \Omega' \subset \Omega $ be a subdomain of $ \Omega $ and $ \Gamma_{r} $ a regular (\footnote{That means there is a countinuous transformation $ \Phi:\Gamma_{r} \to \R{n-1} $ that maps $ \Gamma_{r} \subset \Gamma $ to the $ (n-1) $ dimensional hyper plane in $ \R{n} $.}) boundary portion of $ \Gamma $ such that $ \Gamma \cap \partial \Omega' \subset \overset{\circ}{\Gamma_r} $ in the $ (n-1) $-dimensional sense. 
	\begin{itemize}
		\item[a)] The set $ \Omega' $ is said to satisfy a \textit{hemisphere condition} if there exists a distance $ d>0 $ such that every $ x \in \Omega' $  with $ \dist(x,\Gamma)\leq d $ has a neighborhood $ U_x $ such that
		\begin{itemize}
			\item[i)] $ \overline{U}_{x} \cap \Gamma \subset \Gamma_r$,
			\item[ii)] $ B_{\nicefrac{d}{2}}(x) \subset U_x$,
			\item[iii)] $ \overline{U_x}\cap \overline{\Omega}=\mathbb{T}_{x}(\Sigma_{R(x)})$, ~~~$ \overline{U_x} \cap \Gamma =\mathbb{T}_{x}(F_{R(x)})$, ~~~ $ 0 < R(x)<1  $
		\end{itemize} 
		for some hemisphere $ \Sigma_{R(x)} = \{z \in \R{n} \, | \, \Norm{z}{} \leq R(x),\, z_{n} \geq 0 \} $ and some disk $ F_{R(x)} = \{z \in \R{n} \, | \, \Norm{z}{} \leq R(x),\, z_n =0 \}$ and transformations $ \mathbb{T}_{x},\, \mathbb{T}_{x}^{-1} $.
		\item[b)] If $ \Omega $ is of class $ C^{k,\phi} $, $ k \geq 0 $, $ \phi\in [0,1] $  and the transformations $  \mathbb{T}_{x}  $ are $ C^{k,\phi} $- transformations are then $ \Omega' $ is said to posses a $ C^{k,\phi} $-hemisphere property.
		\item[c)] The transformations $ \mathbb{T}_{x} $ are called \textit{hemisphere transformations}.
	\end{itemize}
\end{defn}

\begin{defn}[Uniform Hemisphere Property]\cite{Agm64,Agm59}
	Let $ \mathcal{O} $ be a family of domains $ \Omega \subset \R{n} $ such that any $ \Omega \in \mathcal{O} $ possesses a $ C^{k,\phi} $-hemisphere property. The hemisphere property is called \textit{uniform} if the transformations satisfy $ \Norm{\mathbb{T}_{x}}{C^{k,\phi}(\Omega,\R{n})} \leq  C_{\mathbb{T},\mathcal{O}} ~ \forall \Omega\in \mathcal{O}, x \in \Gamma $ for only one constant $ C_{\mathbb{T},\mathcal{O}} $ and the distance $ d $ can be chosen uniformly with respect to $ \mathcal{O} $.
\end{defn}

A hemisphere property is a special $ C^{k,\phi} $-boundary condition. 
It claims that the transformations satisfy $ \Psi: U_x \cap \Omega \to \Sigma_{R}(x) \subset \tilde{U}$ which is a special assumption on $ \tilde{U} $.

\noindent It moreover requests not only existence of suitable neighborhoods for the boundary points, but also for the points in $ \Omega $ close to the boundary. These neighborhoods have a minimum diameter of $ d/2 $ and cover $ \Omega^{d}:=\{x \in \Omega \, | \, \dist(x,\Gamma)<d\} $.

In case $ \Omega $ is a bounded domain with $ C^{k,\phi} $ boundary it is nevertheless possible to construct hemisphere transformations that satisfy the assumptions of Definition \ref{PDE_Systems:Def:Hemisphere_Trafo}. In this case both definitions can be transferred one in each other. This is due to the compactness of the boundary, a geometric construction of the neighborhoods $ U_x $ and the possibility to extend the inverse transformations $ \Psi_{x}^{-1}$  onto hemispheres. We also refer to Lemma 6 in \cite{BittGottsch} and Lemma 5.4. in \cite{GottschSchmitz}.

\begin{lem}\label{Ch:Ex_Shape_Deriv_LinEl:Lem: uniform hemisphere property}
	Any bounded domain $ \Omega $ of class $ C^{k,\phi} $ has a $ C^{k,\phi} $-hemisphere property. The choice of the diameter $ d $ depends on the curvature of the boundary of $ \Omega $. 
\end{lem}

\section{Linear elliptic systems of partial differential equations}
\label{PDE_Systems:Sec:Linear Elliptic Systems}

\begin{notation}[PDE Systems] Let $ \Omega \subset \R{n} $, $ n\geq 2 $ be a domain with coordinates $ x=(x_1,x_2,\ldots,x_n) $. Let $ \Xi \in \R{n} $  and then 
	$$ \Xi^{\alpha}:=\Xi_1^{\alpha_1}\Xi_2^{\alpha_2}\cdots\Xi_n^{\alpha_n}$$
	for any multiindex $ \alpha = (\alpha_1,...,\alpha_n) \in \N{n}_{0}  $. Let 
	\begin{align}\label{PDE_Systems:Eq:Polynom}
	p(x,\Xi)&=\sum_{|\alpha|=0}^{d}p^{(\alpha)}(x)\Xi^{\alpha} = \sum_{|\alpha|=0}^{d}p^{(\alpha_1,...,\alpha_n)}(x)\Xi_{1}^{\alpha_1}\cdots \Xi^{\alpha_n}
	\end{align}
	be a polynomial in $ \Xi $ with coefficient functions $ p^{(\alpha)}, 0 \leq |\alpha| =\sum_{i=1}u^{nu} \leq m $ in $ x \in \Omega$ and degree $ d $. Replacing $ \xi_i $ by $ \frac{\partial}{\partial x_i} $ we can associate a partial differential operator of order $ d $ with the polynomial $ p $.
	\begin{align}
	p(x,D)u(x):&= \sum_{|\alpha|=0}^{d}p^{(\alpha_1,...,\alpha_n)}(x)\Big(\frac{\partial}{\partial x_1}\Big)^{\alpha_1}\Big(\frac{\partial}{\partial x_2}\Big)^{\alpha_2}\cdots \Big(\frac{\partial}{\partial x_n}\Big)^{\alpha_n}u(x)\\
	&=\left(\sum_{|\alpha|=0}^{d}p^{(\alpha_1,...,\alpha_n)}\frac{\partial^{|\alpha|}u}{\partial x_1^{\alpha_1} \partial x_2^{\alpha_2} \cdots \partial x_n^{\alpha_n}}\right)(x)\, .
	\end{align}
	
	We now take a matrix of such polynomials $ \mathbf{a}(x,\Xi)=\left(\mathbf{a}_{ij}(x,\Xi)\right)_{i,j=1,\ldots,N}$ where $ \mathbf{a}_{ij}(x,\Xi) = \sum_{|\alpha|=0}^{d_{ij}} \mathbf{a}_{ij}^{(\alpha_1,...,\alpha_n)}(x)\Xi_{1}^{\alpha_1}\cdots \Xi_{n}^{\alpha_n} $ each  of degree $ d_{ij} \geq 0 $ in $ \Xi $
	and define a system of partial differential equations setting
	\begin{align}\label{PDE_Systems:Eq:PDE_System_Volume}
	(\mathbf{a}(x,D)u(x))_{i}&=\sum_{j=1}^{N}\mathbf{a}_{ij}(x,D)u_j(x)=f_i(x),~~~ i=1,\ldots,N.
	\end{align}
	
\end{notation}
\noindent The matrix $ \mathbf{a}(x,\Xi) $ is called the \textit{symbol of the differential operator} $ \mathbf{a}(x,D) $.

With any system of PDEs we associate two systems of weights of integers: $ s_i \leq 0, i=1,2,\ldots,N$ such that $ s_i $ corresponds to the $ i $-th equation and $ t' \geq t_j\geq 0,\, j=1,2,\ldots,N$, $ t'=\max\{t_1,\ldots,t_N\} $ that are related to the unknowns $ u_j $ (\footnote{This can always bee achieved by adding a constant to one system of weights and subtracting it from the other, confer \citep{Agm64}}) by the relation
\begin{align}\label{PDE_Systems:Eq:PDE_System_Weights_Volume}
d_{ij}=\deg(\mathbf{a}_{ij}(x,\Xi)) \leq s_i+t_j ~~~\forall i,j=1,\ldots,N\, .
\end{align}
Note that the systems $ (s_{i})_{i},\, (t_{j})_{j} $ are not uniquely defined and that $ \mathbf{a}_{ij}\equiv 0 $ if $ s_i+t_j<0 $, consider \citep[(1.2)-(1.4)]{Agm64}. Then \[  \mathbf{a}_{ij}(x,\Xi)=\sum_{|\alpha|=0}^{s_i+t_j}\mathbf{a}_{ij}^{(\alpha)}\Xi^{\alpha}.\]  

\begin{defn}[Ellipticity]\citep{Agm64}
	A system of partial differential equations \eqref{PDE_Systems:Eq:PDE_System_Volume} is called 
	\textit{elliptic} if the \textit{characteristic polynomial}
	\begin{equation}\label{PDE_Systems:Eq:Def_elliptic}
	\mathcal{A}(x,\Xi):=\det(\mathbf{a}'_{ij}(x,\Xi)) \neq 0 ~~~ \text{ for all } \Xi \in \R{n}\setminus\{0\}. 
	\end{equation}
	The matrix $ (\mathbf{a}'_{ij}(x,\Xi))_{ij} $ contains only the terms in $ \mathbf{a}_{ij}(x,\Xi) $ that are of order $ s_i+t_j $ and is called \textit{principal symbol of } $ \mathbf{a}(x,D) $.
\end{defn}

\begin{defn}
	An elliptic system of PDEs \ref{PDE_Systems:Eq:PDE_System_Volume} satisfies the \textit{supplementary condition} if the following two conditions are satisfied:
	\begin{itemize}
		\item[i)] The characteristic polynomial $ \mathcal{A}(x,\Xi) $ is of even degree $ 2M,\, M \in \N{} $ in $ \Xi $.
		\item[ii)] The polynomial $ \mathcal{A}(x,\Xi + \tau \Xi') $ in $ \tau \in \mathbb{C},\, x \in \overline{\Omega} $ has exactly $ M $ roots with positive imaginary part for any pair of linearly independent vectors $ \Xi,\, \Xi' \in \R{n} $.
	\end{itemize} 
\end{defn}

The supplementary condition is satisfied whenever the PDE system \eqref{PDE_Systems:Eq:PDE_System_Volume} is a system in three or more independent variables $ u_j,\, j=1,2,3,\ldots,N $. The proof of this statement can be found in \citep[P. 631-632]{Agm59}. Thus the supplementary condition has to be assumed for systems in two variables, only. 
Moreover, it is actually only employed at the boundary $ \Gamma $ of $ \Omega $ with $ \Xi $ tangent and $ \Xi' $ normal to $ \Gamma $ at the point $ x \in \Gamma $, see \cite{Agm64}. 

\begin{defn}[Uniform Ellipticity]\citep{Agm64}
	A system of partial differential equations \eqref{PDE_Systems:Eq:PDE_System_Volume} is called
	\textit{uniformly elliptic} if the characteristic polynomial  $ \mathcal{A}(x,\Xi) $ is of even degree $ \deg(\mathcal{A}(x,\Xi)) =2M $ in $ \Xi $ and there exists a constant $ \varLambda>0 $ sucht that 
	\begin{equation}\label{PDE_Systems:Eq:Def_uniformly_elliptic}
	\varLambda^{-1}\Norm{\Xi}{}^{2M} \leq |\mathcal{A}(x,\Xi)| \leq \varLambda \Norm{\Xi}{}^{2M} ~~~ \forall \,\Xi \in \R{n} \text{ and } x \in \overline{\Omega}.
	\end{equation}
\end{defn}

\noindent Throughout this chapter we assume the folling:

\noindent \textbf{Assumptions:} Let $ \mathbf{a}(x,D) $ be defined according to \eqref{PDE_Systems:Eq:PDE_System_Volume}.
\begin{itemize}
	\item[\textbf{(A1)}] The system weights $ s_1,\ldots,s_N \leq 0$ and $ t'=\max\{t_1,\ldots,t_N\} \geq t_1,\ldots,t_N \geq 0 $  satisfy \eqref{PDE_Systems:Eq:PDE_System_Weights_Volume}.
	\item[\textbf{(A2)}] The system of PDE is elliptic.
	\item[\textbf{(A3)}] The supplementary condition holds, and the number $ M=\frac{1}{2}\deg\mathcal{A}(x,\Xi) $ is positive.
\end{itemize}

\begin{notation}[Boundary Conditions]
	Let $ \Gamma_{r}$ be a regular portion of $ \Gamma=\partial\Omega $. A system of boundary conditions is given by
	\begin{align}\label{PDE_Systems:Eq:PDE_System_Boundray}
	(\mathbf{b}(x,D)u(x))_{h}=\sum_{j=1}^{N}\mathbf{b}_{hj}(x,D)u_i(x)=g_h(x),~~~ h=1,\ldots,M ,\, x \in \Gamma_r,
	\end{align}
	where the $ \mathbf{b}_{hj}(x,\Xi) $ are polynomials in $ \Xi $ with coefficients in $ x $. 
	Thus the differential operator $ \mathbf{b}(x,D) $ depends on the matrix $ \mathbf{b}(x,\Xi) = (\mathbf{b}_{hj}(x,\Xi))_{hj},\, h=1,\ldots,M,\, i=1,\ldots,N  $ with polynomial entries.
	The degrees of the polynomials $ \mathbf{b}_{hj} $ in $ \Xi $   depend on integer weights $ r_1,\ldots,r_M $ 
	and the system $ t_1,\ldots, t_N $ associated to \eqref{PDE_Systems:Eq:PDE_System_Volume} by \eqref{PDE_Systems:Eq:PDE_System_Weights_Volume}
	\begin{align}
	\deg(\mathbf{b}_{hj}(x,\Xi))\leq r_h + t_j ,\, 1 \leq h \leq M,\, 1 \leq j \leq N,
	\end{align}
	the degree of $ \mathbf{b}_{hj}(x,\Xi) $ with respect to $ \Xi \in \R{n} $ and
	$ \mathbf{b}_{h,j}(x,\Xi) $ can be rewritten as
	\[  \mathbf{b}_{hj}(x,\Xi)=\sum_{|\alpha|=0}^{r_h+t_j}\mathbf{b}_{hj}^{(\alpha)}\Xi^{\alpha}.\] If $ r_h+t_j<0 $ then  $ \mathbf{b}_{hj}=0 $. By $ \mathbf{b}_{hj}'(x,\Xi) $ we denote the terms in $ \mathbf{b}_{hj}(x,\Xi) $ which are of the order $ r_h+t_{j} $. Moreover, we set $ r':=\max\{0, r_1+1,\ldots, r_M+1 \}$ and $ r'':=\max\{0, r_1,\, \ldots,r_{M}\} $. 
\end{notation}

Now we investigate so called \textit{boundary value problems} (BVPs)
\begin{equation}\label{PDE_Systems:Eq:PDE_System}
\begin{split}
\mathbf{a}(\cdot,D)u&=f \text{ on } \Omega \\
\mathbf{b}(\cdot,D)u&=g \text{ on }  \Gamma.
\end{split}
\end{equation}
Therefore, we have to assume further continuity conditions. The so called \textit{complementing boundary condition} assures that this system is well-posed.

\begin{notation} 
	Let $ \vec{n}(x) $ denote the outward normal, $ \mathcal{T}_{x} $ the tangential space at $ x $ associated to the manifold $ \Gamma$ and $ \Xi(x) \neq 0,\, \Xi(x) \in \mathcal{T}_{x} $ a tangent to $ x \in \Gamma_{r} $. Then, 
	$ \tau^{+}_{k}(x,\Xi),\, 1 \leq k \leq M $ denote the $ M $ solutions  with positive imaginary part\footnote{The existence of these roots is assured by the supplementary condition.} of the characteristic equation 
	\[ \mathcal{A}(x,\Xi(x) + \tau \vec{n}(x))=\det(\mathbf{a}_{ij}'(\Xi(x) + \tau \vec{n}(x)))=0.\]
	Additionally, define the polynomial 
	$$ \mathbf{M}^{+}(x,\Xi,\tau)=\prod_{k=1}^{M}(\tau - \tau^{+}(x,\Xi)) $$
	in the variable $ \tau $ and let $ (\mathbf{a}_{ij}'^{\ast}(x,\Xi + \tau \vec{n}))_{ij},\, 1 \leq i,j \leq N $ denote the adjoint matrix of $ (\mathbf{a}_{ij}'(x,\Xi + \tau \vec{n}))_{ij},\, 1 \leq i,j \leq N  $.
\end{notation}

\begin{defn}(Complementing Boundary Condition) The system \eqref{PDE_Systems:Eq:PDE_System} is said to satisfy a \textit{complementing} (or complementary) \textit{boundary condition} if the matrix (with polynomial entries in the indeterminate $ \tau $ )
	\[ \mathbf{c}(x,\Xi + \tau \vec{n})_{hk}=\sum_{j=1}^{N}\mathbf{b}_{hj}(x,\Xi + \tau \vec{n})\mathbf{a}_{jk}'^{\ast}(x,\Xi + \tau \vec{n}),\, h=1,\ldots,M,\, k=1,\ldots,N \]
	has linear independent rows modulo $ \mathbf{M}^{+}(x,\Xi,\tau) $, $  \Xi=\Xi(x) $ tangent to $ \Gamma_r $ at  $ x \in \Gamma_{r} $.
\end{defn}

\noindent \textbf{Assumptions:} Let $\mathbf{a} $ and $ \mathbf{b} $ be defined according to \eqref{PDE_Systems:Eq:PDE_System_Volume} and  \eqref{PDE_Systems:Eq:PDE_System_Boundray}, respectively. Then, additionally to \textbf{(A1)} - \textbf{(A3)}, we assume that
\begin{itemize}
	\item[\textbf{(B1)}] the complementing boundary condition holds for the system \eqref{PDE_Systems:Eq:PDE_System}.
\end{itemize}

\vspace{\parindent}
Since $ \mathbf{c}(x,\Xi + \tau \vec{n})$ is a polynomial in each entry we can find coefficients $ \mathbf{c}_{hk}^{(\beta)}(x,\Xi) \in \R{}$ such that 
\[  \mathbf{c}(x,\Xi + \tau \vec{n})_{hk} =\sum_{\beta=0}^{M-1}\mathbf{c}_{hk}^{(\beta)}(x,\Xi) \tau^{\beta} \mod \mathbf{M}^{+}(x,\Xi,\tau) \, .\]

\begin{notation} \cite{Agm64}
	Let $ x \in \Gamma_r \subset \Gamma=\partial \Omega $ and let $ \Xi=\Xi(x) $ be a tangent vector to $ \Gamma_{r} $ at $ x \in \Gamma_{r} $. Set $\mathbf{C}(x,\Xi)_{h,(\beta,k)}:=\mathbf{c}_{hk}^{(\beta)}(x,\Xi) $ (\footnote{Since the absolute value of the minors of this matrix are calculated and the determinant is alternating in the columns (rows) the column sorting is not important here.}) as a matrix with $ M $ rows $ h=1,\ldots, M $ and $ MN $ columns $ \beta=0,\ldots,M-1,\, k=1,\ldots,N $, i.e. 
	\[C= \begin{pmatrix}
	\mathbf{c}_{1,1}^{(0)} & \ldots & \mathbf{c}_{1,N}^{(0)} & \mathbf{c}_{1,1}^{(1)} & \ldots & \mathbf{c}_{1,N}^{(1)} & 	\mathbf{c}_{1,1}^{(M-1)} & \ldots & \mathbf{c}_{1,N}^{(M-1)} \\
	\mathbf{c}_{2,1}^{(0)} & \ldots & \mathbf{c}_{2,N}^{(0)} & \mathbf{c}_{2,1}^{(1)} & \ldots & \mathbf{c}_{2,N}^{(1)} & 
	\mathbf{c}_{2,1}^{(M-1)} & \ldots & \mathbf{c}_{2,N}^{(M-1)} \\
	\vdots & & \vdots & \vdots & &\vdots  & \vdots & &\vdots \\
	\mathbf{c}_{M,1}^{(0)} & \ldots & \mathbf{c}_{M,N}^{(0)} & \mathbf{c}_{M,1}^{(1)} & \ldots & \mathbf{c}_{M,N}^{(1)} & 
	\mathbf{c}_{M,1}^{(M-1)} & \ldots & \mathbf{c}_{M,N}^{(M-1)}
	\end{pmatrix} \]
	By $ \mathcal{M}^1(x,\Xi),\ldots, \mathcal{M}^{\binom{MN}{M}}(x,\Xi) $ we denote $ M $-rowed minors of $ \mathbf{C}(x,\Xi) $ achieved by deleting $ NM-M $ columns at a time and calculating the determinant.
\end{notation}

\begin{lem}[\cite{Agm64}]
	The complementary boundary condition \textbf{(B1)} implies that $ \mathbf{C}(x,\Xi) $, for fixed $ x \in \Gamma_{r},\, \Xi=\Xi(x) \in \mathcal{T}_{x}  $ is a rank-$ M $-matrix and that not all $ M $-rowed minors $ \mathcal{M}^{i}(x,\Xi),\, i=1,\ldots, \binom{NM}{M}, $ are zero. In particular \[ \mathcal{M}^{max}(x,\Xi)=\max_{i=1,\ldots,\binom{NM}{M}}|\mathcal{M}^{i}(x,\Xi)| >0 \]
	and if $ \Gamma_{r} \subset \Gamma$ is compact, then also $ \displaystyle \triangle_{\Gamma_r}':=\inf_{x \in \Gamma_r, \Xi \in \mathcal{T}_{x}} \mathcal{M}^{max}(x,\Xi) >0. $
\end{lem}

\begin{defn}[Minor Constant] \cite{Agm64} $  $
	\begin{itemize}
		\item[i)] 	When $ \Gamma_{r} $ is plane, then the \textit{minor constant} $ \triangle_{\Gamma_r}=\triangle_{\Gamma_r}' $.
		\item[ii)] If $ \Gamma_r $ is not plane and if there exists a change of coordinates $ \Phi: \Gamma_r \to \R{n-1} $ that 'makes' $ \Gamma_r $ plane, then the \textit{minor constant} is defined by $\triangle_{\Gamma_r}=\triangle_{\Phi(\Gamma_r)}' $. 
		\item[iii)] If $ \Gamma $ is not plane and compact and if there exists a family of regular boundary portions $ \Gamma_i, i=1,...,l $ such that $ \Gamma $ is covered by the family $ (\Gamma_i)_{i} $ and transformations $ \Phi_i: \Gamma_i \to \R{n-1} $ that make $ \Gamma_i $ plane, then each portion $ \Gamma_{i} $ has a minor constant as defined in $  ii) $ and the corresponding \textit{minor constant} for the whole boundary $ \Gamma $ is defined by $$\triangle \equiv \triangle_{\Gamma} =\inf_{i=1,\ldots,l}\triangle_{\Gamma_i}.$$
	\end{itemize}
\end{defn}

\noindent \textbf{Notation:} Let $ \triangle_x $  denote the minor constant for $ \overline{U}_x \cap \partial \Omega $ which pertains to the hemisphere transformation $ \mathbb{T}_x $. Then $ \triangle = \inf_{\{x: ~\dist(x,\Gamma)\leq d\}} \triangle_x $.

\begin{exmp}[Pure Traction Problem of Linear Elasticity]\label{PDE_Systems:Exmp:Lin_El}
	Let $ \Omega \subset \R{3} $, $ f:\Omega  \to \R{3} $ and $ u \in C^{2}(\Omega,\R{3}) $. The PDE system of linear elasticity with Dirichlet BC is given by  \begin{equation}\label{PDE_Systems:Eq:Lin_El_Volume}
	\begin{split}
	\Div(\sigma(u))&= f \text{ on } \Omega \\
	u&=0 \text{ on } \Gamma 
	\end{split}
	\end{equation}
	where $ \sigma(u)=\lambda\Div(u)\mathrm{I} + \mu(Du + Du^{\top}) $ for Lamé-Coefficients $ \lambda,\mu>0 $.  Component-wise the system reads
	\begin{equation*}
	\begin{split}
	\left[(\lambda+2\mu)\frac{\partial^2}{\partial x_1^2}+ \mu\left(\frac{\partial^2}{\partial x_2^2}+\frac{\partial^2}{\partial x_3^2}\right)\right]u_1 + (\lambda+\mu)\frac{\partial^2}{\partial x_1\partial x_2}u_2 +   (\lambda+\mu)\frac{\partial^2}{\partial x_1\partial x_3}u_3 &=f_1
	\\[1ex]
	(\lambda+\mu)\frac{\partial^2}{\partial x_1\partial x_2}u_1 +  \left[(\lambda+2\mu)\frac{\partial^{2}}{\partial x_2^{2}}+\mu\left(\frac{\partial^{2}}{\partial x_1^{2}}+\frac{\partial^{2}}{\partial x_3^{2}}\right)\right]u_2 +   (\lambda+\mu)\frac{\partial^{2}}{\partial x_2 \partial x_3}u_3 &= f_2
	\\[1ex]
	(\lambda+\mu) \frac{\partial^{2}}{\partial x_1\partial x_3}u_1  +   (\lambda+\mu)\frac{\partial^{2}}{\partial x_2 \partial x_3}u_2 + \left[(\lambda+2\mu)\frac{\partial^{2}}{\partial x_{3}^2}+\mu\left(\frac{\partial^2}{\partial x_1^{2}}+ \frac{\partial^{2}}{\partial x_2^{2}}\right)\right]u_3&= f_3 
	\end{split}
	\end{equation*}
	and
	\begin{equation*}
	\begin{split}
	u_1 + 0 u_2 + 0u_3 &=0 \\
	0u_1 + u_2 + 0u_3 &=0 \\
	0u_1 + 0u_2 + u_4 &=0
	\end{split}
	\end{equation*}
	on $ \Gamma$. Since the Lamé-Coefficients are constant in $ x $ also
	\begin{align*}
	\mathbf{a}_{i,j}(\Xi)&=\mathbf{a}_{i,j}(x,\Xi)=
	\begin{pmatrix}
	(\lambda+\mu)\xi_1^2 + \mu\Norm{\xi}{}^2 & (\lambda +\mu)\xi_1\xi_2 & (\lambda +\mu)\xi_1\xi_3 \\[1ex]
	(\lambda +\mu)\xi_1\xi_2 & (\lambda+\mu)\xi_2^2 + \mu\Norm{\xi}{}^2   &  (\lambda +\mu)\xi_2\xi_3 \\[1ex]
	(\lambda +\mu)\xi_1\xi_3 & (\lambda +\mu)\xi_2\xi_3 & (\lambda+\mu)\xi_3^2 + \mu\Norm{\xi}{}^2 
	\end{pmatrix}
	\end{align*}
	and 
	\begin{align*}
	\mathbf{b}_{i,j}(\Xi)&=\mathbf{b}_{i,j}(x,\Xi)=
	\begin{pmatrix} 
	1 &0 &0 \\
	0 & 1 & 0 \\
	0 & 0 & 1
	\end{pmatrix}
	\end{align*}
	are constant in $ x $.
	
	Here, we obtain $ s_1=s_2=s_3=0 $ and $ t'=t_1=t_2=t_3=2 $ for the weights (but  $ s_1=s_2=s_3=-1 $ and $ t_1=t_2=t_3=3 $ would also be a reasonable choice). Then $ r_1=r_2=r_3=-2 $, $ r'=\max\{0, r_1+1, r_2+1,r_3+1\}=0 $. 
	The supplementary condition is fulfilled anyway since the PDE system of linearized elasticity is a system in three independent variables $ u_1,\, u_2,\, u_3 $ and the determinant of degree $ \deg(\mathcal{A}(x,\Xi))=2M,\, M=3 $ satisfies
	\begin{align*}
	\mathcal{A}(x,\Xi)=\mu(\lambda+\mu)^2\xi_2^2\xi_3^2\Norm{\Xi}{}^{2} + \mu^2((\lambda+\mu) + \mu^4)\Norm{\Xi}{}^{6}.
	\end{align*}
	
	\noindent Thus, $ \mu^2((\lambda+\mu) + \mu^4) \Norm{\Xi}{}^6 \leq |\mathcal{A}(x,\Xi)| \leq (\mu(\lambda+\mu)^2 + \mu^2(\lambda+\mu) + \mu^6)\Norm{\Xi}{}^{6} $	and $ \varLambda>0 $ can be chosen (depending on $ \lambda $ and $ \mu $) such that $ \varLambda^{-1} \leq \mu^2((\lambda+\mu) + \mu^4) \leq  \mu(\lambda+\mu)^2 + \mu^2(\lambda+\mu) + \mu^6 \leq \varLambda $. Thereof we conclude that the linear elasticity equation \ref{PDE_Systems:Eq:Lin_El_Volume} is a uniformly elliptic system of partial differential equations of second order. The complementary condition is also satisfied, confer Section 6.3. in \citep{Cia1988}.
\end{exmp}

Combining  Observation 3.2  and Theorem 3.5.  in \citep{Geymonat} and the comments thereafter, we can state the following theorem that plays a crucial role in the existence of solutions to \eqref{PDE_Systems:Eq:PDE_System} w.r.t. to higher order Sobolev and classical function spaces:

\begin{thm}[Index Theorem] \label{PDE_Systems:Thm:Index_Theorem}
	Let a system 
	\begin{equation}\label{PDE_Systems:Eq:PDE_System_II}
	\begin{split}
	\mathbf{a}(.,D)u&=f \text{ on } \Omega \\
	\mathbf{b}(.,D)u&=g \text{ on }  \Gamma.
	\end{split}
	\end{equation}
	of partial differential equations be given and suppose that $ \Omega $ is a bounded domain in $\R{n},\, n \geq 2 $ with boundary of class $ C^{r'+t'+k} $, $k \in \N{}_{0}\cup \{\infty\}$. Moreover, assume that the coefficients of $ \mathbf{a} $ and $ \mathbf{b} $, respectively, satisfy
	\begin{align*}
	\mathbf{a}_{i,j}^{(\rho)}  \in 
	\begin{cases}
	C^{r'- s_i+k}(\overline{\Omega}) & \text{ if } |\rho| = s_i + t_j \\
	W^{r'-s_i+k,\infty}(\Omega) & \text{ if } |\rho| < s_i + t_j
	\end{cases}
	\text{ and } ~	
	\mathbf{b}_{i,j}^{(\rho)}  \in 
	\begin{cases}
	C^{r'- r_h+k} & \text{ if } |\varrho| = r_h + t_j \\
	W^{r'-r_h+k,\infty} & \text{ if } |\varrho| < r_h + t_j
	\end{cases}
	\end{align*}
	for $ i,\,j=1,\ldots,N $, $ h =1,\ldots,M $. Then the following two assertions are equivalent:
	\begin{itemize}
		\item[i)] The system \eqref{PDE_Systems:Eq:PDE_System_II} is elliptic  \textbf{(A2)}  and satisfies the supplementary and complementary conditions  \textbf{(A3)}  and \textbf{(B1)};
		\item[ii)] If $1 < p<\infty  $  and $ 0 \leq l \leq k$, the operator 
		\begin{align*}
		A_{l,p}: \prod_{j=1}^{N} W^{t'+t_j+l,p}(\Omega) &\to \prod_{i=1}^{N} W^{t'-s_i+l,p}(\Omega) \times \prod_{h=1}^{M} W^{t'-r_h+l-\nicefrac{1}{p},p}(\Gamma
		) \\
		u & \mapsto  (\mathbf{a}u,\mathbf{b}u) 
		\end{align*}
		is linear and continuous and has a finite index $$ \ind{A_{l,p}}:= \dim(\ker(A_{l,p})) - \dim(\coker{A_{l,p}}) $$
		that depends neither on $ k $ nor on $ p $.
	\end{itemize}
\end{thm}

\noindent Here $ W^{l-1/p,p}(\Gamma)=\mathbf{T_{\Gamma}}(W^{l,p}(\Omega)) =\{\mathbf{T_{\Gamma}}(u) \vert u \in W^{l,p}(\Omega) \} $ with $ \mathbf{T_{\Gamma}} $ the trace operator on $ W^{l,p}(\Omega) $, see also \ref{App: Trace Space}.

A linear, continuous operator $ A: X \to Y $ from one Banach space $ X $ into another one $ Y $ is called \textit{Fredholm operator} if  $ \dim(\ker(A))  $ is finite, $ \im{A} $ is closed 
and the \textit{codimenson} of $ \im{A} $ in $ Y $, i.e. $  \dim(\coker{A}) = \dim( Y/\im{A}) $, is also finite. 

From this point of view, the family of operators $ (A_{l,p})_{l,p} $ is a set of  Fredholm operators with constant index if the assumptions and condition i) of Theorem  \ref{PDE_Systems:Thm:Index_Theorem} are satisfied.

\begin{thm}[Schauder Estimates in Sobolev Spaces]\label{PDE_Systems:Thm:SchauderEstSobolev} \citep[Theorem 10.5]{Agm64} Let $ \Omega \subset \R{n} $ be a bounded domain.
	Suppose that a PDE system \eqref{PDE_Systems:Eq:PDE_System_II} is given on  $ \Omega$ and $ \Gamma $ such that the assumptions $ (\mathbf{A1}) - (\mathbf{A3}) $ and $ (\mathbf{B1}) $ are satisfied. Let $ k \geq r'=\max\{0,r_1+1,\ldots,r_M+1\} $ be a fixed integer and $ 1<p <\infty $.
	
	\noindent Moreover, we make the following assumptions regarding smooth- and boundedness:
	\begin{itemize}
		\item[(S1)] The coefficients $ \mathbf{a}_{i,j}^{(\rho)} $ belong to $ C^{k-s_i}(\overline{\Omega}) $ and the functions $ f_i $ are elements of $ W^{k-s_i,p}(\Omega) $.
		\item[(S2)] The coefficients $ \mathbf{b}_{h,j}^{(\varrho)} $ are elements of $ C^{k-r_h}(\Gamma) $ and $ g_h \in W^{k-r_h-\nicefrac{1}{p},p}(\Gamma) $.
		\item[(S3)] The right hand sides $ f_i $ and $ g_h $ are bounded from above by some constant $ c_{f,g}>0 $ and the coefficients $ \mathbf{a}_{i,j}^{(\rho)} $ and $ \mathbf{b}_{h,j}^{(\varrho)} $ are bounded from above by $ c_{a,b} $ in their respective norms.
		\item[(S4)]Moreover, suppose that $ \Omega $ possesses a $ C^{k+t'} $- hemisphere property with $ t' = \max\{t_j\} $ such that $ \triangle_{\Gamma}>0 $  and such that the hemisphere transformations $ \mathbb{T}_{x} $ have finite $ C^{k + t',\phi} $-norms bounded by some constant $ C_{\mathbb{T}} $ independent of $ x $.
	\end{itemize}
	Then any solution $ u \in W^{r' + t_j}(\Omega) $ of \eqref{PDE_Systems:Eq:PDE_System_II}
	belongs to $ W^{k+t_j}(\overline{\Omega}) $, $ j=1,\ldots, N $ and satisfies
	\begin{align*}
	\Norm{u_{j}}{W^{k+t_j,p}(\Omega)} \leq C \left(\sum_{i=1}^{N} \Norm{f_i}{W^{k-s_i,p}}(\Omega) + \sum_{h=1}^{M}\Norm{g_h}{W^{k-r_h-\nicefrac{1}{p},p}(\Gamma)} + \sum_{k=1}^{N}\Norm{u_{k}}{C^{0}(\Omega)}\right)
	\end{align*}
	for some constant $ C\geq 0 $ that depends on $ c_{a,b},\, \varLambda,\, \triangle_{\Gamma},\, C_{\mathbb{T}},\, d,\, n,\, N,\, M\, \sum| r_h|,\,p,\, k $.
	
	\noindent The term $ \Norm{u_{j}}{C^{0}(\Omega)} $  can be replaced by $ \int_{\Omega} |u_j|\, dx $, $ j=1,\ldots, N $.
\end{thm}

\begin{thm}[Schauder Estimates for Classical PDE solutions] \label{PDE_Systems:Thm:SchauderEstHölder} \citep{Agm64} Let $ \Omega \subset \R{n} $ be a general, nonempty (possibly infinite) domain and $ \Omega' \subset \Omega $ a subdomain abutting a boundary portion $ \Gamma' $ of $ \Omega $ such that $ \partial \Omega' \cap \Gamma \subset \Gamma'  $ in the $ (n-1) $-dimensional sense.
	
	\noindent Suppose that a PDE system \eqref{PDE_Systems:Eq:PDE_System_II} is given on  $ \Omega$ and $ \Gamma' $ such that the assumptions $ (\mathbf{A1}) - (\mathbf{A3}) $ and $ (\mathbf{B1}) $ are satisfied. Let $ k \geq r'' = \max\{0,r_1,\ldots,r_{M}\} $ be a fixed integer and $ \phi \in (0,1) $ a Hölder index.
	Additionally, suppose that $ \Omega' \subset \Omega $ satisfies a hemisphere property such that $ \triangle_{\Gamma'}=\inf_{\{x: ~\dist(x,\Gamma')\leq d\}} \triangle_x  >0 $. Moreover, we make the following assumptions regarding smoothness and boundedness:
	\begin{itemize}
		\item[(S1)] The coefficients $ \mathbf{a}_{i,j}^{(\rho)} $ and the functions $ f_i $ belong to $ C^{k-s_i,\phi}(\overline{\Omega}) $.
		\item[(S2)] The coefficients $ \mathbf{b}_{h,j}^{(\varrho)} $ and $ g_h $ are elements of $ C^{k-r_h,\phi}(\Gamma') $.
		\item[(S3)] The right hand sides $ f_i $ and $ g_h $ are bounded from above by some constant $ c_{f,g}>0 $ and the coefficients $ \mathbf{a}_{i,j}^{(\rho)} $ and $ \mathbf{b}_{h,j}^{(\varrho)} $ are bounded from above by $ c_{a,b} $ in their respective Hölder norms.
		\item[(S4)] The hemisphere transformations $ \mathbb{T}_{x} $ and there inverses are of class $ C^{k+ t'',\phi} $ with $ t'' = \max\{-s_i,\, -r_h,\, t_j\} $ and have finite $ C^{k + t'',\phi} $-norms bounded by some constant $ C_{\mathbb{T}} $ independent of $ x $.
	\end{itemize}
	Then any solution $ u \in C^{r'' + t_j,\phi}(\Omega\cup \Gamma') $ of \eqref{PDE_Systems:Eq:PDE_System_II}
	belongs to $ C^{k+t_j,\phi}(\overline{\Omega'}) $, $ j=1,\ldots, N $ and satisfies
	\begin{align*}
	\Norm{u_{j}}{C^{k+t_j,\phi}(\Omega')} \leq C \left(\sum_{1}^{N} \Norm{f_i}{C^{k-s_i,\phi}}(\Omega) + \sum_{h=1}^{M}\Norm{g_h}{C^{k-r_h,\phi}(\Gamma')} + \sum_{k=1}^{N}\Norm{u_{k}}{C^{0}(\Omega)}\right)
	\end{align*}
	for some constant $ C\geq 0 $ that depends on $ c_{a,b},\, \varLambda,\, \triangle_{\Gamma'},\, C_{\mathbb{T}},\, d,\, n,\, N,\,\phi,\, k $.
	
	\noindent If $ \Omega $ is bounded, $ \Omega'=\Omega $, $ \Gamma'=\Gamma $, then any solution $ u \in C^{r'' + t_j,\phi}(\overline{\Omega}) $ of \eqref{PDE_Systems:Eq:PDE_System_II}
	already belongs to $ C^{k+t_j,\phi}(\overline{\Omega}) $, $ j=1,\ldots, N $ and $ \Norm{u_{j}}{C^{0}(\Omega)} $  can be replaced by $ \int_{\Omega} |u_j|\, dx $, $ j=1,\ldots, N $.
\end{thm}
\pagebreak
\begin{rem}
	To be more accurate, Theorem 7.3. in \cite{Agm59} shows that the constant $ C $ depends on $ \triangle_{\Gamma'}^{-1} $ and that it's dependence on $ \triangle_{\Gamma'} $ is anti proportional: if $ \triangle_{\Gamma'} $ decreases, then $ C $ increses and vice versa.
	
	\noindent The constant $ C $ is obviously independent of the choice of $ \Omega' $ but not of $ \Gamma' $.
\end{rem}

\section{Regularity theory in linear elasticity}
\label{Sec:Linear_Elasticity:Sec:Modeling Elastic Deformations}

Let us now resume the linear elasticity equation which was introduced in Section \ref{Reliability}. 

The set $\Omega\subseteq \R{3}$ is assumed to be a bounded domain with divided boundary $\Gamma=\partial\Omega$: A Dirichlet boundary part denoted by $ \Gamma_{D} $ and a Neumann boundary $ \Gamma_N $ with $ \Gamma_{D} \dot{ \cup} \Gamma_{N}=\Gamma $. 

The displacement field $u=u(\Omega):\overline{ \Omega}\to \R{3}$ depends on the domain and can be derived as a solution of a linear elasticity problem which is given by a system of linear elliptic PDE of second order. The we assume that volume loads $f=f(\Omega):\Omega\to\R{3}$ and surface loads $g_{N}=g(\Gamma_{N}):\Gamma_{N} \to\R{3}$ are acting on $ \Omega $. The lamé coefficients $ \lambda,\, \mu $ shall are assumed to be positive real numbers. 
According to \citep{ErnGuerm04,Cia1988} the \textit{disjoint displacement-traction problem} is given by equation \eqref{Reliability:Eq:LinEl}, i.e.
\begin{align}
\left. 
\begin{array}{rcll}
-\Div( \se(u)) &=&f  &\text{ in } \Omega \\
\se(u) &= &\lambda\Div(u)\mathrm{I}+\mu(Du+Du^{\top}) &\text{ on } \Omega \\
u &=& 0  &\text{ on } \Gamma_{D}  \\
\se(u) \vec{n}&=&g_N &\text{ on }\Gamma_{N}.
\end{array} 
\right.
\end{align}

As already indicated, it is easy to find results on the existence of $ H^1 $-solutions in the literature, see e.g. \citep{Cia1988,ErnGuerm04} or \cite{KnopsPayne1971}. The book of Ciarlet \cite{Cia1988} also provides existence theory for solutions the pure traction problem in higher order Sobolev spaces. 

The situation is somewhat different concerning strong solutions for elliptic PDE systems.
Although the results on the existence of strong solutions are widely known, they are hardly detectable in the literature. 
Nevertheless, it is the aim of this section to provide these statements on regularity results on linear elasticity since they play a crucial role in the study of shape derivatives for linear elasticity in Hölder Spaces. Thus we supplement the results, where we could not find them in the literature.

\subsection{Sobolev-solutions and Schauder estimates}
\label{Sec:Linear_Elasticity:Sec:Regularity_Weak_Solutions}

\begin{thm}[Korn's second inequality]\label{Linear_Elasticity:Thm:Korns_second_ineq}\citep[Theorem 6.3-3 and 6.3-4]{Cia1988} \cite{Nitsche1981}
	Suppose that $ \Omega \subset \R{3}$ is a Lipschitz domain.
	\begin{itemize}
		\item[i)]  For each $ v \in H^{1}(\Omega,\R{3}) $ the strain tensor $ \varepsilon(v)_{ij} \in L^2(\Omega) $ and there is a constant $ c>0 $ such that 
		\begin{equation}\label{Linear_Elasticity:Eq:Korns_second_ineq}
		\Norm{v}{H^1(\Omega,\R{3})}^2 \leq C\left(\Norm{v}{L^2(\Omega,\R{3})}^2 + \Norm{\varepsilon(v)}{L^2(\Omega,\R{3 \times 3})}^2\right)
		\end{equation}
		with  
		\begin{equation}
		\Norm{\varepsilon(v)}{L^2(\Omega,\R{3 \times 3})}=\left(\int_{\Omega} \varepsilon(u):\varepsilon(u) \, dx\right)^{\nicefrac{1}{2}}= \left(\int_{\Omega} \tr(\varepsilon(u)^2) \, dx\right)^{\nicefrac{1}{2}}\, .
		\end{equation}
		\item[ii)] Let $ \Gamma_{D}$ be a measurable subset of $ \Gamma =\partial\Omega $. Then the space
		\[ H^1_{D}(\Omega,\R{3}):=\{ v \in H^1(\Omega,\R{3}) \vert \mathbf{T_{\Gamma}}(v)=0 \,  \text{ on } \Gamma_{D}\} \]
		is a closed subspace of $  H^1(\Omega,\R{3}) $ and therefore a Hilbert space. If the surface measure of $ \Gamma_{0} $ is positive, i.e. $ |\Gamma_{0}|>0 $, there exists a constant $ C_{K}>0 $ such that  
		\begin{equation}\label{Linear_Elasticity:Eq:Folgerung_Korns_second_ineq}
		\Norm{v}{H^1(\Omega,\R{3})} \leq C_{K} \Norm{\varepsilon(v)}{L^2(\Omega,\R{3})} ~~~ \forall v \in H^1_{\Gamma_{0}}(\Omega,\R{3}).
		\end{equation}
	\end{itemize}
\end{thm}

\noindent The original work was published in 1981 by Nitsche \cite{Nitsche1981}.
A good survey over some possibilities to proof Korn's second inequality is provided in  \citep{DiplomDerksen}.

\begin{defn}\label{Linear_Elasticity:Def:Weak_formulation_linear_elasticity}
	The weak formulation of equation \eqref{Reliability:Eq:LinEl} is given by 
	\begin{align}\label{Linear_Elasticity:Eq: B(u,v)=L(v)}
	B(u,v)=L(v) \, \forall v \in H^{1}_{D}(\Omega,\R{3})
	\end{align}
	where
	\begin{align}
	B(u,v):= \int_{\Omega} \lambda \tr(\varepsilon(u)) \tr(\varepsilon(v)) + 2\mu \tr(\varepsilon(u)\varepsilon(v)) \, dx = \int_{\Omega} \varepsilon(u):\sigma(v) \, dx 
	\end{align}
	and 
	\begin{align}
	L(v):=\int_{\Omega} \langle f(\Omega),v \rangle \, dx + \int_{\Gamma_{N}}\langle g(\Gamma),v \rangle \, dS.
	\end{align}
\end{defn}

\noindent As a result of Theorem \ref{Linear_Elasticity:Thm:Korns_second_ineq} the bilinear form $ B $ is strictly coercive. Moreover, $ B $ is continuous on $ H^{1}_{D} $.

\noindent In case $ \Omega $ possesses a Lipschitz boundary, the trace spaces
$ W^{m-1/p,p}(\Gamma)=\mathbf{T_{\Gamma}}(W^{m,p}(\Omega))$ $ =\{\mathbf{T_{\Gamma}}(u) \vert u \in W^{m,p}(\Omega) \} $ are well defined for any $ m \geq 1 $ - see also Definition \ref{App: Trace Space}.

\noindent If $ f \in L^{\nicefrac{6}{5}}(\Omega,\R{3}) $ and $ g \in L^{\nicefrac{4}{3}}(\Gamma_{N},\R{3}) $ then $ L $ is a continuous linear form. This is due to the Sobolev Embedding Theorem, the trace operator $ \mathbf{T_{\Gamma}} $ and H\"olders inequality because
$$ H^{1}(\Omega,\R{3}) = W^{1,2}(\Omega,\R{3})\hookrightarrow L^{p^{\ast}}(\Omega,\R{3}) \text{ if } p^{\ast}=6 $$ 
and
\[ \mathbf{T_{\Gamma}}:H^{1}(\Omega,\R{3})  = W^{1,2}(\Omega,\R{3}) \to W^{1-1/2,2}(\Gamma,\R{3}) \subsetneq L^{p^{\#}}(\Gamma,\R{3}) \text{ with } p^{\#}=4, \]
see the sobolev embedding Theorem \ref{App: Sobolev embedding}, imply
\begin{align}\label{Linear_Elasticity:Def:Con Linearform LinEl}
|L(v)| 
&\leq \Norm{\langle f(\Omega),v \rangle}{L^1(\Omega,\R{3})}+\Norm{\langle g(\Gamma),v \rangle}{L^1(\Gamma_N,\R{3})} \nonumber \\
&\leq  \Norm{f(\Omega)}{L^{\nicefrac{6}{5}}(\Omega,\R{3})} \Norm{v}{L^{6}(\Omega,\R{3})}+ \Norm{g(\Gamma)}{L^{\nicefrac{4}{3}}(\Gamma_N,\R{3})} \Norm{v}{L^{4}(\Omega,\R{3})}\\
&\leq \left(C_1\Norm{f(\Omega)}{L^{\nicefrac{6}{5}}(\Omega,\R{3})} + C_2\Norm{g(\Gamma)}{L^{\nicefrac{4}{3}}(\Gamma_N,\R{3})}\right)\Norm{v}{H^{1}(\Omega,\R{3})}. \nonumber
\end{align}

\noindent Thus, we can state the following:

\begin{thm} \label{Linear_Elasticity:Def:H1_solutions_lin_el} ($ H^1 $-solutions for the disjoint displacement traction \\ problem)\citep[Theorem 6.3-5 ]{Cia1988}
	Let $ \Omega \subset \R{3}$ be a nonempty Lipschitz domain and  $ \Gamma_{D}$ be a measurable subset of $ \Gamma =\partial\Omega $ with $ \int_{\Gamma_D}\,dS= |\Gamma|_{D}>0 $. Then there exists a unique weak solution $  u \in H^{1}_{D}(\Omega,\R{3})  $ such that 
	\begin{equation}
	B(u,v)=L(v)\, \forall v \in H^{1}_{D}(\Omega,\R{3}),
	\end{equation} 
	if $ f \in L^{\nicefrac{6}{5}}(\Omega,\R{3}) $ and $ g_{N} \in L^{\nicefrac{4}{3}}(\Gamma_N,\R{3}) $.
\end{thm}
Thus especially $ f:=f\vert_{\Omega} $ and $ g:=g\vert_{\Gamma_N} $ with $ f,\,g \in C(\overline{ \Oext},\R{3}) $ will be sufficient to obtain a unique $ H^1 $-solution $ u $ to equation \eqref{Linear_Elasticity:Eq: B(u,v)=L(v)}.

\begin{thm}\label{Linear_Elasticity:Thm:Weak_Reg_LinEl}(Regularity weak solutions to the linearized disjoint displacement traction problem)
	Let $ \Omega $ be a $ C^{2} $-domain in $ \R{3} $ and $ \Gamma_{D} \subset \Gamma $ a proportion of the boundary with positive surface measure. Assume that $ \Gamma_N= \Gamma\setminus\Gamma_{D} $ and $ \Gamma_{D} $ have a positive distance $\dist(\Gamma_{D},\Gamma_{N})>0$ (\footnote{Otherwise there may be no such regular solution to \eqref{Linear_Elasticity:Eq: B(u,v)=L(v)}}).
	\begin{itemize}
		\item[i)] Suppose that $ f \in L^{p}(\Omega,\R{3}) $ and $ g_N \in W^{1-\nicefrac{1}{p},p}(\Gamma,\R{3}),\, p\geq \frac{4}{3} $. Then the weak solution $ u \in H^{1}_{D}(\Omega,\R{3}) $ of \eqref{Reliability:Eq:LinEl} is an element of $ W^{2,p}(\Omega,\R{3}) $.
		\item[ii)] Suppose that $ k \geq 1 $ is an integer and $ \Gamma $ is of class $ C^{2+k} $. If $ f \in W^{k,p}(\Omega,\R{3}) $ and $ g_N\in W^{k+1-\nicefrac{1}{p},p}(\Gamma_N,\R{3}),\, p\geq \frac{4}{3} $. Then the weak solution $ u \in H^{1}_{D}(\Omega,\R{3}) $ of \eqref{Reliability:Eq:LinEl} is an element of $ W^{k+2,p}(\Omega,\R{3}) $.
	\end{itemize}
	For any $ k \geq 0 $ the solution $ u \in W^{2+k,p}(\Omega,\R{3}) $ satifies 
	\begin{equation}\label{Linear_Elasticity:Eq:Schauder_Estimate_u_Sobolev}
	\Norm{u}{W^{k+2,p}(\Omega,\R{3})} \leq C \left(\Norm{f}{W^{k,p}(\Omega,\R{3})} + \Norm{g}{W^{k+1-\nicefrac{1}{p},p}(\Gamma,\R{3})} +\Norm{u}{C^0(\Omega,\R{3})}\right).
	\end{equation}
	for some constant $ C>0 $ depending on  $ \lambda,\, \mu,\, n=3,\, N=M=3,\,p\geq \frac{4}{3} ,\, \sum |r_h| = 6 $ the choice of $ k \in \N{} $ and the domain $ \Omega $ or more precisely on $ \triangle_{\Gamma},\, C_{\mathbb{T}}$ and the distance $ d $.
\end{thm}

\begin{proof}
	We transfer the proof for the pure Dirichlet case, see \citep{Cia1988}, to equation \eqref{Reliability:Eq:LinEl}:
	
	\noindent i) (Step 1) The differential operator $ \mathbf{a}(.,D)$ that belongs to the disjoint displacement traction problem \eqref{Reliability:Eq:LinEl},
	\begin{align*}
	\mathbf{a}_{i,j}(\Xi)&=\mathbf{a}_{i,j}(x,\Xi)=
	\begin{pmatrix}
	(\lambda+\mu)\xi_1^2 + \mu\Norm{\xi}{}^2 & (\lambda +\mu)\xi_1\xi_2 & (\lambda +\mu)\xi_1\xi_3 \\
	(\lambda +\mu)\xi_1\xi_2 & (\lambda+\mu)\xi_2^2 + \mu\Norm{\xi}{}^2   &  (\lambda +\mu)\xi_2\xi_3 \\
	(\lambda +\mu)\xi_1\xi_3 & (\lambda +\mu)\xi_2\xi_3 & (\lambda+\mu)\xi_3^2 + \mu\Norm{\xi}{}^2 
	\end{pmatrix},
	\end{align*} compare Example \ref{PDE_Systems:Exmp:Lin_El},
	satisfies 
	\[ \mathbf{a}(x,D)u(x)=-\Div(\sigma(u(x))) \text{ if } x \in \Omega, \]
	is uniformly elliptic and fulfills the supplementary condition (see \citep{Agm64}). The boundary differential operator 
	$ \mathbf{b}(x,D)$ given by
	\begin{align*}
	\mathbf{b}(x,D)u(x)=
	\begin{cases}
	u &\text{ if } x \in \Gamma_{D},\\
	\sigma(u)\vec{n} &\text{ if } x \in  \Gamma_{N}
	\end{cases}
	~~~=\mathds{1}_{\Gamma_D} u + \mathds{1}_{\Gamma_N} \sigma(u)\vec{n}
	\end{align*}
	is defined by the polynomials $\left(\mathbf{b}_{h,j}(x,\Xi)\right),\,h=1,2,3,\, j=1,2,3 $, 
	\begin{align}
	\mathbf{b}_{hj}(x,\Xi)= 
	\begin{cases}
	\mathds{1}_{\Gamma_N}(x)((\lambda + 2\mu)\vec{n}_h \xi_h + \sum_{k \neq h}\mu \vec{n}_k\xi_k)+  \mathds{1}_{\Gamma_D}(x)& \text{ if } h=j \\
	\mathds{1}_{\Gamma_N}(x)(\lambda \vec{n}_h\xi_j + \mu \vec{n}_j\xi_h) & \text{ if } h\neq j 
	\end{cases}
	\end{align}
	and
	satisfies the complementing boundary condition, see \cite{Cia1988}. Let $ f\in L^{2}(\Omega,\R{3}) $ and  $ g_{N} \in W^{1-1/2,2}(\Gamma,\R{3}) $. As Theorem \ref{Linear_Elasticity:Def:H1_solutions_lin_el} shows, the weak PDE formulation \[ B(u,v)=L(v) \, \forall v \in  H^{1}_{D}(\Omega,\R{3}) \]
	then has a unique $ H^{1}_{D} $ solution  $ u $ and \citep{Necas1967} implies, that $ u \in H^{2}(\Omega,\R{3}) $, already.
	
	\noindent (Step 2) Now let $  \nicefrac{4}{3}<p< \infty $ and $ k\geq0 $. The degrees $ s_{i},t_{j} \in \mathbb{Z}$ that are associated to $ \mathbf{a}_{i,j}(x,\Xi) $ 
	can be chosen as $ s_1=s_2=s_3=0 $ and $ t_1=t_2=t_3=2 $. The system $ r_1,r_2,r_3 $ is then
	determined as $ r_1=r_2=r_3=-1 $ and thus $ \deg(\mathbf{b}_{i,j}(x,\Xi)) =1=-1+2=r_h+t_j $ for all $ h,j=1,\,2,\,3. $ 
	Furthermore, $ r'=\max_{h=1,\, 2,\,3}\{0,r_{h}+1\}=0 $. 
	On $  W^{r'+t_j+m,p}_{D}(\Omega)  $ we can also examine the weak differential operator 
	\begin{align*}
	T_{m,p}^{D}: \prod_{j=1}^{3} W^{m+2,p}_{D}(\Omega) & \to \prod_{i=1}^{3} W^{m,p}(\Omega) \times \prod_{h=1}^{3} W^{m+1-\nicefrac{1}{p},p}(\Gamma_N) \\
	u &\mapsto (\mathbf{a}u,\mathbf{b}u) = (-\Div(\sigma(u)),\sigma(u)\vec{n})
	\end{align*}
	for $ 0\leq m \leq k $ instead of 
	\begin{align*}
	T_{m,p}: \prod_{j=1}^{3} W^{m+2,p}(\Omega) & \to \prod_{i=1}^{3} W^{m,p}(\Omega) \times \prod_{h=1}^{3} W^{m+1-\nicefrac{1}{p},p}(\Gamma) \\
	u &\mapsto (\mathbf{a}u,\mathbf{b}u) = \left(-\Div(\sigma(u)), \mathds{1}_{\Gamma_D} u + \mathds{1}_{\Gamma_N} \sigma(u)\vec{n}\right).
	\end{align*}
	In case of linear elasticity with mixed boundary condition and $ m=0 $ we obtain
	\begin{align*}
	T_{p}^{D}=T_{0,p}^{D}: W^{2,p}(\Omega,\R{3}) & \to L^{p}(\Omega,\R{3}) \times W^{1-\nicefrac{1}{p},p}(\Gamma_N,\R{3})\, .
	\end{align*}
	According to Theorem \ref{PDE_Systems:Thm:Index_Theorem} the index \[ \ind{T^{D}_{p}} = \dim (\ker(T^{D}_{p})) -\dim(\coker{T^{D}_{p}})\, , \]
	is independent of $ 1 < p < \infty $ where $$ \coker{T^{D}_{p}} = \left[L^{p}(\Omega,\R{3}) \times W^{1-\nicefrac{1}{p},p}(\Gamma_N,\R{3})\right]/\im{T^{D}_{p}} \, .$$ Thus we can return to the case $ p=2 $:
	Then we already know, that $ T_{p}^{D} $ is a bijection since there exists a unique solution to $ T_{p}^{D}(u)=(f,g_N) $ in  $ H^2 $ given  $ (f,g_N) \in L^{2}(\Omega,\R{3}) \times W^{1-\nicefrac{1}{2},p}(\Gamma_N,\R{3})$. Thus $ \coker{T^{D}_{2}} =\{0\}  $ and $ \ker{T^{D}_{2}}=\{0\} $ implies $ \ind{T^{D}_{2}}=0 $.
	
	\noindent Suppose that $ f=0 $ on $ \Omega $ and $ g_N=0 $ on $ \Gamma_N $. Then the unique solution $ u  \in H^{1}_{D}(\Omega,\R{3}) $ satisfies $ u=0 $. But 
	$ W^{2,p}_{D}(\Omega,\R{3})\hookrightarrow H^{1}(\Omega,\R{3}) $ for $ p\geq 4/3 >6/5 $ implies $ \ker{T^{D}_{p}}=\{0\} $ and hence $ T^{D}_{p} $ is injective. Therefore $$ -\dim(\coker{T^{D}_{p}}) = \dim (\ker(T^{D}_{p})) -\dim(\coker{T^{D}_{p}}) = \ind{T^{D}_{p}} $$ and since the index is independent of $ p $
	\[ -\dim(\coker{T^{D}_{p}}) = \ind{T^{D}_{p}} = \ind{T^{D}_{2}} =0\, .\]
	Thereof we conclude that $ \coker{T^{D}_{p}}=\{0\} $ i.e. $ T^{D}_{p} $ is also surjective. This proofs the assertion for $ l=0 $.
	
	\noindent ii) Now, we investigate the solution operator
	\begin{align*}
	T_{k,p}^{D}: W^{2+m,p}_{D}(\Omega,\R{3}) & \to W^{m,p}(\Omega,\R{3}) \times W^{m+1-\nicefrac{1}{p},p}(\Gamma_N,\R{3})
	\end{align*}
	for  $ m \in \{ 1,\ldots,k\} $ and $ \Omega $ of class $ C^{2+m} $, $ k> 0 $. Since
	\[ \{0\} \subset \ker(T_{k,p}^{D}) \subset \cdots \subset\ker(T_{m,p}^{D}) \subset \cdots \subset \ker(T_{0,p}^{D})=\ker(T_{p}^{D}) =\{0\} \]
	$ T_{m,p}^{D} $ is injective for any $ 1 \leq m \leq k $ and since the index  is independent of $ k $ and $ p $ we deduce $ \ind{T_{m,p}^{D}} = \ind{T_{0,p}^{D}}= \ind{T_{0,2}^{D}}=\{0\} $ and thus the operator $ T_{m,p}^{D} $ is also surjective. 
	
	\noindent The Schauder estimate can directly be derived from Theorem \ref{PDE_Systems:Thm:SchauderEstSobolev}. 
\end{proof}

\begin{rem}
	This proof can also be extended to the equation 
	\begin{align}
	\left. 
	\begin{array}{rcll}
	-\Div( \se(u)) &=&f  &\text{ in } \Omega \\
	u &=& g_D  &\text{ on } \Gamma_{D}  \\
	\se(u) \vec{n} &=&g_N &\text{ on }\Gamma_{N}
	\end{array} 
	\right.
	\end{align}
	with $ f \in W^{k,p}(\Omega,\R{3}),\, g_D \in W^{k+2-1/p,p}(\Gamma_{D},\R{3}) $ and $ g_{N} \in  W^{k+1-1/p,p}(\Gamma_{N},\R{3}) $. Then the operator 
	\begin{align*}
	T_{k,p}: \prod_{j=1}^{3} W^{2+k,p}(\Omega) & \to \prod_{i=1}^{3} W^{k,p}(\Omega) \times \prod_{h=1}^{3} W^{1+k-\nicefrac{1}{p},p}(\Gamma) \\
	u &\mapsto (\mathbf{a}u,\mathbf{b}u) = (-\Div(\sigma(u)), \mathds{1}_{\Gamma_D} u + \mathds{1}_{\Gamma_N} \sigma(u)\vec{n}),
	\end{align*}
	$ k \geq 0 $, has to be examined and application of the same arguments as before leads again to existence of unique solutions. 
\end{rem}

\subsection{Classical solutions and Schauder estimates}
\label{Sec:Linear_Elasticity:Sec:Classical_solutions_and_Schauder}

\begin{thm}\label{Linear_Elasticity:Thm:LinEl_Classical_Sol} Let $ \Omega \subset \R{3} $ be a domain of class $ C^{k+2,\phi} $ for some $ k\in \N{}_{0}$. Let $ \Gamma_{D} $ be a proportion of the boundary with positive surface measure.  Moreover, assume that $ \Gamma_N= \Gamma\setminus\Gamma_{D} $ and $ \Gamma_{D} $ have a positive distance $\dist(\Gamma_{D},\Gamma_{N})>0$.
	Suppose that
	$ f \in C^{k,\phi}(\overline{\Omega},\R{3}) $ and $ g \in C^{k+1,\phi}(\overline{\Gamma}_{N},\R{3})  $, $ \phi \in (0,1) $. Then there exists a unique solution $ u \in C^{k+2,\phi}(\overline{ \Omega},\R{3}) $ of equation \eqref{Reliability:Eq:LinEl}
	and
	\begin{equation}\label{Linear_Elasticity:Eq:Schauder_Estimate_u}
	\Norm{u}{C^{k+2,\phi}(\Omega,\R{3})} \leq C \left(\Norm{f}{C^{k,\phi}(\Omega,\R{3})} + \Norm{g}{C^{k+1,\phi}(\Gamma,\R{3})} +\Norm{u}{C^0(\Omega,\R{3})}\right)
	\end{equation}
	where $ C>0 $ is a constant depending on the constants $ \lambda,\, \mu,\, n=3,\, N=3,\,\phi\in (0,1) ,\,  k $ and the domain $ \Omega $ or more precisely on $ \triangle_{\Gamma},\, C_{\mathbb{T}}$ and the distance $ d $. 
\end{thm}

\begin{proof}
	i) First, let $ k\geq 1 $. Then $ f \in C^{k,\phi}(\overline{\Omega},\R{3})$  $ \subset W^{k,p}(\Omega,\R{3})$ and $ g \in C^{k+1,\phi}(\overline{\Gamma_{N}},\R{3}) $ $\subset W^{k+1-\nicefrac{1}{p},p}(\Gamma_N,\R{3})  $ for any $ p \geq 1 $. Thus we can choose $ p^0>3$ large enough such that $ \phi\leq 1-\nicefrac{3}{p^0} $. By Theorem \ref{Linear_Elasticity:Thm:Weak_Reg_LinEl} there exists a unique weak solution $ u \in W^{k+2,p}(\Omega,\R{3})  $ that can be embedded according to the Sobolev Embedding Theorem into $ C^{k+1,\phi}(\overline{\Omega},\R{3}) $. In the case of linear elasticity Theorem 9.3. in \citep{Agm64} tells us, that if $ f \in C^{k,\phi},\, g\in C^{k+1,\phi} $ and $ u \in C^{2,\phi} $ is a solution, then $ u $ is already an element of $ C^{k+2,\phi} $ and satisfies the Schauder estimate \eqref{Linear_Elasticity:Eq:Schauder_Estimate_u}. Since $ k+1 \geq 2 $ the assertion therefore holds.

	\noindent Now let $ k=0 $, $ f \in C^{0,\phi}(\overline{\Omega},\R{3})   $ and $ g \in C^{1,\phi}(\overline{\Gamma_{N}},\R{3})  $ and $ \varphi<\phi $. Then $ f $ and $ g $ have $ C^{0,\phi} $ and $ C^{1,\phi} $ extensions $ \tilde{f},\, \tilde{g} $, respectively,  to some domain $  \tilde{\Omega}  \supset \supset \Omega$ of class $ C^{2,\phi} $, see Lemma \ref{App:Lem: Hölder Extension Lemma}.  Then we a sequence $ (f_n)_{n} \subset  C^{1,\phi}(\tilde{\Omega},\R{3})  $ and $ (g_{n})_{n} \in C^{2,\phi}(\tilde{\Omega},\R{3}) $ with $ f_{n} \to \tilde{f} \in C^{0,\phi},\, g_{n} \to \tilde{g} \in C^{1,\phi} $ and thus $ f_{n} \to f $ on $ \Omega $, $ g_{n} \to g $ on $ \Gamma_N $. 
	We denote the sequence of solutions associated to $ (f_n,0,g_n) $  by $ (u_n)_{n} \subset C^{3,\phi}(\overline{\Omega},\R{3})$. This sequence satisfies 
	\[ \Norm{u_n}{C^{2,\phi}} \leq C(\Norm{f_{n}}{C^{0,\phi}(\overline{\Omega},\R{3})} + \Norm{g_{n}}{C^{1,\phi}(\overline{\Omega},\R{3})} + \Norm{u_{n}}{C^{0}}) \]
	since especially $ f_n \in C^{0,\phi}(\overline{\Omega},\R{3}) $, $ g_{n}\in C^{1,\phi}(\overline{\Omega},\R{3}) $ and $ u_{n} \in C^{2,\phi}(\overline{\Omega},\R{3}) $ satisfies 
	\begin{align}
	\left. 
	\begin{array}{rcll}
	-\Div( \se(u_n)) &=&f_n  &\text{ in } \Omega \\
	u_n &=& 0  &\text{ on } \Gamma_{D}  \\
	\se(u_n) \vec{n}&=&g_n &\text{ on }\Gamma_{N}.
	\end{array} 
	\right.
	\end{align}
	Now let $0< \delta<\nicefrac{1}{C} $. Since $ \Omega $ satisfies a cone condition, we can deduce from Lemma 5.5  \citep{GottschSchmitz} that there exits a constant $ C(\delta) $ such that 
	$$\Norm{u}{C^{0}(\overline{\Omega},\R{3})} \leq \delta \Norm{u}{C^1(\overline{\Omega},\R{3})} +C(\delta)\Norm{u}{L^1((\Omega,\R{3}))} \, \forall u \in C^{1}(\Omega,\R{3}).$$ 
	Then Hölder's inequality and the definition of the $ C^{2,\phi} $-norm lead to 
	\[ \Norm{u_n}{C^{0}(\overline{\Omega},\R{3})} \leq \delta \Norm{u_n}{C^{2,\phi}(\overline{\Omega},\R{3})} +C(\delta)\Norm{u_n}{H^1(\Omega,\R{3})}\sqrt{|\Omega|} ~~~\forall n \in\N{}.\]
	From the uniform ellipticity of the bilinear form $ B $ we derive
	\begin{align*}
	& && \Norm{u_{n}}{H^1(\Omega,\R{3})}^2 
	&&\leq \Varlambda B(u_n,u_n)=\Varlambda\left(\int_{\Omega} \langle f_n,u_{n} \rangle \, dx  + \int_{\Gamma_{N}} \langle g_n,u_{n} \rangle \, dS\right) \\
	& && &&\leq c \left(\Norm{f_n}{C^{0}(\Omega,\R{3})}\Norm{u_{n}}{H^1(\Omega,\R{3})}+ \Norm{g_n}{C^{0}(\Gamma_{N},\R{3})}\Norm{u_{n}}{H^1(\Omega,\R{3})} \right) \\[1ex]
	& \Leftrightarrow && \Norm{u_{n}}{H^1(\Omega,\R{3})} 
	&&\leq   c \left(\Norm{f_n}{C^{0}(\Omega,\R{3})}+ \Norm{g_n}{C^{0}(\Gamma_{N},\R{3})}\right)  \\
	& && &&\leq c \left(\Norm{f_n}{C^{0,\phi}(\Omega,\R{3})}+ \Norm{g_n}{C^{1,\phi}(\Gamma_{N},\R{3})}\right)
	\intertext{
		Where $ c $ depends on $ \Omega $ and $ \Varlambda $. Therefore,
	}
	& &&\Norm{u_n}{C^{2,\phi}} &&\leq C(\Norm{f_{n}}{C^{0,\phi}} + \Norm{g_{n}}{C^{1,\phi}} + \Norm{u_{n}}{C^{0}}) \\[1ex]
	& && &&\leq \tilde{C}(\delta)(\Norm{f_{n}}{C^{0,\phi}} + \Norm{g_{n}}{C^{1,\phi}}) + C\delta\Norm{u_{n}}{C^{2,\phi}} \\[1ex]
	&\Leftrightarrow && \Norm{u_n}{C^{2,\phi}} &&\leq \frac{\tilde{C}(\delta)}{1-C\delta}\left(\Norm{f_{n}}{C^{0,\phi}} + \Norm{g_{n}}{C^{1,\phi}}\right).
	\end{align*}
	Since $ (f_{n})_{n \in  \N{}} $ and $ (g_n)_{n\in \N{}} $ are convergent sequences in the respective norms the sequences $ \Norm{f_{n} }{C^{0,\phi}(\overline{\Omega},\R{3})} $ and $ \Norm{g_{n}}{C^{1,\phi}(\overline{\Omega},\R{3})} $ are bounded and thus $\Norm{u_n}{C^{2,\phi}(\overline{\Omega},\R{3})}$ is also. Thus there is a constant $ C^{\ast} $ such that, 
	$$ (u_n)_{n \in \N{}} \subset S:=\{u \in C^{2,\phi}(\Omega,\R{3})\, \vert \Norm{u_n}{C^{2,\phi}(\Omega,\R{3})} \leq C^{\ast} \}.$$ As \citep[Lemma 6.36]{GilbTrud} shows, the set $ S $ is precompact in $ C^{2,\varphi}(\Omega,\R{3}) $ for any $  \varphi \in (0,\phi) $ and there is a subsequence $ u_{n_{k}} \to u $ in $ C^{2,\varphi} $. Since $ u_{n_k} $ converges in $ C^2 $  all partial derivatives of $ u_{n_k} $ converge in $ C^{0} $ and we can conclude that $ u $ satisfies \eqref{Reliability:Eq:LinEl}. 
	
	\noindent Now, we show that $ u $ is again an element of $ C^{2,\phi} $ even though it does not necessarily satisfy $ \Norm{u_{n_{k}}-u}{C^{k,\phi}} \to 0 $. Nevertheless, we can investigate the point wise convergence of this sequence and observe that $$ \lim_{k\to \infty}\frac{\partial^{\beta} u_{n_{k}}}{\partial x^{\beta}}(x) = \frac{\partial^{\beta} u}{\partial x^{\beta}}(x) ~~~\text{ where }~~~ \frac{\partial^{\beta}}{\partial x^{\beta}} = \frac{\partial^{|\beta|}}{\partial x^{\beta_1}\partial x^{\beta_2}\partial x^{\beta_3}} $$ for any $ x \in \overline{\Omega} $ and any multiindex $ \beta  \in \N{3}_{0} $ with $ \beta=2 $. Since $ \R{} \to \R{+}_{0},\, x \to |x| $ is continuous we thus obtain 
	\begin{align*}
	C^{\ast} \geq \lim_{k \to \infty}\frac{\Big|\frac{\partial^{\beta} u_{n_{k}}}{\partial x^{\beta}}(x) - \frac{\partial^{\beta} u_{n_{k}}}{\partial x^{\beta}}(x')\Big|}{|x-x'|^{\phi}}
	=\frac{\Big|\frac{\partial^{\beta} u}{\partial x^{\beta}}(x) - \frac{\partial^{\beta} u}{\partial x^{\beta}}(x')\Big|}{|x-x'|^{\phi}}
	\end{align*}
	for any paring $ \forall x \neq x' \in \Omega$. Thus, this inequality carries over to the supremum which exists on any subset of real numbers that is bounded from above. 
	
	\noindent The demanded Schauder estimate then follows directly from \citep{Agm64}.
\end{proof}


\chapter{Calculus in Banach Spaces} \label{Diff_Banach_Space}

In shape optimization many derivatives appear as derivatives of mappings from an open interval $  I$ to some Banach space $ Y $, where $ Y $  usually is a function space.

The differential calculus in $ \R{n} $ is well known. Analogously,  \textit{Gâteaux} and \textit{Fréchet derivatives} in infinite dimensions can be defined and we provide the results taken from \cite{Werner_Funkana} or \cite{Cheney} here.

\section{Gâteaux and Fréchet differentiability}

In this chapter let $ X $ and $ Y $ be Banach Spaces, $ U \subset X,\, U \neq \emptyset $ an open subset and $ F:U \subset X \to Y $ a functional. The normed vector space of linear operators from $ X $ to $ Y $ will be denoted by $ L(X,Y) $ and is equipped with the so called operator norm 
\[ \Norm{F}{L(X,Y)}=\sup_{\Norm{x}{X} \leq 1} \Norm{F(x)}{Y} = \sup_{\Norm{x}{X} = 1} \Norm{F(x)}{Y} =\sup_{x \in X \setminus \{0\}} \frac{\Norm{F(x)}{Y}}{\Norm{x}{X}}.  \]
The space of linear and continuous (i.e. bounded) operators is denoted by $ \mathcal{L}(X,Y) $.  

\begin{defn}[Gâteaux and Fréchet Differentiability]\citep[Def. III.5.1]{Werner_Funkana} \label{Diff_Banach_Space:Def:Gateaux/Frechet Differntial}$  $
	\begin{itemize}
		\item[a)] $ F $ is called \textit{Gâteaux differentiable (G-differentiable)}  at $ x_{0} \in U $ if there exists $ F'(x_{0}) \in \mathcal{L}(X,Y) $ such that
		\begin{equation}
		\lim_{h \to 0} \frac{F(x_{0} +hv)-F(x_{0})}{h} = F'(x_{0})[v] ~~~~ \forall v \in X.
		\end{equation}
		If $ F'(x_{0}) \in \mathcal{L}(X,Y)$ exists for every $x_{0} \in U$, then $ F $ is called \textit{Gâteaux differentiable on $ U $} and $ F':U \to \mathcal{L}(X,Y) $ is called the \textit{Gâteaux differential of $ F $}. Then we write $ D^{g}F $ instead of $ F' $. 
		\item[b)] $ F $ is called \textit{Fréchet differentiable (F-differentiable) at} $ x_{0} \in U $ if there exists  $ F'(x_{0}) \in \mathcal{L}(X,Y) $ such that
		\begin{equation}
		\lim_{h \to 0} \sup_{\Norm{v}{X} \leq 1}\Norm{\frac{F(x_{0} +hv)-F(x_{0})}{h} - F'(x_{0})[v]}{Y}=0 \, .
		\end{equation}
		If $ F'(x_{0}) \in \mathcal{L}(X,Y)$ exists for any $x_{0} \in U$, then $ F $ is called \textit{(Fréchet) differentiable on $ U $} and $ F':U \to \mathcal{L}(X,Y) $ is called \textit{(Fréchet) differential}. Then we write $ DF $ instead of $ F' $. 
	\end{itemize}
\end{defn}

\begin{lem}\label{Diff_Banach_Space:Lem:G_Diff_Linear_Map} \cite{Cheney}  Let  $ F \in \mathcal{L}(X,Y) $. Then $ F $ is Fréchet differentible at $ x_0 \in X $ with differential $  DF: X \to \mathcal{L}(X,Y),\, x_{0} \mapsto F  $. 
	
	Note that $ DF $ is constant and therefore continuous and that $ DF \neq F $! ($ DF:X \to \mathcal{L}(X,Y)  $ and $ F: X \to Y $!)
\end{lem}

\begin{proof}
	The assertion follows from $ F(x_{0} +hv) = F(x_0) +hF(v) $ for all $ v \in X $.
\end{proof}

\begin{exmp}\label{Diff_Banach_Space:Exmp:Diff_Gradient}
	The gradient $ \nabla:C^1(\R{n}) \to C(\R{n},\R{n}) $ is a linear and continuous differential operator with 
	\[ \Norm{\nabla}{L(C(\R{n}), C(\R{n},\R{n}))}=\sup_{\Norm{f}{C^1(\R{n})}\leq 1}\Norm{\nabla f}{C(\R{n},\R{n})}\leq \sup_{\Norm{f}{C^1(\R{n})}\leq 1}\Norm{ f}{C^1(\R{n})}= 1 .\]	
	We can thus apply Lemma \ref{Diff_Banach_Space:Lem:G_Diff_Linear_Map} and obtain the Fréchet differential of $ \nabla 
	$ by $ D\nabla(f_0)[f]$ $=\nabla f \in C(\R{n},\R{n}). $
\end{exmp}

\begin{lem}[Taylor Expansion] \cite[Lemma III.5.2]{Werner_Funkana}:\label{Diff_Banach_Space:Lem:Linear_Approx}
	Let $ F : X \to Y$. Then $ F $ is F-differentiable at $ x_{0}\in X $ if any only if there exists a linear and continuous operator $ F'(x_0) \in \mathcal{L}(X,Y)$ such that
	\begin{equation}
	F(x_{0}+v)=F(x_{0}) + F'(x_{0})[v] + r_{x_{0}}(v) \text{ where  } \frac{r_{x_{0}}(v)}{\Norm{v}{X}} \to 0 \text{ as } \Norm{v}{X} \to 0.
	\end{equation}
	In this case, $ F'(x_0)=DF(x_0) $.
\end{lem}

\begin{rem}\label{Diff_Banach_Space:Rem:Fréchet Diff.=>Contiuous}
	The last Lemma shows that Definition \ref{Diff_Banach_Space:Def:Gateaux/Frechet Differntial} b) is equivalent to the definition of the Fréchet differential in \citep{Cheney}:
	Therein $ F $ is called \textit{Fréchet differentiable at} $ x_{0} \in U $ if there exists a continuous map $ F'(x_{0}) \in \mathcal{L}(X,Y) $ such that
	\begin{equation}
	\lim_{v \to 0} \frac{\Norm{F(x_{0} +v)-F(x_{0})- F'(x_{0})[v]}{Y}}{\Norm{v}{X}}=0. 
	\end{equation}
\end{rem}

A consequence from the preceding Lemma is, that any Fréchet differentiable map is also continuous since
\[ 
\lim
_{x\to x_{0}}F(x) = F(x_0) + \lim_{x \to x_0 } DF(x_{0})[x-x_0] + \lim_{x \to x_0 } r_{x_{0}}(x-x_0) =F(x_0).  \]
Moreover, it is clear that any Fréchet differentiable map is Gâteaux differentiable. 

\begin{rem}
	The notions of G- and F-differentiability coincide with the definitions of differentiability of functions in the euclidean space - Gâteaux derivatives correspond to directional derivatives and Fréchet differentiability to total differentiability.
\end{rem}

Now we return to the general case of arbitrary Banach spaces.
The Gâteaux and the Fréchet differential are linear operators or more precisely:

\begin{lem}\citep[Thm. III.5.4 (a)]{Werner_Funkana}
	Let $ F,\, G:X \to Y $ G- (F-) differentiable. Then also $ F+G:X \to Y $ and $ \lambda F $, $ \lambda \in \R{} $ are G-(F-)differentiable with $ D^g(F+G)=D^gF+ D^gG $ and $ D^g\lambda F=\lambda D^gF $ or $ D(F+G)=DF+ DG $ and $ D\lambda F=\lambda DF $, respectively.
\end{lem}

\begin{lem}\label{Diff_Banach_Space:Lem:Product_rule_Gateaux}
	Let $ W,X,Y,Z $ be Banach spaces and $ F:X \to Y $, $ G:X \to Z $ G-differentiable and continuous at $ x_0\in X $. Suppose that there exists a product $$ \cdot:  Y \times Z \to W,\, (y, z) \mapsto yz  $$ such that $ \lim_{n \to \infty} y_nz_n = (\lim_{n \to \infty} y_n)( \lim_{n \to \infty } z_n)$ for sequences $ (y_n)_{n} \subset Y,\, (z_n)_{n} \subset Z $ if 
	all limits exist in $ W $.
	
	\noindent Then also $ F\cdot G:X \to W $ is G-differentiable and continuous at $ x_0 $ with $ D^g(F\cdot G)(x_0)[v]=F(x_0) \cdot D^gG(x_0)[v]+  D^gG(x_0)[v]\cdot G(x_0),\, v \in X  $.
\end{lem}

\begin{proof} The assertion follows from $$ \frac{(F G)(x_0+hv)-(FG)(x_0)}{h} =  F(x_0+hv)\frac{G(x_0+hv)-G(x_0)}{h}+  \frac{F(x_0+hv)-F(x_0)}{h}G(x_0).$$ 
\end{proof}

We already know  that any Fréchet differentiable map is Gâteaux differentiable.
The converse obviously is not true, but in analogy to the finite dimensional case the following holds: 

\begin{lem}\citep[Thm. III.5.4 (c)]{Werner_Funkana}\label{Diff_Banach_Space:Lem:Gateaux=>Frechet}
	Let $ F:X \to Y $ be G-differentiable on $ U \subset X $ and $ D^{g}F: U \to \mathcal{L}(X,Y) $ continuous. Then $ F $ is F-differentiable on $ U $ with $ D^gF=DF$.
	
	In this case we say that $ F:U \subset X: \to Y $ is continuously F-differentiable and write $ F \in C^1(U,Y) $ in analogy to continuous differentiability on $ \R{n} $. If $ F $ is only continuous we denote this by $ F \in C(U,Y) $.
\end{lem}

%

This connection between the Gâteaux and the Fréchet differential is helpful to illustrate the link to another notion of differentiability that will be needed in the later sections: Differentiability w.r.t. the strong (norm) topology on the Banach space $ X $:\\
\indent Let $ f:\R{} \to X, t \mapsto f(t) $ such that $ X $ is a Banach (or Hilbert) Space and $ I \subset \R{} $ and open interval.
The mapping $f $ is called differentiable w.r.t. the strong (norm) topology on $ X $ at $ t \in I $ if there exists $\left.\frac{d}{dt}f(t)\right\vert_{t=t_0} =\dot{f}(t_0)$ such that 
\[ \lim_{h \to 0}\Norm{\frac{f(t_0+h)-f(t_0)}{h} -\dot{f}(t_0)}{X}=0\, , \]
compare Definition \ref{Parameter_Dep_PDE:Defn: Differentiation on H}. This notion is equivalent to Gâteaux-differentiability: 	
\begin{itemize}
	\item[i)]  Let $f $ be differentiable w.r.t. the strong norm topology on $ X $ then $ f $ is  G-differentiable on $ I$ with $D^{g}u(t)[\alpha]=\alpha \dot{f}(t)$, $ \alpha \in \R{} $: For $ \alpha \neq 0 $
	\begin{equation*}
	\begin{split}
	\lim_{h \to 0}   \frac{\alpha (f(t_0+h\alpha)-f(t_0))}{\alpha h}= \lim_{\tilde{h} \to 0}  \alpha \frac{f(t_0+\tilde{h})-f(t_0)}{\tilde{h}} = \alpha\dot{f}(t_0)
	\end{split}
	\end{equation*}
	and for $ \alpha=0 $ we observe
	$ \lim_{h \to 0} \frac{f(t_0+h\alpha)-f(t_0)}{h} 
	= 0 =0\cdot\dot{f}(t_0)=0.$
	This suggests that $D^{g}u(t_0)\in L(\R{}, X)  $ is the multiplication operator $$M_{\dot{f}(t_0)}: \R{} \to X, \, \alpha \mapsto \alpha \dot{f}(t_0).$$
	This operator is an element of $ L(\R{},X) $ and continuous with 
	$\Vert M_{\dot{f}(t_0)} \Vert_{L(\R{},X)}= \Vert \dot{f}(t_0)\Vert_X.$ 
	Thus, the Gâteaux-differential is given by 
	$$ D^{g}f(t_0)[\alpha]=M_{\dot{f}(t_0)}(\alpha)=\alpha \dot{f}(t_0) . $$
	If $ f $ is conversely Gâteaux-differentiable, then  
	$D^{g}f(t_0)[1]= \lim_{h \to 0} \tfrac{f(t_0+h)-f(t_0)}{h}$ $=\dot{f}(t_0)$
	and
	$ D^{g}f(t_0)[\alpha]=\alpha D^{g}f(t_0)[1] = \alpha \dot{f}(t_0).$
	\item[ii)] If it is additionally supposed that the mapping $ I \to X, t \to \dot{f}(t) $ is strongly continuous on $ X $, then $ f $ is even Fréchet differentiable, because  $ t \mapsto  M_{\dot{f}(t)}$ is continuous then:
	\begin{align*}
	\Norm{D^gf(t)-D^gf(s)}{L(\R{},X)}
	&=  \Norm{M_{\dot{f}(t)}- M_{\dot{f}(s)}}{L(\R{},X)} 
	=  \sup_{|\alpha| \leq 1} |\alpha| \Norm{\dot{f}(t)-\dot{f}(s)}{X}\underset{s \to t}{\to} 0.
	\end{align*}
\end{itemize}

\begin{lem} \label{Diff_Banach_Space:Lem:Gateaux/Frechet_Diff_u^t}
	Let $ f:I \subset \R{}  \to  X$, $ I  $ an open interval and $ X $ a Banach space. 
	\begin{itemize}
		\item[i)] $ f $ is differentiable w.r.t. the strong topology on $ X $ if and only if f is G-differentiable. Then 
		$D^{g}f(t_0)[\alpha]=M_{\dot{f}(t_0)}(\alpha)=\alpha \dot{f}(t_0),\, \alpha \in \R{}.$
		\item[ii)] Suppose that i) is valid and assume that the mapping $ t \in I \mapsto \dot{f}(t)=D^{g}f(t_0)[1] $ $\in X $ is strongly continuous, then $ f $ is F- differentiable. 
	\end{itemize}
\end{lem}

\section{Chain rule and mean value theorems}

%
%

\begin{lem}[Chain Rule]\citep{Cheney}\label{Diff_Banach_Space:Lem:CR}
	Let $ X,Y,Z $ be Banach spaces and $ F:X \to Y $, $ G:Y \to Z $. If $ F $ is F-differentiable on $ U \subset X $ and $ G $ is F-differentiable on $ F(U) $, then $ G \circ F:X \to Z $ is F-differentiable on $ U $ with differential $ D(G \circ F )(x_0) =  DG\left( F(x_0)\right) \circ (DF(x_0)),\, x_0 \in U  $.
\end{lem}

\noindent The proof is analogous to the finite dimensional case, see \cite[Section 3.2. Thm.1]{Cheney}.

\pagebreak
\begin{rem}\label{Diff_Banach_Space:Rem:CR_I} The special case, when 
	$ G\in \mathcal{L}(Y,Z) $ is linear and continuous, naturally is included: Thus, if $ X,Y,Z $ are Banach spaces and $ F:X \to Y $, $ G:Y \to Z $ such that $ F $ is Fréchet differentiable on $ U \subset X $ and $ G $ is linear on $ F(U) $, then $ G \circ F:X \to Z $ is Gâteaux (Fréchet) differentiable on $ U $ with differential $ D^g(G \circ F )(x_0) = G\circ D^gF(x_0),\, x_0 \in U  $.
\end{rem}

\begin{exmp}\label{Diff_Banach_Space:Exmp:G_Diff_Integral}
	Let $ \Omega \subset \R{n} $, $ n \in \N{} $ be a bounded domain.
	The space $ (C(\overline{\Omega}), \Norm{.}{\infty}) $ of continuous functions on $ \overline{\Omega} $ is a Banach space. The mapping 
	\begin{align}
	I_{\Omega}: C(\overline{\Omega}) \to \R{},\,
	f \mapsto \int_{\Omega} f(x) \, dx
	\end{align}
	is a linear and continuous functional $ I_{\Omega} \in C(\overline{\Omega})' $ with operator norm  $ \Norm{I_{\Omega}}{C(\overline{\Omega})'}=|\Omega| $.
	Let $ X$ be a Banach space and $ F:X \to C(\overline{\Omega}) $ Gâteaux differentiable on $ U \subset X $. Then 
	\begin{align}
	I_{\Omega} \circ F: U \to  \R{},\,  
	f \mapsto  \int_{\Omega} F(f)(x)\, dx 
	\end{align} is Gâteaux differentiable on $ U $ according to Lemma \ref{Diff_Banach_Space:Lem:CR} and the remark above with 
	\[ D^g(I_{\Omega} \circ F)(f_0)[f] = \int_{\Omega} (D^gF(f_0)[f])(x) \, dx,\]
	for any $ f_0 \in U $, $ f \in X $.
	If $ F $ is Fréchet differentiable on $ U $ then $ I_\Omega \circ F $ also is and $ D^gF $ can be replaced by $ DF $.  
	Moreover, this example extends to surface integrals and continuous functions. 
	Under these conditions 
	\begin{align}
	I_{\Gamma}: C(\Gamma) \to \R{},\,
	f \mapsto \int_{\Gamma} f(x) \, dS
	\end{align}
	turns also out to be Gâteaux/Fréchet differentiable by analogous arguments. 
\end{exmp}


\begin{thm}[Mean Value Theorem I] \citep[Sec. 3.2, Theorem 2]{Cheney}\label{Diff_Banach_Space:Thm:MWS_G_I}
	Let $ X $ be a Banach space, $ U \subset X $ an open subset and $ F:U \subset X \to \R{} $ a real valued function.  
	Suppose that for the elements $ x_1,x_2 \in I $ also the line segment $ S=[x_1,x_2]:=\{x_1 +\lambda (x_2-x_1) \, | \, \lambda \in [0,1]\}  $ is contained in $ U $. 
	
	If $ F $ is continuous on $ S $ and F-differentiable on the open line segment $(x_1,x_2):=\{x_1 +\lambda (x_2-x_1) \, | \, \lambda \in (0,1)\} $, then for some $ \xi \in (x_1,x_2) $  
	\[ F(x_1)-F(x_2)=DF(\xi)[x_1-x_2]. \]
\end{thm}

\begin{thm}(Mean Value Theorem II)\label{Diff_Banach_Space:Thm:MWS_G_II}\\
	Let $ F:X \to Y $ be G-(F-) differentiable on $ U \subset X $. Let $ x_0 \in U $, $ v \in X $ be fixed such that $ S:=(x_0,x_0+v)\subset U $. If $ F $ is continuous on $ S $ and $ L \in \mathcal{L}(X,Y) $, then 
	$$ \displaystyle \Norm{F(x_0+v) - F(x_0)-L(v)}{Y} \leq \sup_{s \in S } \Norm{D^gF(s)-L}{L(X,Y)} \Norm{v}{X}. $$
	
\end{thm}

\begin{proof} 
	%
	For any element $ y $ of a normed space $ Y $  there exists a functional $ y^{\ast}  $ in the dual space $ Y' $ with $ \Norm{y^{\ast}}{Y'}=1 $ and $ y^{ \ast}(y)=\Norm{y}{Y} $ (\citep[Corollary III.1.6.]{Werner_Funkana}). We apply this to $ y= F(x_0+v)-F(x_0)- L(v) $ and investigate the mapping
	\begin{align*}
	\varphi: [0,1] \to \R{},\, 
	s &\mapsto y^{\ast}(F(x_0+s v)-F(x_0)-s L(v)).
	\end{align*} 
	Especially, $ \varphi(1)= y^{\ast}(F(x_0+ v)-F(x_0)- L(v))=\Norm{F(x_0+ v)-F(x_0)- L(v)}{Y} $.
	Since $ y^{\ast} $ is linear and continuous
	\begin{align*}
	\,\lim_{h \to 0} \frac{\varphi(s + h)-\varphi(s)}{h} 
	=&\,\lim_{h \to 0} y^{\ast}\left(\frac{F(x_0+s v+hv) -F(x_0+s v) }{h} - L(v)\right) \\
	=&\, y^{\ast}\left( \lim_{h \to 0 } \frac{F(x_0+s v+hv) -F(x_0+s v) }{h} - L(v)\right)\\
	=& \, y^{\ast}\left(D^gF(x_0+s v)[v] - L(v)\right)
	\end{align*}
	and hence $ \varphi'(s)= y^{\ast}\left(D^gF(x_0+s v)[v] - L(v)\right)$ is the derivative of $ \varphi $ at $ s \in (0,1) $. Moreover, $\varphi $ is continuous on [0,1] because $ y^{\ast} $ is continuous and $ F $ is continuous on $ S $.
	This gives us the possibility to apply the classical mean value theorem to $ \varphi $: $ \exists \mathcal{ \xi} \in (0,1) $ such that
	$|\varphi(1)-\varphi(0)| =|\varphi'(\mathcal{ \xi})||1-0|  $
	which is equivalent to  
	\[ \Norm{F(x_0 + v)-F(x_0)-L(v)}{Y}=|y^{\ast}\left(D^gf(x_0+\xi v)[v] -L(v)\right)|. \]
	Accordingly, for any $ \xi \in [0,1] $,
	\begin{align*}
	& \Norm{F(x_0 + v)-F(x_0)-L(v)}{Y}
	=|y^{\ast}\left(D^gF(x_0+\xi v)[v]-L(v)\right)|  \\
	& \leq \Norm{y^{\ast}}{Y'}\Norm{D^gF(x_0+\xi v)[v]-L(v)}{Y} 
	\leq \Norm{D^gF(x_0+\xi v)-L}{\mathcal{L}(X,Y)} \Norm{v}{X}.
	\end{align*}
\end{proof}
\vspace{\parindent}

\section{Partial derivatives}

If $ X_1,...,X_n $ are Banach spaces, then also the Cartesian product $ X:=\prod_{i=1}^{n} X_i $ with Norm $ \Norm{(x_1,\ldots,x_n)}{X} =\sum_{i=1,...,n} \Norm{x_i}{X_i} $ is a Banach space.

In case $ U\subset\prod_{i=1}^{n}  X_i  $ is open, $ p\in U $ and $ F:U \to Y $ is a functional with values in some Banach space $ Y $ then we can find an environment $ U_{p_i} \subset X_i $ of $ p_i $ such that the point $  (p_1,...,p_{i-1},x,p_{i+1},...,p_n) \in U   $ for any $ x \in U_{p_i}$. Afterwards we can define the mappings $$ F_{p,i}:U_{p_i} \subset X_i \to Y,\, x \mapsto F(p_1,...,p_{i-1},x,p_{i+1},...,p_n) $$
which satisfy $ F_{p,i}(p_i) =F(p) $.
This justifies the following definition: 

\begin{defn}[Partial derivatives]\citep[p. 339 f]{Heuser_Analysis2}
	Let $ X_1,...,X_n,\,Y $ be Banach spaces, $ U\subset X:=\prod_{i=1}^{n} X_i $ open and $ F:U \to Y $. 
	\begin{itemize}
		\item[a)] The functional $ F $ is called partially
		Gâteaux (Fréchet) differentiable at $ p \in U $ regarding the $ i $-th component if $$ F_{p,i}: x \in U_{p_i} \mapsto F(p_1,...,p_{i-1},x,p_{i+1},...,p_n)  $$
		\begin{itemize}
			\item[i)]  is  Gâteaux differentiable at $ p_i$,
			i.e. the limit value
			\[ \lim_{h \to 0} \frac{F_{p,i}(p_i+hv_i) - F_{p,i}(p_i)}{h} =D^gF_{p,i}(p_i)[v_i] ,\, v_i\in X_i \]
			exists such that $ D^gF_{p,i}(p_i) \in \mathcal{ L}(X_i,Y)$. 
			\item[ii)] is  Fréchet differentiable at $ p_i $, i.e. there exists $ DF_{p,i}(p_i) \in $ $ \mathcal{ L}(X_i,Y)$ such that 
			\[ \lim_{\Norm{v_i}{X_i}\to 0} \frac{\Norm{F_{p,i}(p_i+v_i) - F_{p,i}(p_i) -DF_{p,i}(p_i)[v_i]}{Y} }{\Norm{v_i}{X_i}} =0.  \]
		\end{itemize}
		Then the $ i $-th partial Gâteaux (Fréchet) derivative is given by the mapping $$ 
		\frac{\partial F}{\partial x_i}(p_1,...,p_n):=D^gF_{p,i}(p_i) \in \mathcal{L}(X_i,Y)
		$$ 
		respectively,  
		$$  \frac{\partial F}{\partial x_i}(p_1,...,p_n):=DF_{p,i}(p_i)\in \mathcal{L}(X_i,Y).$$
		\item[b)] $ F $ is called partially Gâteaux (Fréchet) differentiable if $ F $ is partially Gâteaux (Fréchet) differentiable regarding any component $ i = 1,...,n $.
	\end{itemize}
\end{defn}

We now establish the link between the partial derivatives of $ F:X \to Y $, $X= \prod_{i=1}^{n} X_i$ for Banach spaces $ X_i,\, i=1,\ldots,n $ and $ Y $ and the Gâteaux (Fréchet) derivative of $ F $:

\begin{lem}\label{Diff_Banach_Space:Lem:Part_Deriv_Frechet}
	Let $ F: U\subset X \to Y $ be G-(F-) differentiable at $ p=(p_1,...,p_n) \in U $. Then $ F $ is partially G-(F-) differentiable at $ p  $ with 
	\begin{align*}
	DF(p)[v]&=\sum_{i=1}^{n}\frac{\partial F}{\partial x_i}(p)[v_i],\, v=(v_1,...,v_n) \in X.
	\end{align*}
\end{lem}

\begin{proof}
	The assertion follows from $ F(p + h(0,\ldots,v_i,\ldots,0) ) = F(p_1,\ldots,p_i+hv_i,\ldots,p_n) =  F_{p,i}(p_i+hv_i) $ and $ F(p) = F_{p,i}(p_i) $ and the linearity of $  DF(p)[v] $.
\end{proof}

\begin{lem}\label{Diff_Banach_Space:Lem: Cont_Part_Deriv_Gateaux}
	Let $ X_1,..,X_n,\, Y $ be Banach spaces and $ F:U \subset X=\prod_{i=1}^{n}X_i \to \R{} $, $ p \in U $.
	Suppose that the partial G-derivatives $ \frac{\partial F}{\partial x_i} $ exist in an environment $ \mathcal{U}_{p} $ of $p$ such that they are continuous on $ U_{p} $ . Then the F-differential of $ F $ exists at $ p $ and can be calculated by
	\[ DF(p)[v]=\sum_{i=1}^{n}\frac{\partial F}{\partial x_i}(p)[v_i],\, v=(v_1,...,v_n) \in X.\]
\end{lem}

The proof is absolutely analogous to the finite dimensional case.


\section{Properties of parameter integrals} \label{Diff_Banach_Space:Sec:Parameter_Integrals}

To calculate so called \textit{shape derivatives} (see, Chapter \ref{Shape_Grad_LinEl}), we have to analyze under which conditions mappings
\begin{align*}
\R{} \to \R{}, t \mapsto \int_{\Omega} (\mathcal{ F}\circ u(t))(x) \, dx \text{ or } \R{} \to \R{}, t \mapsto  \int_{\Gamma} (\mathcal{ F}\circ u(t))(x) \, dS
\end{align*}
are differentiable w.r.t. the parameter $ t \in I \subset \R{} $ where $ I $ is some open interval. The functions $ u(t),\, t \in I $ are assumed to be mappings $ u(t):\Omega \to \R{n} $ from a bounded domain $ \Omega $ to $ \R{n} $. We suppose that $ \Omega  $ is a $ C^1 $ domain if we consider surface integrals and that $ \mathcal{F} $ is sufficiently smooth. 

There are at least two approaches to this problem: The 'classical' rules for parameter integrals based on the theorem of Lebesgue and another one via Gâteaux or Fréchet differentials, confer Example \ref{Diff_Banach_Space:Exmp:G_Diff_Integral}. In the letter sense
\[ 
\int_{\Omega} (\mathcal{ F}\circ u(t))(x) \, dx =( I_{\Omega} \circ F_{\mathcal{ F}} \circ u)(t)  
\] if we define $  F_{\mathcal{ F}}: C(\overline{ \Omega},\R{n}) \to C(\overline{ \Omega}),\, f \mapsto \mathcal{ F} \circ f .$

Since this work treats Fréchet or Gâteaux derivatives of mappings
$ I \to C^{k,\phi}(\Omega),\, t \mapsto u(t)$
from an Interval into some Hölder-space in Section \ref{Ex_Shape_Deriv_LinEl:Sec: hölder material derivatives} the approach using Fréchet or Gâteaux derivatives is natural in some way. 

Nevertheless, we also establish the "classical" rules for parameter integrals since they will require weaker hypothesis and will be needed in Section \ref{Ex_Shape_Deriv_LinEl:Sec:H1 material derivatives}.
\begin{prop}[Continuity of Parameter Integrals] \label{Diff_Banach_Space:Cor:Continuity of parameter Integrals} $  $\\
	Suppose that $ X $ is a Banach space, $ U \subset X $ is an open subset and $ (Z,\mathcal{Z},\mu) $ a measurable space. Suppose that $ f:X \times Z \to \R{}$. 
	\begin{enumerate}
		\item $ f(x, \cdot)\in L^{1}(Z,\mathcal{Z},\mu)$ for any $ x \in U $, 
		\item $ f(\cdot,z):Y \to \R{},\, x \mapsto f(x,z) $ is continuous on $ U $ for almost every $ z \in Z $
		\item there is an integrable function $ \mathfrak{m}\in L^{1}(Z,\mathcal{Z}, \mu) $ such that for any $ x\in U$ and almost every $ z \in Z $: $|f(x,z)| \leq \mathfrak{m}(z). $
	\end{enumerate}
	Then \[ U \to \R{},\, x \mapsto \int_{Z}f(x,z)d\mu(z) \] is continuous.
\end{prop}

\begin{proof}
	Let $ (x_j)_{j \in \N{}} \subset U  $ be a sequence with $ x_j \to x_0,\,  j \to \infty$. Then apply Lebesgue's Theorem of Dominated Convergence (\cite[Sec. 8.6. Thm. 2 ]{Cheney}) to the sequence $ f_{j}:=f(x_j,\cdot) \in L^1(Z,\mathcal{Z},\mu)  $ and use the continuity of $ f(\cdot,z):Y \to \R{},\, x \mapsto f(x,z) $.
\end{proof}

The following Proposition allows to interchange the order of integrals and partial Fréchet derivatives under weak conditions, as they will be needed in Chapter \ref{Ex_Shape_Deriv_LinEl}.

\begin{prop}[Differentiability of Parameter Integrals]\label{Diff_Banach_Space:Prop:Diff_Parameter_Integrals}
	Assume that $ X $ is a Banach space, $x_0 \in  U \subset X $ is an open subset and $ (Z,\mathcal{Z},\mu) $ a measurable space. Let $ f:X \times Z \to \R{}$. If
	\begin{enumerate}
		\item $ f(x, \cdot) $ is $ d\mu $- integrable on $ Z $ for any $ x \in U$.
		\item the partial F-derivative $ \frac{\partial f}{\partial x}(x_0,z)[v] $ at $ x_0 $ in direction of $ v \in X $ exists for a.e. $ z \in Z $ such that $ \frac{\partial f}{\partial x}(x_0,\cdot)[v]\in L^{1}(Z,\mathcal{Z},\mu) $, 
		\item  there is $  \mathfrak{m}\in L^{1}(Z,\mathcal{Z}, \mu) $ s.t. for any $ x\in U $ and $ \mu $-a.e. $ z \in Z $: $ \Norm{\frac{\partial f}{\partial x}(x,z) }{X'}\leq \mathfrak{m}(z). $
	\end{enumerate}
	Then $ U \to \R{},\, x \mapsto \int_{Z} f(x,z)\, d\mu(z)  $ is F-differentiable with  \[ D\left(\int_{Z}f(\cdot,z)d \mu(z) \right)(x_0)[v] = \int_{Z} \frac{\partial f}{\partial x}(x_0,z)[v] \, d\mu(z).\]
\end{prop}

\begin{proof}
	We have to show that the F-differential of the operator $F:X \to \R{},\, x \mapsto \int_{Z} f(x,z)\, d \mu(z) $ is given by \[ F'(x_0): X \to \R{},\, v \mapsto \int_{Z} \frac{\partial f }{\partial x} (x_0,z)[v] \, d\mu(z).\]
	Therefor, let $ (v_{j})_{j \in N} \subset X$ be a sequence such that $ v_{j} \to 0,\, j \to \infty $ in $ X $ and apply Lebesgue's Theorem to $ f_{j}:=  \frac{f(x_0 +v_j, \cdot) -  f(x_0,\cdot) - \frac{\partial f}{\partial x} (x_0,\cdot)[v_j]}{\Norm{v_{j}}{X}}  \in L^{1}(Z,\mathcal{Z},\mu)$.
	Now choose $ j $ large enough such that $[x_0,x_0+v_j] \subset U  $, then we receive by application of \ref{Diff_Banach_Space:Thm:MWS_G_II}
	\begin{align*}
	\left|f(x_0 + v_j,z) - f(x_0,z) - \frac{\partial f}{\partial x} (x_0,z)[v_j]\right| 
	&\leq \sup_{x \in [x_0,x_0+v_j]} \Norm{ \frac{\partial f}{\partial x} (x,z) - \frac{\partial f}{\partial x} (x_0,z)}{X'} \Norm{v_{j}}{X} \\
	& \leq 2 \mathfrak{m}(z) \Norm{v_{j}}{X} \text{ a.e. on } Z.
	\end{align*}
	Thus $ |f_{j}(z)|\leq \mathfrak{m}(z) $ for $ \mu $-a.e. $ z \in Z $. Since $ \frac{\partial f}{\partial x}(x_0,z)[v_j] $ is the partial F-derivative of $ f(\cdot, z) $ in direction of $ v_{j} \in X $ we obtain $ f_{j}(z) 
	\to f^{*}(z)=0,\, j \to \infty \text{ pointwise a.e. on } Z  $
	and thus 
	\begin{align*}
	&\lim_{j \to \infty}  \frac{\left|\int_{Z} f(x_0 +v_{j},z) -  f(x_0,z) - \frac{\partial f }{\partial x} (x_0,z)[v_{j}] \, d\mu(z) \right|}{\Norm{v_j}{X}}\\
	\leq & \lim_{j \to \infty} \int_{Z}\frac{ \left|f(x_0 +v_j,z) -  f(x_0,z) - \frac{\partial f }{\partial x} (x_0,z)[v_j] \right| }{\Norm{v_j}{X}} \, d\mu(z) =0.
	\end{align*}
	Finally, one has to show that $ F'(x_0) $ is linear and continuous but at least linearity is clear.
	Moreover the continuity of $ F'(x_0) $ is implied by
	\begin{align*}
	|F'(x_0)[v]| 
	& \leq \int_{Z} \left|\frac{\partial f }{\partial x} (x_0,z)[v] \right|\, d\mu(z)   
	\leq  \int_{Z} \Norm{\frac{\partial f }{\partial x} (x_0,z)}{X'} \hspace*{-1mm}\Norm{v}{X} d\mu(z)  \leq \Norm{v}{X} \int_{Z} \hspace*{-1mm} \mathfrak{m}(z)\, d\mu(z). 
	\end{align*}
\end{proof}

\noindent The following Lemma summarizes our conclusions derived from Lebesgue's Theorem:

\begin{lem}(Rules for Parameter Integrals)\label{Diff_Banach_Space:Lem:Cont_Diff_Parameter_Integrals} 
	Suppose that $ t_0 \in \R{} $, $ \Omega \subset \R{n} $ is a Lebesgue measurable set and $ f:\R{} \times \R{n} \to \R{}$. Assumed there is an open interval $ t_0 \in I \subset \R{} $ such that
	\begin{itemize}
		\item[a)] 
		$ f(t,\cdot) \in L^{1}(\Omega)$ for any $ t \in I $, there exists a function $ f^{\ast} $ such that $ \lim_{t \to t_0}f(t,x)$ $ = f^{\ast}(x)$ a.e.\ in $ \Omega $, and there is $ \mathfrak{m}\in L^{1}(\Omega) $ such that for any $ t\in I $: $ \left|f(t,x)\right| \leq \mathfrak{m}(x) \text{ a.e. }$
		then $$ \lim_{t \to t_0}\int_{\Omega}f(t,x)\,dx = \int_{\Omega} f^{\ast}(x) \,dx  .$$
		\item[] If additionally $ t \mapsto f(t,\cdot) $ is continuous at $ t_0 $ then, also $I \to \R{},\, t \mapsto \int_{\Omega} f(t,x) \, dx  $
		is continuous at $  t_0 $.
		\item[b)] Suppose that,  $ f(t,\cdot) \in L^{1}(\Omega)$ for any $ t \in I $, $ f(\cdot,x):\R{} \to \R{}  ,\, t \mapsto f(t,x)$ is differentiable in $ t_0 $ for almost any $ x \in \Omega $ such that $ \frac{d f }{d t}(t,.) \in L^1(\Omega)$,
		and there is a function $ \mathfrak{m}\in L^{1}(\Omega) $ s.t. for any $ t \in I $: $  \left|\frac{d f}{d t}(t,x)\right| \leq \mathfrak{m}(x) \text{ a.e. on } \Omega $.
		Then \[ \left.\frac{d}{d t}\right\vert_{t =t_0}\int_{\Omega}f(t,x)\,dx  = \int_{\Omega} \left.\frac{d f}{d t}(t,x)\right\vert_{t=t_0}\, dx.\]
		\item[] If additionally $ I \to \R{},\, t \mapsto \frac{d f}{d t}(t,x)$  is continuous at $ t_0 $, then also $ t \mapsto \int_{\Omega}f(t,x)\,dx  $ is continuously differentiable at $ t_0 $.
	\end{itemize}
\end{lem}

\begin{rem}
	Of course this statement also holds on $ \partial \Omega $ if $ \Omega \subset \R{n} $ is a bounded $ C^1 $-domain with boundary $ \partial\Omega $. Then $ L^1(\Omega) $ has to be replaced by $ L^{1}(\partial\Omega) $ and the domain integral by a boundary integral.
\end{rem} 

\pagebreak
Even though the assumptions are stronger than those requested
in Proposition \ref{Diff_Banach_Space:Prop:Diff_Parameter_Integrals} and \ref{Diff_Banach_Space:Lem:Cont_Diff_Parameter_Integrals}, the following lemma will be very helpful in the progress of this thesis since the assumptions are easy to check: 

\begin{lem}\label{Diff_Banach_Space:Lem:Gâteaux_diff_Integral}
	\begin{itemize}
		\item[i)] Suppose that $ I \to C(\overline{\Omega}),\,t \mapsto u(t) $  is G-differentiable. Then $ I\to \R{},\, t \mapsto \int_{\Omega} u(t)(x) \, dx $ is G-differentiable with \[ \frac{d}{dt}\int_{\Omega} u(t)\, dx =\int_{\Omega} \frac{d}{dt}u(t) \, dx .\] 
		\item[ii)] Suppose that $ I \to C(\Gamma),\,t \mapsto v(t) $ is G- differentiable. Then $I \to\R{},\, t \mapsto \int_{\Gamma} v(t) \, dA $ is G-differentiable  with \[ \frac{d}{dt}\int_{\Gamma} v(t) \, dx =\int_{\Gamma} \frac{d}{dt}v(t) \, dx .\] 
	\end{itemize}
\end{lem}

\begin{proof}
	We rewrite $ \int_{\Omega} u(t)(x) \, dx = (I_{\Omega} \circ u)(t) $  and  $ \int_{\Gamma} v(t)(x) \, dS = (I_{\Gamma} \circ u)(t) $. Then the assertion follows from Example \ref{Diff_Banach_Space:Exmp:G_Diff_Integral} and chain rule \ref{Diff_Banach_Space:Lem:CR}.
\end{proof}
\vspace{\parindent}

\section{Derivatives of composed domain and boundary functionals}

Analogously to Example 4 in Section 3.1. of \cite{Cheney} we can show the following more general result:

\begin{lem}\label{Diff_Banach_Space:Lem:Fréchet Differential F_F}
	Let $  \mathcal{F} \in C^1(\R{n}) $, $ F_{\mathcal{ F}}: C(M,\R{n}) \to C(M),\, f \mapsto \mathcal{ F} \circ f  $ be defined as above and $M =\overline{\Omega}  $ or $ M=\Gamma $. Then $ F_{\mathcal{ F}}$ is F-differentiable at $ f_0 \in C(M,\R{n}) $ with \[ DF_{\mathcal{F}}(f_0)[f]=\langle (\nabla \mathcal{F} \circ f_0), f \rangle \in C(M) . \]  
\end{lem}

\begin{proof}
	Since $ \mathcal{F} \in C^1(\R{n}) $, it is F-differentiable with 
	$ D\mathcal{F}(x_0)[v]=\langle \nabla \mathcal{F}(x_0),v\rangle$ for any $ x_0 \in \R{n},\,v \in \R{n} $ and $ \R{n} \to \R{n},\, x \mapsto \nabla \mathcal{F}(x) $ is continuous.
	Now, let $ f_0,\, f \in C(M,\R{n}) $ and $ x \in M $ be arbitrary. Then $ F_\mathcal{ F}(f_0 + f),\, F_\mathcal{ F}(f_0) \in C(M) $ are scalar functions and  
	\begin{align*}
	\left( F_{\mathcal{F}}(f_0+f)- F_{\mathcal{F}}(f_0)\right)(x)&= \mathcal{F}(f_0(x)+f(x)) - \mathcal{F}(f_0(x)) \\
	&= \langle \nabla \mathcal{F}(f_0(x)+\lambda(x)f(x)), f(x) \rangle 
	\end{align*}
	for some $ \lambda(x) \in (0,1)  $ by the classical mean value theorem on $ \R{n} $ applied to $ \mathcal{ F} $ and the
	\pagebreak
	
	\noindent points $ f_0(x) \in \R{n} $ and $ f(x) \in \R{n} $. Accordingly,
	\begin{align*}
	&\frac{\Norm{F_{\mathcal{F}}(f_0+f)- F_{\mathcal{F}}(f_0)- \langle\nabla \mathcal{F} \circ f_0,f \rangle}{\infty} }{\Norm{f}{\infty}}
	\\
	\leq & \sup_{x \in M} \frac{\Norm { \nabla \mathcal{F}(f_0(x)+\lambda(x)f(x)) - \nabla \mathcal{F}(f_0(x)) }{\R{n}} \Norm{f(x)}{\R{n}}}{\Norm{f}{\infty}}
	\\
	\leq & \Norm { \nabla \mathcal{F}\circ (f_0+\lambda(\cdot) f) - \nabla \mathcal{F}\circ f_0) }{\infty}.
	\end{align*}
	
	\noindent Combining $ f_0 + \lambda f \overset{unif.}{\to}f_0 $ when $ \Norm{f}{\infty} \to 0 $ and the continuity of $ \nabla \mathcal{F} $ we obtain $\Norm{\nabla \mathcal{F}\circ (f_0+\lambda f) - \nabla \mathcal{F} \circ f_0}{\infty} \to 0 \text{ if }  \Norm{f}{\infty} \to 0.$
	It is clear that $ f \mapsto \langle \nabla\mathcal{ F} \circ f_0,f \rangle $ is linear from $ C(M,\R{n}) \to C(M) $ and the continuity follows from 
	$$  \Norm{\langle \nabla\mathcal{ F} \circ f_0, \cdot \rangle}{\mathcal{L}(C(M,\R{n}),C(M))}\leq \sup_{\Norm{f}{\infty} \leq 1} \Norm{\nabla \mathcal{ F} \circ f_0}{\infty}\Norm{f}{\infty} = \Norm{\nabla \mathcal{ F} \circ f_0}{\infty} < \infty.$$ 
\end{proof}

\begin{lem}\label{Diff_Banach_Space:Lem: Fréchet differnetial F_nach_f mehrdim } 
	Let $ n_1, \ldots,n_k\in\N{} $, $ \mathcal{F} \in C^1(\prod_{i=1}^{k}\R{n_i}) $ a continuously differentiable scalar function. 
	Then
	$F_{\mathcal{ F}}: \prod_{i=1}^{k}C(M,\R{n_i}) \to C(M),
	f=(f_{1},\ldots,f_{k}) \mapsto  \mathcal{F}\circ f$
	is F-differentiable at $ g^{(0)}=(g^{(0)}_1,\dots,g^{(0)}_k) $ with 
	\begin{align*}
	DF_{\mathcal{ F}}(g^{(0)})[g] = \sum_{i=1}^{k} \left\langle  \frac{\partial \mathcal{F}}{\partial z_i} \circ g^{(0)} , g_i\right\rangle_{\R{n_i}}
	\end{align*} 	
	in direction of $ g=(g_{1},\ldots,g_{k}) 
	\in \prod_{i=1}^{k}C(M,\R{n_i}) $.
\end{lem}

\begin{proof} Set $ N:=\sum_{i=1}^{k}n_i $. Then
	$\prod_{i=1}^{k}C(M,\R{n_i})\equiv C(M,\R{N} ) $ and therefore this case is implied by  Lemma \ref{Diff_Banach_Space:Lem:Part_Deriv_Frechet} and  Lemma \ref{Diff_Banach_Space:Lem:Fréchet Differential F_F}.
\end{proof}

\begin{lem}\label{Diff_Banach_Space:Lem:d/dt F nach u(t)}
	Let $ I \subset \R{} $ be an open interval $ u_i: I \to C(M,\R{n_i})$, $ n_i \in \N{},\, i=1,...,k $ F-differentiable. Set $ u=(u_1,\ldots,u_k) $ and let $ \mathcal{ F}:\, \prod_{i=1}^{k}\R{n_i} \to \R{},\, z=(z_1,\ldots,z_k) \mapsto \mathcal{ F}(z_1,\ldots,z_k)  $ be an element of  $ C^1(\prod_{i=1}^{k}\R{n_i}) $. Then 
	\begin{align}
	I \to C(M),\, t &\mapsto \mathcal{ F} \circ u(t) =\mathcal{ F}(u(t)(.))
	\end{align} 
	is differentiable on $ I $ with  differential
	\begin{align}
	\frac{d}{dt} \mathcal{ F}\circ u(t)=\sum_{i=1}^{k}\left\langle \frac{ \partial \mathcal{ F}}{\partial z_i} \circ u(t),\frac{d}{dt}u_i(t) \right\rangle\,.
	\end{align}	
\end{lem}

\begin{proof}
	For any $ t \in I $: $ \mathcal{ F} \circ u(t)=F_{\mathcal{ F}}(u(t)) $. 
	Then chain rule \ref{Diff_Banach_Space:Lem:CR}, Lemma \ref{Diff_Banach_Space:Lem:Part_Deriv_Frechet} and the previous lemma imply 
	\begin{align*}
	\frac{d}{dt} \mathcal{ F}\circ u(t)= \frac{d }{d t}F_{\mathcal{ F}}(u(t))= &~\sum_{i=1}^{k}\dfrac{\partial F_{\mathcal{ F}}}{\partial f_i}(u(t))\left[\frac{d}{dt}u_i(t)\right]   = \left\langle\frac{\partial \mathcal{F}}{\partial z_i}\circ u(t) ,\frac{d}{dt}u_i(t)\right\rangle.
	\end{align*}
\end{proof}
\vspace*{-1ex}
A special case of this situation appears when $ \mathcal{F} $ depends explicitly on the prameter $ t $, i.e. $ \mathcal{F}: I \times\prod_{i=1}^{k}\R{n_i} \to \R{} $ and $ u $ is given as above. Then we can consider the mapping
\begin{align}
I \to C(M),\, t \mapsto \mathcal{F}(t,u(t)(.) )
\end{align}
and its derivative by $ t $. In this view we drive the following general statement:
\begin{lem}\label{Diff_Banach_Space:Lem: d/dt F (t,u(t))}
	Suppose that $ F: \R{} \times X \to \R{},\, (t,x) \mapsto F(t,x) $ and $ u:\R{} \to X, t \mapsto u(t) $ are F-differentiable on an open interval $ I $. Then the map $ \R{}\to \R{}, t \mapsto F(t,u(t)) $ is differentiable on $ I $ and the differential is given by
	\begin{equation}
	\frac{d }{d t}F(t,u(t)) = \dfrac{\partial F}{\partial t}(t,u(t))[1]+\dfrac{\partial F}{\partial x}(t,u(t))\left[\frac{d u}{d t}(t)\right]  \, \forall t \in I.
	\end{equation}
	
\end{lem}

\begin{proof}
	We set $ U: \R{} \to \R{} \times Z \, t \mapsto (t,u^t)  $. The F-differential of $ U $ at $ \alpha = 1  $ is given by 
	\[DU(t)[1]= \frac{d}{dt}U(t)= \left(1,\frac{du}{dt}(t)\right).\] Then the differential can be calculated by chain rule Lemma \ref{Diff_Banach_Space:Lem:CR} and Lemma \ref{Diff_Banach_Space:Lem:Part_Deriv_Frechet}:
	\begin{equation}
	\begin{split}
	D(F  \circ U)(t_0)[1]&=
	DF(U(t))\left[ DU(t)[1]\right]  = DF(U(t))\left[\left(1 , \frac{du}{dt}(t)\right)\right]\\
	&=\dfrac{\partial  F}{\partial t}(t,u(t))[1] + \frac{\partial F}{\partial x}(t,u(t))\left[ \frac{du}{dt}(t)\right].
	\end{split}
	\end{equation}
\end{proof} 
\vspace*{-1ex}
This implies the following:

\begin{lem}\label{Diff_Banach_Space:Lem: d/dt F (t,u(t)) mehrdim}
	Let $ I \subset \R{} $ be an open interval, $ u_i: I \to C(M,\R{n_i})$, $ n_i \in \N{},\, i=1,...,k $ F-differentiable and set $ u= (u_1,\ldots,u_k) $. Moreover, suppose that $ \mathcal{F}: I \times\prod_{i=1}^{k}\R{n_i} \to \R{} $ is differentiable. Then the mapping 
	\begin{align}
	I \to C(M),\, t \mapsto \mathcal{F}(t,u(t)(.) )
	\end{align}
	is differentiable with differential 
	\begin{align}
	\frac{d}{dt}\mathcal{ F} (t,u(t)(.)) 
	= \frac{\partial \mathcal{ F}}{\partial t}(t, u(t)(.)) + \sum_{i=1}^{k}\left\langle \frac{ \partial \mathcal{ F}}{\partial z_i}(t,u(t)(.)),\frac{d u_i}{dt}(t) \right\rangle. 
	\end{align}
\end{lem}

\begin{proof}
	Set $ X:=\prod_{i=1} C(M,\R{n_i})  $ and combine \ref{Diff_Banach_Space:Lem: Fréchet differnetial F_nach_f mehrdim } and  \ref{Diff_Banach_Space:Lem: d/dt F (t,u(t))}.
\end{proof}

\begin{lem}\label{Diff_Banach_Space:Lem:d/dt int f(t)F(t,u(t))}
	Let all hypotheses of Lemma \ref{Diff_Banach_Space:Lem: d/dt F (t,u(t)) mehrdim} be satisfied and suppose in addition that $ f_{v}:\R{} \to C(\overline{\Omega}),\, f_{s}:\R{} \to C(\Gamma) $
	are F-differentiable at $ t$ such that their derivatives are bounded in a neighborhood of $ t $. Then the differentials of
	\begin{align*}
	J_v:\R{} &\to \R{},~
	t\mapsto \int_{\Omega}f_{v}(t) \mathcal{F}(t,u(t)(x))\, dx 
	\intertext{ and }
	J_s:\R{} &\to \R{},~
	t\mapsto \int_{\Gamma} f_{s}(t) \mathcal{ F}(t,v(t)(x))\, dS
	\end{align*}
	at $t\in I $ are given by
	\begin{align*}
	\frac{d}{dt} J_v(t)=& \int_{\Omega} f_{v}(t) \left[\frac{\partial \mathcal{ F}}{\partial t}(t, u(t)(.)) + \sum_{i=1}^{k}\left\langle \frac{ \partial \mathcal{ F}}{\partial z_i}(t,u(t)(.)),\frac{d u_i}{dt}(t) \right\rangle\right] +\frac{df_{v}}{dt}(t)   \mathcal{ F}(t,u(t)(.)) 
	\, dx\, , \\[1em]
	\frac{d}{dt} J_s(t) =& \int_{\Gamma}  f_{s}(t) \left[\frac{\partial \mathcal{ F}}{\partial t}(t, v(t)(.)) + \sum_{i=1}^{k}\left\langle \frac{ \partial \mathcal{ F}}{\partial z_i}(t,v(t)(.)),\frac{d v_i}{dt}(t) \right\rangle\right]+\frac{df_{s}}{dt}(t) \mathcal{ F} (t,v(t)(.)) 
	\, dS\, .
	\end{align*}
\end{lem}

\begin{proof} 
	This follows from Lemma \ref{Diff_Banach_Space:Lem:Gâteaux_diff_Integral}, product rule \ref{Diff_Banach_Space:Lem:Product_rule_Gateaux} and Lemma \ref{Diff_Banach_Space:Lem:d/dt F nach u(t)}.
\end{proof}


\chapter{Theoretical Foundations of Shape Optimization} \label{Transf_shape_opt}

In this chapter we give a basic introduction to shape optimization. We recapitulate the well known theorems from the literature \cite{SokZol92,DelfZol11,ShapeOpt} and make all necessary computational rules in the spaces considered in this work available. Even though most of these results are known, we were not always able to find rigorous proofs in the literature. For the sake of completeness we provide these proofs here. 

\section[Velocity method]{Velocity method - transformation along vector fields}\label{Transf_shape_opt:Sec: Transformation along Vector fields}

In the following, let $ k \in \N{} $, $ k\geq 1 $, 
and $ \Omega^{ext} \subset \R{n} $ a bounded domain with boundary of class $ C^{k} $.

\begin{defn}[Admissible Vector Fields]
	According to Theorem 2.16 in \cite{SokZol92} we define the set
	\begin{equation}\label{Transf_shape_opt:Def: Admissible vectorfields}
	\Vad{k}{\Omega^{ext}}:=\Menge{V\in C^k(\overline{\Oext},\R{n})}{ \langle V,\vec{n}_{ext} \rangle=0 \text{ on } \partial \Omega^{ext}} 
	\end{equation} 
	of admissible $ C^{k} $-vector fields on $ \R{n}$, where $ \vec{n}_{ext} $ denotes the outward unity normal vector field of $ \Omega^{ext} $.\footnote{For an extension to unbounded domains $ \Oext $ see also \cite{ShapeOpt}}
\end{defn}
For such a $ V \in \Vad{k}{\Omega^{ext}} $ the system of ordinary differential equations 
\begin{equation}\label{Transf_shape_opt:Eq: ODE}
\left.\begin{array}{ll}
\dfrac{d}{dt} y(t,x)&=V(y(t,x))\\
y(0,x) &=x 
\end{array} \right. ~~~ \forall  x \in \overline{ \Oext}
\end{equation} 
has a unique solution $ y:I_{V} \times \R{n} $ on the maximal existence interval $ I_{V} $ that depends on the chosen vector field $ V $. For any $ t \in I_{V} \ni \{0{\tiny }\} $, the mapping $ y_{t}:=y(t,\cdot) $ maps $ \Oext$ to $  \Oext $ and the condition $  \langle V,\vec{n}_{ext} \rangle=0 \text{ on } \partial \Omega^{ext} $ ensures that $ y_{t} $ maps $ \overline{ \Oext}  $ to $ \overline{ \Oext} $, see also (2.76) - (2.79) in \cite{SokZol92}. Moreover, there exists a $ \epsilon>0  $ such that $ [0,\delta_{+} ) \subset I_{V}$.

The following formulas are stated in the mentioned book for vector fields $ V \in C^{k}_{0}(\Omega^{ext},\R{n})$, the set of all vector fields with compact support $$\supp(V)=\overline{\{x \in \Omega^{ext} \vert V(x) \neq 0\}} \subset \overline{\Omega^{ext}}. $$ 
Especially, they can be applied to $ V \in \Vad{k}{\Oext} $, since then $ V \in C^k(\Omega^{ext},\R{n}) $ and $ \supp(V) $ is compact as it is a closed and bounded subset of $ \R{n} $.

\begin{lem}[\cite{SokZol92}, Lemma 2.42]\label{Transf_shape_opt:Lem: Flow_Properties_Tt}$  $\\
	Choose $ V \in \Vad{k}{\Omega^{ext}} $, let $ T_{t}[V]:=y_{t},\, t \in I_{V} $ be the mapping induced by \eqref{Transf_shape_opt:Eq: ODE}. Then the following holds:
	\begin{itemize}
		\item[i)] $ T_{t+s}[V]=T_{s}[V] \circ T_{t}[V]=T_{t}[V] \circ T_{s}[V] $ for $ t,s  \in I_{V}$ with $ t+s \in I_{V} $. 
		\item[ii)]  For any $ t \in I_V $ the mapping $ T_{t}[V]: \overline{ \Oext} \to \overline{ \Oext}$ is a one-to-one transformation and the inverse is given by  $T_{t}[V]^{-1}= T_{t}[-V]$.
	\end{itemize} 
\end{lem}

\begin{proof}
	$ i) $ Let $ y_{s}(Y) $ be the unique solution of \eqref{Transf_shape_opt:Eq: ODE} with $ y_{0}(Y)=Y $, $ Y=y_{t}(x) $. Hence, the mapping $ s \mapsto y_{s}(y_{t}(x)) $ solves \eqref{Transf_shape_opt:Eq: ODE} with initial value $ Y= y_{t}(x) $. By differentiation it is clear that the mapping $ s \mapsto y_{s+t}(x) $ also is a solution to the same problem since
	\[ \dfrac{d}{ds}y_{s+t}(x)=\frac{d}{ds}(s+t) \left.\frac{d}{dr} y_{r}(x)\right|_{r=s+t}=V(y_{s+t}(x)) \]
	and at $ s =0 $ it holds that $ y_{s+t}(x)\vert_{s=0}=y_{t}(x)=Y $. Then the assertion follows from the uniqueness of the solution.
	
	\noindent $ ii) $ The fist statement can be found in \cite[P. 51]{SokZol92}. 
	If $ I_V  $ ist symmetric, then for any $ t \in I_V $ also $ -t \in I_{V} $ and
	\[ id=T_{0}[V]=T_{t-t}[V]=T_{t}[V]\circ T_{-t}[V]= T_{-t}[V]\circ T_{t}[V] .\] 
\end{proof}

The following scheme illustrates what was stated in the previous lemma. Here, $ \Omega_0=\Omega $ is some subset of $ \Omega^{ext} $ and $ \Omega_{t}:=T_{t}[V](\Omega) $ for any $ t \in I_V $. 
\begin{center}
	\begin{tikzpicture}
	\node(O_s) at (0,0) {$ \Omega_{t} $};
	\node(O_t) at (4,0) {$ \Omega_{s} $};
	\node(O_0) at (-4,0){$ \Omega_{0} $};
	\path (O_t) edge [bend left,->] node[below] {$T_{s-t}[-V]$}  (O_s);
	\path (O_s) edge [bend left,->] node[above] {$T_{t}[V]$}  (O_t)
	edge [bend left,->] node[below] {$T_{t}[-V]$}  (O_0);
	\path (O_0) edge [bend left,->] node[above] {$T_{s-t}[V]$}  (O_s);
	\end{tikzpicture}
\end{center}

\noindent  Since we fix an arbitrary $ V \in \Vad{k}{\Omega^{ext}} $ and regard the induced transformation mapping $ t \mapsto T_{t}[V] $, we suppress the $ V $-dependence of $ T_{t}[V] $ and $ I_V $  for $ t \in I $  and write $ T_{t} $ and $ I $ instead to abbreviate the notation if possible.

\subsection{Properties of the transformations}

The following properties of the transformations and associated quantities are well known in the case $ t=0 $ but
that many of them will be needed (\footnote{See also 
	the introductory example in Chapter \ref{Parameter_Dep_PDE} and the Theorems \ref{Parameter_Dep_PDE:Thm: Existence Deriv. strong Topology} and \ref{Parameter_Dep_PDE:Thm: Existence of Strong Derivatives u^t }}) also for $ t \neq 0 $ and thus we provide them here.
The proofs can be found partially in \cite{SokZol92, DelfZol11,ShapeOpt}.

\noindent \textbf{Notation:} Whenever $ A \in \R{n \times n} $ is an invertible matrix$ (A^{-1})^{\top} = (A^{\top})^{-1} $ holds. Hence the abbreviation $ (A^{-1})^{\top} = (A^{\top})^{-1} =: A^{-\top} $ is justified.

\begin{lem}\label{Transf_shape_opt:Lem: Properties D_Tt}$  $ \vspace*{-1ex}
	\begin{itemize}
		\item[i)] The mapping $t \mapsto T_{t}$ is an element of $ C^1(I,C^{k}(\overline{\Omega^{ext}},\R{n})) $ with \[ \frac{d}{dt}T_t=V \circ T_t \in C^k(\overline{\Omega^{ext}},\R{n}). \]
		\item[ii)] The mappings $t \mapsto DT_{t}$ and $ t \mapsto (DT_{t})^{-1} $ are in $C^1(I, C^{k-1}(\overline{\Omega^{ext}},\R{n \times n})) $ with 
		\begin{enumerate}
			\item  $ DT_{0}= \mathrm{I} $ where $ \mathrm{I} $ is the unit matrix in $ \R{n\times n} $.
			\item  $\dfrac{d}{dt}DT_{t} =\left(DV \circ T_{t}\right) DT_t \text{ and } \left.\dfrac{d}{dt}DT_{t}\right|_{t=0} =DV.$
			\item  $\dfrac{d}{dt} (DT_{t})^{-1}=-(DT_t)^{-1}(DV \circ T_t) \text{ and }\left.\dfrac{d}{dt} (DT_{t})^{-1}\right|_{t=0} = -DV.$
		\end{enumerate}
	\end{itemize}
\end{lem}

\begin{proof}
	i) Follows from the construction of $ T_t $.\\[1ex]
	ii) 1. follows from $ T_0=id $. \\
	2. It is clear that  $t \mapsto DT_{t}$ and $ t \mapsto (DT_{t})^{-1} $ are in $C^1(I, C^{k-1}(\overline{\Omega^{ext}},\R{n \times n})) $ since $ t \mapsto T_{t} $ is continuous with values in 
	$ C^{k}(\overline{\Omega^{ext}},\R{n}) $ and 
	$$ \frac{d}{dt}DT_t = D\frac{d}{dt}T_t =D (V \circ T_t)=\left(DV\circ T_{t} \right) DT_t  \in C(I, C^{k-1}(\overline{\Omega^{ext}},\R{n \times n})) .$$ 
	3. By chain rule we obatin
	\begin{align}\label{Transf_shape_opt:Eq: Jacobi_Ttinv}
	(DT_{t}[V])^{-1} = D(T_{t}[V]^{-1}) \circ T_t[V] = DT_t[-V] \circ T_t[V],
	\end{align} thus $  DT_t =DT_{t}[V] $ is invertible. Since $ (DT_t)^{-1} DT_t=\mathrm{I} $,
	$$
	0=\dfrac{d}{dt}[(DT_t)^{-1} DT_t]= \left[\dfrac{d}{dt} (DT_t)^{-1} \right] DT_t + (DT_t)^{-1}\, (DV \circ T_t)\, DT_t $$
	implies the assertion.
\end{proof}

Thus we resume, that $ T_t $ is an element of $ Diff^k(\overline{\Omega^{ext}},\overline{\Omega^{ext}}) $, the set of $ k $-diffeo\-mor\-phisms from $ \overline{\Omega^{ext}} $ to $ \overline{\Omega^{ext}} $.


\begin{lem}\label{Transf_shape_opt:Lem: Properties gammat}
	The scalar field $ \gamma_{t}:=\det(DT_{t}): \overline{\Omega^{ext}} \to \R{} $, $ t \in I $ satisfies the following properties:
	\begin{itemize}
		\item[i)] The mapping $t \mapsto \gamma_{t} $ is in $ C(I,C^{k-1}(\overline{\Omega^{ext}})) $ and $ \displaystyle \min_{t \in I} \gamma_{t}>0  \text{ on }\overline{\Omega^{ext}}.$
		In particular $ \Norm{\gamma_t -1}{\infty,\Omega} \to 0 $ as $ t \to 0 $.
		\item[ii)] It holds that $ \gamma_{s+t}=\left(\gamma_{s} \circ T_{t}\right) \gamma_{t} $.
		\item[iii)] The map $t \mapsto \gamma_{t}$ is an element of $ C^1(I,C^{k-1}(\overline{\Omega^{ext}})) $ with derivative
		\begin{align*} 
		\dot{\gamma}_{t}=\left.\left(\dfrac{d}{ds}\gamma_{s+t}\right)\right|_{t=0} = \gamma_{t} \Div(V )\circ T_{t},\, t \in I \text{ and }  \dot{\gamma}_{0} = \Div(V) .
		\end{align*}
	\end{itemize}
\end{lem}

\begin{proof}
	i) $ \gamma_{0}=\det\left(DT_{0}\right)=\det(\mathrm{I})=1$ on $ \overline{\Omega^{ext}} $. Since $ DT_t(x) $ is invertible for any $ t \in I,\, x \in \overline{\Omega^{ext}} $ the determinant $ \det(DT_t)(x) $ is nowhere equal to zero and $ t \mapsto \det(DT_t), \, I \to C^{k-1}(\overline{\Omega^{ext}},\R{}) $ is continuous. Thus the assertion holds by the intermediate value theorem.
	\\[1ex]
	ii) $ \gamma_{t+s}=\det\left(DT_{t+s}\right)=\det\left(DT_{s} \circ T_{t} \right) \det\left(DT_t \right)= \left(\gamma_{s} \circ T_{t}\right) \gamma_{t}$.\\[1ex]
	iii) Can be found in \cite{SokZol92} Lemma 2.31 and Proposition 2.44. 
\end{proof}

Usually the existence interval $ I_V $ of the Transformations is not symmetric and so we have to complete the approach for negative values of $ t $. Therefor we regard the existence interval $ I_{-V} \supset (-\delta_{-},0] $  of $ T_t[-V] $, $ \delta_{-} >0 $. The associated flow runs into the opposite direction of $ V $ and thus it is intuitive to set $ T_t[V]:= T_{-t}[-V] $ for $ t \in I_{-V} $. Now the transformations $ T_t $ are defined for values of $ t $ in some interval $ (-\epsilon,\epsilon) $ with $ 0 < \epsilon \leq \min\{\delta_{-},\delta_{+}\} $ and inherit the proven properties.

Before we investigate the behavior of Gram determinants and normal vector fields under the application of the transformations $ T_t=T_t[V] $
we briefly introduce the tangential differential operators which are needed in this context. 
They are defined as differential operators on the $ n-1 $-dimensional submanifold $ \Gamma $ and can be found in most books on differential geometry \cite{lee2013smooth,Kosinski1993differential,Kuhnel2015differential} and also for example in \cite{SokZol92,DelfZol11}.

\subsubsection{Tangential derivatives at the boundary}

\begin{defn}[Tangential Derivatives] 
	Let $ \Omega \subset \R{n} $ be a domain of class $ C^k,\, k\geq 1 $, $ \Gamma =\partial \Omega $ and let $ \vec{n} $ denote the outward normal vector field $ \vec{n} \in C^{k-1}(\Gamma,\R{n}) $. 
	\begin{itemize}
		\item[i)] For any scalar field $ f\in C^{1}(\Gamma,\R{}) $ the \textit{tangential gradient} $ \nabla_{\Gamma} :  C^{1}(\Gamma,\R{}) \to C(\Gamma,\R{n}) $ is defined by 
		\[ \nabla_{\Gamma}f:= \nabla f - \langle\nabla f ,\vec{n}\rangle \vec{n} \text{ on } \Gamma.\]
		\item[ii)] Let $ v\in C^{1}(\Gamma,\R{m}) $ be a vector field. Then the \textit{tangential Jacobian} is given by the mapping $ D_{\Gamma} :  C^{1}(\Gamma,\R{m}) \to C(\Gamma,\R{m \times n}) $
		with
		\[ D_{\Gamma}v:= D v - D v \,\vec{n}\vec{n}^{\top}\text{ on } \Gamma\]
		where $ \vec{a} \vec{b}^{\top}:= (a_{i}b_{j})_{i,j=1,\ldots,m} $ is the tensor product of two vectors $ \vec{a},\, \vec{b} \in \R{n} $. The \textit{tangential divergence} is given by 
		\[ \Div_{\Gamma}(v)=\tr(D_{\Gamma}v) = \Div(v) - \langle Dv\vec{n},\vec{n}\rangle \text{ on } \Gamma.\]
		\item[iv)] If $ f\in C^{2}(\Gamma,\R{}) $, then the \textit{Laplace-Beltrami operator} $ \Delta_{\Gamma}:  C^{2}(\Gamma,\R{}) \to C(\Gamma,\R{}) $
		\begin{align*}
		\Delta_{\Gamma}(f)=\Div_{\Gamma}(\nabla_{\Gamma}(f)) \text{ on } \Gamma.
		\end{align*} 
	\end{itemize}
\end{defn} 

\begin{rem}
	These definitions can be extended to Sobolev spaces $ H^{s+\frac{1}{2}}(U,\R{m}) $ using extension of $ v $ (or $ V $) from $ \Gamma $ to an neighborhood $ U $ of $ \Gamma $ and the trace operator  $$  \varGamma:  H^{s+\frac{1}{2}}(\Omega,\R{m}) \to H^{s}(\Gamma,\R{m}),\, u \mapsto \varGamma(u) $$
	where $ \varGamma(u)=u\vert_{\Gamma}$ if $u\in H^{s+\frac{1}{2}}(\Omega,\R{n})\cap C(\overline{\Omega},\R{m})$, compare \cite[Prop. 2.55]{SokZol92} and the remark above.  Nevertheless, it is common to identify $ u\vert_{\Gamma} $ with $ \varGamma(u) $ even if $ u \notin C(\overline{\Omega},\R{m}) $.
	
	\noindent The space $ H^1(\Gamma) $ can be defined in the following way, see \cite{SokZol92}: The scalar product $$ \langle u,v \rangle_{H^1(\Gamma)}  =\int_{\Gamma} \langle \nabla_{\Gamma} u, \nabla_{\Gamma}v \rangle + uv \, dS $$
	is well defined for $ u,v \in C^{1}(\Gamma) $. Thus we can define the space $ H^1(\Gamma)$ as the completion of $C^{1}(\Gamma) $ with respect to the Sobolev norm $ \Norm{u}{H^1(\Gamma)}=\sqrt{\langle u,u \rangle_{H^1(\Gamma)}}$.
\end{rem}

\subsubsection{Outward normal vector fields and Gram determinants}
\begin{lem}\label{Transf_shape_opt:Lem: Properties omegat}
	$ \Omega \subset \Omega^{ext} $ be a domain of class $ C^k $ and let $ \mathcal N \in  C^{k-1} (\overline{ \Oext},\R{n}) $ be a unitary extension of the unit outward normal vector field $ \vec{n} \in C^{k-1}(\Gamma,\R{n}) $ to $ \overline{ \Oext}$. Then the outward normal vector field $ \vec{n}_t $ on $ \Gamma_t = \{T_t(x) \vert x \in \Gamma \} =\partial \Omega_t $ is given by
	\[ 
	\vec{n}_{t} \circ T_t = \frac{1}{\Norm{(DT_{t})^{-\top}\vec{n}}{}}(DT_{t})^{-\top}\vec{n}  \text{ on } \Gamma\,.
	\] 
	Let $ \mathcal{M}(T_{t})(x):=\gamma_{t}(x)(DT_t(x))^{-\top}$, $ x \in \overline{\Omega^{ext}} $ be the adjunct matrix of $ DT_t(x) $.
	
	\begin{itemize}
		\item[i)]The mapping $ t \mapsto \omega_{t}:=\Norm{\mathcal{\mathcal{M}}(T_{t}) \mathcal{N}}{} $ is an element of  $C(I,C^{k-1}(\overline{\Omega^{ext}},\R{})) $ and satisfies
		$  \Norm{\omega_t -1}{\infty,\Omega^{ext}} \underset{t \to 0}{\to} 0$ and there exists an environment $ 0 \in U $ such that $ \min_{t \in U} \omega_{t}>0  \text{ on }\overline{\Omega^{ext}}$.
		\item[ii)]   $ \omega_{s+t}=\left(\Norm{\gamma_{s} (DT_{s})^{-\top} \vec{n}_{t} }{} \circ T_{t}\right) \Norm{\mathcal{M}(T_{t}) \vec{n}}{}=(\omega_{s} \circ T_{t})  \omega_{t}$ on $ \Gamma $ 
		.
		\item[iii)]The mapping $ t \mapsto \omega_{t} $ even is in $ C^1(I,C^{k-1}(\overline{\Omega^{ext}},\R{})) $ with
		\begin{align} 
		\dot{\omega}_{t}&=\left.\left(\frac{d}{ds}\omega_{s}\right)\right|_{s=t}= \omega_{t}\Div_{\Gamma_t}(V)\circ T_{t},\, \forall t \in I \text{ and } \dot{\omega}_{0}=\Div_{\Gamma}(V) \text{ on }  \Gamma. 
		\end{align}
	\end{itemize}
\end{lem}

\begin{proof}
	The statement $\vec{n}^{t}:=\vec{n}_{t} \circ T_t =\Norm{(DT_{t})^{-\top}\vec{n}}{}^{-1}(DT_{t})^{-\top}\vec{n} $ can be found in \cite{SokZol92} Proposition 2.48. 
	\\[1ex]
	i) This is due to $ \omega_0 =1 $ and the continuity of the mapping $ t \mapsto \omega_t $.
	\begin{align*}
	\Norm{\omega_t - \omega_0}{\infty,\Omega^{ext}}=\Norm{\Norm{\mathcal{M}(T_{t})\mathcal{N} }{} - \Norm{\mathcal{N}}{}}{\infty,\Omega^{ext}} \leq \Norm{\mathcal{M}(T_{t})\mathcal{N} -\mathcal{N}}{\infty,\Omega^{ext}} \underset{t \to 0}{\longrightarrow}0
	\end{align*}
	ii) On $ \Gamma = \Gamma_{0} $ we have
	\begin{align*}
	\omega_{s+t}
	&= \Norm{\mathcal{M}(T_{s} \circ T_{t})\vec{n}}{} \\
	&= \Norm{\left(\gamma_{s}\circ\T{t}\right) \left(DT_{s}^{-1} \circ T_{t}\right)^{\top} \mathcal{M}(T_{t}) \vec{n}}{}\\
	&= \Norm{\left(\gamma_{s}\circ\T{t}\right) \left(DT_{s}^{-1} \circ T_{t}\right)^{\top} \Norm{\mathcal{M}(T_{t}) \vec{n}}{} \frac{\mathcal{M}(T_{t}) \vec{n}}{\Norm{\mathcal{M}(T_{t}) \vec{n}}{}} }{}\\
	&= \Norm{\left(\gamma_{s}\circ\T{t}\right) \left((DT_{s})^{-1} \circ T_{t}\right)^{\top} \vec{n}_{t} \circ T_{t} }{} \Norm{\mathcal{M}(T_{t}) \vec{n}}{}  \\
	&= \left(\Norm{\gamma_{s} (DT_{s})^{-\top} \vec{n}_{t} }{} \circ T_{t}\right) \Norm{\mathcal{M}(T_{t})\vec{n}}{} \\
	&= \left(\omega_{s}\circ T_{t}\right) \omega_{t}\,.
	\end{align*}
	iii)  $ \dot{\omega}_{0}=\Div(V)-\skp{DV\vec{n}}{\vec{n}} =\Div_{\Gamma}(V)$ 
	is stated in \cite{SokZol92} Lemma 2.49, while $\dot{\omega}_{t}= \omega_{t}\Div_{\Gamma_t}(V)\circ T_{t}$ follows by application of the chain rule.
\end{proof}


\section{Shape derivatives and the Hadamard structure theorem} \label{Transf_shape_opt:Sec: Admissible Shapes}

In the following we consider sets that can be created by application of the transformations $ T_{t}[V] $, $ V \in \Vad{k}{\Omega^{ext}}$ to a set $ \Omega \subset \Omega^{ext}$ of class $ C^k,\, k \in \N{} $. The set $ \mathcal{O} $ is the set of \textit{admissible shapes} that has to be chosen properly with respect to the present problem. Generally, the set $ \mathcal{O} $ and the natural number $ k $ have to be chosen such that $ T_{t}[V](\Omega) \in \mathcal{O} $ for any $ V\in \Vad{k}{\Omega^{ext}} $ and $ t \in I_V $.

Furthermore, we will establish the main notions of shape calculus in spaces consisting of continuous or even differentiable functions. Many of the concepts can be derived under weaker conditions, for example in Sobolev spaces. For further information we refer to the books \cite{SokZol92}, \cite{DelfZol11} or \cite{ShapeOpt}. Nevertheless, we will give some comments regarding these spaces.

\begin{defn}[Shape Functional] \label{Transf_shape_opt:Defn:shape_functional}
	Let $ \mathcal{O} \subset \mathbb{P}(\Omega^{ext}) $. A shape functional is a mapping \[ J:\mathcal{O} \to \R{},\, \Omega \mapsto J(\Omega) ,\] that is well defined for every $ \Omega \in \mathcal{O}$.
\end{defn}

\begin{defn}[Shape Optimization Problem] \label{Transf_shape_opt:Defn:shape_opt_probl}
	Let $ \mathcal{O} \subset \mathbb{P}(\Omega^{ext}) $ be a set of measurable subsets of $ \Omega^{ext} $ and $ J:\mathcal{O}\to \R{}$ be a shape functional. 
	\begin{itemize}
		\item[i)] A \textit{shape optimization problem} is given by the minimization problem $$ \text{ Find }  \Omega^{\ast} \in \mathcal{O} \text{  s.t. }  J(\Omega^{\ast}) \leq J(\Omega) ~~~ \forall \Omega \in \mathcal{O} $$
		respectively,
		\[ \min_{\Omega \in \mathcal{O}} J(\Omega). \]
		\item[ii)] Let a PDE $ (P(\Omega)) $ be given such that there exists a unique solution $ u(\Omega) $ for any $ \Omega \in \mathcal{O} $. Then the problem    
		\begin{eqnarray*}
			&\min & J(\Omega,u(\Omega)) \\
			&\text{ s.t. } & u(\Omega) \text{ solves } P(\Omega),\, \Omega \in \mathcal{O}
		\end{eqnarray*}
		is called \textit{PDE-constraint shape optimization problem}. $ J $ may also depend on first or higher order (weak) derivatives of $ u(\Omega) $.
	\end{itemize}
\end{defn}

\begin{defn}[Shape Derivative]{\cite[Def. 2.19/2.20]{SokZol92}} \label{Transf_shape_opt:Defn:Euler_Deriv}
	Let $ J:\mathcal{O} \to \R{} $ be functional that is well defined for any $ \Omega $ of class $ C^{k} $. 
	\begin{itemize}
		\item[i)] The \textit{shape derivative} (\textit{Euler derivative}) of $ J $ at $ \Omega$ in direction of $ V \in \Vad{k}{\Omega^{ext}} $ is defined as 
		\begin{align}
		dJ(\Omega)[V]:= \left.\frac{d}{dt}J(\Omega_t) \right\vert_{t=0} = \lim_{t \to 0} \frac{J(\Omega_t) -J(\Omega)}{t},\,~~~~~~~  \Omega_t =T_t[V](\Omega),
		\end{align}
		if this limit value exists.
		\item[ii)] The functional $ J $ is called \textit{shape differentiable} if \\[1ex]
		1) $ dJ(\Omega)[V]  $ exists for any $ V \in  C^{k}_{0}(\Oext,\R{n}) \subset C^{k}(\overline{\Oext},\R{n})$ and \\[1ex]
		2) the map $C^{k}_{0}(\Oext,\R{n}) \to \R{},\, V \mapsto dJ(\Omega)[V] $ is linear and continuous.
	\end{itemize}
\end{defn}

The notion of shape differentiability always  has to be adopted to the present optimization problem and the shape functional under consideration. 
In many cases it is enough to claim that $ \Omega $ is a Lipschitz domain and $ V \in C^{0,1}_{0}(\Oext,\R{n})$ \cite{SturmLaurain2016distributed,Zolesio1979identification} as for example in the case of an energy type functional, e.g. $$ J(Du(\Omega)) = \int_{\Omega} Du(\Omega) : Du(\Omega) \, dx.$$

\begin{defn}\label{Transf_shape_opt:Defn:transformed_shapes} $ $\\
	i) Let $k\in \N{}$. The set $ \mathcal{O}_{k}:=\{\Omega \subset \Omega^{ext}\,| \, \Omega \text{ is a } C^k\text{- domain} \} $ contains all\textit{ $k$-admissible shapes}.\\[1ex]	
	ii) The set of all \textit{admissible transformations} that are generated by equation \eqref{Transf_shape_opt:Eq: ODE} is given by
	\begin{align}\label{Transf_shape_opt:Eq: transformation variables}
	\mathcal{T}:=\Menge{T_{t}[W]}{W \in \Vad{k}{\Oext}, t \in I_{W}}
	\end{align}
	where $ I_{W} $ is the maximal existence interval of $ \T{t}[W] \in C^k(\overline{\Omega^{ext}},\R{n}) $.\\[1ex]
	iii) Let $ \Omega \in \mathcal{O}_{k} $ and $ V \in \Vad{k}{\Omega^{ext}} $. Then $ (\Omega_{t}[V])_{t \in I_{V}} $ is the family of all \textit{perturbed domains along $ V $} where \[ \Omega_{t}[V]:=T_{t}[V](\Omega)=\{T_{t}[V](x)\, |\, x \in \Omega\} \text{ with } \Omega_{0}=\Omega.\]
	The family 
	$ (\Gamma_{t}[V])_{t \in I_{V}} $ is the family of all \textit{perturbed boundaries along $ V $} where \[ \Gamma_{t}[V]:=T_{t}[V](\Gamma)=\{T_{t}[V](x)\, |\, x \in \Gamma\} \text{ with } \Gamma_{0}=\Gamma.\]
\end{defn}

\begin{rem}$ $
	i) The set $ \mathcal{O}_{k} $ is closed w.r.t. transformation by $ T_{t}[V] $, $ V \in \Vad{k}{\Omega^{ext}} $.\\[1ex]
	ii) $ \Omega_{0}[V]=\Omega $ for any $ V \in \Vad{k}{\Omega^{ext}} $ and $ \Omega \in \mathcal{O}_{k}$.\\[1ex]
	iii) If $ V  \in \Vad{k}{\Oext}$ and $ \Omega \in \mathcal{O}_{k}$ are fixed we can also investigate the sets
	\begin{align}
	\mathcal{T}[V]:=& \Menge{T_{t}[V]}{ t \in I_{V}} \subset \mathcal{T}, \\
	\mathcal{O}(\Omega;V):=& \Menge{\Omega_t[V]}{t \in I_{V}} \subset \mathcal{O}_{k}.
	\end{align}
	Further, for any closed subset $ \tilde{I}\subset I_{V} $ there exists a constant $ C_{V, \tilde{I}} \geq 0 $  such that 
	\begin{align}\label{Transf_shape_opt:Eq: Bound TtV}
	\sup_{t \in \tilde{I}} \Norm{T_{t}[V]}{\Cm{k}{\overline{\Omega^{ext}},\R{n}} }\leq C_{V,\tilde{I}}.
	\end{align}
\end{rem}

\subsection{The Hadamard formula}

For more transparency we summarize the results in \cite[Section 2.11]{SokZol92} before we turn to the Hadamard Structure Theorem:

Supposed that $ J $ is a shape functional that is defined on the family of measurable subsets of $ \Omega^{ext} $ and is shape differentiable, then $ dJ(\Omega)=\mathscr{G}(\Omega)$ is an element of the topological dual space $ C^{k}_{0}(\Omega^{ext},\R{n})' $ or equivalently 
\begin{equation}\label{Transf_shape_opt:Eq:Existence_G_Omega}
dJ(\Omega)[V] =\mathscr{G}(\Omega)(V)
\, \forall V \in C^{k}_{0}(\Omega^{ext},\R{n}). 
\end{equation}
This representation is also called \textit{domain representation} \cite{DissKW}, \textit{weak shape derivative} \cite{Schmidt2011Dissertation} or \textit{ distributed shape derivative} \cite{SturmLaurain2016distributed}. It turns out, that the distribution $ G(\Omega) $ has only support in the interior of $ \Omega $ and actually only on  $ \Gamma $ if $ J $ is defined and shape differentiable at any $ \Omega $ of class $ C^k $:

Moreover, any vector field $ V $ with $V_{\vec{n}} = \langle V,\vec{n}\rangle=0  $ on $ \Gamma $  is identified as an element of $ \ker(dJ(\Omega)) $ since
$\langle V,\vec{n} \rangle=0  $ implies that $ T_t[V]:\overline{\Omega} \to \overline{\Omega} $ for any $ t \in I_V $ and therefore $ dJ(\Omega)[V]=0 $.  Now we consider the closed subspace \[ F_{\vec{n}}(\Omega) :=\{V \in C^{k}_{0}(\Oext,\R{n})\, \vert \, V_{\vec{n}} = \langle V\vert_{\Gamma},\vec{n} \rangle=0 \text{ on } \Gamma\} \subset C^{k}_{0}(\Oext,\R{n}) \]
and the canonical projection $ \pi:C^{k}_{0}(\Oext,\R{n}) \to C^{k}_{0}(\Oext,\R{n})/F_{\vec{n}}(\Omega), V \mapsto [V]_{\sim}.$ The mapping $ P:C^{k}_{0}(\Oext,\R{n}) \to C^{k}(\Gamma),\, V \mapsto V_{\vec{n}} = \langle V\vert_{\Gamma},\vec{n} \rangle $ is linear and continuous. 
Moreover, $ F_{\vec{n}}(\Omega) = \ker(P) $ and thus there is a unique mapping $$ \hat{P}:C^{k}_{0}(\Oext,\R{n})/F_{\vec{n}}(\Omega) \to  C^{k}(\Gamma),\, [V]_{\sim} \mapsto \langle V,\vec{n} \rangle $$ that is linear and continuous such that the following diagram commutes:
\begin{center}
	\begin{tikzpicture}
	\node(n1) at (0,0){$C^{k}_{0}(\Oext,\R{n})$}; 
	\node(n2) at (0,-1.5){$C^{k}_{0}(\Oext,\R{n})/F_{\vec{n}}(\Omega)$};
	\node(n3) at (3,0){$C^{k}(\Gamma)$}; 
	\draw[->] (n1) -- (n2) node[midway,left]{$ \pi $};
	\draw[->] (1,-1.2) -- (n3) node[midway,right]{$~~ \hat{P} $};
	\draw[->] (n1) -- (n3) node[midway,above]{$ P $};
	\end{tikzpicture}
\end{center} 
Furthermore, we can investigate the following diagram concerning $ dJ(\Omega) $:
\begin{center}
	\begin{tikzpicture}
	\node(n1) at (0,0){$C^{k}_{0}(\Oext,\R{n})$}; 
	\node(n2) at (0,-1.5){$C^{k}_{0}(\Oext,\R{n})/F_{\vec{n}}(\Omega) \cong $};
	\node(n3) at (2.5 ,-1.5){$ C^{k}(\Gamma) $};
	\node(n4) at (4.4,0){$ \R{} $}; 
	\draw[->] (n1) -- (n2) node[midway,left]{$ \pi $};
	\draw[white] (n2) -- (n3) node[midway,above]{$ $};
	\draw[->] (n3) -- (n4) node[near start,right]{$ ~dJ(\Gamma) $};
	\draw[->] (n1) -- (n4) node[midway,above]{$ dJ(\Omega) $};
	\end{tikzpicture}
\end{center} 
The mapping $ dJ(\Omega): C^{k}_{0}(\Oext,\R{n}) \to \R{} $ is linear and continuous and $ F_{\vec{n}}(\Omega) \subset \ker(dJ(\Omega)) $ is closed. Again, by the Fundamental Theorem on Homomorphisms, there is a unique mapping 
$$ dJ(\Gamma):  C^{k}(\Gamma,\R{n}) \cong C^{k}_{0}(\Oext,\R{n})/F_{\vec{n}}(\Omega) \to \R{}  $$ with $dJ(\Omega)= dJ(\Gamma) \circ \hat P \circ \pi $, i.e. 
\begin{equation}\label{Transf_shape_opt:Eq:Existence_G_Gamma}
dJ(\Omega)[V] = dJ(\Gamma)[V_{\vec{n}}]~~~ \forall V \in C^{k}_{0}(\Oext,\R{n}). 
\end{equation}

\begin{thm}[Hadamard Structure Theorem]{\cite[Theorem 2.27]{SokZol92}} \label{Transf_shape_opt:Thm:Hadamard_Thm}\\
	Let $ J $ be a shape functional, that is shape differentiable for any subset of $\Omega^{ext}  $ with boundary of class $ C^k $. Furthermore, suppose that $  \Omega \subset \Omega^{ext}$ is a domain of class $ C^{k+1} $.
	There exists a unique scalar distribution $ \mathcal{G}(\Gamma) \in C^{k}(\Gamma)' $ such that $ dJ(\Omega) \in C^{k}_{0}(\Oext,\R{n})' $ satisfies 
	\begin{align}\label{}
	dJ(\Omega)
	= \mathbf{T}_{\Gamma}'(\mathcal{G}(\Gamma)\vec{n})
	\end{align} 
	where $\mathbf{T}_{\Gamma}:   C^{k}_{0}(\Oext,\R{n}) \to  C^{k}(\Gamma,\R{n}),\,
	v  \mapsto v\vert_{\Gamma}$ 
	is the trace operator, $ \mathbf{T}_{\Gamma}': C^{k}(\Gamma,\R{n})' \to  C^{k}_{0}(\Oext,\R{n})',\, A \mapsto A \circ \mathbf{T}_{\Gamma} $ is the dual oparator of $ \mathbf{T}_{\Gamma}.$ 
\end{thm}

\begin{proof}
	We apply \eqref{Transf_shape_opt:Eq:Existence_G_Omega} and \eqref{Transf_shape_opt:Eq:Existence_G_Gamma}. The mapping $dJ(\Gamma)= dJ(\Omega) \circ \pi^{-1} \circ \hat{P}^{-1}$ is linear and continuous and thus 
	\begin{align*}
	dJ(\Omega)[V]=dJ(\Gamma)[\langle V,\vec{n} \rangle] = dJ(\Gamma)[V_{\vec{n}}] = \mathcal{G}(\Gamma)( \langle \mathbf{T}_{\Gamma}(V),\vec{n} \rangle )
	\end{align*}
	For $ R \in C^{k}(\Gamma)'  $ and $ v \in C^{k}(\Gamma,\R{n})  $ the product $ Rv \in C^{k}(\Gamma,\R{n})' $ is given by
	$ Rv(W)=R(\langle W, v\rangle) $ and thus 
	\begin{align*}
	\mathcal{G}(\Gamma) (\langle \mathbf{T}_{\Gamma}(V),\vec{n} \rangle ) = \left(\mathcal{G}(\Gamma)\vec{n}\right)(\mathbf{T}_{\Gamma}(V) ) = \mathbf{T}_{\Gamma}'\left(\mathcal{G}(\Gamma)\vec{n}\right)(V) .
	\end{align*} 
\end{proof}

One has to take good care in view of the regularities of the domains:
If all shapes $ \Omega $ are of class $ C^{k+1} $, $ J $ is differentiable on any $ \Omega \in \mathcal{O}_{k} $ and $ V \in  C^{k+1}_{0}(\Oext,\R{n}) $ then the Hadamard Structure Theorem implies that the distribution $ \mathcal{G}(\Gamma) $ exists and is uniquely defined in $ C^{k}(\Gamma)' $.

\subsection{$ L^2 $-descent directions}\label{Transf_shape_opt:Sec:L2 decent directions}

The domain and the surface representation can be used to find decent directions for numerical optimization schemes \cite{ShapeOpt,Schulz2016metricsComparison,SturmLaurain2016distributed,GottschSaadi2018,NumShapeCer2017}. For example the "classical" Hadamard shape derivative representation defines such a descent direction:

\begin{defn}
	Suppose that $ J:\mathcal{O}_{k} \to \R{} $ is shape differentiable. Then any $ W \in  C^{k}(\overline{\Omega},\R{n}) $ with 
	$$ dJ(\Omega)[W] = dJ(\Gamma)[W_n] < 0 $$ is called a \textit{descent direction} for $ J $. 
\end{defn} 
Note that it is enough to take descent directions as elements of $ C^{k}(\overline{\Omega},\R{n}) $ respectively $ C^{k}(\Gamma,\R{n}) $ since $ dJ(\Omega) $ has only support on the domain or even more precisely only on $ \Gamma $. Moreover, any function $ W \in C^{k}(\overline{\Omega},\R{n}) $ or $ C^{k}(\Gamma,\R{n}) $ has a $ C^{k} $ extension to $ \Oext $ according to Lemma \ref{App:Lem: Hölder Extension Lemma}. 

Suppose that $ J $ is shape differentiable. Then the induced map $ dJ(\Omega) : C^{k}(\Gamma,\R{n}) \to \R{} $ is linear and continuous. We assume that $ dJ(\Gamma)[.] $ satisfies
\[ |dJ(\Gamma)[V_{\vec{n}}]| \leq C \Norm{V_{\vec{n}}}{L^2(\Gamma)} ~~~ \forall V_{\vec{n}} \in C^{k}(\Gamma,\R{n}).   \]
Since $  C^{k}(\Gamma,\R{n}) $ is dense in $ L^2(\Gamma) $ there exists an extension $ dJ^{ext}(\Omega):L^2(\Gamma) \to \R{} $ of $ dJ(\Omega) $. Then the theorem of Lax-Milgram implies that the variational problem 
\begin{align}\label{Transf_shape_opt:Eq:L2 decent direction Eq}
\langle W,V \rangle_{L^2(\Gamma,\R{n})} = - dJ^{ext}(\Gamma)[V_{\vec{n}}] ~~~ \forall V \in L^2(\Gamma,\R{n}) 
\end{align}
has a solution $ W \in L^{2}(\Gamma,\R{n}) $ and $dJ(\Gamma)[W_n] =-\langle W,W \rangle <0$ holds. Thus, if we suppose that $ G(\Gamma) \in L^2(\Gamma)\setminus\{0\} $ and $  dJ(\Gamma)[V_{\vec{n}}]  $ has the following structure 
$$ dJ(\Gamma)[V_{\vec{n}}] = \int_{\Gamma}  G(\Gamma)V_{\vec{n}} \, dS = \int_{\Gamma} \langle G(\Gamma)\vec{n} ,V\vert_{\Gamma} \rangle \, dS  = \langle G(\Gamma)\vec{n},V \rangle_{L^2(\Gamma,\R{n})} =  \langle G(\Gamma),V_{\vec{n}} \rangle_{L^2(\Gamma)} $$ then  $ W= - G(\Gamma)\vec{n} \in C^{k}(\Gamma) $ is a suitable choice for a descent direction.  

Unfortunately, the $ L^2 $-decent density $ G(\Gamma) $ usually is no element of $ C^{k} $. And even if $G(\Gamma) \in C^{k}(\Gamma) $ holds, then $ G(\Gamma)\vec{n} \in C^{k-1}(\Gamma,\R{n})$ provided that $ \Omega $ is of class $ C^{k} $. Thus the 
direction $ W= -G(\Gamma)\vec{n} $ provides not enough regularity to preserve the domain-regularity during a descent along $ W $. 

Depending on the chosen representation we call either $ G(\Gamma) $ or $ G(\Gamma)\vec{n} $ the \textit{shape gradient} w.r.t. the $ L^{2}(\Gamma) $ respectively the $ L^2(\Gamma,\R{n}) $ scalar product. However, the regularity of $ G(\Gamma) $ can in fact be much higher than only $ L^2 $. The regularity of $ G(\Gamma) $ in the case of shape optimization problems under linear elasticity constraints will be investigated in detail in Chapter \ref{Shape_Grad_LinEl}. Based on these investigations, we will give a short outlook in Section \ref{Outlook}  on the methods that already have been developed \cite{Schulz2016metricsComparison,SturmLaurain2016distributed,GottschSaadi2018,NumShapeCer2017} and a perspective on descent directions representatives that will hopefully be suitable to maintain the shape-regularity along descent flows.  

\newpage

\section{Material and local shape derivatives}\label{Transf_shape_opt:Sec: material and shape derivative}

To motivate the next definition we examine the following example: \\
Suppose that $ k\geq 1 $ and $ \{u(\Omega)|\Omega \in \mathcal{O}_{k}\} $ is a collection of functions, such that $  u(\Omega) \in C^{0}(\overline{\Omega}) $ for any $ \Omega \in \mathcal{O}_{k} $ and let $ J:\mathcal{O}_k \to \R{},\, \Omega \mapsto J(\Omega,u(\Omega)) $ be given by 
\[ J(\Omega,u(\Omega))= \int_{\Omega} u(\Omega) \, dx.\] 
Then 
$$ \frac{J(\Omega_t,u(\Omega_t)) - J(\Omega,u(\Omega))}{t} = \int_{\Omega} u(\Omega_{t}) \circ T_t \frac{(\gamma_{t} -  1)}{t}\, dx + \int_{\Omega} \frac{u(\Omega_t) \circ T_t - u(\Omega)}{t} \gamma_t \, dx.$$ 
To calculate the Eulerian derivative the derivative of $u^t = u(\Omega_t) \circ T_t$ thus has to exist in $ L^1(\Omega) $ since $ \gamma_t $ is bounded if $ t $ is close to $ 0 $. 

\begin{defn}[Material Derivatives in the Volume]\label{Transf_shape_opt:Def: Material_Derivative} $  $ 
	Let $ V \in \Vad{k}{\Omega^{ext}} $, $ \Omega \in \mathcal{O}_k $ and $ T_{t}[V] $ the associated family of transformations and $ u(\Omega_t): \Omega_t \to \R{m} $.
	Assume that there exits an $ \epsilon>0 $ such that the mappings  $ u(\Omega_t) \circ T_t[V]: \Omega \to \R{m} $ are elements of a Banach space $ X_{\Omega} $ for any $ t \in (t_0 -\epsilon, t_0 +\epsilon) \subset I_{V} $.
	
	The $ X_{\Omega} $-material derivative of $ u(\Omega_t) $ at $ t=t_0 $ in direction of $ V$ is defined as the Gâteaux derivative of $t \mapsto u(\Omega_t) \circ T_t[V] $ at $ t=t_0 $ if it exists in the topology on $ X_{\Omega} $, i.e.
	\[ \left.\frac{d}{dt}u(\Omega_t) \circ T_t[V]\right\vert_{t=t_0}=\lim_{t \to t_0}\frac{1}{t}(u(\Omega_t) \circ T_t[V]-u(\Omega_{t_0}) \circ T_{t_0}[V]) \in X_{\Omega}.\]
\end{defn}

\begin{defn}[Material Derivatives at the Boundary]\label{Transf_shape_opt:Def: Material_Derivative_Boundary} 
	Let $ V \in \Vad{k}{\Omega^{ext}} $, $ \Omega \in \mathcal{O}_k $ and $ T_{t}[V] $ the associated family of transformations and $ v(\Gamma_t):\Gamma_t \to \R{m} $.
	Assume that there is an $ \epsilon>0 $ such that the mappings  $  v(\Gamma_t) \circ T_t[V]:\Gamma \to \R{m} $ are elements of a Banach space $ X_{\Gamma} $ for any $ t \in (t_0 -\epsilon,t_0 + \epsilon) \subset I_{V} $.
	
	The $ X_{\Gamma} $-material derivative of $ v(\Gamma_t) $ at $ t=t_0 $ in direction of $ V$ is defined as the Gâteaux derivative of $t \mapsto v(\Gamma_t) \circ T_t $ at $ t=t_0 $ and $ \alpha =1 $ if it exists in the topology on $ X_{\Gamma} $, respectively:
	\[\left.\frac{d}{dt}v(\Gamma_t) \circ T_t[V]\right\vert_{t=t_0}=\lim_{t \to t_0}\frac{1}{t}(v(\Gamma_t) \circ T_t[V]-v(\Gamma_{t_0}) \circ T_{t_0}[V]) \in X_{\Gamma}.\]
\end{defn}

\begin{conv}
	Let $ k\geq 1 $, $ V \in \Vad{k}{\Omega^{ext}} $ be an admissible vector field and $ (T_{t}[V])_{t \in I_{V}} $ the associated transformation family and $ \Omega \in \mathcal{O}_{k} $. If no confusion is possible we will use the simpler notation
	\begin{align}
	T_t =T_t[V]: \Omega \to \Omega_{t}    &&   u_{t}:=u(\Omega_t):\Omega_t \to \R{m}  &&    y_t:=y(\Gamma_t):\Gamma_t \to \R{m} 
	\end{align}
	and further
	\begin{align}
	\begin{array}{l l l}
	u^t:=u_t \circ T_t: \Omega\to \R{m} & &y^t :=y_t \circ T_t:\Gamma \to \R{m} \\[1ex]
	\dot{u}^{t} :=\dot{u}^{t}(\Omega;V):\Omega \to \R{m }& &  \dot{y}^{t} :=\dot{y}^{t}(\Gamma;V):\Gamma \to \R{m }  \\[1ex]
	\dot{u}:= \dot{u}(\Omega;V)=\dot{u}^{0}(\Omega;V)& &  \dot{y}:=\dot{y}(\Gamma;V)=  \dot{y}^{0}(\Gamma;V) \\[1ex]
	u'=u'(\Omega;V)& & y'=y'(\Gamma;V)\,.
	\end{array}
	\end{align}
\end{conv} 

\begin{defn}[Local Shape Derivative]\label{Transf_shape_opt:Def: Shape_Derivative}
	Let $ V \in \Vad{k}{\Omega^{ext}} $ and $ \Omega \in \mathcal{O}_k $. 
	\begin{itemize}
		\item[i)] Let $ X_{\Omega} $ be a Banach space, that is a subspace of $ H^{1}(\Omega,\R{m}) $ or $ C^{1}(\Omega,\R{m}) $. Suppose that the material derivative $ \dot{u}$ of $ u_t$ at $ t=0 $ in direction of $ V $ exists in $ X_{\Omega} $. Then the  \textit{local (volume) shape derivative} of $ u_t $ is defined by
		\[ u'(\Omega;V):= \dot{u}(\Omega;V) - Du(\Omega)V.\]  
		\item[ii)] Let $ X_{\Gamma} $ be a Banach space, that is a subspace of $ H^{1}(\Gamma,\R{m}) $ or $ C^{1}(\Gamma,\R{m}) $. Suppose that the material derivative $ \dot{y}$ of $ y_t $ in direction of $ V $  exists in $ X_{\Gamma} $ exists. Then the \textit{local (boundary) shape derivative} is given by
		\[ y'(\Gamma;V):= \dot{y}(\Gamma;V) - D_{\Gamma}y(\Gamma) V.\]  
\end{itemize}\end{defn}

The material derivative is always calculated on a reference domain $ \Omega $ (Lagrangian coordintates) whereas the local shape derivative is a derivative in local coordinates (Eulerian coordinates). If $ z \in \bigcap_{t \in(-\epsilon,\epsilon)} \Omega_{t} \cap \Omega $ for any $ t \in (-\epsilon,\epsilon) $, then 
\[ u'(z) = \left.\frac{d}{dt} u(\Omega_t)(z)\right\vert_{t=0} \text{ for any } z \in  \bigcap_{t \in(-\epsilon,\epsilon)} \Omega_{t} \cap \Omega\,.\] 

Unfortunately, the local shape derivative usually looses one degree of regularity in comparison with the material derivative:
For example, if $ u_t \in  H^s(\Omega_t,\R{m})$ has material derivatives in this space with $1\leq l \leq k$. Then $ DuV \in H^{s-1}(\Omega,\R{m}) $, since $ V \in C^{k}(\overline{\Omega},\R{m}) $ and thus $ u'= \dot{u}-DuV \in H^{s-1}(\Omega,\R{m}) $. Therefore $ X_{\Omega} = H^{s-1}(\Omega,\R{m}) $ has to be chosen in 
the above definition.

\subsection{Calculation rules for material and local shape derivatives}  \label{Transf_shape_opt:Sec: arithmetic rules shape derivative} 

In this work, we will concentrate on material and local shape derivatives in Hölder function spaces $ C^{l,\phi}(\overline{\Omega},\R{m})  $, $ \phi \in [0,1] $ and thus we will establish the main calculation rules with respect to these spaces. Most of them can be found in the literature, for example in \cite{Berggren2010,DissKW,SokZol92,DelfZol11,Schmidt2011Dissertation,ShapeOpt}. For reasons of completeness, consistency and rigorousness we will give full proofs under assumptions that are appropriate to investigate the singular shape functionals introduced in Chapter \ref{Reliability}. Until the end of this chapter it will be enough to assume $ V \in \Vad{1}{\Oext} $.

However, some of these results can be obtained under weaker conditions which are not covered in this thesis.

\begin{lem} \label{Transf_shape_opt:Lem: arithmetic rules mat deriv I} 
	$  $
	\begin{itemize}
		\item[i)] Assume that the material derivatives of $ u_t,\, v_t: \Omega_t \to \R{},\, u_t$, $v_t \in C^{0}(\Omega_t) $ exists in $ C^{0}(\Omega_t) $. Then, at $ t=0 $, 
		\begin{equation}
		\begin{split}
		a) & (au+bv)\dot{}\,(\Omega;V)  = a\dot{u}(\Omega;V) + b\dot{v}(\Omega;V) \, \forall a,\, b \in \R{}, \\
		b) & ~~~~~~~~ (uv)\dot{}\,(\Omega;V)= \dot{u}(\Omega;V)v +  u \dot{v}(\Omega;V). 
		\end{split}
		\end{equation}
		\item[ii)] If  $ u_t:\Omega_t \to \R{m} \in C^{0}(\Omega_t,\R{m}) $ such that  $ \dot{u} \in C^{0}(\Omega,\R{m}) $ and $ f \in C^{1}(\R{m},\R{r}) $. Then the $ C^{0}(\Omega_t,\R{r}) $ material derivative of $ f \circ u_t $ at $ t=0 $ satisfies 
		\[  (f\circ u)\dot{}\,(\Omega;V)= (Df \circ u)\dot{u}(\Omega;V). \]
		\item[iii)] Let $ u_t:\Omega_t \to \R{m} \in C^{0}(\Omega_t,\R{m}) $ such that  $ \dot{u} \in C^{0}(\Omega,\R{m}) $ and suppose that  $ f(\Omega): \R{m} \to \R{r},\, z \mapsto  f(\Omega)(z)=f(\Omega,z) $ is a vector field for any $ \Omega \in \mathcal{O} $. Define $F(t,.):= f(\Omega_{t})(.):\R{m} \to \R{r} $. 
		
		If $ F:(-\epsilon,\epsilon) \times \R{m} \to \R{n} $ is Fréchet differentiable then the material derivative of $  F(t,u_t)=f(\Omega_t) \circ u_t $ at $ t=0 $ is given by
		\[ (f(\Omega) \circ u)\dot{}\,(\Omega;V) =  \dot{f}(\Omega;V)\circ u + \left(\frac{\partial f(\Omega)}{\partial z} \circ u\right)\dot{u} = \dot{f}(\Omega;V)\circ u+ (Df(\Omega) \circ u) \dot{u}.\]
	\end{itemize}
\end{lem}

\begin{proof}
	i) is stated in \cite[(24)]{Berggren2010} and follows directly from \ref{Diff_Banach_Space:Lem:Product_rule_Gateaux}. ii) Follows from Lemma \ref{Diff_Banach_Space:Lem:CR} and iii) from 
	\ref{Diff_Banach_Space:Lem: d/dt F (t,u(t))}
\end{proof}

\begin{rem}
	Clearly,  these results can be transferred by analogous calculations to boundary material derivatives and material derivatives in other spaces. Note that the regularity of the vector field $ f $ can not be reduced so easily since the Sobolev chain rule requires continuous differentiability. 
\end{rem}

\begin{lem} \label{Transf_shape_opt:Lem: d/dt Du_t circ T_t} 
	Assume that the material derivative of $ u_t \in C^1(\Omega_{t},\R{m})$ exists in $ C^1 $.
	Then the $ C^0 $ material derviative of $  Du_t\in C(\overline{\Omega_t},\R{m\times n})  $ is given by 
	\begin{equation}
	(Du)\dot{}\,(\Omega;V)=\frac{d}{dt }Du_t \circ T_t\vert_{t=0} = D\dot{u}(\Omega;V) -DuDV .
	\end{equation}
	If the mapping  $ (-\epsilon,\epsilon) \to C^1( \Omega_{t},\R{m}),\,t \mapsto \dot{u}^{t} $ is additionally continuous then the mapping $ (-\epsilon,\epsilon) \to C(\overline{\Omega},\R{m\times n}),\, t\mapsto Du_t \circ T_t $ is even Fréchet differentiable in $ C^0 $ at $ t=0 $.
\end{lem}

\begin{proof}
	The $ C^1$-material derivative of $ u_t $ exists at $ t\in (-\epsilon,\epsilon) $ if and only if $t \mapsto u^{t} $ is Gâteaux differentiable in the Banach space $ C^1(\overline{ \Omega},\R{m}) $ at $ t$.
	
	\noindent Since $D: C^k(\R{n},\R{m}) \to C^{k-1}(\R{n},\R{m \times n}) $ is linear and continuous, the Gâteaux differential $ \frac{d}{dt} $ permutes with the spatial derivative $ D $ and therefore $ \frac{d}{dt}Du^{t}= D\frac{d}{dt}u^{t}=D\dot{u}^t$. This, combined with product rule, see Lemma \ref{Diff_Banach_Space:Lem:Product_rule_Gateaux}, and the properties of $ T_t $, see Lemma \ref{Transf_shape_opt:Lem: Properties D_Tt}, leads to
	\begin{align}
	\frac{d}{dt }Du_t \circ T_t
	&= \frac{d}{dt }Du^t (DT_t)^{-1}
	=D\dot{u}^{t}(DT_t)^{-1} +Du^t\frac{d}{dt}(DT_t)^{-1}\\
	&= D\dot{u}^{t}(DT_t)^{-1} -Du^t(DT_t)^{-1}\, (DV \circ T_t) .
	\end{align}
	Evaluated at $ t=0 $ this gives$
	\left.\frac{d}{dt }Du_t \circ T_t\right\vert_{t=0}= D\dot{u} -DuDV
	$, compare also \citep[(25)]{Berggren2010}. 
	The second assertion follows from Lemma \ref{Diff_Banach_Space:Lem:Gateaux/Frechet_Diff_u^t}: Since $ (-\epsilon,\epsilon) \ni t \mapsto Du^{t} \in C(\overline{ \Omega},\R{m \times n}),$ is Gâteaux differentiable, it is also continuous. This holds as well for $ (-\epsilon,\epsilon)\to C^{k}(\overline{\Omega},\R{n}),\, t \mapsto T_t $ and $ t \mapsto (DT_{t})^{-1} \in C^{k-1}(\overline{\Omega},\R{n \times n}) $. Thus, if $(-\epsilon,\epsilon) \to C^1(\overline{ \Omega},\R{m}),\, t \mapsto \dot{u}^{t} $ is continuous then also the mapping $ t \mapsto \frac{d}{dt }Du_t \circ T_t \in C(\overline{\Omega},\R{m\times n}) $ is continuous.
\end{proof}

\begin{cor}\label{Transf_shape_opt:Lem: dot(div), dot(sigma)} 
	Under the assumptions of Lemma \ref{Transf_shape_opt:Lem: d/dt Du_t circ T_t} 
	$$ (\Div(u))\dot{} = \Div(\dot{u}) - \tr(DuDV) \text{ and }  \sigma(u)\dot{}=\sigma(\dot{u}) - (DuDV)^{\sigma} $$
	where 
	\[ M^{\sigma}=\lambda\tr(M)\mathrm{I}+\mu(M+M^{\top}) \text{ for } M \in \R{n \times n}.\]
	Here $ \lambda>0 $ and $ \mu >0 $ are the Lamé coefficients introduced in Chapter \ref{Sec:MechLoad}.\\
	If $ m=1 $ then $$ (\nabla u)\dot{} = \nabla \dot{u} - DV^{\top}\nabla u.$$
	If the boundary of $ \Omega_t $ is of class $ C^1 $ for any $ t \in (-\epsilon,\epsilon) $, then $ \dot{\vec{n}} = -D_{\Gamma}V^{\top}\vec{n} $ at $ t =0 $.
\end{cor}


\begin{lem}	\label{Transf_shape_opt:Lem: Shape derivatives C^l} 
	Let $ \Omega \in \mathcal{O}_{k},\, k \geq 1 $  and  let $ V \in \Vad{k}{\Omega^{ext}} $.
	\begin{itemize}
		\item[i)] Now suppose that $ y(\Gamma) =u(\Omega)\vert_{\Gamma}$ is the restriction of $u(\Omega) \in  C^{l}(\overline{\Omega},\R{m})$ to $ \Gamma $. Then
		\begin{align}
		y'(\Gamma;V) =  u'(\Omega;V)\vert_{\Gamma} + \frac{\partial u(\Omega)}{\partial \vec{n} }V_{\vec{n}}\in C^{l-1}(\Gamma,\R{m}).
		\end{align} 
		
		\item[ii)] Let $ u,\,y\in C^l(\Omega^{ext},\R{m})  $.  Then, the material and shape derivative of \\ $ u(\Omega):=u\vert_{\Omega}$ and  $ y(\Gamma):=y\vert_{\Gamma}$, respectively, are given by
		\begin{itemize}
			\item[I)] $\dot{u}(\Omega;V)=D u \, V \text{ on } \Omega $ and $ u'(\Omega;V)=0 .$ 
			\item[II)] $ \dot{y}(\Gamma;V)=D_{\Gamma} y \, V \text{ on } \Gamma $ and $ y'(\Gamma;V)=\dfrac{\partial y}{\partial \vec{n} } V_{\vec{n}}$.
		\end{itemize}
	\end{itemize}
\end{lem}

\begin{proof}
	i) Since $ u(\Omega_t)\vert_{\Gamma_t} \circ T_t = (u(\Omega_t) \circ T_t)\vert_{\Gamma} $ it is clear that $ \dot{u}(\Omega;V)\vert_{\Gamma} = (u(\Omega)\vert_{\Gamma})\dot{}(\Gamma;V)$. We conclude
	\begin{align}
	y'(\Gamma;V)
	&=(u(\Omega)\vert_{\Gamma})'(\Gamma;V) =(u(\Omega)\vert_{\Gamma})\dot{}(\Gamma;V) -D_{\Gamma}(u(\Omega)\vert_{\Gamma}) V \\
	&=\dot{u}(\Omega;V)\vert_{\Gamma} - Du(\Omega)\vert_{\Gamma} V + Du(\Omega)\vert_{\Gamma} \vec{n} (\vec{n}^{\top} V )= u'(\Omega;V)\vert_{\Gamma} + \frac{\partial u(\Omega)}{\partial \vec{n} }\langle V,\vec{n} \rangle.
	\end{align}
	ii) I) In this case  $ \dot{u}=\left.\dfrac{d}{dt}u \circ T_t \right\vert_{t=0} = DuV \Rightarrow u'(\Omega;V)=DuV-DuV=0.$ \\[1ex]
	\indent \, II) Apply i) to $ y(\Gamma)=u(\Omega)\vert_{\Gamma} $ with $u(\Omega)=u\vert_{\Omega}.$
\end{proof}

\begin{rem}	\label{Transf_shape_opt:Rem: Shape derivatives H^s}
	When dealing with Sobolev spaces $ H^l(\Omega,\R{m}) $ for $ 1\leq l\leq k-\frac{1}{2},\, k\geq 1$, $ \Omega \in C^{k-1,1} $, $ V \in \Vad{1}{\Oext} $ one has to take more care - especially in case ii) and iii) because of the trace operator 
	$$  \mathbf{T_{\Gamma}}:  H^{l+\frac{1}{2}}(\Omega,\R{m}) \to H^{l}(\Gamma,\R{m}),\, u \mapsto \mathbf{T_{\Gamma}}(u) = u\vert_{\Gamma}$$
	where $ \varGamma(u)=u\vert_{\Gamma}$ if $u\in H^{l+\frac{1}{2}}(\Omega,\R{m})\cap C(\overline{\Omega},\R{m})$, compare \cite[Prop. 2.55]{SokZol92} and the remark above.  Nevertheless, it is common to identify $ u\vert_{\Gamma} $ with $ \varGamma(u) $ even if $ u \notin C(\overline{\Omega},\R{m}) $.
	
	Provided that $ y(\Gamma)=u(\Omega)\vert_{\Gamma} $ is the restriction of a function $ u(\Omega) \in  H^{l+\frac{1}{2}}(\Omega,\R{n}) $ then, the shape derivative of $ y(\Gamma) $ satisfies 
	\begin{align}
	y'(\Gamma;V) =  u'(\Omega;V)\vert_{\Gamma} + \frac{\partial u(\Omega)}{\partial \vec{n} }V_{\vec{n}}\in H^{l-1}(\Gamma,\R{n}).
	\end{align} 
	
	Assume that $ u,\, y:\Omega^{ext} \to \R{m} $ with 
	$ u\in H^l(\Omega^{ext},\R{n})  $,
	$ y\in H^{l+\frac{1}{2}}(\Omega^{ext},\R{n}) $, $ 1\leq l\leq k-\frac{1}{2} $.
	Set $ u(\Omega):=u\vert_{\Omega},\, y(\Gamma):=y\vert_{\Gamma}.$ Then again $ \dot{u}(\Omega;V)= D u \, V \text{ on } \Omega   $, $ u'(\Omega;V)=0  $, $ \dot{y}(\Gamma;V)=D_{\Gamma} y \, V \text{ on } \Gamma $ and $ y'(\Gamma;V)=\dfrac{\partial y}{\partial \vec{n} } V_{\vec{n}}$.
\end{rem}

\begin{lem}\label{Transf_shape_opt:Lem: Product and Chain rule for shape derivatives}
	Let $ u(\Omega),\, v(\Omega) \in C^{l}(\overline{ \Omega},\R{m})$, $  y(\Gamma),\, z(\Gamma) \in C^{l}( \Gamma,\R{m})$ for $1 \leq l\leq k   $ such that their $ C^{l} $-material derivatives $ \dot{u}(\Omega;V)$, $ \dot{v}(\Omega;V)$, $ \dot{y}(\Gamma;V)$, $ \dot{z}(\Gamma;V) $ exist w.r.t the  strong (weak) topology on $ C^{l} $. Then 
	\begin{itemize}
		\item[i)] $ u'(\Omega;V) ,\,  v'(\Omega;V)  \in  C^{l-1}( \Omega,\R{m}) $, $ y'(\Gamma;V) ,\,  z'(\Gamma;V) \in  C^{l-1}(\Gamma,\R{m}) $ and
		\begin{align} 
		\langle u,v\rangle'(\Omega;V)&=\langle u'(\Omega;V),v(\Omega)\rangle+\langle u(\Omega),v'(\Omega;V)\rangle\, ,\\
		\langle y,z\rangle'(\Gamma;V)&=\langle y'(\Gamma;V),z(\Gamma)\rangle+\langle y(\Gamma),z'(\Gamma;V)\rangle.
		\end{align}
		\item[ii)] Suppose that $ f \in C^1(\R{m},\R{n}) $, $ n \in \N{} $ 
		\begin{align}
		(f \circ u)'(\Omega;V)&= \left(Df \circ u(\Omega)\right)u'(\Omega,V), \\
		(f \circ y)'(\Gamma;V)&=\left(Df \circ y(\Gamma)\right)y'(\Gamma,V).
		\end{align}
		\item[iii)] Suppose that for any $ \Omega \in \mathcal{O} $, $ f(\Omega): \R{m} \to \R{r},\, z \mapsto  f(\Omega)(z)=f(\Omega,z) $ is a vector field.  If $ (-\epsilon,\epsilon) \times \R{m} \to \R{n}, (t,z) \mapsto f(\Omega_t,z) $ is Fŕechet differentiable the local shape derivative of $ f(\Omega_t,u_t(.))=f(\Omega_t) \circ u_t $ satisfies 
		\[ (f(\Omega) \circ u )'(\Omega;V) =  \dot{f}(\Omega;V)(\Omega) \circ u + (Df(\Omega)\circ u)u' .\]
	\end{itemize}
\end{lem}

\begin{proof}
	i) Let $ u=u(\Omega),\, v=v(\Omega) $. Then,
	\begin{align}
	\langle u,v\rangle'
	&=\langle u,v\rangle\dot{} - \left\langle \nabla\langle u,v\rangle ,  V \right\rangle =\langle \dot{u},v\rangle+ \langle u,\dot{v}\rangle-\langle Du^{\top}v+Dv^{\top}u,V\rangle\\
	&=\langle \dot{u},v\rangle -\langle DuV,v\rangle+ \langle u,\dot{v}\rangle-\langle u, Dv V\rangle =\langle \dot{u}-DuV,v\rangle + \langle u,\dot{v}-Dv V\rangle\, .
	\end{align}
	\noindent ii) Let $ u=u(\Omega) $. Then $(f\circ u)' = (f \circ u)\dot{}-D(f \circ u)V =(Df \circ u)\dot{u} - (Df\circ u)DuV= (Df \circ u)u'$
	where we applied Lemma \ref{Transf_shape_opt:Lem: arithmetic rules mat deriv I}.\\[1ex]
	iii) Under these conditions
	\begin{align}
	(f(\Omega) \circ u)' &= (f(\Omega) \circ u)\dot{} - \frac{\partial }{\partial x}f(\Omega,u(.))V\\
	& = \dot{f}(\Omega) \circ u +(Df(\Omega)\circ u)\dot{u} - (Df(\Omega)\circ u)DuV \\
	&= \dot{f}(\Omega)\circ u + (Df(\Omega)\circ u)u'.
	\end{align}
	
	The results for $ y=y(\Gamma),\, z=z(\Gamma) $ can be derived by the same argumentation. 
\end{proof}

While the material derivative $ \dot{u}(\Omega;V) $ does not commute with spacial derivatives the shape derivative $ u'(\Omega;V) $ does:

\begin{lem}\label{Transf_shape_opt:Lem: Properties u'} 
	Let $ V \in \Vad{k}{\Omega^{ext}} $, $ u(\Omega)=(u_1(\Omega),u_2(\Omega),\ldots, u_m(\Omega))^{\top} \in C^{l}( \Omega,\R{m})$, $ 1 \leq l\leq k  $ such that their strong (weak) $ C^{l} $-material derivative $ \dot{u}(\Omega;V)
	$ exist. Then the following holds: 
	\begin{itemize}
		\item[i)] $ u'(\Omega,V)=(u_{i}'(\Omega,V))_{i=1,\ldots,m} $. Set $ (au +bv)(\Omega):= au(\Omega) + bv(\Omega),\, \forall a,b \in \R{} $. Then   $$ (au +bv)'(\Omega;V):= au'(\Omega;V) + bv'(\Omega;V),\, \forall a,b \in \R{} \, .$$
		\item[ii)]  If $ u(\Omega) \in C^{2}(\Omega,\R{m}) $, then  $ (Du)'(\Omega;V)$ $=D(u'(\Omega;V)) $. 
	\end{itemize}
\end{lem}

\begin{proof}
	$i)$ $ (u_i)'=(u_i)\dot{} -  (\nabla u_i)^{\top}V = (\dot{u} -Du V)_{i} =(u')_{i}.$ Moreover,
	The linearity of the differential operator $ (.)\dot{} $, $ \nabla $ and the bilinearity of $ \langle .,. \rangle  $ imply $ (u+v)'=u'+v' $. Since $ a'=0 $, Lemma \ref{Transf_shape_opt:Lem: Product and Chain rule for shape derivatives} yields $ (au')=((au_i)')^{\top}_{i=1,\ldots,m}$ $= (a'u_i+a u_i')^{\top}_{i=1,\ldots,m} =(a u_i')^{\top}_{i=1,\ldots,m}= au'$.
	\\[1em]
	$ ii) $ Let $ m\geq 1 $. Then, according to Lemma \ref{Transf_shape_opt:Lem: d/dt Du_t circ T_t}
	\begin{equation}
	\left(\frac{\partial u_i}{\partial x_j}\right)^{\hspace*{-2mm}\sbt}=\frac{\partial \dot{u_i}}{\partial x_j} -\sum_{k=1}^{n} \frac{\partial u_i}{\partial x_k} \frac{\partial V_k}{\partial x_j},\,~~~ i=1,\ldots,m,\, j=1,\ldots, n.
	\end{equation}
	This implies
	\begin{align}
	\left( \frac{\partial u_i}{\partial x_j} \right)' \hspace*{-1mm}
	&= \frac{\partial \dot{u_i}}{\partial x_j} -\left[\sum_{k=1}^{n} \frac{\partial u_i}{\partial x_k} \frac{\partial V_k}{\partial x_j}\right]  - \left(\nabla \frac{\partial u_i}{\partial x_j}\right)^{\top} V 
	= \frac{\partial \dot{u_i}}{\partial x_j} -\left[\sum_{k=1}^{n} \frac{\partial u_i}{\partial x_k} \frac{\partial V_k}{\partial x_j}  +  \frac{\partial^2 u_i}{\partial x_k \partial x_j} V_k\right]
	\\
	&= \frac{\partial }{\partial x_j}(\dot{u_i} - \nabla u_i^{\top}V)= \frac{\partial u_i'}{\partial x_j},~~~i=1,\ldots,m,\, j=1,\ldots,n.
	\end{align}
	since the second order partial derivatives are symmetric by Schwartz's Theorem:
	\begin{align}
	\frac{\partial}{\partial x_j}(\nabla u_i^{\top }V) &= \sum_{k=1}^{n} \frac{\partial^2 u_i  }{\partial x_j \partial x_k} V_k + \frac{\partial u_i}{\partial x_k} \frac{\partial V_k}{\partial x_j} \\
	&= \sum_{k=1}^{n}\frac{\partial^2 u_i  }{\partial x_k \partial x_j} V_k+\frac{\partial u_i}{\partial x_k} \frac{\partial V_k}{\partial x_j}.
	\end{align}
\end{proof}
The combination of i) and ii) implies that the shape derivative commutes with any linear differential operator if the regularity of $ u(\Omega) $ is high enough and the material derivative exists in $ C^{l} $ for $1\leq l \leq k$ large enough.


\section{Shape derivatives of local cost functionals} \label{Transf_shape_opt:Sec: Eulerian Derivatives} 

The following formula can be found in \cite{SokZol92,Schmidt2011Dissertation,DelfZol11}.

\begin{thm}\label{Transf_shape_opt:Thm: Tangential Stokes Vector}(Integration by Parts on the Boundary/Tangential Stokes Formula)
	Let $ \Omega \subset \R{n}  $  be a domain of class $ C^2 $ and $ f\in C^{1}(\Gamma) $, $ v \in C^{1}(\Gamma,\R{n}) $. Then 
	\[\int_{\Gamma} \Div_{\Gamma}(fv)\, dS =  \int_{\Gamma} \langle \nabla_{\Gamma} f,v \rangle  + f \Div_{\Gamma}v \, dS = \int_{\Gamma} \kappa  \langle f v,\vec{n}\rangle \, dS\]
	where the \textit{mean curvature} of $ \Gamma $ is given by $ \kappa:=\Div_{\Gamma}\vec{n} $. 
\end{thm}
Thereof, we can derive the matrix-vector valued version of this Theorem:
\begin{cor}\label{Transf_shape_opt:Cor: Tangential Stokes Matrix}
	Let $ \Omega \subset \R{n}  $  be a domain of class $ C^2 $,  $ v \in C^{1}(\Gamma,\R{n}) $ a vector field and $ M \in C^1(\Gamma,\R{n \times n}) $ a matrix field. Then
	\begin{align}\label{Transf_shape_opt:Eq: div_thm_matrix_ vector_fields_boundary}
	\int_{\Gamma} \Div_{\Gamma}(Mv)\, dS =
	\int_{\Gamma} \tr(M D_{\Gamma}v) + \langle \Div_{\Gamma}(M),v \rangle \, dS  =   \int_{\Gamma} \kappa \langle M^{\top}\vec{n},v \rangle \, dS
	\end{align}
	where $ \Div_{\Gamma}(M)= \begin{pmatrix} \Div_{\Gamma}(M_{.,1}), & \cdots, & \Div_{\Gamma}(M_{.,n}) \end{pmatrix}^{\top} $ and $ M_{.,k} $ the $ k $-th column of $ M $.
\end{cor}

\subsection{Reynolds transport theorem in shape calculus}\label{Transf_shape_opt:Sec: Transport Theorem}  

\begin{lem}
	\label{Transf_shape_opt:Lem:Reynolds Theorem Volume} \cite[Section 2.31]{SokZol92}
	Assume that the $ C^0 $-material derivative of $ u(\Omega_t) \in C(\overline{ \Omega_{t}},\R{}) $ in direction of $ V  $ exists at $ t=0 $. Let 
	$$ J(\Omega):=\int_{\Omega} u(\Omega) \, dx.$$
	\begin{itemize}
		\item[i)] The the Eulerian derivative of $ J:\mathcal{O}_{k} \to \R{} $ is given by
		\begin{align}
		dJ(\Omega)[V]=\int_{\Omega} \dot{u}(\Omega;V) + u(\Omega)\Div(V) \, dx.
		\end{align}
		\item[ii)] If the $ C^1 $-material derivative of $ u(\Omega_t)  \in C^1(\overline{ \Omega_{t}},\R{}) $ exists at $ t=0 $ then 
		\begin{align}
		dJ(\Omega)[V] =\int_{\Omega} \Div(u(\Omega)V) + u'(\Omega;V)  \, dx= \int_{\Omega} u'(\Omega;V)\, dx + \int_{\Gamma} u(\Omega) V_{\vec{n}} \, dS.
		\end{align}
	\end{itemize} 
\end{lem}

\begin{proof}
	We apply Lemma \ref{Diff_Banach_Space:Exmp:G_Diff_Integral} and Lemma \ref{Diff_Banach_Space:Lem:Product_rule_Gateaux} to the Gâteaux differentiable maps
	$(-\epsilon,\epsilon) \to C(\overline{\Omega}),\, t \mapsto u(\Omega_t) \circ T_t \gamma_{t}$. Therefore,
	\begin{align*}
	\frac{d}{dt}J(t)= \int_{\Omega} \left(\frac{d}{dt}\gamma_{t}\right)u(\Omega_t)  \circ T_t + \gamma_{t}\frac{d}{dt}(u(\Omega_t)  \circ T_t) \, dx.
	\end{align*}
	At $ t=0 $ this leads to
	\begin{align}
	dJ(\Omega)[V] 
	= & \,  \int_{\Omega} \hspace*{-1mm} \Div(V) u(\Omega) + \dot{u}(\Omega;V)  \, dx 
	= \,  \int_{\Omega} \hspace*{-1mm} \Div(V) u(\Omega;V) + u'(\Omega;V) + Du(\Omega)V  \, dx\\
	=& \,  \int_{\Omega} \Div(u(\Omega)V) + u'(\Omega;V)  \, dx = \,  \int_{\Omega} u'(\Omega;V) \, dx + \int_{\Gamma} u(\Omega)  V_{\vec{n}} \, dS.
	\end{align}
\end{proof}

\begin{lem}
	\label{Transf_shape_opt:Lem:Reynolds Theorem Surface} 
	\cite[Section 2.33]{SokZol92}
	Suppose that the $ C^0 $-material derivative of $ y(\Gamma_t)\in C(\Gamma_t,\R{})  $ in direction of $ V $ exists at $ t=0 $. Let 
	$$ J(\Omega):=\int_{\Gamma} y(\Gamma)\, dS  .$$
	\begin{itemize}
		\item[i)] Then the Eulerian derivative of $ J:\mathcal{O}_{k} \to \R{} $ is given by
		\begin{align}
		dJ(\Omega)[V]=\int_{\Gamma} \dot{y}(\Gamma;V)+ y(\Gamma)\Div_{\Gamma}(V) \, dS.
		\end{align}
		\item[ii)] If the $ C^1 $-material derivative of $ y_t \in C^1(\Gamma_t,\R{})  $ exists, then 
		\begin{align}
		dJ(\Omega)[V] = \int_{\Gamma} y'(\Gamma;V) + \kappa y(\Gamma) V_{\vec{n}} \, dS.
		\end{align} 
	\end{itemize} 
	In the special case of $ y(\Gamma)=v(\Omega)\vert_{\Gamma} $, 
	\begin{align}
	dJ(\Omega)[V]
	=\int_{\Gamma} y'(\Gamma;V)+\Div_{\Gamma}(y(\Gamma)V) \, dS
	= \int_{\Gamma} \hspace*{-1mm} v'(\Omega;V) + \left(\frac{ \partial v(\Omega)}{\partial \vec{n}}  + \kappa v(\Omega)\right)\hspace*{-1mm} V_{\vec{n}}\, dS.
	\end{align}
\end{lem}

\begin{proof}
	We apply Lemma \ref{Diff_Banach_Space:Exmp:G_Diff_Integral} and Lemma \ref{Diff_Banach_Space:Lem:Product_rule_Gateaux} to the Gâteaux differentiable maps $(-\epsilon,\epsilon) \to C(\Gamma),\, t \mapsto y(\Gamma_t) \circ T_t \omega_{t}. $ Therefore,
	\begin{align}
	\frac{d}{dt}J(t)&= \int_{\Omega} \left(\frac{d}{dt}\omega_{t} \right)y(\Gamma_t) \circ T_t + \omega_{t} \frac{d}{dt}(y(\Gamma_t) \circ T_t) \, dS .
	\end{align}
	At $ t=0 $ this leads to
	\begin{align}
	dJ(\Omega)[V]
	= & \,  \int_{\Omega} \Div_{\Gamma}(V)y(\Gamma) + \dot{y}(\Gamma;V)  \, dS
	= \, \int_{\Gamma} y'(\Gamma;V) + \Div_{\Gamma}(V)y(\Gamma) + D_{\Gamma}y(\Gamma)V  \, dS \\
	= & \int_{\Gamma} y'(\Gamma;V)+\Div_{\Gamma}(y(\Gamma)V) \, dS 
	=  \int_{\Gamma} y'(\Gamma;V) + \kappa y(\Gamma) V_{\vec{n}} \, dS 
	\end{align}
	where partial integration on the boundary was used in the last step. 
	
	\noindent The last statement follows directly from Lemma \ref{Transf_shape_opt:Lem: Shape derivatives C^l}.
\end{proof}

In the special case that  $ f,\, g \in C^{1}(\R{n}) $ are independent of the shape, then the shape derivatives of 
\begin{align}
J_{vol}(\Omega)=\int_{\Omega} f(x) \, dx && \text{ and } && J_{sur}(\Omega)=\int_{\Gamma} g(x) \, dx
\end{align}
are given by 
\begin{align}
dJ_{vol}(\Omega)[V]  = \int_{\Omega} \langle \nabla f,V\rangle+f\Div(V), dx= \int_{\Omega} \Div(fV)\, dx = \int_{\Gamma} f V_{\vec{n}} \, dS
\end{align}
and 
\begin{align}
dJ_{sur}(\Omega)[V]  &= \int_{\Gamma} \langle\nabla g,V\rangle+g\Div(V)\, dx = \int_{\Gamma} \Div_{\Gamma}(gV) +\frac{ \partial g}{\partial \vec{n}}   V_{\vec{n}} \, dx  \\
&= \int_{\Gamma} \left[\frac{ \partial g}{\partial \vec{n}}   + \kappa g\right]V_{\vec{n}} \, dS.
\end{align}

\begin{rem}
	\begin{itemize}
		\item[i)] Lemma \ref{Transf_shape_opt:Lem: Product and Chain rule for shape derivatives} i) and Lemma \ref{Transf_shape_opt:Lem: Properties u'} also apply for shape derivatives in appropriate Sobolev spaces $ W^{k,p}(\Omega,\R{m}) $ and $ W^{k,p}(\Gamma,\R{m}) $ if $ l \geq 1 $, see also \cite{SokZol92,DelfZol11}.
		\ref{Transf_shape_opt:Lem: Product and Chain rule for shape derivatives} ii) can not be proved that easily e.g. in $ W^{1,1}(\Omega,\R{m}) $ because there is no general chain rule on Sobolev spaces see, for example, \cite{LeoniMorini}.
		\item[ii)] Of cause material and shape derivatives can be defined according to other notions of differentiability. For example the material derivative can be defined point-wise, analogously to \cite{DissKW}: \\[1ex]
		Let $ u(\Omega_t): \Omega_t \to \R{m} $ be defined for $ t \in (-\epsilon,\epsilon) $ and some $ \epsilon>0 $.
		The (weak) material derivative of $ u(\Omega_t) $ at $ t=t_0 $ is defined as the vector field $ \dot{u}(t_0)(\Omega;V)=\dot{u}^{t_0} $ which is the point-wise derivative of $ u(\Omega_t) \circ T_t $ at $ t=t_0 $:
		\[ 
		\dot{u}^{t_0}(x):=\left.\frac{d}{dt}\left(u(\Omega_t) \circ T_t \right)(x)\right\vert_{t=t_0} \hspace*{-3mm}=\lim_{t \to t_0}\frac{1}{t}([u(\Omega_t) \circ T_t](x)-[u(\Omega_{t_0})\circ T_{t_0}](x)).
		\]
		In this case many of the calculation rules can also be established if $ u=u(\Omega)=u(\Omega_0) $ is differentiable or at least the weak derivative exists. Moreover, Reynolds transport theorem can be derived in the classical way using the differentiation rules for parameter integrals , see Section \ref{Diff_Banach_Space:Sec:Parameter_Integrals}. 
	\end{itemize}
\end{rem}

\subsection{General local cost functionals}
\begin{defn}[Local Cost Functionals]\label{Transf_shape_opt:Defn: Local Cost Functionals}
	Suppose that $ u(\Omega) \in C^l(\overline{\Omega},\R{m})$, $ 0 \leq l\leq k $  and let $ \mathcal{F}_{vol},\, \mathcal{F}_{sur} \in C^1(\R{d}) $ with $ d = n+ m\frac{n^{l+1}-1}{n-1} $ if $ \Omega \subset \R{n}, n \geq 2 $ and $d=1+m(l+1) $ if $n=1$. A \textit{local cost functional of $ l $-th order} is a mapping 
	\begin{align}\label{Transf_shape_opt:Eq: Local Integral Cost Functional}
	J: \mathcal{O}_k \to \R{} ,\,\Omega \mapsto J_{vol}(\Omega) + J_{sur}(\Omega)
	\end{align}
	where
	\begin{align}
	J_{vol}(\Omega) &:= \int_{\Omega} \mathcal{ F}_{vol}(x,u(\Omega)(x),\ldots, D^lu(\Omega)(x)) \, dx < \infty\, \forall \Omega \in \mathcal{O}_k \label{Transf_shape_opt:Defn: Volume_Cost_Functional}\\
	J_{sur}(\Omega) &:= \int_{\partial \Omega} \hspace*{-2mm} \mathcal{ F}_{vol}(x,u(\Omega)(x),\ldots, D^lu(\Omega)(x)) \, dS < \infty \,\forall \Omega \in \mathcal{O}_k . \label{Transf_shape_opt:Defn: Surface_Cost_Functional}
	\end{align}
\end{defn}
Thus we obtain mappings
\begin{align}
t \in (-\epsilon,\epsilon) \to J_{vol}(\Omega_t) &:= \int_{\Omega_{t}} \mathcal{ F}_{vol}(x,u_t(x), Du_t(x),\ldots, D^lu_t(x)) \, dx \, \in \R{} \\[1em]
t \in (-\epsilon,\epsilon) \to J_{sur}(\Omega_t) &:= \int_{\partial \Omega_{t}} \mathcal{ F}_{vol}(x,u_t(x), Du_t(x),\ldots, D^lu_t(x)) \, dS\, \in \R{} \\
\end{align}

\subsubsection{Shape derivatives of first order local cost functionals}

In the following we regard local cost functionals of first order, i.e. 
\begin{equation}\label{Transf_shape_opt:Eq: Loc_Cost_Functionals_first_order}
\begin{split}
J_{vol}(\Omega_t)&=\int_{\Omega_t}\mathcal{ F}_{vol}(x,u_t(x), Du_t(x))\, dx \\	
J_{sur}(\Omega_t)&=\int_{\Gamma_t}\mathcal{ F}_{sur}(x,u_t(x), Du_t(x))\, dS
\end{split}
\end{equation}
or linear combinations thereof where $ \mathcal{F}_{vol},\, \mathcal{F}_{sur} \in C^1(\R{d}) $
\begin{align*}
\mathcal{F}_{vol/sur}:\R{d}\cong\R{n}\times \R{m} \times \R{n \times m}
,\, (z_1,z_2,z_3) \to 	\mathcal{F}_{vol/sur}(z_1,z_2,z_3).
\end{align*}

\begin{lem}[Shape Derivative in Material Derivative Form]\label{Transf_shape_opt:Lem: d/dt int(T_t,u_t,Du_t) in material derivative form} Suppose that $ \Omega \in \mathcal{O}_{1} $, $ u_t \in C^1(\overline{ \Omega_{t}},\R{m}) $, $ t \in (-\epsilon,\epsilon) $ and let $ J:=J_{vol}+J_{sur} $ be defined as above.
	Assumed that the $ C^1 $-material derivative $ \dot{u}^t $ of $ u_t $ in direction of $ V \in \Vad{1}{\Oext} $ exists at any $ t \in (-\epsilon,\epsilon) $ such that $ t \in I \mapsto \dot{u}^{t}\in C^1(\overline{ \Omega},\R{m})  $ is continuous. Then the mapping
	\begin{equation}
	\begin{split}
	\mathcal{J}: I &\to \R{},\,t \mapsto J(\Omega_t)
	\end{split}
	\end{equation}			
	is Fréchet differentiable. At $ t=0 $ the derivative is given by
	\begin{equation} \label{Transf_shape_opt:Eq: d/dt int(T_t,u_t,Du_t) in material derivative form}
	\begin{split}
	dJ(\Omega)[V]
	=&\, \int_{\Omega}
	\Div(V)(x) \mathcal{F}_{vol}(x,u(x),Du(x)) \, dx   \\
	&+ \int_{\Omega} \left\langle \frac{\partial \mathcal{F}_{vol}}{\partial z_1} (x,u(x),Du(x)
	) ,V(x) \right\rangle \, dx   \\
	&+  \int_{\Omega} \left\langle\frac{\partial \mathcal{F}_{vol}}{\partial z_2} (x,u(x),Du(x)
	) ,\dot{u}(x) \right\rangle \, dx \\
	&+  \int_{\Omega} \frac{\partial \mathcal{F}_{vol}}{\partial z_3} (x,u(x),Du(x)
	) : (D\dot{u}(x) - Du(x)DV(x)) \, dx \\
	&+ \int_{\Gamma}
	\Div_{\Gamma}(V)(x) \mathcal{F}_{sur}(x,u(x),Du(x)
	) \, dS  \\
	&+ \int_{\Gamma} \left\langle \frac{\partial \mathcal{F}_{sur}}{\partial z_1} (x,u(x),Du(x)
	) ,V(x) \right\rangle \, dS  \\
	&+  \int_{\Gamma} \left\langle\frac{\partial \mathcal{F}_{sur}}{\partial z_2} (x,u(x),Du(x)
	) ,\dot{u}(x) \right\rangle \, dS \\
	&+  \int_{\Gamma} \frac{\partial \mathcal{F}_{sur}}{\partial z_3} (x,u(x),Du(x)) : (D\dot{u}(x) - Du(x)DV(x)) \, dS.  
	\end{split}
	\end{equation} 
\end{lem}

\begin{proof}
	We apply  Lemma \ref{Transf_shape_opt:Lem:Reynolds Theorem Volume} i) and Lemma \ref{Transf_shape_opt:Lem:Reynolds Theorem Surface} i) to $ \mathcal{ F}_{vol/sur} \circ w(\Omega) $
	where $ w(\Omega):=(w_1(\Omega),w_2(\Omega),w_3(\Omega))
	$ with $ 
	w_1(\Omega_t):= id \vert_{\overline{\Omega}_{t}} \,\in C^\infty(\overline{\Omega}_t,\R{n})$ , $w_2(\Omega_t):= u_t \, \in C^1(\overline{\Omega}_t,\R{n}) $, $
	w_3(\Omega_t):= Du_t\,\in  C(\overline{\Omega}_t,\R{m\times n})$
	and $ \tilde{w}(\Gamma)=w(\Omega)\vert_{\Gamma} $. Then, Lemma \ref{Transf_shape_opt:Lem: d/dt Du_t circ T_t} implies
	\begin{align*}
	\dot{w}(\Omega;V)
	=&(V,\dot{u}(\Omega;V),(Du)\dot{}(\Omega;V)) 
	=(0,\dot{u}(\Omega;V), D\dot{u}(\Omega;V) -DuDV) 
	\end{align*} 
	and Lemma \ref{Transf_shape_opt:Lem: arithmetic rules mat deriv I} leads to 
	\begin{align*}
	&(\mathcal{ F}_{vol/sur} \circ w)\dot{}
	= \sum_{i=1}^{3}\left\langle \frac{\partial \mathcal{F}_{vol/sur}}{\partial z_i} \circ w(\Omega) , (w_i)\dot{} \right\rangle \\
	&=\left\langle \frac{\partial \mathcal{F}_{vol/sur}}{\partial z_1} \circ w(\Omega), V  \right\rangle +\left\langle \frac{\partial \mathcal{F}_{sur}}{\partial z_2} \circ w(\Omega) , \dot{u}\right\rangle  + \left(\frac{\partial \mathcal{F}_{vol/sur}}{\partial z_3} \circ w(\Omega)\right): ( D\dot{u} - DuDV) .
	\end{align*}
	
	Alternatively, apply Lemma \ref{Diff_Banach_Space:Lem:d/dt int f(t)F(t,u(t))} with $ f_{v}(t)=\gamma_{t} $ and $ f_{s}(t)=\omega_t $ and 
	\begin{align*}
	u_1:I &\to C^1(\overline{\Omega},\R{n}),\, t \mapsto T_t,\\
	u_2:I &\to C^1(\overline{\Omega},\R{m}),\, t \mapsto u_t \circ T_t=u^t ,\\
	u_3:I &\to C^{0}(\overline{\Omega},\R{m\times n}),\, t \mapsto Du_t \circ T_t.
	\end{align*}
	
\end{proof}

The following formula can also be found in \cite{SokZol92}. Here we supply some further details of the proof:

\begin{lem}\label{Transf_shape_opt:Lem: d/dt int J(T_t,u_t,Du_t) in shape derivative form}
	Suppose that $ \Omega \in \mathcal{O}_{2} $, $ u_t\in C^2(\overline{ \Omega_{t}},\R{m}),\, t \in (-\epsilon,\epsilon)  $ and let $ J:=J_{vol}+J_{sur} $ be a local cost functional of first order, see \eqref{Transf_shape_opt:Eq: Loc_Cost_Functionals_first_order}.
	Presumed that the $ C^1 $-material deriavtive $ \dot{u}^{t_0} $ of $ u_t $ in direction of $ V \in \Vad{1}{\Oext} $ exists at any $ t_0 \in (-\epsilon,\epsilon) $  such that $ t \in I \mapsto \dot{u}^{t}\in C^1(\overline{ \Omega},\R{m})  $ is continuous. Then the mapping
	\begin{equation}
	\begin{split}
	\mathcal{J}: I \to \R{},\,
	t \mapsto J(\Omega_t)
	\end{split}
	\end{equation}			
	is Fréchet differentiable. At $ t=0 $ the Fréchet derivative is given by 
	\begin{equation}
	\begin{split}\label{Transf_shape_opt:Eq: d/dt int J(T_t,u_t,Du_t) in shape derivative form}
	dJ(\Omega)[V]  =&\,  \int_{\Omega} \left\langle\frac{\partial \mathcal{F}_{vol}}{\partial z_2} (x,u(x),Du(x)
	) ,u'(x) \right\rangle + \frac{\partial \mathcal{F}_{vol}}{\partial z_3} (x,u(x),Du(x)
	) :Du'(x) \, dx \\
	&+  \int_{\Gamma} \mathcal{F}_{vol} (x,u(x),Du(x)) V_{\vec{n}} +  \left\langle\frac{\partial \mathcal{F}_{sur}}{\partial z_1} (x,u(x),Du(x)
	), \vec{n}  \right\rangle  V_{\vec{n}}  \, dS   \\ 
	&+  \int_{\Gamma} \left\langle\frac{\partial \mathcal{F}_{sur}}{\partial z_2} (x,u(x),Du(x)
	), u'(x) +  V_{\vec{n}} \frac{ \partial u}{\partial \vec{n}}       \right\rangle \, dS \\
	&+  \int_{\Gamma} \frac{\partial \mathcal{F}_{sur}}{\partial z_3} (x,u(x),Du(x)
	) :\{Du'(x)+ D(Du)[\vec{n}](x)V_{\vec{n}}\}   \, dS \\
	&+ \int_{\Gamma} \kappa \mathcal{F}_{sur}(x,u(x),Du(x))  V_{\vec{n}} \, dS.
	\end{split}
	\end{equation} 
	Here the following notation was used: \[ (D(A)[\vec{v}])_{ij}=(\langle \nabla a_{i,j}, \vec{v} \rangle)_{ij},\, A \in C^1(\Omega,\R{n \times m}),\,  \vec{v} \in \R{n},\, i=1,...,m,\, j=1,...,n.\]
\end{lem}
\begin{proof}
	
	The proof is analogous to the proof of the previous Theorem. This time, we apply  Lemma \ref{Transf_shape_opt:Lem:Reynolds Theorem Volume} ii) and Lemma \ref{Transf_shape_opt:Lem:Reynolds Theorem Surface} ii). Again we consider $ w(\Omega):=(w_1(\Omega),w_2(\Omega),w_3(\Omega))$ 
	and $ \tilde{w}(\Gamma)=w(\Omega)\vert_{\Gamma} $. Then, Lemma \ref{Transf_shape_opt:Lem: Properties u'} ii) implies $
	w'(\Omega;V)
	=(0,u'(\Omega;V), Du'(\Omega;V)) $ holds and
	\begin{align*}
	\tilde{w}'(\Gamma;V)
	= &\left(0 + \frac{\partial id}{\partial \vec{n}}V_{\vec{n}} ,u'(\Omega;V) + \frac{\partial u(\Omega)}{\partial \vec{n}}V_{\vec{n}},Du'(\Omega;V) + D(Du)[\vec{n}] V_{\vec{n}} \right),
	\end{align*} can be derived from Lemma \ref{Transf_shape_opt:Lem: Shape derivatives C^l}.  Lemma \ref{Transf_shape_opt:Lem: Product and Chain rule for shape derivatives} implies
	\begin{align*}
	(\mathcal{ F}_{vol} \circ w)'(\Gamma;V) 
	& = \sum_{i=1}^{3}\left\langle \frac{\partial \mathcal{F}_{sur}}{\partial z_i} \circ w(\Omega) , w_i'(\Omega;V) \right\rangle \\
	&=\left\langle \frac{\partial \mathcal{F}_{sur}}{\partial z_1} \circ w(\Omega) ,  \vec{n}\langle V,\vec{n} \rangle \right\rangle +\left\langle \frac{\partial \mathcal{F}_{sur}}{\partial z_2} \circ w(\Omega) , u'(\Omega;V)\right\rangle  \\
	&~+ \left(\frac{\partial \mathcal{F}_{sur}}{\partial z_3} \circ w(\Omega)\right): Du'(\Omega;V) . 
	\end{align*}
	and
	\begin{align*}
	(\mathcal{ F}_{sur} \circ \tilde{w})'(\Gamma;V) 
	=& \sum_{i=1}^{3}\left\langle \frac{\partial \mathcal{F}_{sur}}{\partial z_i} \circ w(\Omega) , \tilde{w}_i'(\Gamma;V) \right\rangle \\
	=&\left\langle \frac{\partial \mathcal{F}_{sur}}{\partial z_1} \circ w(\Omega) ,  V_{\vec{n}} \vec{n} \right\rangle +\left\langle \frac{\partial \mathcal{F}_{sur}}{\partial z_2} \circ w(\Omega) , u'(\Gamma;V) + V_{\vec{n}} \frac{\partial u}{\partial \vec{n}} \right\rangle  \\
	&+ \left( \frac{\partial \mathcal{F}_{sur}}{\partial z_3} \circ w(\Omega)\right) : \left(Du'(\Gamma;V) + D(Du)[\vec{n}] V_{\vec{n}} \right). 
	\end{align*}
	where $ (D(Du)[\vec{n}])_{ij} =(Hu_i\vec{n})_{j} $. 
	Finally we apply Lemma  \ref{Transf_shape_opt:Lem:Reynolds Theorem Volume} to $  \mathcal{ F}_{vol} \circ w(\Omega) $ and   \ref{Transf_shape_opt:Lem:Reynolds Theorem Surface} to $  \mathcal{ F}_{vol} \circ \tilde{w}(\Gamma) $.
\end{proof}

\begin{rem}
	In case that $ \mathcal{F}_{vol} $ or $ \mathcal{F}_{sur} $ explicitly depend on $ \Omega $ or $ \Gamma $, then Lemma \ref{Transf_shape_opt:Lem: Product and Chain rule for shape derivatives} implies
	\begin{align}
	dJ(\Omega)[V]
	=&\, \int_{\Omega}
	\Div(V) \mathcal{F}_{vol}(\Omega,.,u,Du) + \dot{\mathcal{F}}_{vol}(\Omega,.,u,Du) \, dx   \\
	&+ \int_{\Omega} \left\langle \frac{\partial \mathcal{F}_{vol}}{\partial z_1} (\Omega,.,u,Du) ,V \right\rangle + \left\langle\frac{\partial \mathcal{F}_{vol}}{\partial z_2} (\Omega,.,u,Du) ,\dot{u} \right\rangle  dx \\
	&+  \int_{\Omega} \frac{\partial \mathcal{F}_{vol}}{\partial z_3} (\Omega,.,u,Du) : (D\dot{u} - DuDV) \, dx \\
	&+ \int_{\Gamma}
	\Div_{\Gamma}(V) \mathcal{F}_{sur}(\Omega,.,u,Du)+\dot{\mathcal{F}}_{sur}(\Omega,.,u,Du) \, dS \\
	&+ \int_{\Gamma} \left\langle \frac{\partial \mathcal{F}_{sur}}{\partial z_1} + \dot{\mathcal{F}}_{vol}(\Omega,.,u,Du) ,V \right\rangle + \left\langle\frac{\partial \mathcal{F}_{sur}}{\partial z_2} + \dot{\mathcal{F}}_{vol}(\Omega,.,u,Du) ,\dot{u} \right\rangle dS  \\
	&+  \int_{\Gamma} \frac{\partial \mathcal{F}_{sur}}{\partial z_3} + \dot{\mathcal{F}}_{vol}(\Omega,.,u,Du) : (D\dot{u} - DuDV) \, dS.   
	\\
	=&\,  \int_{\Omega} \hspace*{-1mm} \dot{\mathcal{F}}_{vol}(\Omega,.,u,Du) + \left\langle\frac{\partial \mathcal{F}_{vol}}{\partial z_2}(\Omega,.,u,Du) ,u' \right\rangle + \frac{\partial \mathcal{F}_{vol}}{\partial z_3} (\Omega,.,u,Du):Du'\, dx \\
	&+  \int_{\Gamma} \hspace*{-1mm} \dot{\mathcal{F}}_{sur}(\Omega,.,u,Du) + \mathcal{F}_{vol} (\Omega,.,u,Du) V_{\vec{n}} + \left\langle\frac{\partial \mathcal{F}_{sur}}{\partial z_1} (\Omega,.,u,Du), \vec{n}  \right\rangle \hspace*{-1mm}  V_{\vec{n}}  \, dS \\ 
	&+  \int_{\Gamma} \left\langle\frac{\partial \mathcal{F}_{sur}}{\partial z_2} (\Omega,.,u,Du), u' +  \frac{\partial u }{\partial \vec{n}}  V_{\vec{n}}   \right\rangle \, dS  \\
	&+  \int_{\Gamma} \frac{\partial \mathcal{F}_{sur}}{\partial z_3} (\Omega,.,u,Du) :\{Du'+ D(Du)[\vec{n}]V_{\vec{n}}\} + \kappa \mathcal{F}_{sur}(\Omega,.,u,Du)  V_{\vec{n}} \, dS.
	\end{align}
\end{rem}


\chapter[Sensitivity Analysis for Parameter Dependent Linear Variational Equations]{Sensitivity Analysis for Parameter Dependent Linear Variational Equations on Hilbert Spaces}
\label{Parameter_Dep_PDE}

Recall the results of the discussion at the end of Chapter \ref{Reliability}: We figured out that solutions in $ W^{2,p} $ with very high values for $ p $ are necessary to assure for example that the LCF-functional is defined. In case that the failure model contains second order derivatives, we illustrated that even $ u \in W^{3,p} $ and thus a strong solution is needed.  

Consider again equation \eqref{Transf_shape_opt:Eq: d/dt int(T_t,u_t,Du_t) in material derivative form}. Once we found a solution $ u \in C^{1,\phi}(\overline{ \Omega},\R{3}) $ the term $$\frac{\partial \mathcal{F}_{sur}}{\partial z_3} (x,u(x),Du(x))$$ is obviously bounded on $ \Gamma $ such that the minimal requirement for the existence of the integral\footnote{It is not clear if this assumption already assures the existence of the shape derivative $ dJ(\Omega)[V] $.}
\[  \int_{\Gamma} \frac{\partial \mathcal {F}_{sur}}{\partial z_3} (x,u(x),Du(x)) : D\dot{u}(x)\, dS. \]
is $ \dot{u} \in  H^{\nicefrac{3}{2}}(\Omega,\R{3}) $. Therefore $ H^1 $-material derivatives are not sufficient here.

Moreover, formula \eqref{Transf_shape_opt:Eq: d/dt int(T_t,u_t,Du_t) in material derivative form} can be extended to functionals containing derivatives of $ k $-th order ($ k \geq 2 $) which are subject to current research e.g.  failure time models that involve notch support \cite{EPNotch1985,HertelVormwNotch2012,NotchFracMech2016,NotchSizeLCF2018}. In this context at least $\dot{u} \in H^{k+\nicefrac{1}{2}}(\Omega,\R{3}) $ is required. 

Apart from that, we anyways need strong assumptions on the domain regularity and the input data $ f(\Omega) $ and $ g(\Gamma_{N}) $ to assure that the solution $ u $ of \eqref{Reliability:Eq:LinEl} provides enough regularity. Thus the question if these assumptions lead to the existence of material derivatives in the necessary or in even higher topologies is obvious.

The aim of this chapter is to derive an abstract functional analytic framework which allows to prove existence of material derivatives first in Hilbert- and in a second step in Banach topologies. In Chapter \ref{Ex_Shape_Deriv_LinEl} this framework is then applied to linear elasticity and we show existence of material derivatives for linear elasticity in Hölder-spaces.

Inspired by the short outlook in \cite[Sec. 4.3]{DissSchmitz2014} and \cite[Sec. 3.5]{SokZol92}, we investigate the behavior of whole families of parameter dependent variational equations on a Hilbert space $ H $ with a parameter $ t $ in an open interval $ I \subset \R{} $, i.e.
\begin{align}\label{Parameter_Dep_PDE:Eq:b^t(u^t,v)=l^t(v)}
b^t(u^t,v)=l^t(v) ~~~~~\forall v \in H,\, t \in I.
\end{align}
In shape optimization these variational equations (VEs) usually appear in form of weak equations of PDE on parameter dependent domains $ \Omega_t $, with solutions in Sobolev spaces and especially in the Hilbert space $ H^1$, consider \cite{SokZol92,DelfZol11,ShapeOpt}. 

\noindent \textbf{Motivation and Example:}
We set $$ H:=H^1_{0}(\Omega)=\{u \in H^1(\Omega)\vert u = 0 \text{ on } \Gamma=\partial\Omega  \} $$ where $ \Omega\Subset \Omega^{ext} \subset \R{n},\, n\geq 2 $ is a bounded domain.  Let $ (T_{t})_{t \in I} $ be a set of $ C^1 $ - diffeomorphisms from $ \overline{\Omega^{ext}} $ onto $ \overline{\Omega^{ext}} $ and set $ \Omega_t := T_t(\Omega)=\{T_{t}(x) \vert x\in \Omega \} $. Moreover we assume that $ f_t:\Omega_t \to \R{},\, t\in I $ is a family of functions on $ \Omega_t$.
The weak formulation 
\[ \int_{\Omega_t} \langle\nabla u_t,\nabla w \rangle \, dx =\int_{\Omega_t} f_t w \, dx ~~~~~ \forall w \in H^{1}_{0}(\Omega_t)  \] 
of the Laplace equation
\begin{align} \label{Parameter_Dep_PDE:Eq:Laplace}
\left.\begin{array}{rcll}
-\Delta u_t &=& f_{t} &\text{ on } \Omega_t \\
u_t &=&0 &\text{ on }\partial \Omega_t
\end{array}\right.
\end{align}
can be reformulated in the following way: Setting $ u^{t}:=u_t \circ T_t $ and $ v:=w \circ T_t \in H^{1}_{0}(\Omega) $ the left-hand side satisfies 
\begin{align*}
\int_{\Omega_t} \langle\nabla u_t,\nabla w \rangle \, dx 
& = \int_{\Omega_t}  \langle\nabla (u^{t} \circ T_t^{-1}),\nabla (v \circ T_t^{-1}) \rangle\, dx
\\
&=\int_{\Omega}  \langle (DT_{t})^{-\top}\nabla u^{t} ,  (DT_{t})^{-\top}\nabla v \rangle |\det(DT_t)|\, dx 
\end{align*}
and with $ f^{t} = f_t \circ T_t $ the right hand side reads
\begin{align*}
\int_{\Omega_t} f_t w \, dx = \int_{\Omega_t} (f_t \circ T_t)   \circ T_t^{-1}(v \circ T_t^{-1}) \, dx  = \int_{\Omega} f^{t} |\det(DT_t)| v \, dx.
\end{align*}
Thus 
\begin{align*}
\begin{array}{l r l l}
& \int_{\Omega_t} \langle\nabla u_t,\nabla w \rangle \, dx &=\int_{\Omega_t} f_t \, w \, dx & \forall w \in H^{1}_{0}(\Omega_t)  \\[1em]
\Leftrightarrow & \underbrace{\int_{\Omega}  \langle (DT_{t})^{-\top}\nabla u^{t} ,  (DT_{t})^{-\top}\nabla v \rangle |\det(DT_t)|\, dx }_{:=b^{t}(u^{t},v)}& = \underbrace{\int_{\Omega} f^{t}\, |det(DT_t)| v \, dx}_{l^{t}(w)} & \forall v \in H^{1}_{0}(\Omega).
\end{array}
\end{align*}
Then $$ b^{t}(u^t,v) = l^{t}(v) \,~~~ \forall v \in H^{1}_{0}(\Omega)$$
defines a VE in terms of \eqref{Parameter_Dep_PDE:Eq:b^t(u^t,v)=l^t(v)}.

Not only in case of the Laplace equation \cite{SokZol92} or \cite{DelfZol11} it is known the Hilbert space material derivatives $ \frac{d}{dt}u^{t}= \frac{d}{dt} u_t \circ T_t=\dot{u}^{t} $ can be calculated as the solution of the associated variational equation
\[ \frac{d}{dt} b^{t}(u^t,v) = \frac{d}{dt}l^{t}(v)~~~ \forall v \in H. \]
But, until now, there was no approach that generalized this technique to whole classes of VEs. We aim to close this gap in this chapter and derive general conditions under which the solution $ q^t $ of the  variational formulation \[ \dot{b}^t(q^t,v)=\dot{l}^t(v)-b(u^t,v)~~~ \forall v \in H \] in a Hilbert space $ H $ is the derivative of the solution $ u^t $ of the original equation \[ b^t(u^t,v)=l^t(v) ~~~ \forall v \in H,\] 
see the Theorems \ref{Parameter_Dep_PDE:Thm: Existence Deriv. Weak Topology} and Theorem \ref{Parameter_Dep_PDE:Thm: Existence Deriv. strong Topology}. 

An additional outcome of this theorem is the continuity of the derivative mapping $ t \to q^t $ w.r.t. the strong topology on $ H $. 
Theorem \ref{Parameter_Dep_PDE:Thm: Existence of Strong Derivatives u^t } shows that these continuity and differentiability properties can be "transported" to higher topologies using compact embeddings, as they appear in Sobolev, Sobolev-Hölder or Hölder-embeddings. In Chapter \ref{Ex_Shape_Deriv_LinEl}
this result will be crucial to derive material derivatives for linear elasticity in higher order Sobolev spaces or even in classical function spaces.

\section{Linear variational equations and topological setup}\label{Parameter_Dep_PDE:Sect: Topological Setting}
%
%
%
%

Let $ H $ be a Hilbert space with scalar product $ \langle.,.\rangle_{H} $ and induced Norm $ \Norm{.}{H}=\sqrt{\langle .,. \rangle_{H}} $.
It is well known, that the topological dual space $ H'=\mathcal{L}(H,\R{}):=\{l:H \to \R{} \,\vert\, l \mathrm{~is~linear~and~continuous} \}  $ equipped with the operator norm $$ \Norm{l}{H^{'}}:=\sup_{\Norm{v}{H}\leq 1} |l(v)|, \, l \in H'  $$ is again a Hilbert space.

By $ \mathrm{B}(H) $ we denote the vector space of bilinear forms $ b:H \times H \to \R{} $. the form $ b(.,.) $ is bilinear if $ b(u,.):H \to \R{} $ and $ b(.,v):H \to \R{} $ are linear for arbitrary $ u,\, v \in H $. We will now introduced a norm $ \Norm{.}{\mathrm{B}(H)} $ using the Hilbert space tensor product $ H \otimes H $ of $ H $ with itself: 

Since $ \otimes: H \times H \to H \otimes H $ posesses the so called universal property there is a unique linear map $ L_{b}: H \otimes H \to \R{} $ such that $ b=L_b \circ \otimes $. 
Then, we can define a norm on $ \mathrm{B}(H) $ by
\begin{equation}\label{Parameter_Dep_PDE:eq: ident. bilinear form and linear form}
\Norm{b}{\mathrm{B}(H)}:= \hspace*{-2mm}\sup_{\Norm{u}{H}\leq 1 \atop \Norm{v}{H} \leq 1} |b(u,v)| = \hspace*{-2mm} \sup_{\Norm{u}{H}\leq 1 \atop \Norm{v}{H} \leq 1}  \left|L_b(u\otimes v )\right| = \hspace*{-2mm}\sup_{\Norm{u \otimes v}{H \otimes H}\leq 1} \hspace*{-5mm} \left|L_b(u\otimes v )\right|= \Norm{L_b}{L(H \otimes H)}
\end{equation}
consider Section 2.4. and Section 2.6 in \cite{KadRing1983}.

\begin{defn}
	For any bilinear map $ b \in \mathrm{B}(H) $ the mapping $ s_{b}:H\times H \to \R{} $ where $ (u,v)\to s_{b}(u,v):=|b(u,v)| $ is a semi norm on $ \mathrm{B}(H) $. Then the weak topology on $ \mathrm{B}(H)  $ is generated by the family $\lbrace s_{b}|\,b \in \mathrm{B}(H) \rbrace $. 
\end{defn}

\begin{defn}
	\begin{itemize}
		\item[i)] A bilinear map $ b \in \mathrm{B}(H) $ is called \textit{continuous} if for any sequence $ (u_{n},v_{n})_{n \in \N{}} \subset H \times H$ with limit value $ (u,v) \in H \times H$: 
		\[ b(u_{n},v_{n}) \underset{n\to \infty}{\longrightarrow} b(u,v)\]
		\item[ii)] $ b $ is called \textit{bounded} if there exists a constant $ C \geq 0 $ such that $ |b(u,v)| \leq C\Norm{u}{H}\Norm{v}{H},\,\forall (u,v) \in H \times H $.
	\end{itemize}
\end{defn}


\begin{lem}
	Let $ b \in \mathrm{B}(H)$. Then, the following statements are equivalent:
	\begin{itemize}
		\item[i)] $ b $ is continuous on $ H\times H $.
		\item[ii)] $ b $ is continuous in $ (0,0)\in H\times H $.
		\item[iii)] $ b $ is bounded.
		\item[iv)] $ b  $ satisfies $ \Norm{b}{\mathrm{B}(H)} < \infty $.
	\end{itemize}
\end{lem}

\begin{proof}
	This assertion becomes clear if equation \eqref{Parameter_Dep_PDE:eq: ident. bilinear form and linear form} is combined with the fact that any linear map $ L: H \otimes H\to  \R{} $ is continuous if and only if it is bounded. Then the statement follows directly from Theorem II.2.1 \cite{Werner_Funkana}.
\end{proof}

Due to the identification of bilinear forms with their associated linear forms
$$ \Norm{b}{\mathrm{B}(H)}=\inf \{C \geq 0: |b(u,v)| \leq C \Norm{u}{H}\Norm{v}{H},u,v \in H \} $$ 
holds for any $ b\in \mathrm{B}(H)$.

In the following we will denote the set of continuous bilinear mappings by $ \mathcal{B}(H) $.
It is now clear that, analogously to linear operators, continuity and boundedness of bilinear forms are equivalent formulations and thus the following definition becomes reasonable:

\begin{defn}
	A set of bilinear forms $ M \subset \mathrm{B}(H) $ is called \textit{equicontinuous} if there exists a constant $ C_{M} \geq 0$ such that 
	\begin{equation*}
	|b(u,v)| \leq C_{M} \Norm{u}{H}\Norm{v}{H} \, \forall b \in M,\, \forall u,\,v \in H.
	\end{equation*}
\end{defn}

\begin{defn}
	\begin{itemize}
		\item[i)]A bilinear form $ b \in \mathrm{B}(H) $ is called \textit{coercive} if there exits a constant $ \Varlambda\geq 0 $ such that $ b(u,u)\geq \Varlambda\Norm{u}{H}^2 \, \forall u \in H $.
		\item[ii)] $ b \in \mathrm{B}(H) $ is called \textit{strictly coercive}  if there exits a constant $ \Varlambda> 0 $ such that $ b(u,u)\geq \Varlambda\Norm{u}{H}^2 \, \forall u \in H $.
		\item[iii)] A set of bilinear forms $ M \subset \mathrm{B}(H) $ is called (strictly) \textit{coercive}  if any $ b \in M $ is a (strictly) coercive.
		It is called \textit{equicoercive}, if there exists a constant $ \Varlambda_M >0$ such that 
		\[ |b(u,u)| \geq \Varlambda_M \Norm{u}{H}^{2}  \, \forall u \in H,\, b \in M. \]
	\end{itemize}
\end{defn}

\begin{thm}[Theorem of Lax-Milgram] \cite[Sec. 6.2 Thm. 1]{Evans} \label{Parameter_Dep_PDE:Thm: Lax-Milgram}
	Let  be a continuous bilinear form.
	\begin{itemize}
		\item[i)] For any $ b \in \mathcal{B}(H) $ there exists a unique operator $ T \in \mathcal{L}(H,H) $ such that $ b(u,v)=\langle u,Tv\rangle.$
		\item[ii)] If $ b $ is additionally strictly coercive with $  b(u,u)\geq \Varlambda\Norm{u}{H}^2 \, \forall u \in H $ for some $ \Varlambda>0 $, then $ T $ is invertible and $ \Norm{T}{\mathcal{L}(H,H)}\leq \frac{1}{\Varlambda} $.
	\end{itemize}
\end{thm}

\begin{proof}
	i) For any $ v \in H $ the linear form $ b(.,v) $ is a bounded and linear functional (with $ \Norm{b(.,v)}{H'} \leq C\Norm{v}{H} $) and therefore the Riesz representation theorem implies that there is a unique $ \tilde{v} \in H $ such that $ b(u,v)=\langle u,\tilde{v} \rangle~~~ \forall u \in H.$ Then define $ T: H \to H,\, v \to \tilde v  $. Then $ T $ is linear because 
	\begin{align*}
	\langle u,T(\lambda_1v_1 + \lambda_2v_2 ) \rangle &= b( u,\lambda_1v_1 + \lambda_2v_2 ) =\lambda_1 b( u,v_1 )+\lambda_2 b( u,v_2 ) \\
	&= \lambda_1 \langle u,Tv_1\rangle + \lambda_2 \langle u,Tv_2\rangle = \langle u,\lambda_1Tv_1 + \lambda_2Tv_2  \rangle ~~~\forall u \in H
	\end{align*}
	and continuous since $\Norm{Tv}{H}^2 =\langle Tv,Tv\rangle = B(Tv,v)\leq C\Norm{Tv}{H}\Norm{v}{H} \, \forall v \in H$ implies $ \Norm{Tv}{H} \leq C \Norm{v}{H} $.\\[1ex]
	The coercivity and continuity of $ b $ imply $ \Varlambda \Norm{v}{H}^2 \leq b(v,v)=\langle v,Tv \rangle \leq \Norm{v}{H}\Norm{Tv}{H} $, whence $ \Varlambda  \Norm{v}{H} \leq \Norm{Tv}{H}$. Let $ v \in H $ with $ v \neq 0 $, then also $ \Norm{Tv}{H} \neq 0  $ and hence $ T $ is injective. Moreover, $ \im{T} $ is closed: Therefore let $ \im{T} \supset y_{n}=Tv_n \underset{n \to \infty}{\to} y \in H  $ be a convergent sequence. Then $ y_{n} $ is a Cauchy sequence and $$\Varlambda \Norm{v_n-v_m}{H} \leq  \Norm{Tv_n -Tv_m}{H} = \Norm{y_n - y_m}{H}\, \forall n,m \in \N{}.$$ But then also $ (v_n)_{n} $ is a Cauchy sequence and converges in the Banach space $ H $ to some $ v \in H $ and because $ T $ is continuous we conclude $ y_n = Tv_n \to Tv =y \in \im{T} $. It is left to show that $ \im{T}=H $. Since $ \im{T} $ is closed there exists its orthogonal complement such that $ H=\im{T}\oplus \im{T}^{\perp} $. Suppose that $ z \in \im{T}^{\perp},\, z \neq 0$, then $0< \Varlambda \Norm{z}{H}^2<B(z,z)=\langle z, Tz \rangle=0 $, which is a contradiction and so $ \im{T}^{\perp}=\{0\} $. Finally, the estimate $ \Norm{T^{-1}}{H'} \leq \frac{1}{\Varlambda} $ follows from $ \Norm{T^{-1}y}{H} = \Norm{v}{H} \leq \frac{1}{\Varlambda}\Norm{Tv}{H}=\frac{1}{\Varlambda}\Norm{y}{H}$ with $ y=Tv $.
\end{proof} 

\begin{lem}[Lemma of Lax-Milgram] \cite[Sec. 6.2 Thm. 1]{Evans} Let $ B \in \mathcal{B}(H)$ be a strictly coercive bilinear form and suppose that  $ l \in H' $. Then there exists a unique solution to \[ b(v,u)=l(v) \, \forall u \in H. \]
\end{lem}

\begin{proof} Since $ l \in H' $ the Riesz representation Theorem implies that there is a unique $ q \in H $ such hat $ l(v)=\langle v,q\rangle \, \forall v \in H$ and and Theorem \ref{Parameter_Dep_PDE:Thm: Lax-Milgram} yields that there exists $ T \in H' $ bijective with $b(v,u)= \langle v,Tu\rangle  $. Thus 
	\begin{align*}
	b(v,u)=l(v) \, \forall v \in H \Leftrightarrow \langle v,Tu\rangle = \langle v,q\rangle \, \forall v\in  H \Leftrightarrow Tu=q \Leftrightarrow u =T^{-1}q
	\end{align*}
	and the assertion follows from the uniqueness of $ q $ and $ T $.
\end{proof}

On base of these well known theorems we can deduce the following lemma which will be a helpful tool throughout this section:

\begin{lem}[Criterion for Weak Convergence]\label{Parameter_Dep_PDE:Lem:Equiv formulation for weak convergence}
	Suppose that $ H $ is a Hilbert space and let $ b(.,.) $ be a strictly coercive and continuous bilinear form. i.e.  \[ \Varlambda \Norm{u}{H}^2 \leq b(u,u) \, \forall u\in H ~~~~\text{ and } ~~~~|b(u,v)|\leq C\Norm{u}{H}\Norm{v}{H}.\]
	for some constants $C, \Varlambda>0 $. Then the following statements are equivalent:
	\begin{itemize}
		\item[i)] The sequence $ (u_{n})_{n \in \N{}} $ converges weakly to zero; $ u_{n}  \rightharpoonup 0$ in $ H $.
		\item[ii)] The sequence $ (u_{n})_{n \in \N{}} $ satisfies  $ b(u_{n},v) \to 0 \, \forall v \in H $. 
	\end{itemize}
\end{lem}

\begin{proof}
	Due to the Theorem \ref{Parameter_Dep_PDE:Thm: Lax-Milgram} of Lax-Milgram  there exists a unique continuous and invertible operator $ T \in H' $ such that $ b(u,v)=\langle u,Tv\rangle_{H} \, \forall u,v \in H.$
	
	\noindent First we show that $ i) \Rightarrow i
	i) $. Therefore let $ l \in H^{\prime} $ be arbitrary. Then, by Riesz representation theorem, there exists a unique $ v \in H $ such that 
	\[ l(.)=\langle .,v\rangle_{H}= \langle .,TT^{-1}v\rangle_{H}=b(.,T^{-1}v).\]
	Suppose that $ (u_{n})_{n \in \N{}}  $ satisfies $ b(u_{n},w) \to 0 \, \forall w \in H$ as $ n \to \infty$. Since $ T^{-1} $ is bijective this is equivalent to $ b(u_{n},T^{-1}w) \to 0  \, \forall w \in H$ as $ n \to \infty $. Since $ v \in H $,
	\[ l(u_{n})= \langle u_{n},v\rangle_{H}= b(u_{n},T^{-1}v) \to 0  \]
	which shows the weak convergence of $ (u_{n})_{n\in \N{}} $.
	
	\noindent $ ii) \Rightarrow i) $: Now let $ l(u_{n}) \to 0 $ for all $ l\in H^{\prime} $ and choose an arbitrary $ v \in H $. Again we regard the operator $ T \in H' $ satisfying $ b(u,v)=\langle u,Tv\rangle_{H} \, \forall u,v \in H.$ 
	The mapping $ x \mapsto \langle x,Tv\rangle_{H} $ is linear and therefore
	\[ \langle u_{n},Tv\rangle_{H}=b(u_{n},v)\to 0  \]
	by assumption. Since $ v \in H $ was arbitrary, the assertion follows.
\end{proof}

\begin{rem}
	For the direction $  ii) \Rightarrow i)  $ in the previous Lemma we only need the continuity of the bilinear form because this already implies the existence of a linear, continuous operator $ T:H \to H $ exists such that $ b(u,v)=\langle u,Tv\rangle_{H} $, confer Theorem \ref{Parameter_Dep_PDE:Thm: Lax-Milgram}. 
\end{rem}

\begin{defn}\label{Parameter_Dep_PDE:Defn: Differentiation on H}
	Let $ (X,\Norm{.}{X}) $ be a Banach space and $u^{(\cdot)}: t \in I \to u^{t} \in X  $.
	\begin{enumerate}
		\item[1)]
		$ i) $ The map $u^{(\cdot)} $ is called \textit{(stongly) continuous} if it is continuous regarding the strong norm topology on $ X $.\\[1ex] 
		$ ii) $  The map $u^{(\cdot)} $ is called \textit{weakly continuous} if it is continuous regarding the weak topology on $ X $ which is generated by the system of semi norms $ \{| l(.)|\}_{l \in X'}  $.
		\item[2)]The mapping $ u^{(\cdot)} $ is  \textit{ differentiable} at $ t \in I $ regarding the strong $ i) $ or weak $ ii)  $ topology if there exists $ \dot{u}^{t} \in X $ such that
		\begin{align*}
		i)  \frac{1}{h}(u^{t+h}-u^t) \underset{h \to 0}{\rightarrow} \dot{u}^t \text{ in } X, && \text{ or } &&
		ii) \frac{1}{h}(u^{t+h}-u^{t}) \underset{h \to 0}{\rightharpoonup} \dot{u}^{t} \text{ in } X. 
		\end{align*}
		If $ \dot{u}^{t} $ exists for all $t \in I$, then  $ u^{(\cdot)} $ is called differentiable on $ I $ regarding the strong or weak topology.
	\end{enumerate}
\end{defn}

\begin{rem}
	Differentiability regarding the strong norm topology is equivalent to Gâteaux differentiability, compare Chapter
	\ref{Diff_Banach_Space}.
\end{rem}

Let $ H $ be a Hilbert space. Then we can similarly regard two topologies on $ H' $: The (strong) norm topology generated by $ \Norm{.}{H'} $ and the weak topology defined by the family of semi norms $ \{|\Phi(.)|\}_{\Phi \in H''} $. Since $ H $ is reflexive as it is a Hilbert space for any $ \Phi \in H'' $  there is an element $ u \in H $ such that $ \Phi(l)=l(u) \, \forall l \in H' $. So, as more handsome alternative, we can consider the weak topology generated by $ \{l \in H' \to |l(u)|\}_{ u \in H }$.

\begin{defn}\label{Parameter_Dep_PDE:Defn: Differentiability on H'}
	Let $ (l^{t})_{t \in I} \subset H' $ be a family of continuous linear forms and regard the associated mapping $ l^{(\cdot)}:I \to H,\, t \to l^{t} $. 
	\begin{itemize}	 
		\item[1)] 
		$ i) $ The mapping $ l^{(\cdot)} $ is called (strongly) continuous on $ H' $ if  
		$$l^{t} \underset{t \to s}{\rightarrow} l^{s} \in H' .$$
		$ ii) $  The mapping $ l^{(\cdot)} $ is called \textit{weakly continuous} on $ H' $ if  $$ l^{t} \underset{t \to s}{\rightharpoonup} l^{s} \in H' :\Leftrightarrow |l^{t}(u)-l^{s}(u)|\underset{t \to s}{\rightarrow} 0 \, \forall u \in H.  $$
		\item[2)]The mapping $ t \to l^{t} $ is called differentiable in $ t  \in I $ reagrding the strong norm topology $ i) $ or the weak topology $ ii) $ on $ H' $ if   there exists $ \dot{l}^{t} \in H' $ such that 
		\begin{align*}
		i)~~ \frac{1}{h}(l^{t+h}-l^t) \underset{h \to 0}{\rightarrow}  \dot{l}^t \text{ in } H' &&  \text{ or } &&
		ii)~ \frac{1}{h}(l^{t+h}-l^{t})\underset{h \to 0}{\rightharpoonup} \dot{l}^{t} \text{ in } H'. 
		\end{align*}
	\end{itemize}
\end{defn}


On $ \mathcal{B}(H) $ we also regard two topologies, which are again the (strong) norm topology generated by $ \Norm{.}{\mathcal{B}(H)} $ and the weak topology created by the semi norms $ s_{(u,v)}(b):=|b(u,v)|,\,b \in \mathcal{B}(H), (u,v) \in H\times H $. 

\begin{defn}\label{Parameter_Dep_PDE:Defn:Differentiability on B(H)}
	Let $ (b^{t})_{t \in I} \subset \mathcal{B}(H) $ be a family of linear forms and consider the associated mapping $
	b^{(\cdot)}:I \to\mathcal{B}(H),\, t \to b^{t} $. 
	\begin{itemize}	 
		\item[1)]   
		$ i) $ The mapping $ b^{(\cdot)} $ is called \textit{(strongly) continuous on $ \mathcal{B}(H) $} if  $$ b^{t} \underset{t \to s}{\rightarrow} b^{s} \in  \mathcal{B}(H).$$
		$ ii) $ The mapping $ b^{(\cdot)} $ is called \textit{weakly continuous on $ \mathcal{B}(H) $} if $$ b^{t} \underset{t \to s}{\rightharpoonup} b^{s} \in\mathcal{B}(H) :\Leftrightarrow |b^{t}(u,v)-b^{s}(u,v)| ~\forall (u,v) \in H \times H \rightarrow 0,\, t \to s.$$
		\item[2)] Differentiability is defined analogously to Definition \ref{Parameter_Dep_PDE:Defn: Differentiability on H'} 2).
	\end{itemize}
\end{defn}

\section{Continuity of solution mappings }

\begin{lem}[Criterion for Strong Continuity]\cite{DissSchmitz2014} \label{Parameter_Dep_PDE:Lem: Crit. for strong cont.}
	Suppose that $ (b^{t})_{t \in I} \subset$ \\$ \mathcal{B}(H) $ is an equicoercive family of bilinear forms on a Hilbert Space $ H $ such that for any $ u \in H$ the mapping $$ I \to H^{\prime}, t \mapsto b^{t}(u,.) $$ is (strongly) continuous on $ H' $.
	Further, let $ (l^{t})_{t \in I} \in H' $ be a family of linear forms, where $ I \to H', t \mapsto l^{t}  $ is (strongly) continuous on $ H' $ and suppose that $ u^{t} \in H $ is the unique solution of \[ b^{t}(u^{t},v)=l^{t}(v)~~~ \forall v \in H \]
	for any $ t \in I $.
	Then $$ t \in I \to u^{t} $$ defines a (strongly) continuous mapping on $ H $.
\end{lem}

\begin{proof} 
	Let $ t \in I $ be arbitrary. The Theorem of Lax-Milgram implies that there is exactly one solution $ u^{t} $ such that $ b^{t}(u,v)=l^{t}(v)\, \forall v \in H $. 
	
	\noindent Now, let $ s\neq t \in I $. Under the given conditions there exists a constant $ \Varlambda>0 $ such that $$\Varlambda\Norm{u^{s}-u^{t}}{H}^{2} 
	\leq b^{s}(u^{s}-u^{t},u^{s}-u^{t})= l^{s}(u^{s}-u^{t})- b^{s}(u^{t},u^{s}-u^{t})$$
	since $ =  b^{s}(u^{s},u^{s}-u^{t}) = l^{s}(u^s -u^{t}) $.
	Therefore,
	\begin{align*}
	\Varlambda\Norm{u^{s}-u^{t}}{H}^{2} 
	&\leq  \underbrace{b^{t}(u^{t},u^{s}-u^{t})- l^{t}(u^{s}-u^{t})}_{=0}+ l^{s}(u^{s}-u^{t})-b^{s}(u^{t},u^{s}-u^{t})\\
	&=(b^{t}-b^{s})(u^{t},u^{s}-u^{t})+  (l^{s}-l^{t})(u^{s}-u^{t})\\[1ex]
	& \leq \left( \vert b^{t}-b^{s})(u^{t},u^{s}-u^{t})\vert +  \vert (l^{s}-l^{t})(u^{s}-u^{t})\vert\right)\\[1ex]
	&\leq\left(\Norm{(b^{s}-b^{t})(u^{t},\cdot )}{H^{\prime}}+ \Norm{(l^{s}-l^{t})(\cdot )}{H^{\prime}} \right)\Norm{u^{s}-u^{t}}{H}.
	\end{align*}
	In case that $ u^{s} = u^{t} $ there is nothing to show. Otherwise $ \Norm{u^{s}-u^{t}}{H} > 0$ implies 
	\begin{align*} 
	\Varlambda \Norm{u^{s}-u^{t}}{H} 
	& \leq \Norm{(b^{s}-b^{t})(u^{t},\cdot )}{H^{\prime}}+ \Norm{(l^{s}-l^{t})(\cdot )}{H^{\prime}}  .
	\intertext{ Then the strong continuity of $ I \to H',\, s \mapsto l^{s}   $ and $ I \to H',\, s \mapsto b^{s}(u^{t},.)  $ implies}
	0 \leq \lim_{s \to t}\Norm{u^{s}-u^{t}}{H} & \leq \frac{1}{\Varlambda}\lim_{s \to t} \left\lbrace \Norm{(b^{s}-b^{t})(u^{t},\cdot )}{H^{\prime}}+ \Norm{(l^{s}-l^{t})(\cdot )}{H^{\prime}} \right\rbrace=0.
	\end{align*}
\end{proof}

\section{Differentiability of solution mappings}\label{Parameter_Dep_PDE:Sect:Ex_Deriv_Weak_Top_H }

In this Section we derive that the differentiability properties of $ t \to u^t $ only depend on the properties of the family of bilinear forms and linear forms, see Theorem \ref{Parameter_Dep_PDE:Thm: Existence Deriv. strong Topology}. 
We start our deduction supposing only that $ (b^{t})_{t \in I} $ is a family of continuous bilinear forms on the Hilbert space $ H $.

In a first step we will show, that under appropriate assumptions the derivative  of $ I \to \R{}, t \to  b^{t}(u^{t},v) $ exists and satisfies 
$$ \frac{d}{dt}b^{t}(u^t,v) = \dot{b}^{t}(u^{t},v) +b^{t}(\dot{u},v) .$$ 
In this sense, we split the differential quotient into two parts: 
\pagebreak

\noindent Let $ (u^{t})_{t \in I} \subset H$. For $ h \in \R{} $ such that $t+h \in I$ we obtain
\begin{equation}\label{Parameter_Dep_PDE:Eq: calc. for deriv. Bt(ut,v)}
\begin{split}
\frac{1}{h} (b^{t+h}(u^{t+h},v)&-b^{t}(u^t,v)) \\
&= \frac{1}{h} \left(b^{t+h}(u^{t+h},v)-b^{t}(u^{t+h},v) + b^{t}(u^{t+h},v) -b^{t}(u^t,v)\right)  \\[1ex]
&=\frac{1}{h} \left(b^{t+h}(u^{t+h},v)-b^{t}(u^{t+h},v) \right) + \frac{1}{h}\left(b^{t}(u^{t+h},v) -b^{t}(u^t,v)\right)  \\[1ex]
&=\frac{1}{h}\left(b^{t+h}-b^{t}\right)(u^{t+h},v) + b^{t}\left(\tfrac{u^{t+h}-u^{t}}{h},v\right) .
\end{split}
\end{equation}

Hence, we have to answer the question under which conditions  the limit values $$ (a) \lim\limits_{h \to 0} \tfrac{1}{h}\left(b^{t+h}-b^{t}\right)(u^{t+h},v) ~~~\text{ and }~~~ (b) \lim\limits_{h \to 0}b^{t}\left(\tfrac{u^{t+h}-u^{t}}{h},v\right)  $$ exist. For the second limit we state the following:

\begin{lem}\label{Parameter_Dep_PDE:Lem:PartI_Bt(ut,v)' } 
	Let $ (b^{t})_{t\in I} \subset \mathcal{B}(H)$  and let $ (u^{t})_{t \in I}\subset H $ such that $ I \to H,\, t \mapsto u^t $ is differentiable w.r.t. the weak topology on $ H $. Then, 
	\[ \lim\limits_{h \to 0}b^{t}\left(\tfrac{u^{t+h}-u^{t}}{h},v\right)= b^{t}(\dot{u}^{t},v) \, \forall t \in I,\, \forall v \in H. \]  
\end{lem}

\begin{proof}
	Let $ t \in I $ and $ |h| $ that small, such that $ t+h \in I $. Since $ b^{t} $ is continuous there exists $ C_{t} >0$ such that $ |b^{t}(u,v)|\leq C_{t} \Norm{u}{H}\Norm{v}{H}  $ for all $ u,v \in H $. Therefore the mapping $ b^t(.,v): H \to H', u \mapsto b^{t}(u,v)  $ is continuous on $ H' $ for any $ v \in H $ with $  \Norm{b^t(.,v)}{H'} \leq C_{t} \Norm{v}{H} $. Additionally, $ \dot{u}^t$ exists in the weak topology.  Then, by Corollary \ref{Parameter_Dep_PDE:Lem:Equiv formulation for weak convergence}
	\[  b^t\left(\tfrac{1}{h}(u^{t+h}-u^{t}),v\right) \underset{h \to 0}{\rightarrow}    b^t(\dot{u}^t,v).\]
\end{proof}
Now we examine the first part $(a)$ of the limit value in \eqref{Parameter_Dep_PDE:Eq: calc. for deriv. Bt(ut,v)}. 
\begin{lem}\label{Parameter_Dep_PDE:Lem:PartII_Bt(ut,v)'}
	Let  $ (b^{t})_{t \in I} \subset \mathcal{B}(H)$ such that
	\begin{enumerate}
		\item[(b1)] $ I \to \mathcal{B}(H),\, t \mapsto b^{t}  $ is differentiable w.r.t. the weak topology on $ \mathcal{B}(H) $ 
		\item [(b2)] The mapping $I \to \mathcal{B}(H),\, t \mapsto \dot{b}^{t} $ is weakly continuous.
		\item[ (b3)] The family $ (\dot{b}^{t})_{t \in I} $ is equicontinuous, i.e. there exists a constant $ C_{I} \geq 0 $ such that $ |b^t(u,v)| \leq C_{I} \Norm{u}{H}\Norm{v}{H} \, \forall t \in I,\, u,\, v \in H$. 
	\end{enumerate}
	and let $u^{(\cdot)}:I \to H, t \mapsto u^{t} $ be a continuous mapping.
	Then 
	\begin{equation}
	\dot{b}^{t}(u^{t},v)=\lim_{h\to 0} \frac{1}{h}(b^{t+h}-b^{t})(u^{t+h},v) \, \forall v \in H.
	\end{equation}
\end{lem}

\begin{proof}
	For any  $ \dot{b}^{t} $ is the weak derivative of $ I \ni t \to b^{t} \in H' $ if
	\[\dot{b}^{t}(u,v):=\lim_{h \to 0} \frac{1}{h}(b^{t+h}(u,v)-b^{t}(u,v)) \]
	exists for all $ (u,v) \in H \times H $. Thus,	we examine limit value of the difference between the two expressions for arbitrary $ v \in H $:
	\begin{align*}
	\left|\frac{1}{h}(b^{t+h}- b^{t})(\right.& \left.u^{t+h},v) -\frac{1}{h}(b^{t+h}-b^{t})(u^{t},v)  \right| \\
	&=\left|\frac{1}{h}(b^{t+h}-b^{t})(u^{t+h}-u^{t},v)\right|=  \left|\frac{1}{h} \int_{t}^{t+h} \dot{b}^{s}(u^{t+h}-u^{t},v) \, ds  \right|\\
	& \leq \sup_{s\in [0,h]} |\dot{b}^{s}(u^{t+h}-u^{t},v)| \leq C_{I}\Norm{u^{t+h}-u^{t}}{H} \Norm{v}{H},
	\end{align*}	
	which tends to zero for $ h \to 0 $.
\end{proof}

\begin{lem}[Chain Rule] \label{Parameter_Dep_PDE:Lem:Chain Rule Bilinear Forms}
	Assume that the mapping $ u^{(.)}: I \to H, t \mapsto u^{t} $ is continuous on $ H $ such that $ \dot{u}^{t} $ exist in the weak topology on $ H $  for any $ t \in I $. Moreover let $ (b^{t})_{t \in I} \subset \mathcal{B}(H) $ such that the assumptions (b1)-(b3) hold. 
	\vspace*{-0.1cm}
	Then, for any $ s \in I $
	\begin{equation}
	\left.\frac{d}{dt} b^{t}(u^{t},v)\right|_{t=s} = \dot{b}^{s}(u^{s},v) + b^{s}(\dot{u}^{s},v) \, \forall v \in H.
	\end{equation}
\end{lem}

\begin{proof} The statement follows directly from Lemma \ref{Parameter_Dep_PDE:Lem:PartI_Bt(ut,v)' }, Lemma \ref{Parameter_Dep_PDE:Lem:PartII_Bt(ut,v)'} and \eqref{Parameter_Dep_PDE:Eq: calc. for deriv. Bt(ut,v)}.
\end{proof}

\begin{lem}\label{Parameter_Dep_PDE:Lem: Addition formula}
	Assume that the mapping $ I \to H, t \mapsto u^{t} $ is continuous on $ H $ such that $ \dot{u}^{t} $ exist in the weak topology on $ H $  for any $ t \in I $. Moreover let $ (b^{t})_{t \in I} \subset \mathcal{B}(H)$ and $ (l^{t})_{t \in I} \subset H'$. Additionally, suppose that the following conditions are satisfied: 
	\vspace*{-0.1cm}
	\begin{itemize}
		\item[I)~~] The mapping $ I \to H',\, t \mapsto l^{t} $ is differentiable regarding the weak topology on $ H' $ with derivative $ \dot{l}^{t} $.
		\item[II)~] The assumptions (b1) - (b3) hold for $ (b^{t})_{t \in I} $.
		\item[ III)] $ u^{t} $ solves $  b^{t}(u^{t},v)=l^{t}(v) $ for any $ t \in I $ and $ v \in H $.
	\end{itemize}
	Under these Conditions
	\begin{equation}
	b^{t}(\dot{u}^{t},v) =\dot{l}^{t}(v) - \dot{b}^{t}(u^{t},v) \, \forall v \in H,\, \forall t\in I.
	\end{equation}
\end{lem}

\begin{proof}
	The assumptions of Lemma \ref{Parameter_Dep_PDE:Lem:Chain Rule Bilinear Forms} are satisfied and thus it implies that 
	$$\frac{d}{dt}b^{t}(u^{t},v)=\lim_{h \to 0}\frac{1}{h}\left(b^{t+h}-b^{t}\right)(u^{t+h},v) + b^{t}\left(\tfrac{u^{t+h}-u^{t}}{h},v\right)=\dot{b}^{t}(u^{t},v) + b^{t}(\dot{u}^{t},v) $$ 
	is valid for any $ v \in H $. Because of assumption II) , III) and \eqref{Parameter_Dep_PDE:Eq: calc. for deriv. Bt(ut,v)}
	\begin{align}
	\dot{l}^{t}(v)&= \lim_{h \to 0} \frac{1}{h}(l^{t+h}(v)-l^{t}(v))  \\
	&=  \lim_{h \to 0} \frac{1}{h}(b^{t+h}(u^{t+h},v) - b^{t}(u^{t},v)) \\
	&=\lim_{h \to 0}  \frac{1}{h}\left(b^{t+h}-b^{t}\right)(u^{t+h},v) + b^{t}\left(\tfrac{u^{t+h}-u^{t}}{h},v\right)\\
	&=\dot{b}^{t}(u^{t},v) + b^{t}(\dot{u}^{t},v)
	\end{align}
	holds for any $ v \in H $ and $ t \in I $.
\end{proof}

Now we are able to formulate a first theorem on the existence of $ \dot{u}^{t} $ in the weak topology on $ H $:

\begin{thm}\label{Parameter_Dep_PDE:Thm: Existence Deriv. Weak Topology}
	Let that $ H $ is a Hilbert space and $ I \subset \R{} $ an open interval, $ (l^{t})_{t \in I} \subset H' $, and $ (b^{t})_{t \in I} \in \mathcal{B}(H) $ an equicoercive family of bilinear forms. Moreover, suppose that the following conditions are fulfilled: 
	\vspace*{-0.1cm}
	\begin{itemize}
		\item[] \begin{itemize}
			\item[(l1)] The mapping $I \to H',\, t \to l^t $ is continuous. 
			\item[(l2)] The derivatives $ \dot{l}^{t} \in H',\, t \in I $ exist in the weak topology on $ H' $. 
			\item [(l3)] The mapping $I \to H',\, t \mapsto \dot{l}^{t} $ is continuous w.r.t. the weak $ H' $-topology.
		\end{itemize}
		\vspace*{1ex}
		\item[]
		\begin{enumerate}
			\item[(b0)]  The mapping $ I \to H',\, t \to b^{t}(u,.) $ is continuous for every $ u \in H $ . 
			\item[ (b1)]$ I \to \mathcal{B}(H),\, t \mapsto b^{t}  $ is differentiable w.r.t. the weak topology on $ \mathcal{B}(H) $.
			\item [(b2)] The mapping $I \to \mathcal{B}(H),\, t \mapsto \dot{b}^{t} $ is weakly continuous on $ \mathcal{B}(H) $.
			\item[(b3)]the family $ (\dot{b}^{t})_{t \in I} $ is equicontinuous. 
		\end{enumerate}
	\end{itemize}
	\begin{itemize}
		\item[a)] Let $ u^{t} $ satisfy $  b^{t}(u^{t},v)=l^{t}(v) \, \forall v \in H $ for any $ t \in H $.  Then \[u^{(.)}: I \to H,\, t \to u^{t} \] is strongly continuous.
		\item[b)] Moreover suppose that $ q^{t} \in H $ is the unique solution of $ b^{t}(q^{t},v)=\dot{l}^{t}(v)-\dot{b}^{t}(u^{t},v) $  $\forall v \in H $. Then $ \dot{u}^{t} $ exists in the weak topology on $ H $ such that
		\begin{equation}
		\dot{u}^{t}=q^{t} \, \forall t \in I.
		\end{equation}
	\end{itemize}
\end{thm}

\begin{proof}
	By the theorem of Lax-Milgram, it is obvious that for any $ t \in I $ the equation $  b^{t}(.,v)=l^{t}(v) \, \forall v \in H $ has a unique solution $ u^{t} \in H $. Thus we have to show the following:
	\begin{itemize}
		\item[a)] The map $ I \to H, t \mapsto u^{t} $ is continuous.
		\item[b)] For any $ t \in I $ it holds $ \dot{u}^{t}=q^{t}$ in the weak topology on $ H $.
	\end{itemize}
	Since (l1) and (b0) hold and the family $ (b^{t})_{t \in I} $ is equicoercive we can apply Lemma \ref{Parameter_Dep_PDE:Lem: Crit. for strong cont.} which directly yields a). 
	
	\noindent Further note that $ \tilde{l}^{t}:=\dot{l}^{t}-\dot{b}^{t}(u^{t},.) \in H' $ and $ b^{t} $ is strictly coercive for any $ t\in I $. Thus, the theorem of Lax-Milgram implies that $ b^{t}(.,v)=\dot{l}^{t}(v)-\dot{b}^{t}(u^{t},v)=\tilde{l}(v) \forall v \in H $ has exactly one solution $ q^{t} \in H $ for any $ t \in I $.
	
	\noindent We now continue by showing that $ l\left(\frac{1}{h}(u^{t+h}-u^{t})-q^{t}\right) \to 0 $ for all $ l \in H'$ and any $ t \in I $:
	Due to the  strict coercivity of $ (b^{t})_{t \in I} $ and continuity of any $ b^{t}, t \in I $ we can refer to Lemma \ref{Parameter_Dep_PDE:Lem:Equiv formulation for weak convergence} and show alternatively that 
	\[ b^{t}\left(\frac{1}{h}(u^{t+h}-u^{t})-q^{t},v\right) \underset{h \to 0}{\to} 0\, \forall v \in H. \] 
	We can now make use of Lemma \ref{Parameter_Dep_PDE:Lem:PartII_Bt(ut,v)'} since $I \ni  t \to u^{t} \in H $ is continuous on $ H $ and (b1)-(b3) applies. Thus 
	\[ \dot{b}^{t}(u^t,v)=\lim_{h \to 0}\frac{1}{h}(b^{t+h}-b^{t})(u^{t+h},v).\]
	Moreover, $ l2) $ applies, $ q^{t} $ satisfies $ b^{t}(q^{t},v)=\dot{l}^{t}(v)-\dot{b}^{t}(u^{t},v)\, \forall v \in H  $
	and thus 
	\begin{align*}
	0&= -\dot{b}^{t}(u^{t},v) -b^{t}(q^{t},v) +\dot{l}^{t}(v) \\
	& =\lim_{h \to 0 }\left[-\frac{1}{h} (b^{t+h}-b^{t})(u^{t+h},v) -b^{t}(q^{t},v)+ \frac{1}{h} (l^{t+h}-l^{t})(v) \right] \, \forall v\in H.
	\end{align*}
	For arbitrary $ v \in H $, we obtain
	\begin{align*}
	-\frac{1}{h}(b^{t+h}-b^{t})(u^{t+h}&,v) -b^{t}(q^{t},v) + \frac{1}{h}(l^{t+h}-l^{t})(v) \\[1ex]
	= ~\,&\frac{1}{h}[ -(b^{t+h}-b^{t})(u^{t+h},v)+b^{t+h}(u^{t+h},v)-b^{t}(u^{t},v)] -b^{t}(q^{t},v) \\
	=~\,&\frac{1}{h}[b^{t}(u^{t+h},v)-b^{t}(u^{t},v)]-b^{t}(q^{t},v)\\
	=~\,&b^{t}\left(\frac{1}{h}(u^{t+h}-u^{t})-q^{t},v\right),
	\end{align*}
	where the left-hand side tends to zero when $ h \to 0 $. Therefore
	\[0=\lim_{h \to 0} b^{t}\left(\frac{1}{h}(u^{t+h}-u^{t})-q^{t},v\right) \, \forall v \in H.\]
	By Lemma \ref{Parameter_Dep_PDE:Lem:Equiv formulation for weak convergence} we finally conclude
	$ \frac{1}{h}(u^{t+h}-u^{t}) \rightharpoonup_{H}  q^{t} \text{ as } h \to 0$
	for arbitrary $ t \in I $. This means $ \dot{u}^{t}=q^{t} $ where $ \dot{u}^{t} $ is the derivative of $ u^{t} $ w.r.t. the weak topology 
	on $ H $.
\end{proof}

\begin{thm}\label{Parameter_Dep_PDE:Thm: Existence Deriv. strong Topology}
	Let $ H $ be a Hilbert space, $ (l^{t})_{t \in I} \subset H' $ a family of continuous linear forms, and  $ (b^{t})_{t \in I} \subset \mathcal{B}(H) $ a family of continuous and equicorcive bilinear forms. Additionally, suppose that the following conditions are satisfied: 
	\vspace*{-0.1cm}
	\begin{itemize}
		\item[] \begin{itemize}
			\item[(l1)~] $ \dot{l}^{t} \in H' $ are the derivatives of $ I \ni t \to l^t \in H' $ w.r.t. the weak $ H' $-topology.
			\item [(l2')]
			The mapping $I \to H',\, t \mapsto \dot{l}^{t} $ is continuous w.r.t. the strong $ H' $-topology.
		\end{itemize}
		\vspace*{1ex}
		\item[]
		\begin{enumerate}
			\item[(b1)~]The derivatives $ \dot{b}^{t} \in \mathcal{B}(H) $ exist for any $ t \in I $ in terms of the the weak topology on $ \mathcal{B}(H) $.
			\item [(b2')] The mapping $I \to \mathcal{B}(H),\, t \mapsto \dot{b}^{t} $ is strongly continuous on $ \mathcal{B}(H) $.  
			\item[(b3)~] The family $ (\dot{b}^{t})_{t \in I} $ is equicontinuous.\\ 
		\end{enumerate}
	\end{itemize}
	\begin{itemize}
		\item[a)] Suppose that $ u^{(.)}: I \to H,\, t \to u^{t} $ is the mapping of unique solutions to $ b^{t}(.,v)=l^{t}(v) \, \forall v \in H $. Then $ u^{(.)}: I \to H,\, t \to u^{t} $ is strongly continuous.
		\item[b)] Additionally, let $ q^{t} \in H $ be the unique solution of $ b^{t}(.,v)=\dot{l}^{t}(v)-\dot{b}^{t}(u^{t},v) $ $\forall v \in H.$
		Then $ \dot{u}^{t} $ exists in the strong topology on $ H $ and
		\begin{equation}
		\dot{u}^{t}=q^{t} \, \forall t \in I.
		\end{equation}
	\end{itemize}
\end{thm}

\begin{proof}
	By the theorem of Lax-Milgram, it is obvious that for any $ t \in I $ the equation $  b^{t}(.,v)=l^{t}(v) \, \forall v \in H $ has a unique solution $ u^{t} \in H $. We  have to show the following:
	\begin{itemize}
		\item[\textit{1)}~~~\,] The map $ I \to H, t \mapsto u^{t} $ is continuous.
		\item[\textit{2)i)}~] It holds $ \frac{d}{dt}u^{t} = \dot{u}^{t}=q^{t} \, \forall t \in I $ weakly in $ H $.
		\item[\textit{2)ii)}] It holds $ \frac{d}{dt}u^{t} = \dot{u}^{t}=q^{t} \, \forall t \in I $ strongly in $ H $.
	\end{itemize}
	1):  We aspire to apply Lemma \ref{Parameter_Dep_PDE:Lem: Crit. for strong cont.} to the respective map and therefore we have to check if the required assumptions are satisfied:
	Since $ (b^{t})_{t \in I} $ is equicoercive by assumption it is left to show that 
	(l1) $ I \to H', t \mapsto l^{t} $ is continuous and that
	(b0) $ I \to H', t \mapsto b^{t}(u,.) $ is continuous for any $ u \in H $
	are satisfied: Since $ (l1) $ and $ (l2') $ mean that $ t \mapsto l^t $ is continuously differentiable w.r.t. the weak $ H' $-topology, the mapping $ t \mapsto l^t $ is continuous. The same holds for the mapping $ t \mapsto b^t(u,.) $. This shows that 1) and therefor assertion a) holds true.\\[1em]
	\noindent \textit{2)i)} can now be directly derived from Theorem \ref{Parameter_Dep_PDE:Thm: Existence Deriv. Weak Topology} since the preceding steps show that the assumptions (l1)-(l3)  and (b0)-(b3) are satisfied.\\[1em] 
	\textit{2)ii)}: We will now show, that even $ \frac{1}{h}(u^{t+h}-u^{t}) \to_{H}  q^{t} \text{ as } h \to 0$ or in other words $ q^t=\dot{u}^t $ in the strong $ H $-topology:\\
	By the conditions (l3') and (b2') the mappings $ I \to H', t \mapsto \dot{b}^{t}(u,.) $ \footnote{This shows that the statement remains true if (b2') is replaced by the claim (b2''): $ I \to H', t \mapsto \dot{b}^{t}(u,.) $ is strongly continuous for any $ u \in H $. } and $ I \to H', t \mapsto \dot{l}^{t}  $ are continuous for any $ u \in H $. Then $ t \to q^{t} $ is  strongly continuous.
	This follows directly 
	\pagebreak
	
	from Lemma \ref{Parameter_Dep_PDE:Lem: Crit. for strong cont.} and step 1) if one considers the continuous map $ I\to H',\, t \mapsto \tilde{l}^t=\dot{l}^t-\dot{b}^t(u^t,.)$ instead of $  I\to H',\,t \to l^{t} $.

	\noindent Let $ l \in H' $ be arbitrary. Then, by the weak convergence of $ \tfrac{u^{t+h}-u^{t}}{h}  \rightharpoonup q^{t}:=\dot{u}^{t} $ for $ h \to 0 $ and every $ t \in I $, it follows that
	\begin{align*}
	\dfrac{d}{dt} (l\circ u^{(.)})(t) 
	&=\lim_{h \to 0} \frac{l(u^{t+h})-l(u^{t})}{h}=\lim_{h \to 0} l(\tfrac{u^{t+h}-u^{t}}{h})=l(q^{t}) =(l\circ q^{(.)})(t) .
	\end{align*} 
	Because $q^{(.)} $ is a continuous mapping on $ H $ also $ l\circ q^{(.)}$ is continuous and therefore the Bochner integral $ \int_{0}^{t}(l\circ q^{(.)})(s) \, ds $ exists. Thus, 
	\begin{align*}
	l(u^{t}) 
	&=(l\circ u^{(.)})(t) = (l\circ u^{(.)})(0) + \int_{0}^{t}\frac{d}{ds}(l\circ u^{(.)})(s) \, ds\\
	&=l(u^{0})+\int_{0}^{t}(l\circ q^{(.)})(s) \, ds =l\left( u^{0}+\int_{0}^{t}q^{s} \, ds \right)  
	\end{align*}
	applies for all $l\in H'$ and hence $ u^{t} =u^{0}+\int_{0}^{t}q^{s}\,ds  $. But this already means that $ \frac{1}{h}(u^{t+h}-u^{t}) \to_{H}  q^{t} \text{ as } h \to 0$ holds: Because of the strong continuity of $ t \to q^{t} $ we conclude that 
	\begin{eqnarray*}
		\Norm{q^{t} -\frac{u^{t+h}-u^{t}}{h}}{H} 
		&= \Norm{\dfrac{1}{h}\int_{t}^{t+h}q^{s}\, ds -q^{t}}{H} 
		& = \Norm{\int_{0}^{1}q^{sh+t} -q^{t}\, ds}{H}\\
		&\leq  \int_{0}^{1} \Norm{q^{sh+t}-q^{t}}{H} \, ds
		& \leq \sup_{s \in [0,1]}  \Norm{q^{sh+t}-q^{t}}{H} 
	\end{eqnarray*}
	there the right hand side tends to zero when $ h \to 0 $ and thus $ q^{t}=\dot{u}^{t} $ w.r.t. the strong topology on $ H $. 
\end{proof}

In conclusion: If the conditions of Theorem \ref{Parameter_Dep_PDE:Thm: Existence Deriv. strong Topology} are satisfied and $ t \in I $, then the sought derivative $ \dot{u}^{t} $ in $ H $ is the unique solution $ q^{t} $ of $$ b^{t}(q^t,v)=\dot{l}^{t}(v)-\dot{b}^{t}(u^{t},v) \, \forall v \in H.$$
\vspace*{1mm}
\section{Continuity and differentiability w.r.t. higher topologies } \label{Parameter_Dep_PDE:Sect: Existence of Derivatives w.r.t. higher topologies}

In this section we show that under suitable assumptions parameter mappings maintain their continuity and differentiability properties, which they have w.r.t a topology, also in stronger topologies: 
\newpage

\begin{lem}\label{Lem: Topo-Lemma}
	Let $ \mathcal{T}_{1} \prec \mathcal{T}_{2}\prec\mathcal{T}_{3} $ three Hausdorff topologies on a set $ M $ and let $ I \subset \R{} $ be an open interval. If
	\begin{enumerate}
		\item[i)] $ U:I \to M $ is continuous with respect to $ \mathcal{T}_{1} $ and
		\item[ii)] $ \overline{U(I)}^{\mathcal{T}_3} $  is relatively compact in $\mathcal{T}_{2}$,
	\end{enumerate}
	then $ U:I \to M $ is continuous with respect to $ \mathcal{T}_{2} $.
\end{lem}

\begin{proof}
	The map $ U $ is continuous w.r.t. $ \mathcal{T}_{1} $. Thus, we state that $u_{i}:= U(t_{i}) \underset{i \in I}{\rightarrow} U(t)=:u $ holds for any directed net $ (t_{i})_{ i \in \mathcal{I}} \subset I $ with limit value $ t \in I $.
	Now suppose that $ U $ is not continuous w.r.t. $ \mathcal{T}_2 $. Then there is
	and neighbourhood $ \mathcal{U}_{\mathcal{T}_{2}} (u) $ of $ u $ such that there is $  \mathcal{J} \subset \mathcal{I} $ where $ u_{j}:=U(t_{j}) \notin \mathcal{U}_{\mathcal{T}_{2}} (u)$ for  $ j \in \mathcal{J} $.
	The closure of  $\overline{{U(I)^{~}}^{\mathcal{T}_3}} $ w.r.t. $ \mathcal{T}_3 $ is compact in $ \mathcal{T}_{2} $ by assumption. 
	Hence, the net $ (u_{j})_{j \in \mathcal{J}} $ has a converging subnet 
	$$ (u_{l})_{l \in \mathscr{L}} \overset{\mathcal{T}_{2}}{\rightarrow} u^{\ast} \in \overline{\overline{{U(I)^{~}}^{\tau_3}}^{~\mathcal{T}_2}}$$ but since $ \mathcal{T}_{1} \prec \mathcal{T}_{2}  $ also $ (u_{l})_{l \in \mathscr{L}} \overset{\mathcal{T}_{1}}{\rightarrow} u^{\ast} $ and therefore $ u^{\ast} =U(t) $ by the Hausdorff property of $ \mathcal{T}_{1} $. Hence $ U:I \to M $ has to be continuous w.r.t. $ \mathcal{T}_{2} $.
\end{proof}

\begin{thm}\label{Parameter_Dep_PDE:Thm: Existence of Strong Derivatives u^t }
	Suppose that $ (X_{3},\Norm{.}{X_{3}} )\subset (X_{2},\Norm{.}{X_{2}} ) \subset (X_{1},\Norm{.}{X_{1}} )$ are Banach spaces and let $ u^{(\cdot)}:I \rightarrow X_{3},\, t \mapsto u^{t}  $ such that the following conditions are satisfied
	\begin{itemize}	
		\item[i)~~] $ u^{(\cdot)}  $ is differentiable w.r.t. the (weak) topology on $ X_{1} $ with derivative $ \dot{u}^{t} $ at $ t \in I $,
		\item[ii)~] the map $\dot{u}^{(\cdot)}:I \rightarrow X_{3}, \,  t \to \dot{u}^{t} $ is continuous w.r.t. the (weak) topology on $ X_{1} $,
		\item[iii)] The closure of $ \dot{u}^{I} $ in $ (X_{3},\Norm{.}{X_{3}}) $ is relatively compact in $ (X_{2},\Norm{.}{X_{2}}) $.
	\end{itemize}
	Then $ \dot{u}^{t} $ is the derivative of $ u^{(.)} $ at $ t \in I $ in $(X_{2},\Norm{.}{X_{2}} )$. Moreover $ t \in I \to \dot{u}^{t} \in X_3 $ is continuous regarding the strong $ X_2 $ topology.
\end{thm}
\vspace*{-1em}
\enlargethispage*{1em}
\begin{proof}
	Any Banach space $ X $ equipped with the weak topology on $ X $ is a Hausdorff space. Hence, the letter is true for the space $ X_{1} $ with weak topology. Let $ \mathcal{T}_{X_{1}} $ be the  weak topology on $ X_{1} $, $ \mathcal{T}_{X_{2}} $ the norm topology on $ X_{2} $  and $ \mathcal{T}_{X_{3}} $ the norm topology on $ X_{3} $. Hence $  \mathcal{T}_{X_{1}} \prec \mathcal{T}_{X_{2}}\prec\mathcal{T}_{X_{3}} $ on $ X_{1} $ and Lemma \ref{Lem: Topo-Lemma} becomes applicable. Therefore, $ \dot{u}^{(.)}:I \to X_{2} $ is continuous on $ X_{2} $.\\
	\noindent Let $ l \in X_{1}' $ be arbitrary. Then $ l $ is also an element of $ X_2' $. By the (weak) convergence of $ \tfrac{u^{t+h}-u^{t}}{h}  \to q^{t}:=\dot{u}^{t} $ for $ h \to 0 $ and every $ t \in I $ in $ X_1 $, it follows that
	\begin{align*}
	\dfrac{d}{dt} (l\circ u^{(.)})(t) =\lim_{h \to 0} \frac{l(u^{t+h})-l(u^{t})}{h}=\lim_{h \to 0} l\left(\frac{u^{t+h}-u^{t}}{h}\right)=l(q^{t}) =(l\circ q^{(.)})(t) .
	\end{align*} 
	Because $q^{(.)} $ is a continuous mapping on $ X_{2} $ also $ l\circ q^{(.)}$ is continuous and therefore the Bochner integral $ \int_{0}^{t}(l\circ q^{(.)})(s) \, ds $ exists. Thus, 
	$$l(u^{t}) 
	= l(u^{0}) + \int_{0}^{t}\frac{d}{ds}l(u^s) \, ds=l(u^{0})+\int_{0}^{t}l(q^s) \, ds  =l\left[ u^{0}+\int_{0}^{t}q^{s} \, ds \right]  
	$$ 
	applies for all $l\in X_{1}'$ and this implies $ u^{t} =u^{0}+\int_{0}^{t}q^{s}\,ds  $. Then the continuity of $ t \to q^{t} $ w.r.t the strong topology on $ X_2 $ leads to  $ q^{t}=\dot{u}^{t} $ w.r.t. the strong topology on $ X_2 $ since
	\begin{align*}
	\Norm{q^{t} -\frac{u^{t+h}-u^{t}}{h}}{X_2} 
	&= \Norm{\frac{1}{h}\int_{t}^{t+h}q^{s}-q^{t}\, ds}{X_2} 
	\leq \sup_{s \in [0,1]}\Norm{q^{sh+t}-q^{t}}{X_2}  \underset{h \to 0}{ \longrightarrow} 0 \,.
	\end{align*} 
\end{proof}


\chapter{Material and Shape Derivatives in Linear Elasticity} \label{Ex_Shape_Deriv_LinEl}

\section{Preliminaries}\label{Sec: Prel.  Mat. deriv. LinEl} 

The aim of this chapter is to apply the results of Chapter \ref{Transf_shape_opt} and \ref{Parameter_Dep_PDE} to the PDE of linear elasticity on the variable, parameter dependent sets $ \Omega_t $ introduced in  \ref{Transf_shape_opt:Sec: Transformation along Vector fields}. 

As central outcomes, the Theorems \ref{Ex_Shape_Deriv_LinEl:Thm: q^t C^3,phi material derivative of u_t} and \ref{Ex_Shape_Deriv_LinEl:Thm: Shape derivatives for linear elasticity} show existence of material and (local) shape derivatives in Hölder spaces $ C^{k,\phi}$ for $ k \geq 2 $. A brief outlook on the deduction of  material derivatives under lower regularity assumptions, also consider Section \ref{outlook: Reduces Requirements}.  

Initially, we take vector fields $ V\in \Vad{1}{\Oext} $ such that $ T_t[V] \in $ $ C^{1}(\overline{\Oext},\overline{\Oext}) $ with $ k\in \N{0} $ and a starting shape $ \Omega \in \mathcal{O}_{1} $. The maximal existence interval of $ T_t=T_t[V] $ is as always denoted by $ I_V $. \\[1em]
Let $ \{f(\Omega):\Omega \to \R{3} \vert  \Omega \in \mathcal{O}_{1} \} $ and $ \{g(\Gamma_N):\Gamma_N  \to \R{3}  \vert \Omega \in \mathcal{O}_{1} \} $ be some families of vector fields. We always assume that $ |\Gamma_{N}| =\int_{\Gamma_{N}}\, dS >0 $ and that $ \Gamma_{D}=\Gamma \setminus \Gamma_{N} $ and $ \Gamma_{N} $ have a positive distance.  The variational formulation of the disjoint displacement-traction problem 
\begin{align}\label{Ex_Shape_Deriv_LinEl:Eq:LinEl Omega_t}
\left. 
\begin{array}{r c l l}
-\Div( \se(u_t)) & = & f(\Omega_t)   &\text{ in } \Omega_t \\
u_t &=& 0  &\text{ on } \Gamma_{D,t}  \\
\se(u_t) \vec{n}_t &=&g(\Gamma_{N,t}) &\text{ on }\Gamma_{N,t}.
\end{array} 
\right.
\end{align} 
on $ \overline{\Omega}_{t} $ is given by $ B_{t}(u,v)=L_{t}(v) $, $ u,v \in $ $ H^{1}_{D,t}(\O{t},\R{3}) $ where
\begin{equation}\label{Ex_Shape_Deriv_LinEl:Eq:Def_Bt}
B_{t}(u,v)=
\int_{\O{t}} \lambda \tr(\varepsilon(u))\tr(\varepsilon(v)) + 2 \mu \tr(\ve(u)\ve(v))\, dx   = \int_{\O{t}} \tr(\sigma(u) \varepsilon(v)) \, dx
\end{equation}
and 
\begin{equation}\label{Ex_Shape_Deriv_LinEl:Eq: Def Lt}
L_{t}(v)=L_{\O{t}}(v)=\int_{\O{t}}  \langle f(\Omega_t),v \rangle  \, dx + \int_{\Gamma_{N,t}}  \langle g(\Gamma_{N,t}),v \rangle \, dS,
\end{equation} 
with $ \Omega_t =T_t(\Omega)$, $ \Gamma_{D,t}=T_{t}(\Gamma_{N}) $ and $ \Gamma_{N,t}=T_{t}(\Gamma_{N}) $. For notational simplicity we write $ f_{t} = f(\Omega_t) $ and $ g_{N,t}:=g(\Gamma_{N,t}) $.

Via composition with the transformations $ \T{t}$ the weak or strong solution with domain $ \overline{\Omega}_{t} $ can be pulled back such that it is defined on $ \overline{\Omega} $. Accordingly, we define the mapping 
$$ u^{t}:=u_{t} \circ \T{t},~~~ u^{t}: \Omega \to \R{3} $$ if the solution $ u_t $ exists in the weak or strong sense. 

The bilinear form $ B_{t} $ and the linear form $ L_{t} $  are defined for functions in $ H^{1}_{D,t}(\O{t},\R{3}) $ with different domains. In this case the \textit{pull-back} to one joint definition set can be realized setting $ B^{t}(u,v):=B_{t}(u \circ \T{t}^{-1}, v \circ T_t^{-1}) ,\, u,v \in H^{1}_{D}(\Omega,\R{3})$ and $ L^{t}(v):= L_{t}(v \circ \T{t}^{-1}),\, v \in H^{1}_{D}(\Omega,\R{3}) $. These linear and bilinear forms are independent of $ t $:
\begin{equation}\label{Ex_Shape_Deriv_LinEl:Eq: B^t(u^t,v)=L^t(v)}
B^{t}(u^{t},v)=L^{t}(v) ~ \forall v \in H^{1}_{D}(\Omega,\R{3})
\end{equation}
with
\begin{align}
&L^{t}(v)=\int_{\Omega_t}\skp{f_t}{v \circ T_t^{-1}} \,dx + \int_{\Gamma_N,t}\skp{g_{N,t}}{v \circ T_t^{-1}} \,dS, \label{Ex_Shape_Deriv_LinEl:Def:L^t} \\
&B^{t}(u,v)
=\int_{\Omega_{t}}\lambda \tr(\varepsilon(u \circ T_{t}^{-1}))\tr(\varepsilon(v \circ T_{t}^{-1})) + 2 \mu \tr(\varepsilon(u \circ T_{t}^{-1})\varepsilon(u \circ T_{t}^{-1})) \, dx.\label{Ex_Shape_Deriv_LinEl:Def:B^t}
\end{align}


The following lemma is a conclusion from the statement concerning change of coordinates in  \cite[Thm. 3.41]{AdamsFournier}. 
In our case we obtain in particular the following estimates under change of coordinates : 

\begin{lem}\label{Ex_Shape_Deriv_LinEl:Lem:H1_estimates_ut}
	Let $ V \in \mathcal{V}^{ad}_{1} $, $ T_t =T_t[V] $ be the associated transformation mapping and $ \Omega \in \mathcal{O}_{1} $. Then there are constants $ C_i(T_t),\, i=1,2,3 $ depending on $ T_t $ such that 
	\begin{align}
	\Norm{u \circ T_t^{-1}}{L^2(\Omega_t,\R{3})}& \leq C_1(T_t) \Norm{u}{L^2(\Omega,\R{3})} && \forall u \in L^2(\Omega,\R{3})  \\
	\Norm{u \circ T_{t}^{-1}}{H^{1}(\O{t},\R{3})}&\leq C_2(T_t) \Norm{u}{H^{1}(\Omega,\R{3})} && \forall  u \in H^1(\Omega,\R{3}) \\
	\Norm{u \circ T_{t}^{-1}}{L^2(\Gamma_{t},\R{3})}&\leq C_3(T_t)\Norm{u}{L^2(\Gamma,\R{3})} && \forall  u \in L^2(\Gamma,\R{3}) .
	\end{align}
	If $ I  \Subset I_{V} $ ( $ \overline{I} \subset I$) then $ C_1,\, C_2,\, C_3 $ can be chosen uniformly with respect to $t \in I $.
\end{lem}

\begin{proof}
	i) For $ u\in L^2(\Omega,\R{n}) $ we obtain
	\begin{align*}
	\Norm{u \circ T_t^{-1}}{L^{2}(\Omega_t)}^2 
	\leq \Norm{\gamma_t}{\infty,\Omega} \Norm{u}{L^{2}(\Omega)}^2
	\end{align*}
	by Lemma \ref{Lem: H1_Estimate_u}
	Since $ \gamma_{(.)}:t \to \gamma_t$ is in $  C^1(I_V, C^{0}(\overline{\Oext})) $ (\ref{Transf_shape_opt:Lem: Properties gammat})  there is a constant $ C>0 $ such that $ \Norm{\gamma_t}{\infty, \Omega} \leq C_1 $  for any $ t $ in the colsed interval $ \overline{I}$. This implies the first statement. 
	
	\noindent ii) Now let $ u \in H^1(\Omega,\R{3}) $. Since $ T_t $ is a $ C^1 $-diffeomorphism on the compact set $ \overline{\Omega} $ there is a constant  $ C(T_t)>0 $ such that 
	\begin{align*}
	\Norm{u \circ T_t^{-1}}{H^{1}(\Omega_t;\R{3})}^2 
	&\leq \Norm{\gamma_t}{\infty} \left(1+\Norm{T_t^{-1}}{C^{1}(\Omega_t,\R{3})}^2\right) \Norm{u}{H^{1}(\Omega,\R{3})}^2 \\
	& \leq C(T_t)^2  \Norm{u}{H^{1}(\Omega,\R{3})}. 
	\end{align*}
	
	\noindent 	Again Lemma \ref{Transf_shape_opt:Lem: Properties gammat} and \ref{Transf_shape_opt:Lem: Properties D_Tt} maintain that the mappings $ t  \to \gamma_{t}=\det(DT_{t})$ and $  t  \to D(T_{ t}^{-1})=DT_{-t} $ are continuous from $ I_V $ to $ C^{0}(\overline{\Oext}) $ and $ C^{0}(\overline{\Oext},\R{3 \times 3}) $ respectively. Therefore there exists a constant $ C_2 $ such that $ C(\T{t}) \leq C_2 $ for all $ t \in I $.\\
	
	\noindent 	iii) Finally, let $  u \in L^2(\Gamma,\R{3}) $
	\begin{align*}
	\Norm{u \circ T_t^{-1}}{L^{2}(\Gamma_t,\R{3})}^2 
	=\int_{\Gamma_t} |(u \circ T_t^{-1})|^2 \,dS 
	= \int_{\Gamma}  |u|^2 |\omega_t| \, dx 
	\leq \Norm{\omega_t}{\infty,\overline{\Omega}} \Norm{u}{L^{2}(\Gamma,\R{3})}^2.
	\end{align*}
	Lemma \ref{Transf_shape_opt:Lem: Properties omegat} states that the mapping $ t \to \omega_{t} $ is continuous on I with values in $C^{0}(\overline{\Omega^{ext}},\R{}) $. Thus $ \Norm{\omega_t}{\infty} $ is bounded by $ C_3 $ independent of $ t \in I $.
\end{proof}

\noindent We can now show that $ u_t \in  H^{1}_{D,t}(\Omega_t,\R{3})$ solves \eqref{Ex_Shape_Deriv_LinEl:Eq: B_t(u_t,v)=L_t(v)} iff $ u^{t}:=u_t \circ T_t $ solves \eqref{Ex_Shape_Deriv_LinEl:Eq: B^t(u^t,v)=L^t(v)}:

\noindent 
We investigate the mapping $ a_{t}: C^{1}(\Omega_t,\R{3}) \to C^{1}(\Omega,\R{3}),\, u \mapsto u \circ T_t $ which extends to $$A_{t}: H^{1}_{D_t}(\Omega_t,\R{3}) \to H^{1}_{D}(\Omega,\R{3}),\, u \mapsto u \circ T_t $$ by completion of the space $ C^1(\Omega_t,\R{3}) $ and $ C^1(\Omega,\R{3}) $ w.r.t. the respective Sobolev-Norms. Moreover the properties of the Trace operators $ \mathbf{T}_{\Gamma,D_t}: H^1(\Omega_t,\R{3}) \to H^{\nicefrac{1}{2}}(\Gamma_{D,t},\R{3})$ ensure that the mapping $ A_t $ is reasonably defined and obviously linear. Analogously to Lemma \ref{Ex_Shape_Deriv_LinEl:Lem:H1_estimates_ut} the estimate
\[  \Norm{A_{t}(u)}{H^1(\Omega,\R{3})} = \Norm{u \circ T_t}{H^1(\Omega,\R{3})} \leq C(T_t^{-1}) \Norm{u}{H^{1}(\Omega_t,\R{3})} \forall u \in  H^{1}(\Omega_t,\R{3}) \]
holds true what implies that $ A_{t} $ is continuous for any $ t \in I_{V} $.  Further,  given a  $ v \in H^{1}(\Omega,\R{3}) $ an inverse image of $ v $ is given by $ u = v \circ T_t^{-1}$ since $ A_{t}(u) = v \circ T_t^{-1} \circ T_t = v $. Thus $ A_t $ is surjective. Finally $ A_{t} $ is injective since $  u \circ T_t = 0 $ implies that
$$ \Norm{u}{H^{1}(\Omega_{t},\R{3})}  = \Norm{u \circ T_t \circ  T_t^{-1}}{H^{1}(\Omega_{t},\R{3})} \leq C(T_t) \Norm{u \circ T_t}{H^{1}(\Omega,\R{3})} = 0  ~~~ \Rightarrow ~~~ u=0.$$
\noindent This shows that 
\begin{align} \label{Ex_Shape_Deriv_LinEl:Eq: equivalence of B^t(u^t,v)=L^t(v) and B_t(u_t,v)=L_t(v)}
\begin{array}{lrcll}
& B_{t}(u_t,v)&=&L_{t}(v)& \forall v \in H^{1}_{D,t}(\O{t},\R{3}) \\
\Leftrightarrow & ~ B^{t}(u_t \circ T_t, v \circ T_t)&=& L^{t}(v \circ T_t) & \forall v \in H^{1}_{D,t}(\O{t},\R{3}) \\
\Leftrightarrow & ~ B^{t}(u^{t},w)&=&  L^{t}(w) & \forall w \in H^{1}_{D}(\O{},\R{3}).
\end{array}
\end{align}



\section{$ H^1 $-material derivatives}  \label{Ex_Shape_Deriv_LinEl:Sec:H1 material derivatives}

As already explained, it is our aim to show existence of material derivatives in Hölder spaces. The following  diagram illustrates the idea of the proof and shows the flexibility (see also Section \ref{outlook: Reduces Requirements}) of the framework presented in Chapter \ref{Parameter_Dep_PDE} at the same time. 

\tikzstyle{block} = [rectangle, draw, fill=cyan!20, 
text width=17em, text centered, rounded corners, minimum height=2em]
\tikzstyle{solblock} = [rectangle, draw, fill=green!20, 
text width=10em, text centered, rounded corners, minimum height=2em]
\tikzstyle{line} = [draw, -latex']
\tikzstyle{cloud} = [draw, ellipse,fill=red!20, node distance=4cm,
minimum height=2em,text width=6em, text centered]

\vspace*{2em}

\begin{center}
	\begin{tikzpicture}[node distance = 2cm, auto]
	\node[block] (wPDE) {write PDE in weak form on $ \Omega_t $, pull-back to $ \Omega $ $ \rightarrow  B^t(u^{t},v) = L^{t}(v) $};
	\node[block, below of=wPDE, node distance =1.8cm] (dotb) {derive $ \dot{l}^t,\, \dot{b}^t $ $ \rightarrow $ \\$ B^t(u^{t},v) = \dot{B}^{t}(q^{t},v) - \dot{L}^{t}(v)$};
	\node[below of=wPDE,node distance=4.2cm] (P1){};
	\node[cloud, left of=P1, node distance=2cm ] (Prop) {Thm. \ref{Parameter_Dep_PDE:Thm: Existence Deriv. strong Topology}};
	\node[solblock, left of=Prop, node distance= 7.75cm](H1sol) {existence of $ H^1 $ solutions $ u^t,\, q^{t} $};
	\node[block, below of=dotb, node distance=5cm] (H1deriv) {existence and continuity of $ H^1 $ material derivative mappings};
	\node[below of=H1deriv,node distance=2cm] (L1){};
	\node[cloud, left of=L1, node distance=2.2cm ] (Lem) {Thm. \ref{Parameter_Dep_PDE:Thm: Existence of Strong Derivatives u^t }};
	\node[solblock, left of=Lem, node distance= 7.6cm](Cksol) {existence of solutions $ u^{t},\, q^{t} $ in higher topologies ($ C^k $, $ C^{k,\phi}$, Sobolev,$ \ldots $ )};
	\node[block, below of=H1deriv, node distance=4cm] (Ckderiv) {existence of material derivatives in higher topologies};
	\path [line,dashed] (wPDE) -| (H1sol) node [near end, right] (LM) {} node [near start] {$ u^{t} $};
	\path [line, dashed] (dotb) -- (LM) node [midway] {$ q^{t} $};
	\path [line] (wPDE) -- (dotb);
	\path [line,dashed] (H1sol) -- (Prop);
	\path [line] (dotb) -- (H1deriv);
	\path [line, dashed] (Prop) -- (P1);
	\path [line] (H1deriv) -- (Ckderiv);
	\path [line,dashed] (Cksol) -- (Lem)  node[midway, below] {\multirow{4}{3.5cm}{\small uniform Schauder estimates and compact embeddings}};
	\path [line,dashed] (Lem) -- (L1);
	\path [line, dashed] (H1sol) -- (Cksol) node [near start, right] {\multirow{3}{3cm}{\small regularity theory for elliptic PDE}};
	\end{tikzpicture}
\end{center}

\noindent In this sense we start with the derivation of $ H^1 $-material derivatives $ \dot{u}^{t} $.
The crucial steps when showing the existence of material or shape derivatives in the context of linear elasticity are bundled in Theorem \ref{Parameter_Dep_PDE:Thm: Existence Deriv. strong Topology}. Following the single requirements on the bilinear forms $ B^{t} $, the linear forms $ L^t $ and the solution space $ H^{1}_{D}$, we will show that these derivatives exists in the strong topology on $ H^{1}_{D} $.
In a second step, see Section \ref{Ex_Shape_Deriv_LinEl:Sec: hölder material derivatives}, we then apply Theorem \ref{Parameter_Dep_PDE:Thm: Existence of Strong Derivatives u^t } to obtain the required material derivatives in classical function spaces. 

\subsection{Linear and bilinear form: Continuity and ellipticity properties}

According to the explanations on Page \pageref{Transf_shape_opt} we find an $ \epsilon>0  $ and an open interval $ (-\epsilon,\epsilon) \Subset I_{V} $ - this interval shall be fixed in the following. 

We will show that the pulled back forms $ L^t $ and $ B^{t} $ are still continuous linear and bilinear forms, respectively, with bounds independent of $ t \in (-\epsilon,\epsilon) $. Moreover the familiy $ (B^{t})_{t \in (-\epsilon,\epsilon)} $ is shown to be equicoercive and thus at the end of this section Lemma \ref{Parameter_Dep_PDE:Lem: Crit. for strong cont.} can be applied.
\\[1ex]
By application of change of coordinates to the integral terms, see for example the book of  \cite[Thm 3.41]{AdamsFournier}, we can state the following regarding $ L^t $ and $ B^t $:

\begin{lem}\label{Ex_Shape_Deriv_LinEl:Lem: continuity Lt}	
	For $ t \in I_{V} $ let $ f_{t}=f(\Omega_t)\in L^{\nicefrac{6}{5}}(\Omega_t,\R{3}) $, $ g_{N,t} = g(\Gamma_{N,t})\in L^{\nicefrac{4}{3}}(\Gamma_{N,t},\R{3})$. 
	\begin{itemize}
		\item[i)] The linear form $ L^t $ is continuous on $ H^{1}_{D}(\Omega,\R{3}) $ and satisfies
		\begin{equation}\label{Ex_Shape_Deriv_LinEl:Eq: Lt}	
		L^t(v)=\int_{\Omega} \left\langle f^{t}\gamma_{t},v \right\rangle \, dx + \int_{\Gamma_N} \left\langle g^{t}\omega_{t},v \right\rangle\, dS ~~~~~ \forall v  \in H^{1}_{D}(\Omega,\R{3})
		\end{equation}
		where $f^{t}=f_t \circ T_{t} \in L^{\nicefrac{6}{5}}(\Omega,\R{3})$ and $ g^{t}=g_{N,t} \circ T_{t} \in L^{\nicefrac{4}{3}}(\Gamma_N, \R{3}) $. 
		\item[ii)] If $ f_t $ is uniformly bounded in $ L^{\nicefrac{6}{5}}(\Omega_t,\R{3})  $ and $ g_{N,t} $ is uniformly bounded in \linebreak $ L^{\nicefrac{4}{3}}(\Gamma_{N,t},\R{3}) $ for $ t \in(-\epsilon,\epsilon) $, then $ (L^t)_{t \in (-\epsilon,\epsilon) } $ is equicontinuous.
	\end{itemize}
\end{lem}

\begin{proof}
	\textit{i)} The equality \eqref{Ex_Shape_Deriv_LinEl:Eq: Lt} 
	follows straightly from change of coordinates, see for example \cite[Thm. 3.41]{AdamsFournier}.
	
	\noindent \textit{ii)} According to the Sobolev Embedding Theorem \ref{App: Sobolev embedding}, the embedding $ H^{^1}(\Omega) \hookrightarrow L^6(\Omega) $ is continuous as well as the trace operator $ \mathbf{T_{\Gamma}}:H^1(\Omega) \to L^4(\Gamma) $, consider Theorem \ref{App: Trace Operator}. This, combined with Hölder's inequality implies $$ \Norm{\left\langle f_t,v \circ T_t^{-1}\right\rangle }{L^1(\Omega_t)} \leq \Norm{f_t}{L^{\nicefrac{6}{5}}(\Omega_t,\R{3})} \Norm{v \circ T_t^{-1}}{L^6(\Omega_{t},\R{3})} $$ and $$ \Norm{\left\langle g_{N,t},v \circ T_t^{-1} \right\rangle}{L^{1}(\Gamma_{N,t})} \leq \Norm{g_{N,t}}{L^{\nicefrac{4}{3}}(\Gamma_{N,t},\R{3})} \Norm{v \circ T_t^{-1}}{L^4(\Gamma_{N,t},\R{3})} $$ and thus we conclude
	\begin{align*}
	|L^{t}(v)| 
	&\leq C\Norm{v \circ T_t^{-1}}{H^1(\Omega,\R{3})}
	\left( \Norm{f_t}{L^{\nicefrac{6}{5}}(\Omega_t,\R{3})} + \Norm{g_{N,t}}{L^{\nicefrac{4}{3}}(\Gamma_{N,t},\R{3})} \right) \\
	& \leq C(T_t)C\Norm{v}{H^1(\Omega,\R{3})} \left( \Norm{f_t}{L^{\nicefrac{6}{5}}(\Omega_t,\R{3})} + \Norm{g_{N,t}}{L^{\nicefrac{4}{3}}(\Gamma_{N,t},\R{3})} \right).
	\end{align*}
	By Lemma \ref{Ex_Shape_Deriv_LinEl:Lem:H1_estimates_ut} $ C(T_t) $ can be chosen independent of $ t \in [-\epsilon,\epsilon] $ and $ \Norm{f_t}{L^{\nicefrac{6}{5}}(\Omega_t,\R{3})} $ and $ \Norm{g_{N,t}}{L^{\nicefrac{4}{3}}(\Gamma_{N,t},\R{3})}  $ are uniformly bounded by assumption.
\end{proof}
\pagebreak

\begin{cor}\label{Ex_Shape_Deriv_LinEl:Cor:Criterion_equi_continuity_Lt} $  $
	\begin{itemize}
		\item[i)] 	Let $ t \to f_t \in C(I_V,C(\overline{\Omega_t},\R{3})) $ and $ t \to g_{N,t}\in C(I_V,C(\overline{\Gamma}_{N,t},\R{3}))$ such that $ [-\epsilon,\epsilon] \subset I_V $.  Then $ (L^t)_{t \in (-\epsilon,\epsilon) } $  is equicontinuous.
		\item[ii)] 	Let $ f\in C(\overline{\Omega^{ext}},\R{3}) $, $ g\in C(\overline{\Omega^{ext}},\R{3})$ and $ f_t=f\vert_{\overline{\Omega}_{t}} $ and $ g_{N,t}=g\vert_{\overline{\Gamma}_{N,t}} $. Then $ (L^t)_{t \in (-\epsilon,\epsilon) } $ is equicontinuous.
	\end{itemize}
\end{cor}

\begin{proof}
	
	i) From the assumptions and the properties of $ \gamma_t $ we deduce $ t \to\gamma_t f_t \circ T_t\in C([-\epsilon,\epsilon],C(\overline{\Omega},\R{3})) $ and $ t \to \omega_t g_{N,t} \circ T_t \in C([-\epsilon,\epsilon],C(\overline{\Gamma}_{N},\R{3})) $
	\[ \Norm{f_t}{L^{\nicefrac{6}{5}}(\Omega_t,\R{3})}^{\nicefrac{6}{5}}=\int_{\Omega} \gamma_{t} |f_t \circ T_t|^{\nicefrac{6}{5}} \, dx \leq   \sup_{t\in [-\epsilon,\epsilon]} \Norm{ \gamma_{t} (f_t \circ T_t)^{\nicefrac{6}{5}}}{\infty,\Omega} |\Omega|, \]
	\[ \Norm{g_{N,t}}{L^{\nicefrac{4}{3}}(\Omega_t,\R{3})}^{\nicefrac{4}{3}}=\int_{\Gamma_{N,t}} \omega_{t} |g_{N,t} \circ T_t|^{\nicefrac{4}{3}} \, dx \leq \sup_{t\in [-\epsilon,\epsilon]} \Norm{ \omega_{t} (g_{N,t} \circ T_t)^{\nicefrac{4}{3}}}{\infty,\Gamma_{N}} \hspace*{-1mm}|\Gamma_{N}| .
	\]
	
	\noindent ii) In this case the mappings $ t \to f_t \in C([-\epsilon,\epsilon],C(\overline{\Omega_t},\R{3})) $ and  $ t \to g_{N,t}\in $ \\ $ C([-\epsilon,\epsilon],C(\overline{\Gamma}_{N,t},\R{3}))$ satisfy the assumption in i). 
\end{proof}

\begin{defn}
	For $ u \in H^{1}(\Omega,\R{3}) $ and Lamé constants $ \lambda,\, \mu>0  $ we define 
	\begin{align*}
	\varepsilon^{t}(u)&=\frac{1}{2}\left[Du(DT_{t})^{-1}+ (Du(DT_{t})^{-1})^{\top}\right] = \varepsilon(u \circ {T_{t}}^{-1}) \circ T_{t}, \\[1em]
	\sigma^{t}(u)&=\lambda\tr(\varepsilon^{t}(u))\mathrm{I}+2\mu\varepsilon^{t}(u) =\sigma(u \circ T_t^{-1}) \circ T_t.
	\end{align*}
\end{defn}

\begin{lem}\label{Ex_Shape_Deriv_LinEl:Lem: continuity Bt}
	The bilinear form $ B^t $ on $ H^{1}_{D}(\Omega,\R{3}) $, $\lambda>0$, $ \mu>0 $ satisfies
	\begin{align*}
	B^{t}(u,v)=\int_{\Omega} \hspace*{-1mm}\gamma_t&\tr(\sigma^{t}(u)\varepsilon^{t}(v)) \, dx = \int_{\Omega}\hspace*{-1mm}[\lambda \tr(\varepsilon^{t}(u))\tr(\varepsilon^{t}(v)) +2\mu \tr(\varepsilon^{t}(u)\varepsilon^{t}(v))]\gamma_{t} \, dx.
	\end{align*}
	The set $ (B^{t})_{t \in I_V} $ is a family of continuous bilinear forms on $ H^1_{D}(\Omega, \R{3}) $ and the subfamily $ (B^{t})_{t \in (-\epsilon,\epsilon)} $ is equicontinuous.
\end{lem}

%

\begin{proof}
	Since $ D(u \circ T_t^{-1}) \circ T_t = Du(DT_t)^{-1}  $ holds
	the identity $$ B^{t}(u,v) =  \int_{\Omega} \hspace*{-1mm}\gamma_t\tr(\sigma^{t}(u)\varepsilon^{t}(v)) \, dx$$ is a direct consequence of change of coordinates.
	We show that $$ |B^{t}(u,v)| \leq C \Norm{u}{H^1(\Omega,\R{3})}\Norm{v }{H^1(\Omega,\R{3})} \, \forall u,v \in H^1_{D}(\Omega,\R{3}),\, t \in [-\epsilon,\epsilon]$$ for a constant $ C>0 $:
	Triangle inequality, equations \eqref{App: Norm_H1_tr_epsuepsv}, \eqref{App: Norm_H1_Divu_Divv} and Lemma \ref{Ex_Shape_Deriv_LinEl:Lem:H1_estimates_ut} applied to $ \tilde{u}_{t}=u \circ T_t^{-1} $ and $ \tilde{v}_{t}=v\circ T_t^{-1} $ then lead to
	\begin{align*}
	|B^{t}(u,v)| &\leq \lambda \int_{\Omega_t} |\tr(\varepsilon(u \circ T_t)) |\, |		
	\tr(\varepsilon(v \circ T_t))|\, dx + 2\mu \int_{\Omega_t} |\tr(\varepsilon(u \circ T_t)\varepsilon(v \circ T_t))| \, dx\\
	&\leq(\lambda+2\mu)\Norm{u \circ T_t^{-1}}{H^{1}(\Omega_t,\R{3})}\Norm{v \circ T_t^{-1}}{H^{1}(\Omega_t,\R{3})} \\
	&\leq (\lambda+2\mu) C(T_t) \Norm{u}{H^{1}(\Omega,\R{3})}\Norm{v}{H^{1}(\Omega,\R{3})}
	\end{align*}
	for any $ t \in I_{V} $ and $ C(T_t)>0 $. Then $ C(T_t) \leq C $ can be chosen uniformly w.r.t. $ [-\epsilon,\epsilon] $, as shown in Lemma \ref{Ex_Shape_Deriv_LinEl:Lem:H1_estimates_ut}.
\end{proof}

\begin{lem}\label{Ex_Shape_Deriv_LinEl:Lem: Coercivity Bt}
	The set $(B^{t})_{t \in (-\epsilon,\epsilon)}$ is a family of equicoercive bilinear forms.
\end{lem}
\begin{proof}
	Let $ C_{K}(\Omega) $ denote the constant in Korn's inequality \eqref{Linear_Elasticity:Eq:Folgerung_Korns_second_ineq}. In the first step we show that for fixed $ t \in (-\epsilon,\epsilon) $ there exists a constant $ \Varlambda_{t}>0 $ such that $ B^{t}(u,u) \geq \Varlambda_t \Norm{u}{H^1(\Omega,\R{3})}^2 $:
	\begin{align} \label{Ex_Shape_Deriv_LinEl:Eq: coercivity B^t}
	B^{t}(u,u) 
	&=\lambda \int_{\Omega_t} \hspace*{-2mm} \tr(\varepsilon(u \circ T_t^{-1}))^2 dx + 2\mu \int_{\Omega_t} \hspace*{-2mm} \tr(\varepsilon(u \circ T_t^{-1})^2) dx \nonumber
	\geq 2\mu \int_{\Omega_t}\hspace*{-2mm}  \tr(\varepsilon(u \circ T_t^{-1})^2) dx \nonumber 
	\\
	&\geq 2 \mu C_{K}(\Omega_t) \Norm{u \circ T_t^{-1}}{H^1(\Omega_t,\R{3})}^2
	\geq \underbrace{2\mu C_{K}(\Omega_t)^2 C(T_t)^{-2}}_{:=\Varlambda_t} \nonumber \Norm{u}{H^1(\Omega,\R{3})}^2.
	\end{align}
	because Lemma \ref{Ex_Shape_Deriv_LinEl:Lem:H1_estimates_ut} yields 
	$$\Norm{u}{H^1(\Omega,\R{3})} = \Norm{(u \circ T_t^{-1} ) \circ T_t}{H^{1}(\Omega,\R{3})} \leq C(T_t)^{-1} \Norm{u \circ T_{t}^{-1}}{H^1(\Omega_t,\R{3})}. $$ Hence, $ \Varlambda_t $ is a possible choice for the required constant.
	
	\noindent It is left to show, that $ \Varlambda_t $ can be chosen independently of $ \Omega_t,\, t \in [-\epsilon,\epsilon] \subset I_{V}$:
	Analogously to  Lemma \ref{Ex_Shape_Deriv_LinEl:Lem:H1_estimates_ut} we obtain a constant $ C $ such that  $ C(T_t) \leq C  ~ \forall t \in[-\epsilon,\epsilon]$. At this point we keep in mind that this implies \[ C^{-1}\Norm{u}{H^1(\Omega,\R{3})} \leq \Norm{u \circ T_{t}^{-1}}{H^1(\Omega_t,\R{3})}. \]
	\noindent Therefore, it is left to show that $ C_{K}(\Omega_t)\geq c>0 $  for $ t \in [-\epsilon,\epsilon] $. Let us assume that this is false. Then there is a sequence $ (t_{n})_{n\in \N{}} \subset [-\epsilon,\epsilon] $ such that $ C_{K}(\Omega_{t_n}) \to 0,\, n \to \infty $. 
	By Lemma  \ref{Ex_Shape_Deriv_LinEl:Lem:H1_estimates_ut} 
	\begin{align*}
	C^{-1} \Norm{v}{H^{1}(\Omega)} 
	&\leq \Norm{v \circ T_t^{-1}}{H^{1}(\Omega_t)} 
	\leq C_{K}(\Omega_{t}) \Norm{\varepsilon(v \circ T_t^{-1})}{L^2(\Omega_t)} \\
	&\leq C_{K}(\Omega_{t}) \Norm{v \circ T_t^{-1}}{H^{1}(\Omega_t)} 
	\leq C_{K}(\Omega_{t}) \tilde{C}(T_t)^{-1} \Norm{v}{H^{1}(\Omega)}
	\end{align*}
	for any $ v \in H^1_{D}(\Omega) $ where  $ \tilde{C}(T_t)^{-1} $ is bounded from  above  by $ \tilde{C} $ on $ [-\epsilon,\epsilon] $ and hence
	\begin{align*}
	C^{-1} \Norm{v}{H^{1}(\Omega)} 
	\leq C_{K}(\Omega_{t}) \tilde{C} \Norm{v}{H^{1}(\Omega)}
	\end{align*}
	for any $ v \in H^{1}_{D} $ and $ t \in [-\epsilon,\epsilon] $. Now, let $ v \in H^{1}_{D}(\Omega,\R{3}) $ with  $ \Norm{v}{H^1(\Omega,\R{3})}=1 $ then $0<C  
	\leq C_{K}(\Omega_{t_n}) \tilde{C} \to 0,\, n \to \infty.$ But this leads to $ C^{-1}=0  $ which is a contradiction.  
\end{proof}

\begin{prop}\label{Ex_Shape_Deriv_LinEl:Prop: Link between u_t ant u^t}
	Let $ V \in \mathcal{V}^{ad}_{1}(\Oext) $ and $ (T_{t})_{t \in I_{V}}$ be the associated familiy of tranformations, $ \Omega \in \mathcal{O}_{1} $ and let $ f_t\in L^{\nicefrac{6}{5}}(\Omega_t,\R{3}) $, $ g_{N,t}\in L^{\nicefrac{4}{3}}(\Gamma_{N,t},\R{3})$.
	Then there is a unique solution $ u^{t} \in H^{1}_{D}(\Omega,\R{3}) $ of 
	\begin{equation}\label{Ex_Shape_Deriv_LinEl:Eq: B^t(u,v)=L^t(v)}
	B^{t}(u,v)=L^{t}(v) ~~~ \forall v \in H^{1}_{D}(\Omega,\R{3}).
	\end{equation}
	for any $ t \in I_V $  and the function $ u_t := u^{t} \circ T_t^{-1} $ uniquely satisfies 
	\begin{equation}\label{Ex_Shape_Deriv_LinEl:Eq: B_t(u_t,v)=L_t(v)}
	B_{t}(u_t,v)=L_{t}(v)~~~ \forall v \in H^{1}_{D,t}(\O{t},\R{3}).
	\end{equation}
\end{prop}

\begin{proof}
	By Lemma \ref{Ex_Shape_Deriv_LinEl:Lem: continuity Bt} $ B^{t} $ is a continuous elliptic bilinear form (Lemma \ref{Ex_Shape_Deriv_LinEl:Lem: Coercivity Bt}) on $ H^{1}_{D}(\Omega,\R{3}) $ for any $ t \in I_V $ and $ L^{t} $ is also continuous on $ H^{1}_{D}(\Omega,\R{3}) $. By means of Lax Milgram's Theorem there is a unique solution $ u^{t} \in H^{1}_{D}(\Omega,\R{3}) $ to \eqref{Ex_Shape_Deriv_LinEl:Eq: B^t(u,v)=L^t(v)}
	and the assertion then follows from equation \eqref{Ex_Shape_Deriv_LinEl:Eq: equivalence of B^t(u^t,v)=L^t(v) and B_t(u_t,v)=L_t(v)}.
\end{proof}

\subsection{Linear and bilinear form: Differentiability properties}

\begin{lem}\label{Ex_Shape_Deriv_LinEl:Lem:Derivative Lt}
	Let $V \in \Vad{1}{\Oext} $, $ \Omega \in \mathcal{O}_{1} $, $ f\in C^{1}(\overline{\Omega^{ext}},\R{3})$ and $  g\in C^{1}(\overline{\Omega^{ext}},\R{3})$. 
	Let $ \epsilon>0 $ such that $ 
	[-\epsilon,\epsilon] \subset I_V $. Then the mapping  $$ (-\epsilon,\epsilon) \to H^{1}_{D}(\Omega,\R{3})' ,\, t \mapsto L^{t} $$ 
	is differentiable in the weak $ H^{1}_{D}(\Omega,\R{3})'$-topology with derivative 
	\begin{align}\label{Ex_Shape_Deriv_LinEl:Eq:dotLt}
	\dot{L}^{t}(v)&=\int_{\O{t}}\langle f_{V},v \circ T_t^{-1} \rangle\, dx + \int_{\Gamma_{N,t}} \langle g_{V,t},v \circ T_t^{-1} \rangle \, dS.
	\end{align}
	where $ f_{V}=\Div(V)f + Df V $ and $ g_{V,t}=\Div_{\Gamma_t}(V)g + Dg V = \Div_{\Gamma_t}(V)g + D_{\Gamma}gV + \frac{\partial g}{\partial \vec{n}} V_{\vec{n}} $.
	Additionally, there exists a constant $ C>0 $ such that $ \Vert\dot{L}^{t}\Vert_{H^{1}_{D}(\Omega,\R{3})'} \leq C \, \forall t \in (-\epsilon,\epsilon) .$
\end{lem}

\begin{proof}
	Let $ t \in (-\epsilon,\epsilon)  $ and $ h \in \R{} $ such that $ t+h \in (-\epsilon,\epsilon) $. Then change of coordinates leads to
	\begin{align*}
	L^{t+h}(v) 
	&= \int_{\Omega_{t+h}} \hspace*{-4mm}\skp{f}{v \circ T_{t+h}^{-1}} \, dx + \int_{\Gamma_{N,t+h}} \hspace*{-4mm} \skp{g}{v \circ T_{t+h}^{-1}} \, dS \\
	&= \int_{\Omega_t} \skp{f^h\gamma_{h}}{v \circ T_t^{-1}} \, dx  +  \int_{\Gamma_{N,t}} \skp{g^h\omega_h}{v \circ T_t^{-1}} \, dx 
	\end{align*}
	with $ f^{h} =f \circ T_h $ and  $ g^{h} =g \circ T_h $, compare Lemma \ref{Ex_Shape_Deriv_LinEl:Lem: continuity Lt}. Now we apply Corollary \ref{Diff_Banach_Space:Lem:Cont_Diff_Parameter_Integrals} to the volume and the surface integral: Since $ v \circ T_t^{-1} $ is independent of $ h $ the point wise derivative of $ \skp{ f^{h}\gamma_h}{v \circ T_t^{-1}} $ exists almost everywhere on $ \Omega_{t} $ and can be calculated with the rules established in Section \ref{Transf_shape_opt:Sec: Transformation along Vector fields}. We obtain
	\begin{align*}
	\frac{d}{dh} f^{h}\gamma_h 
	&=  \dot{\gamma}_h f \circ T_h  + \gamma_{h}(Df \circ T_h)\dot{} = \gamma_h \Div(V)\circ T_h f \circ T_h  + \gamma_{h}(Df \circ T_h)V \circ T_h  \\
	&= \gamma_h \left[\Div(V)f + Df V\right] \circ T_h 
	\end{align*}
	according to Lemma \ref{Transf_shape_opt:Lem: arithmetic rules mat deriv I} ii) and thus
	\begin{align*}
	\frac{d}{dh}\skp{ f^{h}\gamma_h}{v \circ T_t^{-1}}
	&= \skp{\frac{d}{dh}  (f\circ T_h)\gamma_h }{v \circ T_t^{-1}}
	= \skp{ \gamma_h \left[\Div(V)f + Df V\right] \circ T_h }{v \circ T_t^{-1}} 
	\end{align*}
	holds pointwise a.e. on $ \Omega_t $ with $ \frac{d}{dh}\skp{f^h}{v \circ T_t^{-1}} \in L^1(\Omega_t)  $. Therefore,
	\begin{align*}
	\frac{d}{dh}\skp{f^h(x)\gamma_{h}(x)}{(v \circ T_t^{-1})(x)} 
	&\leq \Norm{\left(\gamma_h \left[\Div(V)f + Df V\right] \circ T_h\right)(x)}{} \Norm{\left(v \circ T_t^{-1}\right)(x)}{} \\
	& \leq C \Norm{\left(v \circ T_t^{-1}\right)(x)}{}  \in L^1(\Omega_t) 
	\end{align*}
	holds a.e. on $ \Omega_t $ for any $ h $ close enough to $ 0 $ and $ C \Norm{\left(v \circ T_t^{-1}\right)(\cdot)}{}  \in L^{1}(\Omega_t)$ dominates $ \frac{d}{dh}\skp{ f^{h}\gamma_h}{v \circ T_t^{-1}} $ in this case. Thus in case of the volume integral the order of integration and differentiation can be exchanged and we obtain 
	\begin{align}
	\left.\frac{d}{dh}\right\vert_{h=0} \int_{\Omega_t} \skp{ f^{h}\gamma_h}{v \circ T_t^{-1}} \, dx 
	&= \int_{\Omega_{t}} \skp{ \Div(V)f + Df V }{v \circ T_{t}^{-1}} \, dS. \nonumber
	\end{align}
	In the case of the surface integral we can argue analogously but have to keep in mind the following calculations can only be done in terms of the trace operator 
	$ \mathbf{T_{\Gamma}}: W^{1,1}(\Omega_{t}) \to L^1(\Gamma_{t}) $, see \ref{App: Trace Operator}, and a suitable representation of $ \skp{g^h\omega_{h}}{v \circ T_t^{-1}} $ in $ W^{1,1}(\Omega_t) \cap C^{1}(\overline{\Omega}_{t}) $. Nevertheless, 
	\begin{align*}
	\frac{d}{dh}\skp{g^h\omega_{h}}{v \circ T_t^{-1}}
	= \skp{\hspace*{-1mm}\frac{d}{dh} (g \circ T_h) \omega_{h}}{v \circ T_t^{-1}\hspace*{-1mm}}= \skp{ \omega_h \hspace*{-1mm} \left[\Div_{\Gamma_t}(V)g \hspace*{-1mm}+ \hspace*{-1mm}Dg V\right] \circ T_h }{v \circ T_t^{-1}}
	\end{align*}
	exists a.e. on $ \Gamma_{t} $ in this sense and again there es a constant $ C>0 $ such that
	\begin{align*}
	\frac{d}{dh}\skp{g^h(x)\omega_{h}(x)}{(v \circ T_t^{-1})(x)} 
	&\leq \Norm{\left(\omega_h \left[\Div_{\Gamma_t}(V)g + Dg V\right] \circ T_h\right)(x)}{} \Norm{\left(v \circ T_t^{-1}\right)(x)}{} \\
	& \leq C \Norm{\left(v \circ T_t^{-1}\right)(x)}{}  \Norm{v \circ T_t^{-1}}{} \in L^1(\Gamma_{t})
	\end{align*} 
	with $\Norm{v \circ T_t^{-1}}{} \in L^1(\Gamma_{t})$ a.e. on $ \Gamma $ for any $ h $ close to $ 0 $.  Since $ \mathbf{T_{\Gamma}}:H^1(\Omega,\R{3}) \to L^2(\Gamma,\R{3}) \subset L^1(\Gamma,\R{3})  $ if $ \Gamma $ has a finite surface measure - which is the case here - the function $  C \Norm{v \circ T_t^{-1}}{} \in L^1(\Gamma_t,\R{3}) $ is suitable for application of Lemma \ref{Diff_Banach_Space:Lem:Cont_Diff_Parameter_Integrals}. Therefore, 
	\begin{align}
	\left.\frac{d}{dh}\right\vert_{h=0} \int_{\Gamma_{N,t}} \skp{g^h}{v \circ T_t^{-1}} \, dS 
	&= \int_{\Gamma_{N,t}} \skp{ \Div_{\Gamma_t}(V)g + Dg V }{v \circ T_{t}^{-1}} \, dS\nonumber.
	\end{align}
	This shows that equation \eqref{Ex_Shape_Deriv_LinEl:Eq:dotLt} holds. Obviously, $ \dot{L}^{t} $ is a continuous linear form on $ H^1_{D}(\Omega,\R{3}) $.
	
	Note that $ \Div_{\Gamma_t}(V) = \Div(V) - \langle DV \vec{n}_t,\vec{n}_t\rangle $  can be extended to $ \Oext $ by a unitary extension of the outward normal vector field $ \vec{n}_t $ to $ \mathcal{N}_{t} \circ T_t = \frac{1}{\Norm{(DT_{t})^{-\top}\mathcal{N}}{}}(DT_{t})^{-\top}\mathcal{N}  \text{ on } \Omega^{ext}$
	where $ \mathcal{N} $ is a unitary extension of the unit outward normal vector field $ \vec{n} $ on $ \Gamma $ to $ \Oext $. In this sense Corollary \ref{Ex_Shape_Deriv_LinEl:Cor:Criterion_equi_continuity_Lt} applied to $ f_V=f\Div(V)+DfV \in C(\overline{\Oext},\R{3}) $ and $ t \to g_{V,t}=g\Div_{\Gamma_t}(V)+DgV \in C([-\epsilon,\epsilon], C(\Gamma_{N,t},\R{3}) )$  shows that $ (\dot{L}^{t})_{t \in [-\epsilon,\epsilon]} $ is equicontinuous.  
\end{proof}

\begin{lem}\label{Ex_Shape_Deriv_LinEl:Lem: derivative Bt}
	Let $V \in \Vad{1}{\Oext}$, $ \Omega \in \mathcal{O}_{1} $ and $\epsilon>0$ such that $ (-\epsilon,\epsilon) \Subset I_{V} $. Then the mapping $ (-\epsilon,\epsilon)  \to \mathcal{B}(H^{1}_{D}(\Omega,\R{3})) $, $t \mapsto B^{t} $ is differentiable in the weak topology on $\mathcal{B}(H^{1}_{D}(\Omega,\R{3}))  $. 
	With $ \tilde{u}_{t}:=u \circ \T{t}^{-1},\,\tilde{v}_{t}:=v \circ \T{t}^{-1} $ the derivative is given by
	\begin{align}\label{Ex_Shape_Deriv_LinEl:Lem: dotBt}
	\dot{B}^{t}(u,v)&=\int_{\O{t}} \tr(\dot{\varepsilon}(\tilde{u}_{t})\sigma(\tilde{v}_{t})+\sigma(\tilde{u}_{t})\dot{\varepsilon}(\tilde{v}_{t})+\Div(V)\sigma(\tilde{u}_{t})\varepsilon(\tilde{v}_{t}))\, dx
	\end{align}
	where $$ \dot{\varepsilon}(u):= \frac{d}{dh} \varepsilon^{h}(u) =  -\frac{1}{2}(DuDV+ (DuDV)^{\top})$$
	pointwise a.e. in $ \Omega_t $ (\footnote{Note that $ \dot{\varepsilon}(u) \neq (\varepsilon(u))\dot{} $}).
	Moreover, $ (B^{t})_{t \in (-\epsilon,\epsilon)} $ is a family of equicontinuous bilinear forms. 
\end{lem}

\begin{proof}
	With $ \varepsilon(u \circ T_{t+h}^{-1}) \circ T_h= \varepsilon(u \circ T_t^{-1} \circ T_{h}^{-1}) \circ T_h = \varepsilon^{h}(u \circ T_t^{-1}) = \varepsilon^{h}(\tilde{u}_{t}) $ we deduce
	\begin{align*}
	B^{t+h}(u,v) =& \int_{\Omega_{t+h}} \hspace*{-4mm}\lambda \tr(\varepsilon(u \circ T_{t+h}^{-1}))\tr(\varepsilon(v \circ T_{t+h}^{-1})) + 2\mu \, \tr(\varepsilon(u \circ T_{t+h}^{-1})\varepsilon(v \circ T_{t+h}^{-1}))\, dx 
	\\
	=&\int_{\Omega_t} \hspace*{-2mm} \gamma_h \hspace*{-1mm} \left[ \tr(\varepsilon^{h}(\tilde{u}_{t}))\tr(\varepsilon^{h}(\tilde{v}_{t})) + 2\mu \, \tr(\varepsilon^{h}(\tilde{u}_{t})\varepsilon^{h}(\tilde{v}_{t}))\right]dx.
	\end{align*}
	
	Without loss of generality we consider the case $ t =0 $. Thus we replace $\tilde{u}_{t}= u \circ T_t^{-1} $ by $ u $, $\tilde{v}_{t}= v \circ T_t^{-1} $ by $ v $ and $ \Omega_t $ by $ \Omega_0=\Omega $. 
	
	\noindent The expression $\varepsilon^{h}(u) = \frac{1}{2}\left[Du(DT_{h})^{-1} + (Du(DT_{h})^{-1})^{\top}\right]$ exists pointwise a.e. on $ \Omega $ 
	and since $ Du $ is independent of $ h $ and $ (DT_h)^{-1} \in C^{1}([-\epsilon,\epsilon], C(\Omega,\R{3 \times 3}))$ we can immediately conclude that \[ \left.\frac{d}{dh}\varepsilon^{h}(u)\right\vert_{h=0} = -\frac{1}{2} \left[DuDV +(DuDV)^{\top}\right]\] 
	holds a.e. on $ \Omega $. Since $ \tr(\cdot) $ is a linear operator on $ \R{3 \times 3} $ and due to invariance of the trace w.r.t. cyclic permutation and transposition 
	\begin{align*}
	&\left.\frac{d}{dh} \right\vert_{h=0}\gamma_{h}\tr(\varepsilon^{h}(u))\tr(\varepsilon^{h}(v)) \\  
	& = \Div(V)\tr(Du)\tr(Dv) - \tr(DuDV)\tr(Dv)-\tr(Du)\tr(DvDV) \\
	& =\Div(V)\tr(\varepsilon(u))\tr(\varepsilon(v)) + \tr(\dot{\varepsilon}(u))\tr(\varepsilon(v)) +  \tr(\varepsilon(v))\tr(\dot{\varepsilon}(u))
	\intertext{and}
	&2\left.\frac{d}{dh}\right\vert_{h=0} \hspace*{-2mm}\gamma_{h}\tr(\varepsilon^{h}(u)\varepsilon^{h}(v)) \\
	& =  \Div(V)\tr[(Du+Du^{\top})Dv] - \tr[Dv(DuDV +(DuDV)^{\top})]  -\tr[(Du +Du^{\top}) DvDV]\\[1ex]
	& = 2\Div(V)\tr(\varepsilon(u)\varepsilon(v)) + \tr(\varepsilon(v)\dot{\varepsilon}(u)) + \tr(\varepsilon(u)\dot{\varepsilon}(v))
	\end{align*}
	holds point wise a.e. on $ \Omega $. On the other hand we have to find suitable majorants in a neighborhood of $ h=0 $. With $ \tr(\varepsilon^{h}(u)) = \tr(Du(DT_{h})^{-1}) = \sum_{k=1}^{3} \langle \nabla u_{k}, (DT_h)^{-1}_{.k}\rangle$, where $(DT_h)^{-1}_{.k}=\frac{\partial T_{h}^{-1}}{\partial x_{k}}\circ T_h$ is the $ k $-th column of $ (DT_h)^{-1} $, we obtain by application of the Cauchy-Schwarz inequality
	\begin{align*}
	&\left|\frac{d}{dh}\gamma_{h}(x)\tr(\varepsilon^{h}(u(x)))\tr(\varepsilon^{h}(v(x)))\right| 
	\leq C \sum_{k,l=1}^{3}  \Norm{\nabla u_{k}(x)}{} \Norm{\nabla v_{l}(x)}{}
	\end{align*} 
	for a.e. $ x \in \Omega $. The functions $ \Norm{\nabla u_{k}}{} \Norm{\nabla v_{l}}{}$ are elements of $L^1(\Omega)$ for all indices $ k,l$. Moreover $ h \to \gamma_h \in C^{1}([-\epsilon,\epsilon],C^{0}(\overline{\Oext}))$ by Lemma \ref{Transf_shape_opt:Lem: Properties gammat} and $ h \to T_h,\,T_h^{-1} \in C^{1}([-\epsilon,\epsilon], C^{1}(\overline{\Oext},\R{3})) $ by Lemma \ref{Transf_shape_opt:Lem: Properties D_Tt}, what implies the existence of the constant $ C>0 $. Analogously, we obtain
	$\tr(\varepsilon^{h}(u)\varepsilon^{h}(v))  = \sum_{k,l=1}^{3} 
	\skp{\nabla u_k}{(DT_h)^{-1}_{.l}}(\skp{\nabla v_l}{(DT_h)^{-1}_{.k}} + \skp{\nabla v_k}{(DT_h)^{-1}_{.l}})$
	and application of the same arguments leads to
	\begin{align*}
	\left|\frac{d}{dh} \gamma_{h}(x)\tr(\varepsilon^{h}(u(x))\varepsilon^{h}(v(x)))\right| \leq C \sum_{k,l=1}^{3} \Norm{\nabla u_k(x)}{}\left(\Norm{\nabla v_l(x)}{} + \Norm{\nabla v_k(x)}{}\right).
	\end{align*}
	where $ \Norm{\nabla u_k}{}\left(\Norm{\nabla v_l}{} + \Norm{\nabla v_k}{}\right) \in L^1(\Omega) $ for each $ k,l=1,2,3 $. 
	We can now apply Lemma \ref{Diff_Banach_Space:Lem:Cont_Diff_Parameter_Integrals} ii) and change the order of differentiation and integration. 
	
	This leads to 
	\begin{align} \label{Ex_Shape_Deriv_LinEl:Eq: dotB}
	\dot{B}(u,v) =&\lambda \int_{\Omega} \tr(\varepsilon(u))\tr(\varepsilon(v))\Div(V) + \tr\left(\dot{\varepsilon}(u)\right)\tr(\varepsilon(v)) + \tr\left(\dot{\varepsilon}(v)\right)\tr(\varepsilon(u)) \, dx\nonumber \\ 
	&+ 2 \mu \int_{\Omega} \tr\left(\varepsilon(u)\varepsilon(v)\Div(V)
	+ \varepsilon(u)\dot{\varepsilon}(v) 
	+\dot{\varepsilon}(u)\varepsilon(v)\right)  \, dx\\
	=& \int_{\Omega} \tr\left(\dot{\varepsilon}(u) \sigma(v)+ \sigma(u)\dot{\varepsilon}(v)+\Div(V)\sigma(u)\varepsilon(v) \right)  \, dx. \nonumber
	\end{align}
	Finally we replace $ \Omega $ by $ \Omega_{t} $ and $ u $ by $ \tilde{u}_{t}=u \circ T_t^{-1}$ and $ v $ by $ \tilde{v}_{t}=v \circ T_t^{-1} $ and obtain equation \eqref{Ex_Shape_Deriv_LinEl:Lem: dotBt}.
	Now we apply analogous arguments to those used to prove Lemma \ref{Ex_Shape_Deriv_LinEl:Lem: continuity Bt}  and use $ V \in C^{1}(\overline{ \Oext},\R{3}) $ to derive a constant $ C_{t,V} $ such that $$|\dot{B}^{t}(u,v)| \leq C_{t,V} \Norm{u}{H^1(\Omega,\R{3})} \Norm{v}{H^1(\Omega,\R{3})}.$$ If $ t \in (-\epsilon,\epsilon) $ then this constant can be chosen independently from  $ t $. 
\end{proof}


\begin{lem} \label{Ex_Shape_Deriv_LinEl:Lem: Continuity t-> dot(L)^t and t-> dot(B)^t }
	Suppose that $V \in \Vad{1}{\Oext} $, $ \Omega \in \mathcal{O}_{1} $ and let $ [-\epsilon,\epsilon] \subset I_V $.
	\begin{itemize}
		\item[i)]   
		Let  $ f\in C^{1}(\overline{\Omega^{ext}},\R{3})$ and $  g\in C^{1}(\overline{\Omega^{ext}},\R{3})$. Then the map $ (-\epsilon,\epsilon) \to H^{1}_{D}(\Omega,\R{3})' ,$ $ t \mapsto \dot{L}^{t}$ is strongly continuous, where 
		$ \dot{L}^{t} $ as defined in Lemma \ref{Ex_Shape_Deriv_LinEl:Lem:Derivative Lt}. 
		\item[ii)] Also the mapping $ (-\epsilon,\epsilon) \to \mathcal{B}(H^{1}_{D}(\Omega,\R{3}) ),\, t \mapsto \dot{B}^{t} $ with
		$ \dot{B}^{t}  $ as in Lemma \ref{Ex_Shape_Deriv_LinEl:Lem: derivative Bt} is strongly continuous. 
	\end{itemize}
\end{lem}

\begin{proof}
	i) We have to show that $ t \to s $ implies $ \Vert\dot{L}^{t}-\dot{L}^{s}\Vert_{H^{1}_{D}(\Omega,\R{3})' } \to 0 $. W.l.o.g. let $ s=0 $. We have $$ \dot{L}^{t}(v)-\dot{L}(v) 
	= \int_{\Omega} \langle \gamma_{t}f_{V}\circ T_t-f_{V} ,v \rangle\, dx +   \int_{\Gamma_{N}} \langle \omega_{t}g_{V,t}\circ T_t-g_{V} ,v \rangle \, dS$$ and thus
	\begin{align*}
	|\dot{L}^{t}(v)-\dot{L}(v)| 
	\leq &\Norm{\gamma_{t}f_{V}\circ T_t-f_{V}}{\infty,\Omega}\Norm{v}{H^1(\Omega,\R{3})} + C\Norm{\omega_{t}g_{V,t}\circ T_t-g_{V}}{\infty,\Gamma_N}\Norm{v}{H^{1}(\Omega,\R{3})}
	\end{align*}
	where $ C>0 $ depends on the trace operator $ \mathbf{T_{\Gamma}}:H^1(\Omega) \to L^2(\Gamma) $ and $ f_{V},\, g_{V,t} $ are defined as in Lemma \ref{Ex_Shape_Deriv_LinEl:Lem:Derivative Lt}. Therefore, 
	\begin{align*}
	\Vert\dot{L}^{t}-\dot{L}\Vert_{H^{1}_{D}(\Omega,\R{3})' } 
	\leq \Norm{\gamma_{t}f_{V}\circ T_t-f_{V}}{\infty,\Omega} + C\Norm{\omega_{t}g_{V,t}\circ T_t-g_{V}}{\infty,\Gamma_N}
	\end{align*}
	with 
	\begin{align*}
	&\Norm{\gamma_{t}f_{V}\circ T_t-f_{V}}{\infty,\Omega} \leq \Norm{\gamma_{t}-1}{\infty,\Omega}\Norm{f_{V}\circ T_t}{\infty,\Omega} +  \Norm{f_{V}\circ T_t-f_{V}}{\infty,\Omega}
	\intertext{and}
	&\Norm{\omega_{t}g_{V,t}\circ T_t-g_{V}}{\infty,\Gamma_N}  \leq \Norm{\omega_{t}-1}{\infty,\Gamma_N}\Norm{g_{V,t}\circ T_t}{\infty,\Gamma_N} +  \Norm{g_{V,t}\circ T_t-g_{V}}{\infty,\Gamma_N}.
	\end{align*}
	From  $ \gamma_{t} \to 1 $  and $ f_{V}\circ T_t \to f_{V}  $ uniformly on $ \Omega $ we conclude $ \Norm{\gamma_{t}f_{V}\circ T_t-f_{V}}{C^{0}(\Omega,\R{3})}  \underset{t \to 0}{\to} 0 $. It is also clear that  $ \omega_{t} \to 1 $ uniformly on $ \Omega^{ext} $, compare Lemma \ref{Transf_shape_opt:Lem: Properties omegat}. Then $ g_{V,t} \circ T_t = (\Div_{\Gamma_t}(v)g+DgV)\circ T_t= [\tr(DV- DV\,\vec{n}_t\vec{n}_{t}^{\top})g]\circ T_t+(DgV)\circ T_t $ on $ \Gamma_{N} $ can be extended to $ \Omega^{ext} $ by
	\begin{align*}
	g_{V,t} \circ T_t  
	&= \tr(DV\circ T_t\left[\mathrm{I}-\mathcal{N}_t\circ T_t(\mathcal{N}_{t}\circ T_t)^{\top} \right])g\circ T_t+(DgV)\circ T_t
	\end{align*}
	where $ \mathcal{N}_{t} \circ T_t = \frac{1}{\Norm{(DT_{t})^{-\top}\mathcal{N}}{}}(DT_{t})^{-\top}\mathcal{N}  \text{ on } \Omega^{ext}$
	and $ \mathcal{N} $ is a unitary extension of the unit outward normal vector field $ \vec{n} $ on $ \Gamma $ to $ \Oext $.
	Obviously, $ (DgV)\circ T_t \to DgV $, $ g \circ T_t \to g  $ and $ DV \circ T_t \to DV $ uniformly on $ \Omega^{ext} \supset \Gamma_{N} $ when $ t \to 0 $ and thus we have to analyze  
	\begin{align*}
	(\mathcal{N}_t\circ T_t)(\mathcal{N}_{t}\circ T_t)^{\top} = \frac{1}{\Norm{(DT_{t})^{-\top}\mathcal{N}}{}^2}(DT_{t})^{-\top}\mathcal{N}\mathcal{N}^{\top}(DT_{t})^{-1}
	\end{align*}
	but since $ T_t \to id $ and $DT_{t} \to \mathrm{I} $ uniformly when $ t \to 0 $ it follows $ (\mathcal{N}_t\circ T_t) (\mathcal{N}_{t}\circ T_t)^{\top}\underset{t \to 0}{\longrightarrow}  \mathcal{N}\mathcal{N}^{\top} $ in $ C^{0}(\Omega^{ext},\R{3 \times 3}) $ and thus also $ \Norm{\omega_{t}g_{V,t}\circ T_t-g_{V}}{\infty,\Gamma_N}  \underset{t \to 0}{\longrightarrow} 0$. \\[1ex]
	ii) We have
	\begin{align}
	|\dot{B}&^{t}(u,v)-\dot{B}(u,v)| \nonumber\\
	\leq 
	&\int_{\Omega} \big|  \tr[\dot{\varepsilon}(\tilde{u}_{t})\sigma(\tilde{v}_{t})+\sigma(\tilde{u}_{t})\dot{\varepsilon}(\tilde{v}_{t})+\Div(V)\sigma(\tilde{u}_{t})\varepsilon(\tilde{v}_{t})] \circ T_t \big| \, dx\, \Norm{\gamma_{t}-1}{\infty,\Omega} \nonumber\\
	& + \int_{\Omega} \big|  \tr[\dot{\varepsilon}(\tilde{u}_{t}) \circ T_t\sigma(\tilde{v}_{t})\circ T_t - \dot{\varepsilon}(u)\sigma(v)] \big|\, dx  \nonumber
	\\
	&+\int_{\Omega} \big|\tr[\sigma(\tilde{u}_{t})\circ T_t \dot{\varepsilon}(\tilde{v}_{t})\circ T_t -\sigma(u)\dot{\varepsilon}(u)] \big|\, dx 
	\\
	&+\int_{\Omega} \big|\tr[\sigma(\tilde{u}_{t})\circ T_t\varepsilon(\tilde{v}_{t})\circ T_t]\big|\, dx\, \Norm{\Div(V)\circ T_t-\Div(V)}{\infty,\Omega}\nonumber
	\\
	&+ \int_{\Omega} \big| \Div(V)\tr[\sigma(\tilde{u}_{t})\circ T_t\varepsilon(\tilde{v}_{t})\circ T_t-\sigma(u)\varepsilon(u)]\big| \, dx,\nonumber
	\end{align}
	where 
	\begin{align*}  \dot{\varepsilon}(\tilde{u}_{t}) \circ T_t &= \dot{\varepsilon}(u \circ T_t^{-1}) \circ T_t = -\frac{1}{2}\left(Du (DT_t)^{-1}(DV \circ T_t) +[Du (DT_t)^{-1}(DV \circ T_t)]^{\top} \right), 
	\intertext{ and }
	\sigma(\tilde{u}_{t}) \circ T_t &= \sigma(u \circ T_t^{-1}) \circ T_t = \lambda\tr(Du (DT_t)^{-1})\mathrm{I}+\mu(Du (DT_t)^{-1}+[Du (DT_t)^{-1}]^{\top})
	\end{align*}
	We only perform the necessary estimates for the second term in detail since the calculations concerning the other terms proceed analogously but become very extensive:
	\begin{align*}
	&\tr[\dot{\varepsilon}(\tilde{u}_{t}) \circ T_t\sigma(\tilde{v}_{t})\circ T_t - \dot{\varepsilon}(u)\sigma(v)]\\[1ex]
	= ~&  \tr[ (\dot{\varepsilon}(\tilde{u}_{t})\circ T_t  - \dot{\varepsilon}(u))\sigma(v)] + \tr[\dot{\varepsilon}(\tilde{u}_{t}) \circ T_t(\sigma(\tilde{v}_{t}) \circ T_t - \sigma(v))].
	\end{align*}
	The first term therein can be reformulated as
	\begin{align*}
	\tr[ (\dot{\varepsilon}(\tilde{u}_{t})\circ T_t  - \dot{\varepsilon}(u))\sigma(v)] 
	&= \tr[Du(DV-(DT_{t})^{-1}DV \circ T_t)\sigma(v)]\\
	& =\sum_{i,j,k=1}^{3} Du_{i,j}(DV- (DT_t)^{-1}DV \circ T_t)_{j,k}\sigma(v)_{k,i}\, .
	\end{align*}
	Since $ \sigma(u) $ is symmetric we obtain
	\begin{align*}
	\sum_{i,j,k=1}^{3} \big|Du_{i,j}(x)\sigma(v)_{k,i}(x)\big|
	= \sum_{i,j,k=1}^{3} \big|Du_{i,j}(x)\sigma(v)_{i,k}(x)\big| 
	\leq   \sum_{j,k=1}^{3} \Norm{\nabla u_j(x)}{}\Norm{\sigma(v)_{.,k}(x)}{}
	\end{align*}
	where $ \sigma(v)_{.,k}  $ denotes the $ k $-th column of $ \sigma(v) $,
	this implies
	$$ \big|\tr[ (\dot{\varepsilon}(\tilde{u}_{t})\circ T_t(x)  - \dot{\varepsilon}(u(x)))\sigma(v(x))] \big| 
	\leq C(V,T_t)  \sum_{j,k=1}^{3} \Norm{\nabla u_j(x)}{}\Norm{\sigma(v)_{.,k}(x)}{} $$
	for a.e. $ x \in \Omega $, with $C(V,T_t) = \Norm{DV- (DT_t)^{-1}(DV \circ T_t)}{\infty,\Omega}.$
	Hence 
	\begin{align*}
	\int_{\Omega}|\tr[ (\dot{\varepsilon}(\tilde{u}_{t})\circ T_t - \dot{\varepsilon}(u))\sigma(v)]|\, dx 
	&\leq C(V,T_t)  \sum_{j,k=1}^{3} \int_{\Omega} \Norm{\nabla u_j(x)}{}\Norm{\sigma(v)_{.,k}(x)}{} \, dx \\
	&\leq C \,C(V,T_t) \Norm{u}{H^1(\Omega,\R{3})} \Norm{v}{H^1(\Omega,\R{3})} 
	\end{align*}
	for a combinatorial constant $ C=C(\lambda,\mu,n=3) $ independent of $ t $ and thus
	\begin{align*}
	\sup_{\Norm{u}{H^1}\leq 1\atop \Norm{v}{H^1}\leq 1}  \left|\int_{\Omega}\tr[ (\dot{\varepsilon}(\tilde{u}_{t})\circ T_t  - \dot{\varepsilon}(u))\sigma(v)]\, dx\right| 
	&\leq C C(V,T_t) \underset{t \to 0}{\to} 0.
	\end{align*}
	For term $  \tr[\dot{\varepsilon}(\tilde{u}_{t}) \circ T_t(\sigma(\tilde{v}_{t}) \circ T_t - \sigma(v))] $ we analogously obtain
	\begin{align*}
	\int_{\Omega}|\tr[\dot{\varepsilon}(\tilde{u}_{t}) \circ T_t(\sigma(\tilde{v}_{t}) \circ T_t - \sigma(v))]|\, dx 
	&\leq C(V)C(T_t)\Norm{u}{H^1(\Omega,\R{3})} \Norm{v}{H^1(\Omega,\R{3})} 
	\end{align*}
	where $ C(T_t) \to 0  $ for $ t \to 0 $ and thus we end up with
	\begin{align*}
	\int_{\Omega} \big|  \tr[\dot{\varepsilon}(\tilde{u}_{t}) \circ T_t\sigma(\tilde{v}_{t})\circ T_t - \dot{\varepsilon}(u)\sigma(v)] \big|\, dx \underset{t \to 0}{\to 0}.
	\end{align*}
	The third and fifth term in equation (6.18) can be estimated analogously and vanish when $ t \to 0 $.
	
	\noindent For the terms $$ \int_{\Omega} \big|  \tr[\dot{\varepsilon}(\tilde{u}_{t})\sigma(\tilde{v}_{t})+\sigma(\tilde{u}_{t})\dot{\varepsilon}(\tilde{v}_{t})+\Div(V)\sigma(\tilde{u}_{t})\varepsilon(\tilde{v}_{t})] \circ T_t \big| \, dx $$ and $$ \int_{\Omega} \big|\tr[\sigma(\tilde{u}_{t})\circ T_t\varepsilon(\tilde{v}_{t})\circ T_t]\big|\, dx $$  we obtain 
	\begin{align*}
	&\int_{\Omega} \big|  \tr[\dot{\varepsilon}(\tilde{u}_{t})\sigma(\tilde{v}_{t})+\sigma(\tilde{u}_{t})\dot{\varepsilon}(\tilde{v}_{t})+\Div(V)\sigma(\tilde{u}_{t})\varepsilon(\tilde{v}_{t})] \circ T_t \big| \, dx  \leq \tilde{C}(T_t,V) \Norm{u}{H^1}\Norm{v}{H^1}
	\intertext{and}
	& \int_{\Omega} \big|\tr[\sigma(\tilde{u}_{t})\circ T_t\varepsilon(\tilde{v}_{t})\circ T_t]\big|\, dx \leq C(T_t) \Norm{u}{H^1}\Norm{v}{H^1} ,
	\end{align*}
	where $ \tilde{C}(T_t,V)>0 $ and $ C(T_t)>0 $ (depend also on $ \lambda $ and $ \mu $) can be bounded from above by a constant $ C $  independently of $ t \in [-\epsilon,\epsilon] $. Then
	$ \Norm{\Div(V) \circ T_t-\Div(V)}{\infty,\Omega} \to 0 $ and $ \Norm{\gamma_t - 1}{\infty,\Omega} \to 0 $ for $ t \to 0 $ imply the assertion.
\end{proof}

\subsection{Existence of $H^1$-material derivatives}

\noindent Until now, we verified all claims  of Theorem \ref{Parameter_Dep_PDE:Thm: Existence Deriv. strong Topology} despite of the existence of solutions $ q^{t} \in H^{1}_{D}(\Omega,\R{3}) $  apply for the equation $$ B^{t}(u,v)=L^{t}(v) \, \forall v \in H^{1}_{D}(\Omega,\R{3}) $$ with $ B^{t} $ and $ L^{t} $ defined by the equations \eqref{Ex_Shape_Deriv_LinEl:Eq: B^t(u^t,v)=L^t(v)} and \eqref{Ex_Shape_Deriv_LinEl:Def:L^t}, \eqref{Ex_Shape_Deriv_LinEl:Def:B^t} if $ f,\, g \in C^{1}(\overline{ \Oext},\R{3}) $. Hence the solution mapping $$(-\epsilon,\epsilon) \to H^{1}_{D}(\Omega,\R{3}), t \mapsto u^{t} $$  is continuous by Theorem  \ref{Parameter_Dep_PDE:Thm: Existence Deriv. strong Topology} a). Thus we now have to prove that there exist $H^{1}_{D}(\Omega,\R{3})$-solutions $ q^{t} $ to the family of equations
\begin{align}\label{Ex_Shape_Deriv_LinEl:Eq: Weak equation for dot_u_t}
B^{t}(q^{t},v)=\dot{L}^{t}(v) - \dot{B}^{t}(u^{t},v) :=L_{u^{t}}^{t}(v) \, \forall v \in H^{1}_{D}(\Omega,\R{3}).
\end{align}
to apply Theorem  \ref{Parameter_Dep_PDE:Thm: Existence Deriv. strong Topology}, also.
For the application of the Theorem of Lax-Milgram we only have to show that $ L_{u^{t}}^{t} \in H^{1}_{D}(\Omega,\R{3})' $ for any $ t \in I$ since Lemma \ref{Ex_Shape_Deriv_LinEl:Lem: Coercivity Bt} shows the coercivity of $ B^{t} $ and Lemma \ref{Ex_Shape_Deriv_LinEl:Lem: continuity Bt} its continuity. Also $ L_{u^{t}}^{t} \in H^{1}_{D}(\Omega,\R{3})' $ is clear since we have
\begin{align}\label{Ex_Shape_Deriv_LinEl:Eq: continuity L_neu}  
|L_{u^{t}}^{t}(v)| &\leq |\dot{L}^{t}(v)|+ |\dot{B}^{t}(u^{t},v)| 
\leq \Vert\dot{L}^{t}\Vert_{(H^{1})'}\Norm{v}{H^{1}} + \Vert\dot{B}^{t}\Vert_{\mathcal{B}(H^{1})}\Norm{u^{t}}{H^{1}}\Norm{v}{H^{1}} 
\end{align}
from Lemma \ref{Ex_Shape_Deriv_LinEl:Lem:Derivative Lt} and Lemma \ref{Ex_Shape_Deriv_LinEl:Lem: derivative Bt}.If $ t \in (-\epsilon,\epsilon)  $ Lemma \ref{Ex_Shape_Deriv_LinEl:Lem:Derivative Lt} and  Lemma \ref{Ex_Shape_Deriv_LinEl:Lem: derivative Bt} imply $$ \left|L_{u^{t}}^{t}(v)\right| \leq \left(C_1+ C_2\Norm{u_t}{H^{1}(\Omega_t)}\right) \Norm{v}{H^{1}}.$$ 
But, we can even prove that $ \Norm{L^{t}_{u^{t}}}{H^{1}_{D}(\Omega,\R{3})'} $ is uniformly bounded w.r.t. $(-\epsilon,\epsilon)  $ since the latter also holds for the norms $ \Norm{u_t}{H^{1}(\Omega_t,\R{3})} $.

\begin{lem}\label{Ex_Shape_Deriv_LinEl:Lem: uniform bounds u_t and u^t in H^1 }
	Let $ u_t \in H^{1}_{D,t}(\Omega_t,\R{3})$ solve \eqref{Ex_Shape_Deriv_LinEl:Eq: B_t(u_t,v)=L_t(v)} and $ u^{t}=u_{t} \circ T_t $ solve \eqref{Ex_Shape_Deriv_LinEl:Eq: B^t(u^t,v)=L^t(v)} then there exist constants $ C_1 >0 $ and $ C_2 >0$ independent of $ t $ such that
	\begin{align}
	&\Norm{u_{t}}{H^{1}(\Omega_t)}\leq C_1 \, \text{ for all } t \in (-\epsilon,\epsilon),
	\intertext{and}
	&\Norm{u^{t}}{H^{1}(\Omega ,\R{3})} \leq C_2  , \text{ for all } t \in (-\epsilon,\epsilon).
	\end{align}
\end{lem}

\begin{proof}
	From Lemma \ref{Ex_Shape_Deriv_LinEl:Lem: Coercivity Bt} we first derive
	\begin{align*}  
	\Varlambda\Norm{u^{t}}{H^{1}(\Omega,\R{3})}^2  
	\leq B^{t}(u^{t},u^{t}) 
	=  L^{t}(u^{t}) 
	\leq \Norm{L^{t}}{H^{1}_{D}(\Omega,\R{3})'} \Norm{u^{t}}{H^{1}(\Omega,\R{3})} 
	\end{align*} 
	since Lemma \ref{Ex_Shape_Deriv_LinEl:Lem: continuity Lt} implies that $ \Norm{L^{t}}{H^{1}_{D}(\Omega,\R{3})'} \leq C $ for all $ t\in (-\epsilon,\epsilon) $. Thus  \linebreak
	$ \Norm{u^{t}}{H^{1}(\Omega_t)} $ $\leq  \tilde{C} $.
	Finally, application of Lemma \ref{Ex_Shape_Deriv_LinEl:Lem:H1_estimates_ut} to $ u^{t}=u_t \circ T_t $ proves the statement. 
\end{proof}

\noindent Therefore, we can now conclude the following:

\begin{lem}\label{Lem: continuity L_neu}  
	Let $ \Omega \in \mathcal{O}_{1} $ and $V \in  \Vad{1}{\Oext}$ be some admissible vector field. Suppose that  $ f\in C^{1}(\overline{\Oext},\R{3}) $ and $ g\in C^{1}(\overline{\Oext},\R{3})$ and let $ u^{t} \in H^{1}_{D}(\Omega,\R{3}) $ be the unique solution of \eqref{Ex_Shape_Deriv_LinEl:Eq: B^t(u^t,v)=L^t(v)}. Then the linear form $
	L_{u^{t}}^{t} $ is equicontinuous on  $ (-\epsilon,\epsilon) \subset I_V $. 	
\end{lem}

\begin{prop}\label{Ex_Shape_Deriv_LinEl:Prop: existence weak solutions weak PDE for dot_u_t}
	Let $ \Omega \in \mathcal{O}_{1}  $ and  $V \in \Vad{1}{\Omega^{ext}}$ an admissible vector field. Suppose that  $ f\in C^{1}(\overline{\Oext},\R{3}) $ and $ g\in C^{1}(\overline{\Oext},\R{3})$ and let $ u^{t} \in H^{1}_{D}(\Omega,\R{3}) $ be the unique solution of \eqref{Ex_Shape_Deriv_LinEl:Eq: B^t(u^t,v)=L^t(v)}. Then the problem \eqref{Ex_Shape_Deriv_LinEl:Eq: Weak equation for dot_u_t}
	\begin{align*}
	B^{t}(q,v) =L_{u^{t}}^{t} \, \forall v \in H^{1}_{D}(\Omega,\R{3})
	\end{align*}
	has a unique solution $ q^{t} \in H^{1}_{D}(\Omega,\R{3}) $ for any $ t \in I_V $. Additionally, the mapping $ t \in (-\epsilon,\epsilon) \mapsto q^{t} \in H^{1}_{D}(\Omega,\R{3}) $ is continuous w.r.t. the strong topology on $ H $.
\end{prop}

\begin{proof} Let $ t \in (-\epsilon,\epsilon) $.
	Then, the bilinear form $ B^{t} $ defined in \eqref{Ex_Shape_Deriv_LinEl:Def:B^t}  is continuous and coercive and the linear form $ L^{t}_{u^{t}} $ is continuous as \eqref{Ex_Shape_Deriv_LinEl:Eq: continuity L_neu} shows. Thus we can apply the Theorem of Lax-Milgram. 
	
	To see the second statement, apply Lemma \ref{Parameter_Dep_PDE:Lem: Crit. for strong cont.} to $ B^{t}(q,v)=L^{t}_{u^t}(v)=\dot{L}^{t}(v) - \dot{B}^{t}(u^t,v) $:  For the left hand side everything is shown and from the continuity of $ t \to \dot{L}^t $ and $ t\to \dot{B}^{t}(u^t,.)  $ follows that $ t \to L^{t}_{u^{t}} $ is continuous.
\end{proof}

\begin{thm}\label{Ex_Shape_Deriv_LinEl:Thm: existence of strong H^1 material derivatives}
	Let  $V \in  \mathcal{V}^{ad}_{1}(\Oext)$ be an admissible vector field and $ \Omega \in \mathcal{O}_{1}  $. Suppose that  $ f\in C^{1}(\overline{\Oext},\R{3}) $ and $ g\in C^{1}(\overline{\Oext},\R{3})$. Furthermore, assume that $ [-\epsilon,\epsilon] \subset I_V $ and $ (u^{t})_{t \in I_V} \subset H^{1}_{D}(\Omega,\R{3}) $ is the family of unique solutions of \eqref{Ex_Shape_Deriv_LinEl:Eq: B^t(u^t,v)=L^t(v)}
	$$B^{t}(u,v)=L^{t}(v) \, \forall v \in H^{1}_{D}(\Omega,\R{3}),\, t \in I_V.$$ 
	Then there exists the material derivative  $ \dot{u}^{t} \in H^{1}_{D}(\Omega,\R{3}) $ w.r.t. the strong topology on $ H^{1}_{D}(\Omega,\R{3}) $ for every $ t \in (-\epsilon,\epsilon) $ and it is given by the unique solution $ q^{t}\in H^{1}_{D}(\Omega,\R{3}) $ of equation \eqref{Ex_Shape_Deriv_LinEl:Eq: Weak equation for dot_u_t}
	\begin{align*}
	B^{t}(q,v) =\dot{L}^{t}(v) - \dot{B}^{t}(u^{t},v)=L_{u^{t}}^{t}(v) \, \forall v \in H^{1}_{D}(\Omega,\R{3}).
	\end{align*}
	Moreover, $ t \to \dot{u}^{t}=q^t $ is strongly continuous on $ (-\epsilon,\epsilon) $ w.r.t. $ H^{1}_{D}(\Omega,\R{3}) $ and the shape derivative $ u'=\dot{u}-DuV $ is an element of $ L^2(\Omega,\R{3}) $.
\end{thm}

\begin{proof}
	
	We have to verify that all conditions that were posed in Theorem \ref{Parameter_Dep_PDE:Thm: Existence Deriv. strong Topology} are satisfied:  
	
	\noindent The continuity of any linear form $ L^{t}$ and any bilinear form $ B^{t}$, $ t\in (-\epsilon,\epsilon) $  has been proved in Lemma \ref{Ex_Shape_Deriv_LinEl:Lem: continuity Lt} and Lemma \ref{Ex_Shape_Deriv_LinEl:Lem: continuity Bt}. The equicoercivity of $(B^{t})_{t \in (-\epsilon,\epsilon)}$ has been shown in Lemma \ref{Ex_Shape_Deriv_LinEl:Lem: Coercivity Bt}. Proposition \ref{Ex_Shape_Deriv_LinEl:Prop: Link between u_t ant u^t} assures the existence of a unique solution to $ B^{t}(u^t,v)=L^t(v) $ for any $ v \in H^{1}_{D}(\Omega,\R{3}) $. 
	
	\noindent Moreover, it is necessary that the families $ (\dot{L}^{t})_{t \in (-\epsilon,\epsilon)} $ and $ (\dot{B}^{t})_{t \in (-\epsilon,\epsilon)} $ exist regarding the weak topologies on $ H^{1}_{D}(\Omega,\R{3})' $ and $ \mathcal{B}(H^{1}_{D}(\Omega,\R{3})) $, respectively. Additionally these families have to be equicontinuous, see the points (l1') and (b3') of Theorem \ref{Parameter_Dep_PDE:Thm: Existence Deriv. strong Topology}. We proofed these claims in Lemma \ref{Ex_Shape_Deriv_LinEl:Lem:Derivative Lt} and Lemma \ref{Ex_Shape_Deriv_LinEl:Lem: derivative Bt}.  Lemma \ref{Ex_Shape_Deriv_LinEl:Lem: Continuity t-> dot(L)^t and t-> dot(B)^t } proves that conditions (l3') and (b2') hold and
	Lemma \ref{Ex_Shape_Deriv_LinEl:Prop: existence weak solutions weak PDE for dot_u_t} shows the existence of unique solutions of \eqref{Ex_Shape_Deriv_LinEl:Eq: Weak equation for dot_u_t}
	\[ B^{t}(q,v)=\dot{L}^{t}(v) - \dot{B}^{t}(u^{t},v)=L_{u^{t}}^{t}(v) \, \forall v \in H^{1}_{D}(\Omega,\R{3}) \] for any $ t \in I_{V} $.
	
	\noindent The continuity of $ t \to \dot{u}^{t} =q^t $  on $ (-\epsilon,\epsilon) $ regarding the strong norm topology on $ H^{1}_{D} $ finally follows from the previous Proposition \ref{Ex_Shape_Deriv_LinEl:Prop: existence weak solutions weak PDE for dot_u_t}.
\end{proof}

%

\section{Material derivatives in H\"older spaces} \label{Ex_Shape_Deriv_LinEl:Sec: hölder material derivatives}

Let $ k \in \N{} $ be an integer $ k\geq 2 $, $ V \in \Vad{k+1}{\Oext} $ and $ (T_t)_{t \in I_V}=(T_t[V])_{t \in I_V} $ the associated transformation family.

To assure that equation \eqref{Ex_Shape_Deriv_LinEl:Eq: LinEl Omega_t} below always possesses a strong solution  we have to ensure that the two boundary parts $ \Gamma_D $ and $ \Gamma_N $ of the baseline design $ \Omega $ have a positive distance. Since the Transformations $ T_{t} $ depend continuously on $ t $ and on $ x \in \overline{\Omega} $ they maintain these property over a finite time interval. This is implied by equation (6.29) in \cite{GilbTrud}: There is a constant $ C_t $ depending on the norms of the transformations $ T_t $ such that 
\begin{align}
C_{t}^{-1}\Norm{x-y}{} \leq \Norm{T_t(x)-T_t(y)}{} \leq C_t\Norm{x-y}{}.
\end{align}
For $ t \in (-\epsilon,\epsilon) \Subset I_V$ the constant $C_t $ can be chosen independently from $ t $ since
$ \sup_{t \in (-\epsilon,\epsilon)} \Norm{T_{t}[V]}{\Cm{k+1}{\overline{\Omega^{ext}},\R{3}} }\leq C_{V,\epsilon} $
and therefore
\begin{align}\label{Ex_Shape_Deriv_LinEl:Eq: uniform bounds distance}
C_{V,\epsilon}^{-1}\Norm{x-y}{} \leq \Norm{T_t(x)-T_t(y)}{} \leq C_{V,\epsilon}\Norm{x-y}{}
\end{align}
if $ \epsilon\geq 0 $ and $ V $ are fixed.

\begin{defn}[Baseline Design]\label{Ex_Shape_Deriv_LinEl:defn: baseline design}
	We say $ \Omega \in \mathcal{O}_{k} $, $ k\geq 1  $ is a \textit{baseline design} if  there is $ \Omega_b \in \mathcal{O}_{k} $ and a domain $ \Omega_{int} \Subset \Omega_b $ of class $ C^k $ and $ D>0 $ such that
	\begin{itemize}
		\item[i)~~] $ \dist(\partial\Omega_{int},\partial \Omega_b ) \geq D$, 
		\item[ii)~] $ \Omega = \Omega_b\setminus \overline{\Omega_{int}} $
		\item[iii)]$  \Omega $ possesses a uniform $ C^{k} $-hemisphere property.
	\end{itemize}
	By $\mathcal{O}^{b}_{k} \subset \mathcal{O}_k$
	we denote the set of admissible baseline designs.
\end{defn}
Then we define the two boundary parts: $ \Gamma_N=\partial\Omega_N = \partial(\Omega_{int}) $ is the \textit{interior boundary} where the component $ \Omega $ is clamped and $ \Gamma_D=\Gamma\setminus \Gamma_{N} $ is the \textit{exterior boundary} part. 

\begin{lem}\label{Ex_Shape_Deriv_LinEl:Lem: uniform hemisphere property Omega_t}
	The set $ (\Omega_t)_{t \in (-\epsilon,\epsilon)} $ satisfies a uniform hemisphere condition for any $ \Omega \in  \mathcal{O}^{b}_{k}  $ and $ [-\epsilon,\epsilon] \subset I_{V} $.
\end{lem}

\begin{proof}
	Let $ \Omega \in  \mathcal{O}^{b}_{k} $ be arbitrary. The boundary of $ \Omega $ is compact since $ \Gamma $ is a bounded domain of class $ C^{k+1}$. This implies that the number of hemisphere transformations $ \mathbb{T}_{x} $, see Definition \ref{PDE_Systems:Def:Hemisphere_Trafo}, and the number of points  $ x $ within a distance $ 0<d<\nicefrac{D}{2} $ of the boundary $ \Gamma $ needed, is finite. The mapping $ T_t:\overline{\Omega} \to \overline{\Omega_t} $ is a bijection for any $ t \in I_V $ and we can choose $ U_{T_t(x)}:= T_t(U_x) $ as neighborhood of $ T_t(x) \in \Gamma_{t} $ what implies
	\[ U_{T_t(x)}\cap \overline{\Omega}_{t} =T_t(\overline{U_x} \cap \overline{\Omega}) = T_t(\mathbb{T}_{x}(\Sigma_{R(x)})) \] and analogously for the boundary. Thus we can choose $ 0<d_t \leq C^{-1}_{t}d  $ as an appropriate distance and $ T_t \circ \mathbb{T}_x $ as appropriate hemisphere transformations for $ \Omega_t $. Since the transformations $ T_t $ are bounded in their $ C^{k} $-norms and $ 0< C_{V,\epsilon} \leq C^{-1}_{t} $ the set $ (\Omega_t)_{t \in (-\epsilon,\epsilon)} $ satisfies the required uniform hemisphere condition.
\end{proof}

\begin{prop}\label{Ex_Shape_Deriv_LinEl:Prop: properies of solutions u_t of DTP on Omega_t}
	Let $ k \geq 2 $, $  \Omega \in \mathcal{O}_{k+1}^{b}  $, $ V\in V^{ad}_{k+1}(\Oext)$, $ f \in C^{k-2,\phi}(\overline{ \Oext}) $ and $ g \in C^{k-1,\phi}(\overline{ \Oext})  $. Then, for any $ t \in I_V$ there exists a unique solution $ u_t \in C^{k,\phi}(\overline{ \Omega_t},\R{3}) $ of 
	\begin{align}\label{Ex_Shape_Deriv_LinEl:Eq: LinEl Omega_t}
	\left. 
	\begin{array}{r c l l}
	-\Div( \se(u_t)) & = & f   &\text{ in } \Omega_t \\
	u_t &=& 0  &\text{ on } \Gamma_{D,t}  \\
	\se(u_t) \vec{n}_t &=&g &\text{ on }\Gamma_{N,t}.
	\end{array} 
	\right.
	\end{align} 
	\begin{itemize}
		\item[i)]The norm of  $ u_t $ is bounded according to \[ \Norm{u_t}{C^{k,\phi}(\Omega_t,\R{3})} \leq C_t \left(\Norm{f}{C^{k-2,\phi}(\Omega_t,\R{3})} + \Norm{g}{C^{k-1,\phi}(\Gamma_t,\R{3})} +\Norm{u_t}{C(\Omega_t,\R{3})}\right) \]
		for some constant $ C_t>0 $, $ t \in I_V $. 
		\item[ii)] If $ (-\epsilon,\epsilon) \Subset I_V $ and $ t \in (-\epsilon,\epsilon)$ then $ C=C_t $ can be chosen uniformly w.r.t. the parameter $ t $.
		\item[iii)] By means of Lemma \ref{Ex_Shape_Deriv_LinEl:Lem: uniform cone lemma} the term  $ \Norm{u_t}{C(\Omega_t,\R{3})} $ can be replaced by \\ $  \Norm{u_t}{L^{1}(\Omega_t,\R{3})}  $ and there even exists a constant $ C_u>0 $ such that 
		\begin{align}\label{Ex_Shape_Deriv_LinEl:Eq: uniform bound u_t in C^3phi}
		\Norm{u_t}{C^{k,\phi}(\Omega_t,\R{3})} \leq C_u
		\end{align} 
		uniformly in $ t \in (-\epsilon,\epsilon) $.
	\end{itemize}
\end{prop}

\begin{proof}
	i) As Theorem \ref{Linear_Elasticity:Thm:LinEl_Classical_Sol} already suggests, the dependence of $ C_t $ on the domain $ \Omega_t $ is due to the hemisphere transformations i.e. $ \varLambda_t,\, \triangle_{\Gamma_t},\, C_{\mathbb{T}_t}$ and $ d_t $. Lemma \ref{Ex_Shape_Deriv_LinEl:Lem: uniform hemisphere property Omega_t} implies, that $ C_{\mathbb{T}_t} $ and $ d_t $ can be chosen uniformly. The minor constant also depends on the mappings $ T_t $ since it depends on the transformations that make $ \Gamma_t $ plane - these transformations are the hemisphere transformations that are already known to be uniformly bounded. The ellipticity constant $ \varLambda_t $ is also bounded from above and below, what can be shown analogously to the last part of the proof of Lemma 5.6. in \cite{GottschSchmitz}.
	
	\noindent The existence of the solution is then immediately clear by Theorem \ref{Linear_Elasticity:Thm:LinEl_Classical_Sol} and it is also clear that ii) holds for any $ t \in (-\epsilon,\epsilon) $. 
	
	\noindent iii) Moreover, $ \mathcal{O}_{k+1}^{b} $ satisfies a uniform cone property: $ \Omega $ is a Lipschitz domain and therefore satisfies a cone property as explained in \cite{Chen75,AdamsFournier}. Even the second derivatives of $ T_t,\, t \in(-\epsilon,\epsilon) $ and therefore the curvature of the boundary $ \Gamma_t $ is uniformly bounded. Thus we can apply Lemma \ref{Ex_Shape_Deriv_LinEl:Lem: uniform cone lemma} and use the same arguments as in the proof of Theorem \ref{Linear_Elasticity:Thm:LinEl_Classical_Sol}.
\end{proof}

\enlargethispage{\baselineskip}

\begin{lem}\label{Ex_Shape_Deriv_LinEl:Lem: uniform bounds ut}
	Let $ k\geq 2 $, $ \Omega \in \mathcal{O}_{k+1}^{b}$ and $ \Omega_{t}=\T{t}(\Omega), \, t \in I_V  $ be associated to the admissible vector field $V \in  \mathcal{V}^{ad}_{k+1}(\Oext)$. Suppose that  $ f\in C^{k-2,\phi}(\overline{\Oext},\R{3}) $, $ g\in C^{k-1,\phi}(\overline{\Oext},\R{3})$ for some $ \phi\in (0,1) $ and let $ u_{t} \in C^{k,\phi}(\O{t},\R{3}) $ be the unique solution of \eqref{Ex_Shape_Deriv_LinEl:Eq: LinEl Omega_t} on $ \O{t} $. 
	
	\noindent Under these assumptions the mapping $ t \in   (-\epsilon,\epsilon) \to u^{t}=u_{t} \circ T_t \in C^{k,\phi}(\Omega,\R{3})  $ is uniformly bounded in $ C^{k,\phi}(\Omega,\R{3}) $ for any interval $ [-\epsilon,\epsilon] \subset I_V $.
\end{lem}

\begin{proof}
	The previous Proposition yields that there exists a constant $ C $ independent of $ t \in (-\epsilon,\epsilon) $ such that $\Norm{u_{t}}{C^{k,\phi}(\O{t},\R{3})}\leq C_u$. Application of equation (6.29) in \cite{GilbTrud} now leads to
	\[ \Norm{u^{t}}{C^{k,\phi}(\Omega,\R{3})} = \Norm{u_{t} \circ \T{t}}{C^{k,\phi}(\Omega,\R{3})} \leq C \Norm{u_{t}}{C^{k,\phi}(\O{t},\R{3})}\leq C C_{u} \]
	where $ CC_{u} $ is independent of $ t $, see also inequality (24) in \cite{BittGottsch}.  
\end{proof}

We will now consider the strong PDE formulation that belongs to the weak equation given in equation \eqref{Ex_Shape_Deriv_LinEl:Eq: Weak equation for dot_u_t}. Therefore, we need the formula for integration by parts for vector valued functions $ z \in C^1(\Omega,\R{3})  $ and matrix valued functions $ A \in C^1(\Omega,\R{3 \times 3}) $, see \eqref{App: Eq: divergence  matrix valued functions}, i.e.
\begin{align}\label{Ex_Shape_Deriv_LinEl:Eq: divergence thm f. matrix valued functions and vector fields }
\int_{\Omega} \tr(A Dz) \, dx =   \int_{\Omega} \skp{-\Div(A)}{z} \, dx +  \int_{\Gamma} \skp{A \vec{n}}{z} \, dS.
\end{align}

\begin{prop}\label{Ex_Shape_Deriv_LinEl:Prop: strong fomulation of the DTP for q_t} $  $ \\
	Suppose that $ k \geq 2 $, $ V \in \mathcal{V}^{ad}_{k+1}(\overline{ \Oext}) $,  $ f \in C^{k-1,\phi}(\overline{ \Oext}) $, $ g \in C^{k,\phi}(\overline{ \Oext})  $ and $ \Omega\in \mathcal{O}_{k+1}^{b} $. 
	\begin{itemize}
		\item[i)] 
		Let $ u\in C^{k,\phi}(\Omega,\R{3}) $ be the unique solution to \eqref{Reliability:Eq:LinEl}. Then 
		\begin{align}\label{Ex_Shape_Deriv_LinEl:Eq: PDE for q} 
		\left. 
		\begin{array}{r c l l}
		-\Div( \se(q)) &=&f_{V}+f_{u}   &\text{ in } \Omega \\
		q &=& 0  &\text{ on } \Gamma_{D} \\
		\se(q)  \vec{n} &=&g_{V}-G_{u}\vec{n} &\text{ on }\Gamma_{N}
		\end{array} 
		\right.
		\end{align}
		where 
		\begin{align*}
		\begin{array}{l l}
		f_V = Df V +f\Div(V)  &  f_{u}= \Div\left[ DV\sigma(u) +\dot{\sigma}(u) + \Div(V)\sigma(u)\right] \\
		g_V= Dg V + g\Div_{\Gamma}(V) & G_{u}=
		\sigma(u)DV^{\top}+\dot{\sigma}(u) + \Div(V)\sigma(u)
		\end{array}
		\end{align*} 
		and  $$ \dot{\sigma}(u) = \lambda\tr(\dot{\varepsilon}(u))\mathrm{I}+2\mu\dot{\varepsilon}(u) $$
		has a unique solution $ q \in C^{k,\phi}(\Omega,\R{3}) $. 
		\item[ii)] If $ t \in I_V $ and 
		$ u=u_t\in C^{k,\phi}(\Omega_t,\R{3})$ is the unique solution of \eqref{Reliability:Eq:LinEl} on $  \Omega_t  $ and $ q=q_t \in C^{k,\phi} (\Omega_t,\R{3})$ is the unique solution of \eqref{Ex_Shape_Deriv_LinEl:Eq: PDE for q} on $ \Omega_t $,  then $ q^{t}=q_{t} \circ T_t $ solves the weak formulation of \eqref{Ex_Shape_Deriv_LinEl:Eq: PDE for q} which is given by \eqref{Ex_Shape_Deriv_LinEl:Eq: Weak equation for dot_u_t}.  
\end{itemize}\end{prop}

\begin{proof}
	\textit{i)} First we show that there exists a unique $ C^{k,\phi} $-solution to  
	\begin{align}\label{Ex_Shape_Deriv_LinEl:Eq: PDE for q_t} 
	\left. 
	\begin{array}{rcll}
	-\Div( \sigma(q)) &=&f_{V}+f_{u_t}   &\text{ in } \Omega_t \\
	q &=& 0  &\text{ on } \Gamma_{D,t}  \\
	\se(q) \vec{n}_t &=&g_{V}-G_{u_t}\vec{n}_t &\text{ on }\Gamma_{N,t}
	\end{array} 
	\right.
	\end{align}
	for arbitrary $ t \in I_V $. Therefor we need the volume force to be a $C^{k-2,\phi}(\overline{ \Oext})$-vector and a the surface force to be a $C^{k-1,\phi}(\overline{ \Oext})$-vector field, see Theorem \ref{Linear_Elasticity:Thm:LinEl_Classical_Sol}. \\
	From $ f \in C^{k-1,\phi}(\overline{ \Oext},\R{3})  $ and $ V \in C^{k+1}(\overline{ \Oext},\R{3}) $ follows $$ f_{V} = Df V + f \Div(V) \in C^{k-2,\phi}(\overline{ \Oext},\R{3})  .$$ 
	The vector field $$ f_{u_t}=\Div[
	DV\sigma(u_t) +\dot{\sigma}(u_t) + \Div(V)\sigma(u_t)] $$ contains second order derivatives of $ u_{t} \in C^{k,\phi}(\Omega_{t},\R{3}) $ and first order derivatives of $ V $ and therefore $ f_{u_t} $ is an element of $ C^{k-2,\phi}(\Omega_{t},\R{3}) $ which implies $ f_V + f_{u_t} \in C^{k-2,\phi}(\Omega_t,\R{3}) $.\\
	Now we regard a $ C^{k} $-extension of the $ C^{k} $-outward normal vector field on $ \Gamma_{t} $ to $ \overline{\Oext} $ i.e.  $ \mathcal{N}_t $ given by 
	$$ \mathcal{N}_{t} \circ T_t = \frac{1}{\Norm{((DT_{t})^{-1})^{\top}\mathcal{N}}{}}((DT_{t})^{-1})^{\top}\mathcal{N}\in C^{k}(\overline{\Oext},\R{3})$$ 
	where $ \mathcal{N}\in C^{k}(\overline{\Oext},\R{3}) $ is an extension of the outward normal vector field $ \vec{n} $ on $ \Gamma $ to $ \overline{\Oext} $. Therefore, $$ \Norm{\vec{n}_t}{C^{k}(\Gamma_t,\R{3})}=\Norm{\mathcal{N}_t \vert_{\Gamma_t}}{C^{k}(\Gamma_t,\R{3})} \leq \Norm{\mathcal{N}_t}{C^{k}(\overline{\Oext},\R{3})} \leq C$$ due to the properties of $ T_t $.\\
	Moreover, $ g  \in C^{k,\phi}(\overline{ \Oext},\R{3})$, $ V \in C^{k+1}(\overline{\Oext},\R{3})$ and $ \mathcal{N}_t \in C^{k}(\overline{\Oext},\R{3})$ which implies  $g^{ext}_V:=  DgV +g (\Div(V) -\skp{DV \mathcal{N}_t}{\mathcal{N}_t}) \in C^{k-1,\phi}(\overline{\Oext},\R{3}) $ and thus $ g^{ext}_V\vert_{\Gamma_t}=g_V \in C^{k-1,\phi}(\Gamma_t,\R{3})   $.\\
	The matrix field $G_{u_t}= \sigma(u_t)DV^{\top}+\dot{\sigma}(u_t) + \Div(V)\sigma(u_t)$ contains first order derivatives of $ u_{t} \in C^{k,\phi}(\overline{\Omega_t},\R{3}) $ and first order derivatives of $ V \in C^{k+1}(\overline{\Oext},\R{3}) $ and therefor $ G_{u_t}\vec{n}_t $ is an element of $ C^{k-1,\phi}(\Gamma_t,\R{3}) $. 
	
	\noindent Thus Theorem \ref{Linear_Elasticity:Thm:LinEl_Classical_Sol} implies the existence of a unique solution $ q_{t} \in C^{k,\phi}(\overline{\Omega},\R{3})$ for any $ k \geq 2 $.

	\noindent \textit{ii)}  We first show that equation \eqref{Ex_Shape_Deriv_LinEl:Eq: PDE for q} is in fact the partial differential equation belonging to 
	$$ 
	B_{t}(u,v)=L_{u_t,t}(v) \, \forall v \in H^{1}_{D,t}(\Omega_t,\R{3})
	$$ 
	where $ B_t(u,v) = B^{t}(u \circ T_t,v \circ T_t) \text{ for } u,\, v \in H^{1}_{D,t}(\Omega_t,\R{3}) $
	and 
	\[L_{u_t,t}(v) = L^{t}_{u^{t}} (v \circ T_t) 
	= \dot{L}^{t}(v \circ T_t) -  \dot{B}^{t}(\underbrace{u_{t} \circ T_t}_{=u^{t}},v \circ T_t ) \text{ for } v \in H^{1}_{D,t}(\Omega_t,\R{3}).  
	\] 
	With $ u^{t} =u_t \circ T_t $ (see section \ref{Sec: Prel.  Mat. deriv. LinEl}) the definitions of $L^{t}_{u^{t}}$ \eqref{Ex_Shape_Deriv_LinEl:Eq: Weak equation for dot_u_t}, $ L^{t} $ and $ B^{t} $ (see Definition \ref{Ex_Shape_Deriv_LinEl:Def:L^t}, \eqref{Ex_Shape_Deriv_LinEl:Def:B^t}) lead to
	\begin{align*}
	B_{t}(q_t,v)&=\int_{\O{t}} \lambda \,\tr(\varepsilon(q_t))\tr(\varepsilon(v)) +2\mu \,\tr(\varepsilon(q_t)\varepsilon(v)) \, dx = \int_{\Omega_t} \tr(\sigma(q_t)\varepsilon(v))\, dx 
	\intertext{and} 
	L_{u_t,t}(v)
	= &\int_{\O{t}} \langle f_V, v \rangle 
	- \tr\left( \dot{\varepsilon}(u_t )\sigma(v) + \sigma(u_t) \dot{\varepsilon}(v) + \sigma(u_t)\varepsilon(v) \Div(V) \right) \, dx +\int_{\Gamma_{N,t}} \hspace*{-2mm}\langle g_V ,v \rangle \, dS   .
	\end{align*}
	Thus $L_{u_t,t}$ is a linear form on $ H^{1}_{D,t}(\Omega_t,\R{3}) $.
	
	We have shown that $ q_t,\, u_t$ are elements of the Hilbert space $H:=H^{1}_{D,t}(\Omega_t,\R{3}) $. By application of  the divergence theorem \eqref{Ex_Shape_Deriv_LinEl:Eq: divergence thm f. matrix valued functions and vector fields } we can derive the variational formulation
	\begin{align*}
	&~~~~~\int_{\Omega_t} \skp{- \Div(\sigma(q_t))}{v} \, dx &= \int_{\Omega_t}  \tr(\sigma(q_t)\varepsilon(v))\, dx - \int_{\Gamma_{N,t}}  \langle\sigma(q_t)\vec{n}_t,v\rangle \, dS~~~  \forall v \in H
	\end{align*}
	from \eqref{Ex_Shape_Deriv_LinEl:Eq: PDE for q_t}, i.e. $ \forall v \in H $:
	\begin{align*}
	\int_{\Omega_t}\tr(\sigma(q_t)\varepsilon(v)) \, dx =& \int_{\Omega_t} \skp{f_V}{v} \, dx + \int_{\Gamma_{N,t}} \skp{g_V}{v} \, dS \\
	& + \int_{\Omega_t} \skp{f_{u_t}}{v} \, dx  
	- \int_{\Gamma_{N,t}}  \skp{G_{u_t}\vec{n}_t}{v} dS .
	\end{align*}
	It is left to show that for any $ v \in H$
	\begin{align*}
	\int_{\Omega_t}\tr\left( \dot{\varepsilon}(u_t )\sigma(v) + \sigma(u_t) \dot{\varepsilon}(v)  + \sigma(u_t)\varepsilon(v) \Div(V) \right)  dx 
	& = \int_{\Omega_t} \hspace*{-2mm} \skp{-f_{u_t}}{v} \, dx + \int_{\Gamma_{N,t}} \hspace*{-2mm} \skp{G_{u_t}\vec{n}_t}{v}  dS.  
	\end{align*}
	We will do this by integration by parts:
	\begin{align*}
	\int_{\Omega_t} \hspace*{-2mm} \skp{-f_{u_t}}{v} \, dx + \int_{\Gamma_{N,t}} \hspace*{-2mm} \skp{G_{u_t}\vec{n}_t}{v}  dS 
	= &\int_{\Omega_t} \hspace*{-2mm} \skp{- \Div[
		DV\sigma(u)
		+\dot{\sigma}(u_t) + \Div(V)\sigma(u_t)] }{v} \, dx \\ 
	&+  \int_{\Gamma_{N,t}} \hspace*{-2mm}\skp{[
		\sigma(u_t)DV^{\top}
		+\dot{\sigma}(u_t) + \Div(V)\sigma(u_t)]\vec{n}_t}{v}  dS\\
	=& \int_{\Omega_t} \tr\left((DV\sigma(u_t) + \dot{\sigma}(u_t) + \Div(V)\sigma(u_t))Dv \right) \, dx .
	\end{align*}
	Now we compare the integrands: Invariance of the trace w.r.t. transpositions and cyclic permutations and the symmetry of $ \sigma(u_t) $ yield
	$$ \tr(DV\sigma(u_t)Dv) = \tr(\sigma(u_t)DvDV) =  \tr(\sigma(u_t)(DvDV)^{\top}) $$ and thus $$ \tr(\sigma(u_t)\varepsilon(v)) = \frac{1}{2}\left( \tr(\sigma(u_t)DvDV) +  \tr(\sigma(u_t)(DvDV)^{\top})\right) = \tr(DV\sigma(u_t)Dv).  $$
	Now we replace $\dot{\sigma}(u_t) $ by $ \lambda\tr(\dot{\varepsilon}(u_t))\mathrm{I}+2\mu\dot{\varepsilon}(u_t) $  
	and with the analogous arguments we used to calculate the first summand we obtain   
	\begin{align*}
	\tr(\dot{\sigma}(u_t)Dv) 
	&= \lambda \tr(\tr(\dot{\varepsilon}(u_t)Dv)) + 2 \mu\tr( \dot{\varepsilon}(u_t)Dv)
	= \lambda \tr(\dot{\varepsilon}(u_t)) \Div(v) +2\mu\tr(\dot{\varepsilon}(u_t)Dv) \\
	&= \tr( (\lambda \Div(v)\mathrm{I} + 2\mu \varepsilon(v))\dot{\varepsilon}(u_t))
	\end{align*}
	for the second term and $$ \tr(\Div(V) \sigma(u_t)Dv) =  \tr( \Div(V) \sigma(u_t) Dv) = \tr( \Div (V)\sigma(u_t) \varepsilon(v))$$
	for the third. Thus 
	\begin{align*}
	\tr\left[(DV\sigma(u_t) + \dot{\sigma}(u) + \Div(V)\sigma(u_t))Dv \right]
	=\tr\left[ \dot{\varepsilon}(u_t )\sigma(v) + \sigma(u_t) \dot{\varepsilon}(v)  + \sigma(u_t)\varepsilon(v) \Div(V) \right].
	\end{align*}
	
	\noindent  Let $ u_{t} \in  C^{k,\phi}(\overline{\Omega}_{t},\R{3}) $ be the unique solution to displacement traction problem \eqref{Ex_Shape_Deriv_LinEl:Eq: LinEl Omega_t} on  $ \Omega_t $, according to Theorem \ref{Linear_Elasticity:Thm:LinEl_Classical_Sol}. As a consequence, the unique solution of \eqref{Ex_Shape_Deriv_LinEl:Eq: B^t(u^t,v)=L^t(v)} is given by $ u^{t}=u_{t} \circ T_t \in C^{k,\phi}(\overline{ \Omega},\R{3})  $  for any $ t \in I_V $. The same argument applies to the solution $ q_t $ of  \ref{Ex_Shape_Deriv_LinEl:Eq: PDE for q_t} and $ q^{t} =q_t \circ T_t \in C^{k,\phi}(\overline{ \Omega},\R{3}) $ which then uniquely solves \eqref{Ex_Shape_Deriv_LinEl:Eq: Weak equation for dot_u_t}.
	Since the composition operator $ A_t: : H^{1}_{D_t}(\Omega_t,\R{3}) \to H^{1}_{D}(\Omega,\R{3}),\, u \mapsto u \circ T_t $ (compare \eqref{Ex_Shape_Deriv_LinEl:Eq: equivalence of B^t(u^t,v)=L^t(v) and B_t(u_t,v)=L_t(v)}) we obtain:
	\begin{align*}
	\begin{array}{lrcll}
	& B_{t}(q_t,v) &=& L_{t,u_t}(v) & \forall v\in H^{1}_{D,t}(\Omega_t,\R{3})\\
	\Leftrightarrow & \,  B^{t}(q^{t},w) &=& \dot{L}^{t}(w) - \dot{B}^{t}(u^t,w) & \forall w\in H^{1}_{D}(\Omega,\R{3}).
	\end{array}
	\end{align*}  
\end{proof}

\begin{rem}
	Since $ \Gamma $ is "only" of class $ C^{k+1} $ the regularity is not high enough to obtain $ C^{k+1,\phi} $- solutions $ u_t $. Not until $ V \in \Vad{k+2}{\Oext} $, $ \Omega \in \mathcal{O}_{k+2}^{b} $ the solution $ u_t $ is an element of $ C^{k+1,\phi}(\Omega_{t},\R{3}) $ but still $ q_{t} \in C^{k,\phi}(\Omega_{t},\R{3}) $.
\end{rem}

\begin{prop}\label{Ex_Shape_Deriv_LinEl:Prop: properies of solutions q_t of DTP 2 on Omega_t}
	Let $ V\in V^{ad}_{k+1}(\Oext) $, $ f \in C^{k-1,\phi}(\overline{ \Oext},\R{3}) $ and $ g \in C^{k,\phi}(\overline{ \Oext},\R{3})$ for some $ k \geq 2 $.  Let $ \Omega \in \mathcal{O}_{k+1}^{b}$ and $ u_t \in C^{k,\phi}(\overline{\Omega_t},\R{3}) $ be the unique solution of \eqref{Reliability:Eq:LinEl} on $ \Omega_t$. Let $ f_{V,u_t}:=f_V+f_{u_t} $ and $ g_{V,u_t}:= g_{V}-G_{u_t}\vec{n}_{t} $. 
	
	\noindent Then, the unique solution $ q_t \in C^{k,\phi}(\overline{ \Omega_t},\R{3}) $ of \eqref{Ex_Shape_Deriv_LinEl:Eq: PDE for q_t} satisfies the following:
	\begin{itemize}
		\item[i)] The norm of $ q_t $ is bounded according to 
		\[ \Norm{q_t}{C^{k,\phi}(\Omega_t,\R{3})} \leq C_t \left(\Norm{f_{V,u_T}}{C^{k-2,\phi}(\Omega_t,\R{3})} + \Norm{g_{V,u_t}}{C^{k-1,\phi}(\Gamma_t,\R{3})} +\Norm{q_t}{C(\Omega_t,\R{3})}\right) \]
		for some constant $ C_t> 0 $, $ t \in I_V $. If $ t \in (-\epsilon,\epsilon) $ then $ C=C_t $ can be chosen uniformly on the interval $ (-\epsilon,\epsilon) \Subset I_V $.
		\item[ii)]  By means of Lemma \ref{Ex_Shape_Deriv_LinEl:Lem: uniform cone lemma} the term $ \Norm{q_t}{C(\Omega_t,\R{3})} $ can be replaced by $  \Norm{q_t}{L^{1}(\Omega_t,\R{3})}  $ and there even exists a constant $ C_{q}> 0 $ such that 
		\begin{align}\label{Ex_Shape_Deriv_LinEl:Eq: uniform bound u_t}
		\Norm{q_t}{C^{k,\phi}(\Omega_t,\R{3})} \leq C_{q}
		\end{align} 
		uniformly in $ t \in (-\epsilon,\epsilon)$.
	\end{itemize}
\end{prop}

\begin{proof} 
	i) Apply Proposition \ref{Ex_Shape_Deriv_LinEl:Prop: properies of solutions u_t of DTP on Omega_t} to $f= f_V + f_{u_t}  \in C^{k-2,\phi}$ and $ g= g_V - G_{u_t}\vec{n}_{t} \in C^{k-1,\phi} $.
	ii) Let $ t \in (-\epsilon,\epsilon)  $. Then we can derive the following from  i), triangle inequality and Lemma \ref{Ex_Shape_Deriv_LinEl:Lem: uniform cone lemma}: (For the moment we write $ C^{k-2,\phi}(\Omega_t)^3 $ instead of $ C^{k-2,\phi}(\Omega_t,\R{3}) $ to abbreviate the notation)
	\begin{align*}
	\Norm{q_t}{C^{k,\phi}(\Omega_t)^3} 
	\leq & C \left(\Norm{f_V + f_{u_t}}{C^{k-2,\phi}(\Omega_t)^3} + \Norm{g_V - G_{u_t}\vec{n}_t}{C^{k-1,\phi}(\Gamma_t)^3}\right. \left.+\Norm{q_t}{C(\Omega_t)^3}\right)
	\\[1ex]
	\leq & C \left(\Norm{f_V}{C^{k-2,\phi}(\Omega_t)^3} + \Norm{f_{u_t}}{C^{k-2,\phi}(\Omega_t)^3} + \Norm{g_V}{C^{k-1,\phi}(\Gamma_t)^3} 
	\right. 
	\\
	&   +\left.\Norm{G_{u_t}\vec{n}_t}{C^{k-1,\phi}(\Gamma_t)^3}+\delta\Norm{q_t}{C^1(\Omega_t)^3} + C_{\delta}\Norm{q_t}{L^{1}(\Omega_t)^3}\right)
	\\[1ex]
	\leq &   C\left(\Norm{f_{u_t}}{C^{k-2,\phi}(\Omega_t)^3}  + \Norm{G_{u_t}}{C^{k-1,\phi}(\Gamma_t)^{3 \times 3}}\Norm{\vec{n}_t}{C^{k-1,\phi}(\Gamma_t)^3}+\delta\Norm{q_t}{C^1(\Omega_t)^3}\right. \\
	&+ \left. C_{\delta}\Norm{q_t}{L^{1}(\Omega_t)^3}\right) +C_{f,g,V}+ C_g \Norm{ \Div_{\Gamma}(V)}{C^{k-1,\phi}(\Gamma_t)}.
	\end{align*}
	for any $ \delta>0 $ and an aggregated constant $ C> 0 $. In the following we choose $ \delta \in $ $(0, \frac{1}{C})$ to guarantee that $ 1-C\delta>0. $ Since $ \Norm{q_t}{C^1(\Omega_t)^3} \leq \Norm{q_t}{C^{k,\phi}(\Omega_t)^3} $ it follows that  
	\begin{align*}
	(1-C\delta)\Norm{q_t}{C^{k,\phi}(\Omega_t)^3} 
	\leq&  C\left(\Norm{f_{u_t}}{C^{k-2,\phi}(\Omega_t)} + c\Norm{G_{u_t}}{C^{k-1,\phi}(\Gamma_t)^{3 \times 3}}\Norm{\vec{n}_t}{C^{k-1,\phi}(\Gamma_t)^3} \right. \\
	&+ \left. C_{\delta}\Norm{q_t}{L^{1}(\Omega_t)^3}\right) 
	+C_{f,g,V}+C_g \Norm{ \Div_{\Gamma}(V)}{C^{k-1,\phi}(\Gamma_t)}.
	\end{align*}
	Herein $\Norm{q_t}{L^{1}(\Omega_t)}  $ can be replaced by $ \Norm{q_t}{H^{1}(\Omega_t)} $ since we can deduce from Hölder's inequality and $ \sqrt{|\Omega_t|}\leq C\, \forall t \in (-\epsilon,\epsilon) $ that
	\[ \Norm{q_t}{L^{1}(\Omega_t)} \leq \sqrt{\vert \Omega_t\vert} \Norm{q_t}{L^{2}(\Omega_t)} \leq C \Norm{q_t}{H^{1}(\Omega_t)}. \]
	Now we apply the same arguments as in the proof of Lemma \ref{Ex_Shape_Deriv_LinEl:Lem: uniform bounds u_t and u^t in H^1 } and achieve  
	$ \Norm{q^{t}}{H^{1}(\Omega,\R{3})} \leq c \Norm{L^{t}_{u^{t}}}{H^{1}_{D}(\Omega,\R{3})'} \leq \tilde{C} $
	because we proved in Lemma \ref{Lem: continuity L_neu}  that $ \Norm{L^{t}_{u^{t}}}{H^{1}_{D}(\Omega,\R{3})'} \leq C $ for some $ C>0 $ and any $t \in (-\epsilon,\epsilon) $.   Thus $  \Norm{q^{t}}{H^{1}(\Omega,\R{3})} \leq \tilde{C} $ and 
	\begin{align*}
	\Norm{q_t}{C^{k,\phi}(\Omega_t,\R{3})} 
	&\leq C_{f,g,V}+C_g \Norm{ \Div_{\Gamma}(V)}{C^{k,\phi}(\Gamma_t)} \\
	&~~~+ C\left(\Norm{f_{u_t}}{C^{k-2,\phi}(\Omega_t)^{3}} + c\Norm{G_{u_t}}{C^{k-1,\phi}(\Gamma_t)^3}\Norm{\vec{n}_t}{C^{k-1,\phi}(\Gamma_t)^3} + \tilde{C}_{\epsilon}\right).
	\end{align*}
	The the vector field
	$$ f_{u_t}= \Div((DV\sigma(u_t))^{\top} + \Div(V)\sigma(u_t)) + \Div(\dot{\sigma}(u_t)) $$ contains only derivatives of $ u_t$ and $ V $ which are at most of order two.
	
	The first part $ \Div(DV\sigma(u_t)+\Div(V)\sigma(u_t))  $ is a vector field with components consisting of sums over products of first and  second order partial derivatives of $ V $ and $ u_{t} $ and therefore their $ C^{k-2,\phi}(\Omega_t) $-Norms can be estimated by $ C \Norm{V}{C^{k,\phi}(\Omega_t)^3} \Norm{u_t}{C^{k,\phi}(\Omega_t)^3} $ whereas $\Div(\dot{\sigma}(u_t))$ contains only second order partial derivatives of $ u_t $ and thus $$ \Norm{\Div(\dot{\sigma}(u_t))}{C^{k-2,\phi}(\Omega_t)^3} \leq C \Norm{u_t}{C^{k,\phi}(\Omega_t)^3}.$$ Hence, 
	\[ \Norm{f_{u_t}}{C^{k-2,\phi}(\Omega_t)} \leq C\Norm{u_t}{C^{k,\phi}(\Omega_t)^3}(\Norm{V}{C^{k,\phi}(\Omega_t)^3} + 1)\leq CC_{u}(\Norm{V}{C^{k,\phi}(\Oext)^3} + 1)\]
	by Proposition \ref{Ex_Shape_Deriv_LinEl:Prop: properies of solutions u_t of DTP on Omega_t}. The argumentation for the boundedness of $ G_{u_t} $ is analogous.
	The outward normal vector field $ \vec{n}_{t} $ satisfies $$ \Norm{\vec{n}_t}{C^{k-1,\phi}(\Gamma_t)^3} \leq \Norm{\vec{n}_t}{C^{k,\phi}(\Gamma_t)^3}  \leq C$$ according to the proof of Proposition \ref{Ex_Shape_Deriv_LinEl:Prop: strong fomulation of the DTP for q_t}. Therefore and because the derivatives of $ V $ contained in $ \Div_{\Gamma}(V)=\Div(V)-\langle DV\vec{n}_{t},\vec{n}_{t}\rangle $ are only of first order, also $ \Div_{\Gamma}(V) $ is bounded in $ C^{k-2,\phi} $. Hence, $ \Norm{q_t}{C^{k,\phi}(\Omega_t)^3} \leq C_{q} $
	for $ C_{q}> 0 $ depending on $ (-\epsilon,\epsilon)\Subset I_V $.
\end{proof}

\begin{lem}\label{Ex_Shape_Deriv_LinEl:Lem: uniform bounds q^t}
	Let $ k\geq 2 $,  $ \Omega \in \mathcal{O}_{k+1}^{b} $, $ \Omega_{t}=\T{t}(\Omega), \, t \in I_V  $ for some admissible vector field $V \in  \mathcal{V}_{k+1}^{ad}(\Oext)$. Suppose that  $ f\in C^{k-1,\phi}(\overline{\Oext},\R{3}) $ and $ g\in C^{k,\phi}(\overline{\Oext},\R{3})$ for some $ \phi\in (0,1) $ and let $ u_{t} \in C^{k,\phi}(\O{t},\R{3}) $ be the unique solution of \eqref{Reliability:Eq:LinEl} on $ \O{t} $. Moreover let $ q_t \in C^{k,\phi}(\overline{ \Omega}_{t},\R{3}) $ be the unique solution of \eqref{Ex_Shape_Deriv_LinEl:Eq: PDE for q_t} on $ \Omega_t $.
	Under these assumptions the mapping $$ t \in (-\epsilon,\epsilon) \mapsto u^{t}=u_t \circ T_t \in C^{k,\phi}(\Omega,\R{3}) $$ is bounded in $ C^{k,\phi} $ and the mapping  $$ t \in  (-\epsilon,\epsilon)\to q^{t}=q_{t} \circ T_t \in C^{k,\phi}(\Omega,\R{3})  $$ is uniformly bounded in $ C^{k,\phi}(\Omega,\R{3}) $ for any interval $ (-\epsilon,\epsilon) \Subset I_{V} $.
\end{lem}

\begin{proof}
	In the proof of Proposition \ref{Ex_Shape_Deriv_LinEl:Prop: properies of solutions q_t of DTP 2 on Omega_t} we showed that there exists a constant $ C_{q}> 0 $ independent of $ t \in 
	(-\epsilon,\epsilon) $ such that $$ \Norm{q_t}{C^{k,\phi}} \leq C_{q} $$ for any $ t \in (-\epsilon,\epsilon)$. Thus, the analogous arguments as used in the proof of Lemma \ref{Ex_Shape_Deriv_LinEl:Lem: uniform bounds ut} lead to
	\[ \Norm{q^{t}}{C^{k,\phi}(\Omega,\R{3})} = \Norm{q_{t} \circ \T{t}}{C^{k,\phi}(\Omega,\R{3})} \leq C \Norm{q_{t}}{C^{k,\phi}(\O{t},\R{3})}\leq C C_{q}=\tilde{C} \] 
	where $ \tilde{C} $ can be chosen uniformly in $ t \in (-\epsilon,\epsilon)$.
\end{proof}

\begin{thm}[H\"older Material Derivatives for Linear Elasticity]\label{Ex_Shape_Deriv_LinEl:Thm: q^t C^3,phi material derivative of u_t} $  $
	\\
	Let the assumptions of Lemma \ref{Ex_Shape_Deriv_LinEl:Lem: uniform bounds q^t} be satisfied. 
	Then, $ q^t=q_t \circ T_t \in C^{k,\phi}(\overline{ \Omega_t},\R{3}) $ is the strong $ C^{k,\varphi} $- material derivative of the unique solution $ u_t \in C^{k,\phi}(\overline{ \Omega_t},\R{3}) $ of \eqref{Ex_Shape_Deriv_LinEl:Eq: LinEl Omega_t} for any $ 0<\varphi<\phi $ and any $ t \in (-\epsilon,\epsilon) \Subset I_{V} $. At $ t=0 $ it satisfies the partial differential equation
	\begin{align}\label{Ex_Shape_Deriv_LinEl:Eq: PDE for dot_u} \tag{P2}
	\left. 
	\begin{array}{rcll}
	-\Div( \sigma(q)) &=&f_{V}+f_{u}   &\text{ in } \Omega \\
	q &=& 0  &\text{ on } \Gamma_{D}  \\
	\se(q) \vec{n} &=&g_{V}-G_{u}\vec{n} &\text{ on }\Gamma_{N}.
	\end{array}
	\right.
	\end{align}
	Further, the mapping $ t \in (-\epsilon,\epsilon) \to q^{t}=\dot{u}^{t}\in C^{k,\phi}(\overline{ \Omega_{t}},\R{3)} $ is continuous w.r.t. the strong $C^{k,\varphi} $-topology. 
\end{thm}

\begin{proof}
	We have to retrace the conditions $ i) -iii)$ of Theorem \ref{Parameter_Dep_PDE:Thm: Existence of Strong Derivatives u^t }:
	We choose $X_{1}=H^{1}_{D}(\Omega,\R{3})  $, $ X_{2}=C^{k,\varphi}(\Omega,\R{3}) $, $ X_{3}=C^{k,\phi}(\Omega,\R{3}) $, the associated strong norm topologies and  $0\leq\varphi<\phi<1 $. 
	\begin{itemize}
		\item[i)] Proposition \ref{Ex_Shape_Deriv_LinEl:Prop: properies of solutions u_t of DTP on Omega_t} shows the existence of unique $ C^{k,\phi} $-solutions $ u_t $ of \eqref{Reliability:Eq:LinEl} on $ \Omega_t $ and $ u^{t}=u_t \circ T_t $ solves \eqref{Ex_Shape_Deriv_LinEl:Eq: B^t(u^t,v)=L^t(v)}, see \eqref{Ex_Shape_Deriv_LinEl:Eq: equivalence of B^t(u^t,v)=L^t(v) and B_t(u_t,v)=L_t(v)} or Proposition \ref{Ex_Shape_Deriv_LinEl:Prop: Link between u_t ant u^t}.
		
		In Theorem \ref{Ex_Shape_Deriv_LinEl:Thm: existence of strong H^1 material derivatives} we have shown that $(-\epsilon,\epsilon) \to  C^{k,\phi}(\overline{ \Omega},\R{3}) $ is differentiable regarding the strong  $ H^{1}_{D}(\Omega,\R{3}) $-topology such that the solution $ q^t $ of \eqref{Ex_Shape_Deriv_LinEl:Eq: Weak equation for dot_u_t}  is the strong $ H^{1} $ material derivative of $ u_t $ at any $ t \in (-\epsilon,\epsilon)$.
		\item[ii)]In Proposition \ref{Ex_Shape_Deriv_LinEl:Prop: properies of solutions q_t of DTP 2 on Omega_t} we proved the existence of unique solutions $ q_{t} \in C^{k,\phi} $ of 
		\eqref{Ex_Shape_Deriv_LinEl:Eq: PDE for q_t}
		such that $ q^{t}=q_t \circ T_t $ uniquely solves the associated weak formulation \eqref{Ex_Shape_Deriv_LinEl:Eq: Weak equation for dot_u_t} (see the proof of Proposition \ref{Ex_Shape_Deriv_LinEl:Prop: properies of solutions q_t of DTP 2 on Omega_t}). Moreover, Theorem \ref{Ex_Shape_Deriv_LinEl:Thm: existence of strong H^1 material derivatives} assures that $ t \in (-\epsilon,\epsilon) \to q^{t}=\dot{u}^{t}\in C^{k,\phi} $ is continuous regarding the strong $ H^{1}_{D}(\Omega,\R{3})  $-topology. 
		\item[iii)] Finally $ t \to q^t $ and $ t \to u^t $ are uniformly bounded in $ C^{k,\phi'}(\Omega,\R{3}) $, see Lemma \ref{Ex_Shape_Deriv_LinEl:Lem: uniform bounds ut} and Lemma \ref{Ex_Shape_Deriv_LinEl:Lem: uniform bounds q^t}, and finally
		the argumentation in Theorem \ref{Linear_Elasticity:Thm:LinEl_Classical_Sol} (see also \cite{GilbTrud}) shows that the unit sphere (and the image of $ q^{t} $) in $ C^{k,\phi'} $ is relatively compact in $ C^{k,\phi} $. 
	\end{itemize}
	Thus $ (-\epsilon,\epsilon) \to C^{k,\phi}(\Omega,\R{3}), \,t \mapsto u^{t} $ and $  (-\epsilon,\epsilon) \to C^{k,\phi}(\Omega,\R{3}), \,t \mapsto q^{t} $ together with $X_{1}=H^{1}_{D}(\Omega,\R{3})  $ , $ X_{2}=C^{k,\varphi}(\Omega,\R{3}) $, $ X_{3}=C^{k,\phi}(\Omega,\R{3}) $ satisfy the hypotheses of  Theorem \ref{Parameter_Dep_PDE:Thm: Existence of Strong Derivatives u^t } and thus $ \dot{u}^{t} =q^{t} $ w.r.t. the norm topology on $ C^{k,\varphi}(\Omega,\R{3}) $. Additionally, Theorem  \ref{Parameter_Dep_PDE:Thm: Existence of Strong Derivatives u^t } shows the continuity of $ t \in (-\epsilon,\epsilon) \to \dot{u}^{t} \in C^{k,\phi}(\Omega,\R{3}) $ w.r.t. the strong norm topology on $ X_2= C^{k,\varphi}(\Omega,\R{3}) $.
\end{proof}

\begin{cor}\label{Ex_Shape_Deriv_LinEl:Cor: Fréchet diff u^t}
	Let the hypothesis of Theorem \ref{Ex_Shape_Deriv_LinEl:Thm: q^t C^3,phi material derivative of u_t} be satisfied. Then the mapping $ t \in (-\epsilon,\epsilon) \to u^{t}=u_t \circ T_t \in C^{k,\phi}(\overline{ \Omega},\R{3}) $ is Gâteaux differentiable in $ C^{k,\varphi} $ for any subinterval $ (-\epsilon,\epsilon) \Subset  I_{V} $ and the mapping $ t \in (-\epsilon,\epsilon) \to \dot{u}^{t}\in C^{k,\phi}(\overline{\Omega},\R{3})  $ is continuous. This means that $ t \in (-\epsilon,\epsilon) \to u^{t}\in C^{k,\phi}(\overline{\Omega},\R{3}) $ is Fréchet differentiable in $ C^{k,\varphi}(\overline{\Omega},\R{3})  $.
\end{cor}


\pagebreak

\begin{rem} The regularity of the boundary restricts the regularity of the PDE solution in a natural way. Thus,
	if $ \Omega \in \mathcal{O}_{k+2}^{b} $ is assumed in 
	\ref{Ex_Shape_Deriv_LinEl:Thm: q^t C^3,phi material derivative of u_t} instead of $ \Omega \in \mathcal{O}_{k+1}^{b} $, then $ u_{t} \in C^{k+1,\phi}(\overline{\Omega}_{t},\R{3}) $, but still $ q_{t}\in C^{k,\phi}(\overline{\Omega}_{t},\R{3}) $. Hence, it is also natural that the material derivative of a second order elliptic PDE under consideration has one degree less regularity then the solution itself. 
\end{rem}

\section{Local shape derivatives in H\"older spaces}

Besides the existence of material derivatives in Hölder spaces we can also prove existence of local shape derivatives in Hölder spaces provided that the boundary of $ \Omega $ and the input data $ f $ and $ g $ are smooth enough: 

\begin{thm}\label{Ex_Shape_Deriv_LinEl:Thm: Shape derivatives for linear elasticity}
	Suppose that $ k\geq 2 $, $ \Omega \in \mathcal{O}_{k+1}^{b} $, $ V\in V^{ad}_{k+1}(\Oext) $, $ f \in C^{k-1,\phi}(\overline{\Omega},\R{3}) $ and $ g \in C^{k,\phi}(\Gamma,\R{3})  $ such that $ f'(\Omega;V) \in C^{k-2,\phi}(\overline{\Omega},\R{3}) $ and $ g'(\Gamma;V)\in C^{k-1,\phi}(\Gamma,\R{3}) $. Let $ u(\Omega)\in C^{k,\phi}(\overline{\Omega},\R{3}) $ be the unique solution of \eqref{Reliability:Eq:LinEl}. Then the (local) shape derivative $ u'(\Omega;V) = \dot{u}(\Omega;V) - DuDV $ is an element of $ C^{k-1,\phi}(\overline{ \Omega},\R{3}) $.
	
	\begin{itemize}
		\item[i)] If $ k\geq 3 $, then $ u'(\Omega;V) $ satisfies 	
		\begin{align}\label{Ex_Shape_Deriv_LinEl:Eq: PDE for u' g(Gamma), f(Omega)} \tag{P3-0}
		\left. 
		\begin{array}{rcll}
		\Div( \se(u'(\Omega;V))) &=&f'(\Omega;V)   &\text{ in } \Omega \\
		u'(\Omega;V) &=& -V_{\vec{n}} Du(\Omega)\, \vec{n}  &\text{ on } \Gamma_{D} \\
		\se(u'(\Omega;V)) \,\vec{n} &=&(f(\Omega)+\kappa g(\Gamma))V_{\vec{n}} +g'(\Gamma;V) &\text{ on }\Gamma_{N}\\
		& &~~~ + \Div_{\Gamma}(V_{\vec{n}}\sigma_{\Gamma}(u(\Omega)))
		\end{array} 
		\right.
		\end{align}
		where $ \sigma_{\Gamma}(u)=\sigma(u) -  \sigma(u)\vec{n}\vec{n}^{\top}$
		is the tangential proportion of $ \sigma(u) $ and $ V_{\vec{n}}= \langle V,\vec{n} \rangle $. 
		
		\noindent If $ k=2 $, then $ u' \in C^{1,\phi}(\Omega,\R{3})$ is a weak solution of \eqref{Ex_Shape_Deriv_LinEl:Eq: PDE for u' g(Gamma), f(Omega)} but still satisfies 
		\begin{align*}
		\left. 
		\begin{array}{rcll}
		\se(u'(\Omega;V)) \,\vec{n} &=&(f(\Omega)+\kappa g(\Gamma))V_{\vec{n}} +g'(\Gamma;V) + \Div_{\partial\Omega}(V_{\vec{n}}\sigma_{\Gamma}(u(\Omega))) &\text{ on }\Gamma_{N}\\
		u'(\Omega;V) &=& -V_{\vec{n}} Du(\Omega)\, \vec{n}  &\text{ on } \Gamma_{D}
		\end{array} \right.
		\end{align*}
	\end{itemize}
\end{thm}

\begin{proof}
	Recall the displacement-traction problem \eqref{Reliability:Eq:LinEl}
	with it's unique $ C^{k,\phi} $-solution $ u $. Taking an arbitrary function $ v \in C^{\infty}( \Omega) $  such that $ v\vert_{\Gamma_D}=0 $ the weak reformulation of \eqref{Reliability:Eq:LinEl}  is given by \eqref{Linear_Elasticity:Eq: B(u,v)=L(v)}.
	We already know that the shape derivative $ u' $ exists in $ C^{k,\phi} $ and thus by Lemma \ref{Transf_shape_opt:Lem:Reynolds Theorem Surface} and Lemma 
	\ref{Transf_shape_opt:Lem:Reynolds Theorem Volume}
	\begin{align*}
	\int_{\Omega} (\sigma(u):\varepsilon(v))'& \, dx + \int_{\Gamma_{N}}\sigma(u):\varepsilon(v) V_{\vec{n}} \, dS  \\  
	= & \int_{\Omega} \langle f, v\rangle' \, dx + \int_{\Gamma_{N}}\langle  f, v\rangle V_{\vec{n}}  \, dS  
	+ \int_{\Gamma_{N}} \langle g, v\rangle' \, dS + \int_{\Gamma_{N}}  \langle \kappa g, v\rangle  V_{\vec{n}}\, dS .	\end{align*} 
	Since $ v=v(\Omega):=v\vert_{\overline{\Omega}} $, we conclude that $v'=0 $ on $ \overline{\Omega} $, see Lemma \ref{Transf_shape_opt:Lem: Shape derivatives C^l} . Hence, Lemma \ref{Transf_shape_opt:Lem: Product and Chain rule for shape derivatives} implies
	\begin{align*}
	\int_{\Omega} (\sigma(u):\varepsilon(v))' \, dx  
	= \int_{\Omega} \langle f', v\rangle \, dx + \int_{\Gamma_{N}} \hspace*{-1ex} \langle  f V_{\vec{n}}+g'+\kappa gV_{\vec{n}}, v\rangle -\sigma(u):\varepsilon(v) V_{\vec{n}} \, dS \,. 
	\end{align*} 
	By Lemma \ref{Transf_shape_opt:Lem: Properties u'} and $ v'=0 $  iii) follows
	$ (\sigma(u):\varepsilon(v))'   =   \sigma(u'):\varepsilon(v)$
	and hence
	\begin{align*}
	\int_{\Omega}  \sigma(u'):\varepsilon(v) \, dx 
	= \int_{\Omega} \langle f', v\rangle \, dx + \int_{\Gamma_{N}} \hspace*{-1ex} \langle  f V_{\vec{n}}+g'+\kappa gV_{\vec{n}}, v\rangle -\tr(\sigma(u)\varepsilon(v)) V_{\vec{n}} \, dS .  
	\end{align*} 
	Now assume $ Dv\, \vec{n}=0 $ on $ \Gamma_N $. As a direct consequence we can replace $ Dv $ by $ D_{\Gamma}v $ on $ \Gamma_D $ and write $ \sigma(u):\varepsilon(v) = \sigma(u):Dv= \sigma(u):D_{\Gamma}v$. This allows us to integrate the term $ \sigma(u):\varepsilon(v) =\tr(\sigma(u)\varepsilon(u)) $ by parts on $ \Gamma_N $, i.e.
	\begin{align*}
	\int_{\Gamma_N}  \sigma(u):\varepsilon(v) V_{\vec{n}} \, dS 
	&= \int_{\Gamma_N}  \tr(V_{\vec{n}}\sigma(u)D_{\Gamma}v) \, dS 
	= \int_{\Gamma_N} \hspace*{-2mm}- \Div_{\Gamma}(V_{\vec{n}} \sigma(u))v + \kappa \langle V_{\vec{n}} \sigma(u)\vec{n} , v \rangle  \, dS.
	\end{align*}
	Thus,
	\begin{align*}
	\int_{\Omega}   \sigma(u'):\varepsilon(v)\, dx 
	=  & \int_{\Omega} \langle f', v\rangle \, dx + \int_{\Gamma_{N}} \hspace*{-1ex} \langle  f V_{\vec{n}}+g'+\kappa gV_{\vec{n}}, v\rangle -    \kappa\langle V_{\vec{n}} \sigma(u)\vec{n} , v \rangle  dS \\
	&+\int_{\Gamma_N}  \Div_{\Gamma}(V_{\vec{n}}  \sigma(u))v\, dS
	\\
	= &  \int_{\Omega} \langle f', v\rangle \, dx + \int_{\Gamma_{N}} \hspace*{-1ex} \langle  f V_{\vec{n}}+g'+\kappa gV_{\vec{n}}, v\rangle-    \kappa\langle V_{\vec{n}} \sigma(u)\vec{n} , v \rangle  dS \\  
	&+ \int_{\Gamma_N}   \langle \Div_{\Gamma}(V_{\vec{n}}  \sigma_{\Gamma}(u)) + \kappa  V_{\vec{n}} \sigma(u) \vec{n} , v \rangle  dS  
	\\
	= &  \int_{\Gamma_{N}} \hspace*{-1ex} \langle  f V_{\vec{n}}+g'+\kappa gV_{\vec{n}}, v\rangle+ \langle\Div_{\Gamma}(V_{\vec{n}}  \sigma_{\Gamma}(u)), v \rangle dS  +\int_{\Omega} \langle f', v\rangle \, dx
	\end{align*} 
	or equivalently 
	\begin{align*}
	\int_{\Omega}  &-\Div(\sigma(u'))v \, dx + \int_{\Gamma_N} \langle \sigma(u') \vec{n},v\rangle \, dS  \\
	&=   \int_{\Gamma_{N}} \hspace*{-1ex} \langle  f V_{\vec{n}}+g'+\kappa gV_{\vec{n}}, v\rangle\, dS + \int_{\Gamma_{N}} \hspace*{-1ex} \langle\Div_{\partial\Omega}(V_{\vec{n}}  \sigma_{\Gamma}(u)), v \rangle \, dS + \int_{\Omega} \langle f', v\rangle \, dx.
	\end{align*}
	Since we assume $ v=0  $ on $ \Gamma_{D}$ we have to consider the equations $$ u_t=0 \text{ on } \Gamma_{D,t} $$ seperately: We set \[ J(\Omega):= \int_{\Gamma_D} \langle u,v\rangle \, dS =0\] and choose an arbitrary $ v \in C^{\infty}_{0}(\R{n},\R{n}) $ with $ \frac{\partial v}{\partial \vec{n}}=0 $ on $ \Gamma_D $. With $ y(\Gamma_D):= \langle u,v\rangle'\vert_{\Gamma_D}$ and $ u=0 $ on $ \Gamma_D $ we can deduce the following identity from 
	Lemma \ref{Transf_shape_opt:Lem:Reynolds Theorem Volume} and \ref{Transf_shape_opt:Lem: Product and Chain rule for shape derivatives}:
	\begin{align*}
	\frac{d}{dt}J(\Omega_t)\vert_{t=0}
	0&=	\int_{\Gamma_D} \langle u,v\rangle'\vert_{\Gamma_D}+ \left(\frac{\partial }{\partial \vec{n}} (uv) + \kappa uv\right) V_{\vec{n}}  \, dS  \\
	&=	\int_{\Gamma_D} \langle u',v\rangle + \langle u,v'\rangle + \left(\langle Du \, \vec{n},v \rangle + \langle Dv \,\vec{n},u \rangle+\kappa uv \right)V_{\vec{n}} \, dS\\
	&=	\int_{\Gamma_D} \langle u',v\rangle + \left(\langle Du \, \vec{n},v \rangle 
	+\kappa uv \right)V_{\vec{n}} \, dS\\
	&= \int_{\Gamma_D} \langle u',v\rangle + \langle Du \, \vec{n},v \rangle V_{\vec{n}}\, dS .
	\end{align*} 
	Thus, in strong formulation
	$$ u' =-V_{\vec{n}} \, Du \, \vec{n}   ~~~   \text{ on } \Gamma_D.$$
	Thus $ u' $ is a weak solution and an element of $ C^{k-1,\phi}(\Omega,\R{3}) $, $ k-1 \geq 2 $ which implies that $ u' $ already is a strong solution of \eqref{Ex_Shape_Deriv_LinEl:Eq: PDE for u' g(Gamma), f(Omega)}.
	
	\noindent For $ k=2 $ the assertion follows from the fundamental lemma of variational calculus.
\end{proof}

\begin{rem}
	i) $u'(\Omega;V)=-V_{\vec{n}} \, Du \, \vec{n}$ on $ \Gamma_D $ follows also directly from Lemma	\ref{Transf_shape_opt:Lem: Shape derivatives C^l} ii) with $ z(\Gamma_D):=u(\Omega)\vert_{\Gamma_D} $. Then 
	
	\[ z'(\Gamma;V)=u'(\Omega;V)\vert_{\Gamma_D} + \frac{\partial}{\partial \vec{n}}u(\Omega)V_{\vec{n}}=0 \text{ on } \Gamma_{D}.\]
	ii) $ f'(\Omega;V)=\Div(\sigma(u))'(\Omega;V)=\Div(\sigma(u'(\Omega;V))) $ also follows from Lemma \ref{Transf_shape_opt:Lem: Properties u'}.
\end{rem}

\begin{cor}\label{Ex_Shape_Deriv_LinEl:Cor: Shape derivatives for linear elasticity}
	Suppose that $ k\in \N{},\, k\geq 3 $, $ V\in V^{ad}_{k+1}(\Oext) $, $ f \in C^{k-1,\phi}(\overline{ \Oext}) $ and $ g \in C^{k,\phi}(\overline{ \Oext})  $. Let $ u(\Omega)\in C^{k,\phi}(\overline{\Omega},\R{3}) $,  $ \Omega \in \mathcal{O}_{k+1}^{b} $ be the family of unique solutions of \eqref{Reliability:Eq:LinEl}. Then, the shape derivative $ u'=u'(\Omega;V) $ exists and is an element of $ C^{k-1,\phi}(\overline{ \Omega},\R{3}) $. Moreover $ u' $ satisfies 	
	\begin{align} \label{Ex_Shape_Deriv_LinEl:Eq: PDE u'} \tag{P3}
	\left. 
	\begin{array}{rcll}
	\Div( \se(u')) &=&0   &\text{ in } \Omega \\
	u' &=& -V_{\vec{n}} Du\, \vec{n}  &\text{ on } \Gamma_{D} \\
	\se(u') \, \vec{n} &=&(f+\kappa g +Dg\vec{n})V_{\vec{n}} +\Div_{\Gamma}(V_{\vec{n}}\sigma_{\Gamma}(u)) &\text{ on }\Gamma_{N}.
	\end{array} 
	\right.
	\end{align}
	
	\noindent If $ k=2 $, then $ u' \in C^{1,\phi}(\Omega,\R{3})$ is a weak solution of (P3)  and still satisfies $ \se(u') \, \vec{n} =(f+\kappa g +Dg\vec{n})V_{\vec{n}} +\Div_{\Gamma}(V_{\vec{n}}\sigma_{\Gamma}(u)) $ on $ \Gamma_{n} $ and $ u'=  -V_{\vec{n}} Du\, \vec{n}$ on $ \Gamma_{D} $.
\end{cor}

\begin{proof}
	The first statement follows directly from Lemma \ref{Ex_Shape_Deriv_LinEl:Thm: q^t C^3,phi material derivative of u_t} by $ u'(\Omega;V)=\dot{u}(\Omega;V)-Du(\Omega)V $.
	
	Let $ f(\Omega)=f\vert_{\overline{ \Omega}} $ and $ g(\Omega)=g\vert_{\overline{ \Omega}}$, $ g(\Gamma)=g\vert_{\Gamma} =g(\Omega)\vert_{\Gamma} $. From Lemma \ref{Transf_shape_opt:Lem: Properties u'} we conclude $f'(\Omega;V)=0 $ and  $$ g'(\Gamma;V)= g'(\Omega)\vert_{\Gamma} +V_{\vec{n}} Dg(\Omega)\vert_{\Gamma}\, \vec{n} = V_{\vec{n}}Dg(\Omega)\vert_{\Gamma}\, \vec{n}
	$$
	If the boundary shall be clamped at the part $ \Gamma_D $ of the boundary, then $ V_{\vec{n}} $ has to be zero there and thus $ -V_{\vec{n}} Du(\Omega)\, \vec{n}=0  $. Replacing the respective terms in equation \eqref{Ex_Shape_Deriv_LinEl:Eq: PDE for u' g(Gamma), f(Omega)} thus directly leads to equation \eqref{Ex_Shape_Deriv_LinEl:Eq: PDE u'}.
\end{proof}

\begin{rem}
	For a different derivation of this equation see \cite{EpplerElasticity} equations (9) and (14) where the following description is given, which is equivalent to equation \eqref{Ex_Shape_Deriv_LinEl:Eq: PDE u'}:
	\begin{align}
	\left. 
	\begin{array}{r c l l}
	\Div( \se(u')) &= & 0   &\text{ in } \Omega \\
	u'& = & - V_{\vec{n}} Du\, \vec{n}  &\text{ on } \Gamma_{D} \\
	\se(u')  \vec{n}& = & [Dg \,\vec{n} - D(\sigma(u))[\vec{n}]\vec{n} - \Div_{\Gamma}(\sigma_{\Gamma}(u))]V_{\vec{n}}  &\text{ on }\Gamma_{N}.\\
	& & +\Div_{\Gamma}(V_{\vec{n}}\sigma_{\Gamma}(u)) & 
	\end{array} .
	\right.
	\end{align}
	This formulation can be obtained using $f =-\Div(\sigma(u)),\, g=\sigma(u)\vec{n} $ and
	\begin{align*}
	\Div_{\Gamma}(\sigma_{\Gamma}(u))=\Div(\sigma(u)) - \kappa \sigma(u)\vec{n} - D(\sigma(u))[\vec{n}]\vec{n}\,.
	\end{align*}
	This implies (compare \cite[Prop. 2.68]{SokZol92} and \cite[Lemma 7]{EpplerElasticity})
	\begin{align*}
	\Div_{\Gamma}(\sigma_{\Gamma}(u)) = -f-\kappa g - D(\sigma(u))[\vec{n}]\vec{n} 
	\Leftrightarrow  \Div_{\Gamma}(\sigma_{\Gamma}(u)) + D(\sigma(u))[\vec{n}]\vec{n} = -f-
	\kappa g \, .
	\end{align*}
	Another possible representation for the Neumann boundary condition can be obtained by the product rule for the tangential divergence which implies
	\begin{align*}
	[Dg \,\vec{n} - D(\sigma(u))[\vec{n}]\vec{n} - \Div_{\Gamma}(\sigma_{\Gamma}(u))]V_{\vec{n}}  +&\Div_{\Gamma}(V_{\vec{n}}\sigma_{\Gamma}(u))\\
	&=[Dg \vec{n} - D(\sigma(u))[\vec{n}]\vec{n}]V_{\vec{n}} + \sigma_{\Gamma}(u) \nabla_{\Gamma}V_{\vec{n}}\,. 
	\end{align*}

	Moreover, the PDE for $ \dot{u} $  and $ u' $ are consistent. 
	Especially, 
	\begin{align*}
	&u'(\Omega;V)=-Du(\Omega) \vec{n} V_{\vec{n}} && \text{on } \Gamma_D \\
	\Leftrightarrow  ~~~ &\dot{u}(\Omega;V) - Du(\Omega)V  =-Du(\Omega) \vec{n} V_{\vec{n}} &&  \text{on } \Gamma_D \\
	\Leftrightarrow ~~~  &\dot{u}(\Omega;V) = Du(\Omega) V - Du \vec{n} \vec{n}^\top V = D_{\Gamma}u V && \text{on } \Gamma_D.
	\end{align*}
	But, since $ u(\Omega) $ is constant along $ \Gamma_D$ the tangential derivative is $ 0 $ on $ \Gamma_{D} $ which implies \[ \dot{u}(\Omega;V) = D_{\Gamma}u(\Omega) V = 0 \text{ on } \Gamma_D.\]
\end{rem}

\section{Shape derivatives for local cost functionals w.r.t. linear elasticity}\label{Ex_Shape_Deriv_LinEl:Eq: Shape Deriv Jvol Jsur}

In the following we will turn our attention to a whole class of functionals to which the LCF and the Ceramic reliability functional belong:\\[1ex]
For $ u\in W^{1,m}(\Omega,\R{3})$, where $ m $ defends on the material properties of the ceramic material, we have 
\begin{equation}\label{Ex_Shape_Deriv_LinEl:Eq: Def Jcer}
J^{\mathrm{cer}}(\Omega):=  \frac{1}{4 \pi} \int_{\Omega} \int_{S^2} \left(\frac{\langle\sigma(u)\mathfrak{n},\mathfrak{n}\rangle^{+}}{\sigma_{0}}\right)^m \, dS_{\mathbb{S}^{2}} \, dx,
\end{equation}
where $ \mathfrak{n} \in \mathbb{S}^2 $ is a normal direction, see Chapter 1 and \cite{GBS_Ceramic2014}, for
$ u\in C^{1,\phi}(\overline{\Omega},\R{3}) $ or $ u\in C^{1,\phi}(\Gamma,\R{3}) $, $ 1\geq \phi \geq 1-\frac{1}{4.7m} $ we define 
\begin{align}\label{Ex_Shape_Deriv_LinEl:Eq: Def Jlcf}
J^{\mathrm{lcf}}(\Omega):= \int_{\Gamma} \left(\frac{1}{N_{det}(\sigma(u))}\right)^m \, dS,
\end{align}
see Chapter \ref{Reliability}.

\noindent This class of functionals can be defined according to equation \eqref{Transf_shape_opt:Eq: Local Integral Cost Functional}, since $ \sigma(u)=\lambda\tr(Du)I + \mu(Du+Du^{^\top}) $, i.e.
\begin{align*}
J(\Omega) &= J_{vol}(\Omega)+J_{sur}(\Omega) \\
&=\int_{\Omega} \mathcal{F}_{vol}(.,u,\sigma(u))\, dx + \int_{\Gamma} \mathcal{F}_{sur}(.,u,\sigma(u))\, dS \\
&=\int_{\Omega} \tilde{\mathcal{F}}_{vol}(.,u,Du)\, dx + \int_{\Gamma} \tilde{\mathcal{F}}_{sur}(.,u,Du)\, dS
\end{align*}
where 
\begin{align*}
\tilde{\mathcal{F}}_{sur}&=\mathcal{F}_{sur} \circ L_{\sigma}\\
\tilde{\mathcal{F}}_{vol}&=\mathcal{F}_{vol} \circ L_{\sigma},
\end{align*}
and
\begin{align*}
L_{\sigma}:\R{3} \times \R{3} \times \R{3 \times 3} &\to \R{3} \times \R{3} \times \R{3 \times 3}, \\
(x,y,M)~~~~~ &\mapsto (x,y,\lambda\tr(M)I + \mu(M+M^{\top})).
\end{align*}
Therefore,
\[ 
J^{\mathrm{cer}}(\Omega)=\int_{\Omega} \mathcal{F}^{\mathrm{cer}}(\sigma(u)) \, dx,~~~~
\mathcal{F}^{\mathrm{cer}}(\sigma(u))= \frac{1}{4 \pi}\int_{S^2} \left(\frac{\langle \sigma(u)\mathfrak{n},\mathfrak{n}\rangle^{+}}{\sigma_{0}}\right)^m \, dS_{\mathbb{S}^2} 
\]
and
\[ 
J^{\mathrm{lcf}}(\Omega)=\int_{\Gamma} \mathcal{F}^{\mathrm{lcf}}(\sigma(u)) \, dS ,~~~~
\mathcal{F}^{\mathrm{lcf}}(\sigma(u)) = \left(\frac{1}{N_{det}(\sigma(u))}\right)^m.
\]

In a first step, we will calculate Euler derivatives for this class of functionals under appropriate conditions which naturally follow the conditions of Lemma \ref{Transf_shape_opt:Lem: d/dt int J(T_t,u_t,Du_t) in shape derivative form}. Therefore, recall the PDE of linear elasticity \eqref{Reliability:Eq:LinEl} and the PDE
determining its shape derivative $ u' $ \eqref{Ex_Shape_Deriv_LinEl:Eq: PDE u'}.
If $ V_{\vec{n}} =0$ on $ \Gamma_D $ then the Dirichlet boundary condition becomes $ u'=0 $ on $ \Gamma_D .$

\begin{prop}\label{Shape_Grad_LinEl:Prop: Shape gradient material derivative form}
	Let $ k\geq 2 $, $ \Omega \in \mathcal{O}_{k+1}^{b} $, $ \Omega_{t}=\T{t}(\Omega), \, t \in I_V  $ for some admissible vector field $V \in  \mathcal{V}_{k+1}^{ad}(\Oext)$ such that $ V_{\vec{n}}=0 $ on $ \Gamma_D $. Suppose that  $ f\in C^{k-1,\phi}(\overline{\Oext},\R{3}) $ and \linebreak $ g\in C^{k,\phi}(\overline{\Oext},\R{3})$ for some $ \phi\in (0,1) $. 
	Let $ u_{t} \in C^{k,\phi}(\O{t},\R{3}) $ be the unique solution of \eqref{Reliability:Eq:LinEl} on $ \O{t} $ and $ \dot{u}_t \in C^{k,\phi} $ be the unique solution of \eqref{Ex_Shape_Deriv_LinEl:Eq: PDE for dot_u}. Moreover let $ \mathcal{F} \in C^{1}(\R{3} \times \R{3} \times \R{3 \times 3}) $.
	
	\noindent Then the shape derivative of
	\begin{align}\label{Shape_Grad_LinEl:Eq: def JOmega}
	J(\Omega) =\int_{\Omega} \mathcal{F}_{vol}(.,u,\sigma(u))\, dx + \int_{\Gamma} \mathcal{F}_{sur}(.,u,\sigma(u))\, dS.
	\end{align}
	is given by $  $
	\begin{align}
	dJ(\Omega)[V]
	=&\, \int_{\Omega}
	\Div(V)(x) \mathcal{F}_{vol}(x,u(x),\sigma(u(x))) +\left\langle \frac{\partial \mathcal{F}_{vol}}{\partial z_1} (x,u(x),\sigma(u(x))) ,V(x) \right\rangle \, dx \nonumber  \\
	&+  \int_{\Omega} \left\langle\frac{\partial \mathcal{F}_{vol}}{\partial z_2} (x,u(x),\sigma(u(x))) ,\dot{u}(x) \right\rangle \, dx\nonumber \\
	&+  \int_{\Omega} \frac{\partial \mathcal{F}_{vol}}{\partial z_3} (x,u(x),\sigma(u(x)) : [\sigma(\dot{u}(x)) - (Du(x)DV(x))^{\sigma}] \, dx \\
	&+ \int_{\Gamma}
	\Div_{\Gamma}(V)(x) \mathcal{F}_{sur}(x,u(x),\sigma(u(x)
	) \, dS \nonumber \\
	&+ \int_{\Gamma} \left\langle \frac{\partial \mathcal{F}_{sur}}{\partial z_1} (x,u(x),Du(x)) ,V(x) \right\rangle + \left\langle\frac{\partial \mathcal{F}_{sur}}{\partial z_2} (x,u(x),Du(x)) ,\dot{u}(x) \right\rangle  dS\nonumber \\
	&+  \int_{\Gamma} \frac{\partial \mathcal{F}_{sur}}{\partial z_3} (x,u(x),\sigma(u(x)) : [\sigma(\dot{u}(x)) - (Du(x)DV(x))^{\sigma}] dS. \nonumber 
	\end{align} 
\end{prop}

\begin{proof}
	Apply Theorem \ref{Ex_Shape_Deriv_LinEl:Thm: q^t C^3,phi material derivative of u_t} and Lemma \ref{Transf_shape_opt:Lem: d/dt int(T_t,u_t,Du_t) in material derivative form} to 	
	\begin{align*}
	\mathcal{ J}:t \in \tilde{I} \mapsto J(\Omega_t) =\int_{\Omega} \tilde{\mathcal{F}}_{vol}(.,u,Du)\, dx + \int_{\Gamma} \tilde{\mathcal{F}}_{sur}(.,u,Du)\, dS.
	\end{align*}
	where $
	\tilde{\mathcal{F}}_{sur}=\mathcal{F}_{sur} \circ L_{\sigma}$ and $  \tilde{\mathcal{F}}_{vol}=\mathcal{F}_{vol} \circ L_{\sigma}.
	$
\end{proof}

\begin{prop}\label{Shape_Grad_LinEl:Prop: Shape Deriv shape derivative form} Let the hypotheses of Proposition \ref{Shape_Grad_LinEl:Prop: Shape gradient material derivative form} be given. Then  
	the local shape derivative $ u' $  is an element of $ C^{k-1,\phi}(\overline{\Omega},\R{3}) $,\, $ k-1 \geq 1 $. Let $J(\Omega)$ be defined according to \eqref{Shape_Grad_LinEl:Eq: def JOmega}.
	
	\noindent Then the shape derivative of $J$ 
	is given by 
	\begin{align*}
	dJ(\Omega)[V]=&\,  \int_{\Omega} \left\langle\tfrac{\partial \mathcal{F}_{vol}}{\partial z_2} (.,u,\sigma(u)) ,u' \right\rangle + \tfrac{\partial \mathcal{F}_{vol}}{\partial z_3} (.,u,\sigma(u)):\sigma(u') \, dx
	\\
	&+  \int_{\Gamma} \mathcal{F}_{vol} (.,u,\sigma(u))V_{\vec{n}} \, dS 
	\\
	&+  \int_{\Gamma} \left\langle\tfrac{\partial \mathcal{F}_{sur}}{\partial z_1} (.,u,\sigma(u)), \vec{n} V_{\vec{n}}   \right\rangle + \kappa \mathcal{F}_{sur}(.,u,\sigma(u))) V_{\vec{n}} \, dS
	\\
	&+  \int_{\Gamma} \left\langle\tfrac{\partial \mathcal{F}_{sur}}{\partial z_2} (.,u,\sigma(u)), u' +  Du\, \vec{n} V_{\vec{n}}   \right\rangle \, dS
	\\
	&+  \int_{\Gamma} \tfrac{\partial \mathcal{F}_{sur}}{\partial z_3} (.,u,\sigma(u)
	) :(\sigma(u')+ D(\sigma(u))[\vec{n}] V_{\vec{n}})  \, dS.
	\end{align*} 
\end{prop}

\begin{proof} Apply Theorem \ref{Ex_Shape_Deriv_LinEl:Thm: q^t C^3,phi material derivative of u_t} and Lemma \ref{Transf_shape_opt:Lem: d/dt int J(T_t,u_t,Du_t) in shape derivative form} to 
	\begin{align*}
	\mathcal{ J}:t \in \tilde{I} \mapsto J(\Omega_t) =\int_{\Omega} \tilde{\mathcal{F}}_{vol}(.,u,Du)\, dx + \int_{\Gamma} \tilde{\mathcal{F}}_{sur}(.,u,Du)\, dS
	\end{align*}
	where $
	\tilde{\mathcal{F}}_{sur}=\mathcal{F}_{sur} \circ L_{\sigma}$ and $  \tilde{\mathcal{F}}_{vol}=\mathcal{F}_{vol} \circ L_{\sigma}.
	$
\end{proof}

\begin{rem}
	Note that the results shown in Lemma \ref{Transf_shape_opt:Lem: d/dt int J(T_t,u_t,Du_t) in shape derivative form}
	and Proposition \ref{Shape_Grad_LinEl:Prop: Shape gradient material derivative form} can be extended analogously to shape functionals of $ l $-th order for $ l \geq 2 $. It is also straight forward to show that these functionals are shape differentiable when constraint of the minimization problem is given by a linear elasticity equation.
\end{rem}

Now we show that the LCF functional and the Ceramic reliability functional satisfy the differentiability requirements of the previous Propositions:

\begin{lem}\label{Shape_Grad_LinEl:Prop: Differentiability LCF}
	The mapping 
	\begin{align*}
	\R{3 \times 3} &\to \R{+}_{0},\,~~~~~
	\sigma \mapsto \mathcal{F}^{\mathrm{lcf}}(\sigma) = \frac{1}{N_{det}(\sigma)^{m}} 
	\end{align*}
	with the parameters $c<b<0,\,0<  \hat{n} \leq \frac{1}{7}$, $ - 4\frac{2b}{m}<1 $ and $ \hat{\sigma}_f,\, \hat{\epsilon}_{f},\, E,\, K >0 $ (see Section \ref{Reliability:Sec:Simulation})
	is at least four times continuously differentiable on $ \R{3 \times 3} \setminus \ker(TF) $ and the first derivative can be continuously extended to $ \ker(TF) $. The first derivative is given by 
	\[ \frac{\partial \mathcal{F}^{\mathrm{lcf}}}{\partial \sigma} (\sigma_0) = (\tilde{CMB}^{-1} \circ \tilde{RO} \circ \tilde{SD}^{-1})^{(1)}(VM^{2} \circ  TF(\sigma_0)) \cdot 3 TF(\sigma_0).  \]
	where
	\begin{align*}
	\tilde{SD}: \R{+}_{0} \to \R{+}_{0},\,&  x \mapsto  x +\frac{E}{K^{\frac{1}{\hat{n}}}} x^{\frac{\hat{n}+1}{2\hat{n}}} 
	\\[1ex]
	\tilde{RO}: \R{+}_{0} \to \R{+}_{0},\, & x \mapsto \frac{x}{E^2} + \frac{2}{EK^{\frac{1}{\hat{n}}}} x^{\frac{\hat{n}+1}{2\hat{n}}} + \frac{1}{K^{\frac{2}{\hat{n}}}} x^{\frac{1}{\hat{n}}} 
	\\[1ex]
	\tilde{CMB}: \R{+}_{0} \to \R{+}_{0},\, & x \mapsto \left(\frac{2^{b}\sigma_{f}'}{E}\right)^2  x^{-\frac{2b}{m}} +  \frac{\sigma_{f}'\epsilon_{f}'2^{1+b+c}}{E} x^{-\frac{b+c}{m}} + 2^{2c}\epsilon_{f}' x^{-\frac{2c}{m}}.
	\end{align*} 
\end{lem}

\begin{proof}
	Let us first recall the equations and the definition of $ N_{det} $:
	\begin{align*}
	\left.
	\begin{array}{l l}
	\hat{\sigma}= TF(\sigma)=\displaystyle \sigma-\frac{1}{3}\tr(\sigma)\mathrm{I}, & TF : \R{3 \times 3} \to \R{3 \times 3} \\
	\displaystyle \sigma_v= VM(\hat{\sigma})=\sqrt{\frac{3}{2} \hat{\sigma}: \hat{\sigma}}, &  VM: \R{3 \times 3} \to \R{+}_{0}  \\
	\displaystyle \sigma_v=SD(\sigma^{el-pl})=\sqrt{( \sigma^{el-pl})^2+\frac{E}{K^{\nicefrac{1}{\hat{n}}}} \left(\sigma^{el-pl}\right)^{1+1/\hat{n}}}, & SD: \R{+}_{0} \to \R{+}_{0}  
	\end{array} \right.\\
	\left.
	\begin{array}{l l}
	\displaystyle 
	\varepsilon^{el-pl}=RO(\sigma^{el-pl})=\frac{ \sigma^{el-pl}}{E}+\left(\frac{\sigma^{el-pl}}{K}\right)^{1/\hat{n}}, & RO:\R{+}_{0} \to \R{+}_{0}  \\
	\displaystyle\varepsilon^{el-pl}=CMB(N_{det})=\frac{\sigma_f'}{E}(2N_{det})^b+\varepsilon_f'(2N_{det})^c, & CMB:\R{+} \to \R{+}.
	\end{array} \right.
	\end{align*}
	
	\noindent Noting that $ \hat{\sigma} = 0 $ if and only of $\sigma_{v} =0$
	it is obvious that $\mathcal{F}^{\mathrm{lcf}} $ is continuously differentiable in case of $ \hat{\sigma}\neq 0  $ since then $ \mathcal{F}^{\mathrm{lcf}} $ is a composition of continuously differentiable functions. Thus, we only have to care about the case when $ \hat{\sigma} = 0 $ since $ VM $ is not differentiable at $ \hat{\sigma}=0 $.  We can manage this situation rewriting the function such that it becomes quadratic in $ VM $:
	
	\noindent We define
	\begin{align*}
	\tilde{SD}: \R{+}_{0} \to \R{+}_{0}, &&  
	\tilde{RO}: \R{+}_{0} \to \R{+}_{0},\, && 
	\tilde{CMB}: \R{+}_{0} \to \R{+}_{0}
	\end{align*}   
	as above. All these functions are strictly monotonically increasing, continuous on $ \R{+} $ and thus bijective. Moreover, 
	\begin{align*}
	\sigma_{v}^{2} = ( \sigma^{el-pl})^2+\frac{E}{K^{\frac{1}{\hat{n}}}} \left(\frac{(\sigma^{el-pl})^{2}}{K}\right)^{\frac{\hat{n}+1}{2\hat{n}}} \vspace*{-5mm}=  \tilde{SD}((\sigma^{el-pl})^2) \Leftrightarrow  \tilde{SD}^{-1}(\sigma_v^2) = &(\sigma^{el-pl})^2, \\[1ex]
	RO(\sigma^{el-pl})^2  =  \left(\frac{\sigma^{el-pl}}{E}\right)^2 + \frac{2}{EK^{\frac{1}{\hat{n}}}} (\sigma^{el-pl})^{\frac{(\hat{n}+1)}{\hat{n}}} + \frac{1}{K^{\frac{2}{\hat{n}}}} (\sigma^{el-pl})^{\frac{2}{\hat{n}}}   =\tilde{RO}&((\sigma^{el-pl})^2)
	\intertext{ and }
	(\varepsilon^{el-pl})^2 = CMB(N_{det}(\sigma))^2 = \tilde{CMB}\left(\frac{1}{N_{det}(\sigma)^m}\right) \Leftrightarrow \frac{1}{N_{det}(\sigma)^m}  = \tilde{CMB}^{-1}&((\varepsilon^{el-pl})^2).
	\intertext{Therefore $ (\varepsilon^{el-pl})^2 =RO(\sigma^{el-pl})^2$ and $ \sigma_{v}^2 =  (VM^{2} \circ TF)(\sigma) $  implies}
	\frac{1}{N_{det}(\sigma)^m} = (\tilde{CMB}^{-1} \circ \tilde{RO} \circ \tilde{SD}^{-1})(\sigma_v^2) =( \tilde{CMB}^{-1} \circ \tilde{RO} \circ \tilde{SD}^{-1})( ( VM^{2}& \circ TF)(\sigma)).
	\end{align*}
	The mappings $ VM^{2} $ and $ TF $ are everywhere (especially at any $ M $ with $ TF(M)=0 $) continuously differentiable and therefore we only have to show that $  \tilde{CMB}^{-1} \circ \tilde{RO} \circ \tilde{SD}^{-1} $ can be extended to a continuously differentiable function at $ \sigma_v^2=x= 0 $. We argue by the rule on differentiability of inverse functions: For notational simplicity we denote the $ k $-th order derivative of a function $ f:\R{} \to \R{} $ by $ f^{(k)}(x) $.  
	
	We have $ \tilde{SD}(0) = 0 $ and $ \tilde{SD} $ is differentiable at $ 0 $ with derivative $ SD^{(1)}(0) = 1 $ which leads to
	$$ (SD^{-1})^{(1)}(0) = \frac{1}{\tilde{SD}^{(1)}(\tilde{SD}^{-1}(0))} = 1$$
	and the continuity of $  \tilde{SD}^{(1)} $ at $ x=0 $ and $ \tilde{SD}^{-1} $ at $y= 0 $ imply that $ \tilde{SD}^{-1} $ is continuously differentiable at $y= 0 $. It is obvious that $ \tilde{RO} \in C^{1}(\R{+}_{0},\R{+}_{0}) $ with $ \tilde{RO}^{(1)}(0) = \frac{1}{E^2} $. Therefore we only have to show that $ \tilde{CMB}^{-1} $ satisfies the requirements:
	$ \tilde{CMB} $ is continuous and bijective with $ \tilde{CMB}(0)=0 $. Moreover it is continuously differentiable on $ \R{+} $ and satisfies $ \lim_{x \to 0^{+}} \tilde{CMB}^{(1)}(x) = \infty $. Thus $$ \left(\tilde{CMB}^{-1}\right)^{(1)}(y) = \frac{1}{\tilde{CMB}^{(1)}(\tilde{CMB}^{-1}(y))} $$ is continuous on $ \R{+} $ and since $ \lim_{y \to 0^{+}} \tilde{CMB}^{-1}(y)=0 $ we have
	\[ \lim_{y \to 0^{+}}  \left(\tilde{CMB}^{-1}\right)^{(1)}(y) =\lim_{y \to 0^{+}} \frac{1}{\tilde{CMB}^{(1)} (\tilde{CMB}^{-1}(y))} =0.  
	\] 
	Thus we define $ \left(\tilde{CMB}^{-1}\right)^{(1)}(0) :=0. $
	
	\noindent For the higher order derivatives we precede analogously but we only carry out the case of second order derivatives in detail. We use 
	\begin{align*}
	(f^{-1})^{(2)}(y) &=- \frac{f^{(2)}(x)}{f^{(1)}(x)^3}\\[1em]
	(f^{-1})^{(3)}(y) &= \frac{f^{(3)}(x)}{f^{(1)}(x)^{4}} - \frac{3 f^{(2)}(x)^2}{f^{(1)}(x)^5} \\[1em]
	(f^{-1})^{(4)}(y) &= \frac{f^{(4)}(x)}{f^{(1)}(x)^5} + \frac{2 f^{(2)}(x) f^{(3)}(x)}{f^{(1)}(x)^6} - \frac{15 f^{(2)}(x)^3}{f^{(1)}(x)^7}
	\end{align*}
	with $ x = f^{-1}(y) $. 
	For $ SD^{-1} $ we observe that   
	\[ \left(\tilde{SD}^{-1}\right)^{(2)}(0) = \left(\tilde{SD}^{-1}\right)^{(3)}(0) =  \left(\tilde{SD}^{-1}\right)^{(4)}(0)= 0  \]
	since $ \tilde{SD}^{(1)}(0)=1 $ and the higher order derivatives (up to degree $  k=\lfloor \frac{\hat{n} +1}{2\hat{n}}\rfloor  $) of $ \tilde{RO} $ at $ \tilde{SD}^{-1}(0)=0 $ are also 0. 
	Hence, when calculating the second derivative of $ \tilde{CMB}^{-1} \circ \tilde{RO} \circ \tilde{SD}^{-1} $  at $ \sigma_v^2=0 $ by the chain rule, all terms despite of $ \frac{1}{E^2}\left(\tilde{CMB}^{-1}\right)^{(2)} $ vanish and we only have to consider 
	$$ \lim_{y \to 0^{+}}\left(\tilde{CMB}^{-1}\right)^{(2)}(y) = \lim_{y \to 0^{+}} \frac{\tilde{CMB}^{(2)}(\tilde{CMB}^{-1}(y))}{ \tilde{CMB}^{(1)}(\tilde{CMB}^{-1}(y))^3 } = \lim_{x \to 0^{+}}  \frac{\tilde{CMB}^{(2)}(x)}{ \tilde{CMB}^{(1)}(x))^3 } .$$ This is due to the continuity of $ \tilde{CMB}^{-1}(y) $ at $ y=0^{+} $.  The $ k $-th order derivative of $ \tilde{CMB} $ is given by 
	\[ \tilde{CMB^{(k)}(x)} = \frac{c_1^{k} + c_2^{k}x^{\frac{b-c}{m}} + c_3^{k}x^{\frac{2b-2c}{m}}}{x^{\frac{2b}{m}+k}} ,\, x>0\]
	for some constants $ c_i^{k} $. This leads to 
	\begin{align*}
	\frac{\tilde{CMB}^{(2)}(x)}{ \tilde{CMB}^{(1)}(x))^3 }
	=
	\frac{\left(c_1^{2} + c_2^{2}x^{\frac{b-c}{m}} + c_3^{2}x^{\frac{2b-2c}{m}}\right)}{\left(c_1^{1} + c_2^{1}x^{\frac{b-c}{m}} + c_3^{1}x^{\frac{2b-2c}{m}}\right)^{3}} \frac{x^{3(\frac{2b}{m}+1)}}{x^{\frac{2b}{m}+2}}   
	=
	\frac{\overbrace{\left(c_1^{2} + c_2^{2}x^{\frac{b-c}{m}} + c_3^{2}x^{\frac{2b-2c}{m}}\right)}^{ \to c_1^{2}}}{\underbrace{\left(c_1^{1} + c_2^{1}x^{\frac{b-c}{m}} + c_3^{1}x^{\frac{2b-2c}{m}}\right)^{3}}_{\to c_{1}^{1}}} \underbrace{x^{\frac{4b}{m} +1}}_{\to 0} 
	\end{align*} 
	because $ c<b<0 $ implies that $ b+3c< 2b + 2c < 3b+c <0 $ and thus $ \frac{b+3c}{m} + k< \frac{2b + 2c}{m} + k < \frac{3b+c}{m}  + k $. Moreover, we use $ \frac{4b}{m} +1 >0 $.
	Thus $ (\tilde{CMB}^{-1})^{(2)}(0):=0 $. We repeat these arguments to deduce the assertion for higher order derivatives. 
	
	\noindent To calculate the first Gâteaux-derivative, we apply chain rule making use of $ D\tr(\sigma)[M] = \tr(M)  $ and $ D (\sigma:\sigma)[M] =2\sigma:M $ for any $ \sigma,\, M \in \R{3 \times 3} $.  Computing $$ (VM^{2} \circ TF)(\sigma) =\frac{3}{2} \tr(\sigma^2) - \frac{1}{2}\tr(\sigma)^2 $$ thus leads to $ D[VM^2 \circ TF](\sigma)[M] = 3(\sigma - \frac{1}{3}\tr(\sigma)I):M = 3TF(\sigma):M. $
\end{proof}


\begin{lem}\label{Shape_Grad_LinEl:Prop: Differentiability Ceramic}
	Let $ \mathfrak{n} \in  \mathbb{S}^{2} $ and $ m>1 $. The functional 
	\[ \R{3 \times 3} \to \R{}, ~~~~~ \sigma \mapsto \mathcal{F}^{\mathrm{cer}}(\sigma) = \int_{\mathbb{S}^2} \left(\frac{\sigma^{+}_{n} }{\sigma_c}\right)^m \, dS_{\mathbb{S}^{2}}  ~~~\text{ with } \sigma^{+}_{n} =\max\{ 0 , \sigma_n \},\, \sigma_n:= \mathfrak{n}^{\top} \sigma \mathfrak{n} \] is $ \lfloor m \rfloor $ times differentiable w.r.t. $ \sigma \in \R{3 \times 3} $.
	
	The first derivative is given by 
	\[ \frac{\partial \mathcal{F}^{\mathrm{cer}}}{\partial \sigma}(\sigma_0) 
	:= \begin{cases}
	\int_{\mathbb{S}^2} m (\mathfrak{n}^{\top} \sigma_0 \mathfrak{n})^{m-1} \,\mathfrak{n} \mathfrak{n}^{\top} \, dS_{\mathbb{S}^2}& \text{ if }  \mathfrak{n}^{\top} \sigma_{0} \mathfrak{n} >0 \\ 
	0 & \text{ if } \mathfrak{n}^{\top} \sigma_{0} \mathfrak{n} \leq 0. 
	\end{cases} \]  
\end{lem}

\begin{proof} We have to show that  $f^{\mathrm{cer}}(\sigma) = \left(\frac{\sigma^{+}_{n} }{\sigma_c}\right)^m$ is continuously differentiable,  since then the order of the Gâteaux derivative and the integral over $ \mathbb{S}^2 $ can be changed.
	
	We set $ f(x):=\max \{0,x\}^{m} $, $ x \in \R{} $ and $ F(\sigma)= \mathfrak{n}^{\top} \sigma \mathfrak{n} $, $ \sigma \in \R{3 \times 3} $. Then 
	$ f^{\mathrm{cer}}(\sigma) = \sigma_c^{-m} f(\mathfrak{n}^{\top} \sigma \mathfrak{n}) = f \circ F (\sigma)$. The mapping, $ F:  \R{3 \times 3} \to \R{},\, \sigma \mapsto \mathfrak{n}^{\top} \sigma \mathfrak{n} $ is linear and continuous, consider also Example \ref{Diff_Banach_Space:Exmp:Diff_Gradient}. Thus $ F $ is Fréchet differentiable with $ DF(\sigma_0)[\sigma] = \mathfrak{n}^{\top} \sigma \mathfrak{n} =( \mathfrak{n}^{\top} \mathfrak{n} ): \sigma $, i.e. $\frac{\partial F}{\partial \sigma_{ij}}(\sigma) = \mathfrak{n}\mathfrak{n}^{\top} $. Regarding the function  $ f $ we only have to consider the case $ x =0 $ since $ f^{(1)}(x) = 0 $ on $ (-\infty,0) $ and  $ f^{(1)}(x) = mx^{m-1} $ on $ (0, \infty) $. But, because $ m >1 $, we observe that  $\lim_{h \to 0} \frac{f(0+h) - f(0)}{h} = \lim_{h \to 0} h^{m-1} =0 $
	is the derivative if $ f $ at $ x=0 $. Thus $ f^{(1)} $ is also continuously differentiable at $ x=0 $. Then, the assertion follows by chain rule. 
\end{proof}

\begin{prop}\label{Shape_Grad_LinEl:Prop: Shape Deriv Llcf Jcer}
	Let $ \Omega $ be of class $C^{3} $, $ V \in C^{3}_{0}(\Oext,\R{3}) $, $ f \in C^{1,\phi}(\overline{\Oext},\R{3}) $ and $ g \in C^{2,\phi}(\overline{\Oext},\R{3}) $ for some $ 0 < \phi<1 $. \begin{itemize}
		\item[i)] If $ m>1 $, then $ J^{\mathrm{cer}} $ is shape differentiable.
		\item[ii)] Suppose that the assumptions of Lemma \ref{Shape_Grad_LinEl:Prop: Differentiability LCF} is satisfied. Then also $ J^{\mathrm{lcf}} $ are shape differentiable.
	\end{itemize}
\end{prop}

\begin{proof}
	Under the given assumptions Theorem \ref{Ex_Shape_Deriv_LinEl:Thm: q^t C^3,phi material derivative of u_t} is applicable and the matiaterial derivative $ \dot{u} $ exists in $ C^{2,\phi}(\overline{\Omega},\R{3}) $ w.r.t. the $ C^{2,\varphi} $ topology $ 0 \leq \varphi <\phi $ and $ u' \in C^{1,\phi}(\overline{\Omega},\R{3}) $ according to Corollary \ref{Ex_Shape_Deriv_LinEl:Cor: Shape derivatives for linear elasticity}. Since Lemma \ref{Shape_Grad_LinEl:Prop: Differentiability LCF} and Lemma \ref{Shape_Grad_LinEl:Prop: Differentiability Ceramic} show that the necessary differentiability requirements are satisfied we can apply Proposition \ref{Shape_Grad_LinEl:Prop: Shape gradient material derivative form} and Proposition \ref{Shape_Grad_LinEl:Prop: Shape Deriv shape derivative form}. This implies the assertion.
\end{proof}


\chapter{Hadamard Shape Derivative and Adjoint Equations}\label{Shape_Grad_LinEl}

Let us suppose that $ J $ is a shape differentiable shape functional on some set $ \mathcal{O} $.
Depending on the regularity of the shape and other input data, it is possible to derive a Hadamard decomposition 
\[ dJ(\Omega)[V] = \int_{\Gamma} G(\Gamma)V_{\vec{n}} \, dS =  \int_{\Gamma} \langle G(\Gamma) \vec{n}, V \rangle\, dS \] 
for some function $ G(\Gamma):\Gamma \to \R{} $. If this is not the case, we can still derive the volume representation also known as distributed or weak formulation of the shape derivative \cite{Schmidt2018weak,SturmLaurain2016distributed,DissKW}. 

In the case of smooth domains, the shape space $$ B_{e} := Diff^\infty(\mathbb{S}^{1},\R{2})/ Diff^\infty(\mathbb{S}^{1},\mathbb{S}^{1}) $$ \cite{Schulz2014riemannian} can be considered \cite{Michor2003ShapeManifold,Michor2007Metrics,MichorBauer2014constructing} as a manifold of shapes . In this sense, $ \Gamma $ is the image of an embedding $ e: \mathbb{S}^{1}\to\R{}$ and the intrinsic structure of the shape manifold $ B_{e} $ provides metrics which allow to define a shape gradient w.r.t. a chosen metric, consider \cite{MichorBauer2014constructing,Schulz2016Steklov,Schulz2016metricsComparison}. This concept can be extended to embeddings from $ \mathbb{S}^2 \to \R{3} $ \cite{Michor2003ShapeManifold} and allows to construct gradients w.r.t. the chosen metric on the manifold $ B_e $.

Unfortunately, $ Diff^k(\mathbb{S}^{2},\R{3})/ Diff^k(\mathbb{S}^{2},\mathbb{S}^{2}) $ is no manifold \cite{Michor2003ShapeManifold} since $ Diff^k(\mathbb{S}^{2},\mathbb{S}^{2}) $ is no Lie-Group and thus we have to consider other possibilities to generate descent directions. 
Section \ref{Transf_shape_opt:Sec:L2 decent directions} shows that, under certain circumstances, it seems to be possible to take the $ L^2 $-direction $ - g \vec{n} = W $ as a decent direction in a optimization scheme. If there is no Hadamard decomposition available, then the approach proposed by \cite{SturmLaurain2016distributed} can be applied to find a suitable decent direction. 

In this chapter we discuss under which conditions a Hadamard decomposition 
\[ dJ(\Omega)[V] = \int_{\Gamma} G(\Gamma)V_{\vec{n}} \, dS =  \int_{\Gamma} \langle G(\Gamma) \vec{n}, V \rangle\, dS \] 
can be derived for a general problem
\begin{align*}
\min_{\Omega \in \mathcal{O}_{k}}~~~&  J(\Omega,u,\sigma(u)):= \int_{\Omega} \mathcal{ F}_{vol}(x,u,\sigma(u)) \, dx + \int_{\Gamma} \mathcal{ F}_{sur}(x,u,\sigma(u)) \, dS  \\
\text{ s.t. }~~~ & u = u(\Omega) \text{ solves } \eqref{Reliability:Eq:LinEl} \text{  on } \Omega
\end{align*}
and in which cases only the distributed shape derivative exists. Especially the functionals $ J^{\mathrm{lcf}}(\Omega,u,\sigma(u))$ and 
$ J^{\mathrm{cer}}(\Omega,u,\sigma(u))$, defined in \eqref{Ex_Shape_Deriv_LinEl:Eq: Def Jcer} and \eqref{Ex_Shape_Deriv_LinEl:Eq: Def Jlcf} are taken into account here. Further, we discuss which regularity the direction $ W(\Gamma) = - G(\Gamma) \vec{n} $ actually provides. This step is crucial in the treatment of descent-flows or gradient-flows \cite[Sec. 6.3]{Sturm2015shape} which are defined as the solution $ \Phi_t,\, t \in [0,\epsilon) $ (provided that one exists in some Banach space) of the differential equation
\begin{align}\label{Shape_Grad_LinEl: Gradient Flow Equation}
\frac{d}{dt} \Phi_t = - W(\Gamma_t) \circ \Phi_t  \text{ on } \Gamma
\end{align}
where $ \Gamma $ is the boundary of the initial domain $ \Omega $. Consider \cite[Sec. 6.3]{Sturm2015shape} for a rigorous definition. We will see, as many other authors before \cite{Schulz2016Steklov,Schulz2016metricsComparison,Schulz2015PdeShapeManifolds,SturmLaurain2016distributed,DissKW}, that case $ W(\Gamma) = - g(\Gamma) \vec{n} $ is not a good choice in this context. 

In any case, the so called adjoint approach \cite{DelfZol11,DissKW,ShapeOpt,Troltzsch2005optimale,SturmLaurain2016distributed,Sturm2015shape} is indispensable since finding a descent direction means solving a variational problem
\begin{align}\label{Shape_Grad_LinEl: VarProblem Descent Dir}
\mathscr{B}(W,V) = dJ(\Omega)[V] ~~~\forall V\in \tilde{H}
\end{align}
like \eqref{Transf_shape_opt:Eq:L2 decent direction Eq}.
Here, $ \tilde{H} $ is some suitable Hilbert space and $ \mathscr{B} $ a bilinear form which is reasonable for the application of the Lax-Milgram Theorem, also consider \cite{SturmLaurain2016distributed}.
Supposed that $ dJ(\Omega)[V] $ is given in the form of \eqref{Transf_shape_opt:Eq: d/dt int(T_t,u_t,Du_t) in material derivative form} or \eqref{Transf_shape_opt:Eq: d/dt int J(T_t,u_t,Du_t) in shape derivative form}, the material derivative $ \dot{u} = \dot{u}(\Omega;V) $ or the local shape derivative $ u'=u'(\Omega;V) = \dot{u}(\Omega;V) - DuV $ has to be calculated for all $ V \in C^{k}_{0}(\Oext,\R{3}) $. This is very expensive in a numerical scheme \cite{ShapeOpt} and thus the adjoint method has been developed to avoid this step.
The idea of this approach is the following:

\noindent We assume that the shape derivative takes the form \[dJ(\Omega)[V] = l_1(V) + l_2(u'(\Omega;V) ),~~~ V\in \tilde{H}\] where $ l_1 \in \tilde{H}' $  and $ l_2 \in H' $ for some second Hilbert space $ H $. Further we assume that the local shape derivative $ u'(\Omega;V) \in H $ solves a variational equation 
\begin{equation}\label{Shape_Grad_LinEl:Eq u'}
b(u',v) = l^{\text{shape}}(v) \, \forall v \in H.
\end{equation}
where $ b \in \mathcal{B}(H) $ and $  l^{\text{shape}}= l^{\text{shape}}_V $ is a linear form on $ H $ may depend on $ V $.
Here, this equation corresponds to equation \eqref{Ex_Shape_Deriv_LinEl:Eq: PDE u'}. 
\pagebreak

\noindent The \textit{adjoint equation} is then defined by 
\begin{equation}
b(\vartheta,p) = l_2(\vartheta), \, \forall \vartheta \in \tilde{H}.
\end{equation}
and the unique solution  $ p \in \tilde{H}  $, supposed that it exists, is called \textit{adjoint state}.
Finally the shape derivative can be calculated by \[dJ(\Omega)[V] = l_1(V) + l^{\text{shape}}_{V}(p),\]
where $ p \in \tilde{H} $ no longer depends on $ V \in H $. If $ l^{\text{shape}}  $ depends linearly on $ V \in \tilde{H} $ then equation \eqref{Shape_Grad_LinEl: VarProblem Descent Dir} can be solved using the Lax-Milgram Theorem. This approach can be applied analogously to the material derivative $ \dot{u}(\Omega;V) $.

Alternatively to the direct approach used in this work, also a Lagrangian approach and the theorem of Correa and Seeger can be applied to show shape differentiability, consider \cite{Cea1986Lagrange,Ito2008variational,Sturm2013lagrange,Sturm2015shape,Delfour1991Velocity}, to calculate the adjoint equation and the shape derivative. But, it is not yet clear if this method is applicable for the special class of shape functionals which are discussed in this work, and the answer to this question is left open as a task for future investigations.

Numerically, as well the \textit{discretize-then optimize} approach as the \textit{optimize - then discretize} approach can be used to compute shape derivatives. In case of the first ansatz, all objects (PDE, objective functional) are first discretized and then a discrete adjoint approach is applied \cite{GottschSaadi2018,NumShapeCer2017,giannakoglou2008adjoint,brezillon2009aerodynamic}.
Otherwise, the adjoint equation is derived from the original PDE in function spaces and then implemented in a numerical scheme. We follow the second approach here. 

%

\section{$ L^2 $-Hadamard decomposition}\label{Shape_Grad_LinEl:Sec: Hadamard Shape Gradients}
We will now present an adjoint approach for the derivation of $ L^2$-shape gradient based on a method that was proposed in the scope of a "computational guide"  \cite{Eppler07} for shape optimization problems with general local cost functionals of first order and Poisson or Poisson-type equations as state equation. This approach can be transferred to linear elasticity equation as PDE constraint, but the derivation of surface representations becomes significantly more difficult, especially in the case of surface integral functionals.  

We will now give adjoint equations to both cases - general local volume and surface cost functionals of first order with linear elasticity as state equation. The regularity of the adjoint states and the $ L^2(\Gamma) $-shape gradient $ G(\Gamma) $, see \eqref{Transf_shape_opt:Eq:L2 decent direction Eq},  will be analyzed and specified. Since a distinction of all possible cases would be too extensive we concentrate on those which are the most important in the scope of this thesis. 
\noindent We motivate our proceeding by the following formal calculations which will be substantiated in the curse of this Section.  In the case a functional of the type 
\begin{equation}\label{Shape_Grad_LinEl:Eq:Jvol}
J_{vol}(\Omega) =\int_{\Omega} \mathcal{F}_{vol}(.,u,Du)\, dx
\end{equation} 
is given, the so called adjoint equation in weak form can be defined straightly setting
\begin{align}\label{Shape_Grad_LinEl:Eq:Weak Adj Jvol}
\int_{\Omega}\vspace*{-2mm} \sigma(\vartheta):\varepsilon(p) \, dx = \int_{\Omega} \left\langle\frac{\partial \mathcal{F}_{vol}}{\partial z_2}(.,u,Du),\vartheta \right\rangle + \frac{\partial \mathcal{F}_{vol}}{\partial z_3}(.,u,Du):D\vartheta\, dx\, \forall \vartheta \in H^1_{D}.
\end{align}
If $ \mathcal{F}_{vol} $ is continuously differentiable and $ u $  are regular enough, such that $ p \in H^1_{D}(\Omega,\R{3}) $ is the uniquely determined adjoint state. Moreover we assume that $ u' $ is a weak solution of \eqref{Ex_Shape_Deriv_LinEl:Eq: PDE u'} in $ H^1_{D}(\Omega,\R{3}) $. Then
\begin{equation} \label{Shape_Grad_LinEl:Eq: Shape derivative volume - replace by BC u'}
\begin{split}
&\int_{\Omega} \left\langle\frac{\partial \mathcal{F}_{vol}}{\partial z_2}(.,u,Du),u' \right\rangle + \frac{\partial \mathcal{F}_{vol}}{\partial z_3}(.,u,Du):Du'\, dx  \\
&= \int_{\Omega}  \sigma(u'):\varepsilon(p) \, dx =  \int_{\Gamma_{N}}\left\langle(f+\kappa g +Dg \, \vec{n})V_{\vec{n}} +\Div_{\Gamma}(V_{\vec{n}}\sigma_{\Gamma}(u)), p \right\rangle \, dS.
\end{split}
\end{equation} 
This implies 
\begin{align*}
dJ_{vol}(\Omega)[V]
= \int_{\Gamma_{N}} \left\langle(f+\kappa g +Dg \, \vec{n})V_{\vec{n}} +\Div_{\Gamma}(V_{\vec{n}}\sigma_{\Gamma}(u)) ,p \right\rangle + \mathcal{F}_{vol}  V_{\vec{n}}  \, dS.
\end{align*} 
Now we would like to apply integration by parts and a trace theorem, but therefore $ p \in H^1(\Omega,\R{3}) $ is not enough. And we can not rewrite the shape derivative in $ L^2 $-product form 
\[ dJ_{vol}(\Omega)[V] = \int_{\Gamma_N} G(\Gamma) V_{\vec{n}} \, dS. \]

The case of a surface cost functional that depends on derivatives of the state 
\begin{align}\label{Shape_Grad_LinEl:Eq:Jsur}
J_{sur}(\Omega) =\int_{\Gamma} \mathcal{F}_{sur}(.,u,Du)\, dS
\end{align}
is even more sophisticated:  The straightforward approach would mean to define the adjoint equation in the weak form by
\begin{align}\label{Shape_Grad_LinEl:Eq: Adjoint surface direct approach}
\int_{\Omega} \hspace*{-2mm} \sigma(\vartheta):\varepsilon(p) \, dx = \int_{\Gamma_N}\hspace*{-1.5mm} \left\langle \frac{\partial \mathcal{F}_{sur}}{\partial z_2}(.,u,Du),\vartheta \right\rangle + \frac{\partial \mathcal{F}_{sur}}{\partial z_3}(.,u,Du):D\vartheta \, dS\, ~~~ \forall \vartheta \in H^{1}_{D}.
\end{align}
Unfortunately this equation is never even defined for elements in $ H^{1}(\Omega,\R{3}) $ but on $ H^{\nicefrac{3}{2}}(\Omega,\R{3}) $. On this space the Theorem of Lax-Milgram is no longer applicable since the bilinear form looses its coercivity on this space.

Numerically, at first sight, this seems to be no problem, since the equation is defined for piece wise differentiable functions. This is enough to solve the equation on the finite element space $ CG_{1} $ (continuous Galerkin) for a fixed mesh. But, when the mesh size converges to zero it is no longer clear if the numerical solution converges in $ H^1 $ to the analytical solution (if one exists).  For literature concerning the finite element method we refer to  \cite{ErnGuerm04,Braess2007finite,Johnson2012numerical}, for numerical methods for PDE to \cite{WLGleichung,BrennerScott2007mathematical}, and for optimization with PDE e.g. to \cite{Ulbrich2012constrained,Troltzsch2005optimale}.
Thus the continuously adjoint equation for surface functionals \ref{Shape_Grad_LinEl:Eq:Jsur}
has to be defined in a different way, which will be illustrated in Section \ref{Shape_Grad_LinEl:Sec:Surface Functional Gradient}.

\section{$ L^2 $-gradient regularity for local volume cost functionals} 

We start with the derivation and regularity classification of $ L^2 $-descent directions for shape functionals of the volume type
$ J_{vol}(\Omega) =\int_{\Omega} \mathcal{F}_{vol}(.,u,Du)\, dx$
w.r.t. linear elasticity constraints.

\begin{thm}[Hadamard Decomposition for Volume Functionals]\label{Shape_Grad_LinEl:Thm: Shape Gradient Vol Functional u, sigma(u)}$  $\\
	Let $ k\geq 2 $, $ \Omega \in \mathcal{O}_{k+1}^{b} $, $ \Omega_{t}=\T{t}(\Omega), \, t \in I_V  $ for some admissible vector field $V \in  C^{k+1}_{0}(\Oext,\R{3})$. Suppose that  $ f\in C^{k-1,\phi}(\overline{\Oext},\R{3}) $ and $ g\in C^{k,\phi}(\overline{\Oext},\R{3})$ for some $ \phi\in (0,1) $. Let $ u=u(\Omega) \in C^{k,\phi}(\overline{\Omega},\R{3}) $ be the unique solution of \eqref{Reliability:Eq:LinEl}.
	
	\begin{itemize}
		\item[a)]  Suppose that $ \mathcal{F}_{vol}:\R{3} \times \R{3} \times \R{3\times 3} \to \R{}$ is continuously differentiable and let  
		$ J_{vol}(\Omega) $  be defined as in \eqref{Shape_Grad_LinEl:Eq:Jvol}.
		Then the shape derivative of exists and a weak adjoint equation is given by \eqref{Shape_Grad_LinEl:Eq:Weak Adj Jvol}
		and the unique adjoint state is an element of $ H^1(\Omega,\R{3}) $. The surface representation of the shape derivative is given by the distribution
		\begin{align*}
		\mathcal{G}(\Gamma):~& C^{1}(\Gamma) \to  \R{} ,\,w \mapsto  \int_{\Gamma} \left\langle (f+\kappa g +Dg \, \vec{n})w +\Div_{\Gamma}(w\sigma_{\Gamma}(u))   ,p \right\rangle + \mathcal{F}_{vol}(Du) w \,dS.
		\end{align*}
		\item[b)] If $ \mathcal{F}_{vol} $ is additionally two times differetiable in $ z_3 $, then the adjoint equation in strong form reads
		\begin{align}\label{Shape_Grad_LinEl:Eq: Adjoint Eq Elasticity Vol u, sigma(u)}
		\left. 
		\begin{array}{r c l l} \tag{AV}
		\Div( \sigma(p)) & = & \tfrac{\partial \mathcal{F}_{vol}}{\partial z_2}(.,u,Du)-\Div\left(\tfrac{\partial \mathcal{F}_{vol}}{\partial z_3}(.,u,Du)^{\top}\right)   &\text{ in } \Omega \\
		p & = & 0  &\text{ on } \Gamma_{D} \\
		\sigma(p) \, \vec{n}& = & \tfrac{\partial \mathcal{F}_{vol}}{\partial z_3}(.,u,Du)\, \vec{n}  &\text{ on }\Gamma_{N}.
		\end{array} 
		\right.
		\end{align}
		\begin{itemize}
			\item[1.1)] If $ ~\tfrac{\partial \mathcal{F}_{vol}}{\partial z_2}(.,u,Du)-\Div\left(\tfrac{\partial \mathcal{F}_{vol}}{\partial z_3}(.,u,Du)\right) \in  L^{q}(\Omega,\R{3})$ and $\tfrac{\partial \mathcal{F}_{vol}}{\partial z_3}\, \in  W^{1-\nicefrac{1}{q},q}(\Gamma,\R{3})$ for some $ \nicefrac{6}{5}\leq q < \infty $ then the (weak) adjoint state $ p\in  H^1_{D}(\Omega,\R{3}) \cap W^{2,q}(\Omega,\R{3}) $.\\
			\item[1.2)]  If additionally $ q \geq 4 $  then the (weak) adjoint state  $ p $ has a representation in $\in C^{1}(\overline{\Omega},\R{3}) $. If even $ 0< \phi < 1 -\frac{3}{q}  $ then $ p  \in C^{1,\phi}(\overline{\Omega},\R{3}) $.
			The shape derivative is given by $dJ_{vol}(\Omega)[V] =  \langle G(\Gamma) \vec{n}, V \rangle_{L^{2}(\Gamma,\R{3})}$
			where 
			\begin{align}\label{Shape_Grad_LinEl:Eg:Haramard Representation Lvol}
			G(\Gamma)= \mathcal{F}_{vol} (.,u,Du) +  \langle f+\kappa g +Dg \, \vec{n} ,p \rangle  - D_{\Gamma}p:\sigma_{\Gamma}(u)   \text{ on } \Gamma.
			\end{align}
			\item[2)] If even $ \mathcal{F}_{vol} \in C^{l+1,\psi}(\R{3} \times \R{3} \times \R{3 \times 3}) $ for $  1 < l +\psi$,  then there exists a unique strong solution $ p$ of the   adjoint equation  \eqref{Shape_Grad_LinEl:Eq: Adjoint Eq Elasticity Vol u, sigma(u)}  in $C^{l+1,\psi}(\overline{ \Omega},\R{3})\cap C^{k,\phi}(\overline{ \Omega},\R{3}) $.
		\end{itemize}
	\end{itemize}
\end{thm}

\begin{proof}
	a) The Eulerian derivative of $ J_{vol} $ reads
	\begin{align*}
	dJ_{vol}(\Omega)[V]
	=& \int_{\Omega} \hspace*{-1mm}\left\langle\frac{\partial \mathcal{F}_{vol}}{\partial z_2}(.,u,Du) ,u' \right\rangle + \frac{\partial \mathcal{F}_{vol}}{\partial z_3}(.,u,Du) :Du'\, dx + \int_{\Gamma_{N}}  \mathcal{F}_{vol}(.,u,Du)V_{\vec{n}} \,dS.
	\end{align*} 
	
	\noindent Since $ \frac{\partial \mathcal{F}_{vol}}{\partial z_2}  $ and $ \frac{\partial \mathcal{F}_{vol}}{\partial z_3} $ are continuous and  $ u \in C^{2,\phi} $ the mapping 
	\[ \vartheta \in H^1_{D}(\Omega,\R{3}) \to \int_{\Omega} \left\langle\frac{\partial \mathcal{F}_{vol}}{\partial z_2}(.,u,Du),\vartheta \right\rangle + \frac{\partial \mathcal{F}_{vol}}{\partial z_3}(.,u,Du):D\vartheta \, dx  \]
	is continuous. Thus the existence of the unique adjoint state follows from the Theorem of Lax-Milgram.
	The vector field $ u'\in  C^{1,\phi}(\overline{\Omega},\R{3}) \subset H^1(\Omega,\R{3}) $ is a weak solution of \eqref{Ex_Shape_Deriv_LinEl:Eq: PDE u'} and thus
	\begin{align*}
	\int_{\Omega} \left\langle\frac{\partial \mathcal{F}_{vol}}{\partial z_2},u' \right\rangle + \frac{\partial \mathcal{F}_{vol}}{\partial z_3}:Du'\, dx 
	= \int_{\Gamma_{N}}\left\langle(f+\kappa g +Dg \, \vec{n})V_{\vec{n}} +\Div_{\Gamma}(V_{\vec{n}}\sigma_{\Gamma}(u)), p \right\rangle \, dS
	\end{align*}
	according to \eqref{Shape_Grad_LinEl:Eq: Shape derivative volume - replace by BC u'}. Moreover the mapping
	\begin{align*}
	\mathcal{G}(\Gamma):~& C^{1}(\Gamma) \to  \R{} ,\,w \mapsto  \int_{\Gamma} \left\langle (f+\kappa g +Dg \, \vec{n})w +\Div_{\Gamma}(w\sigma_{\Gamma}(u))   ,p \right\rangle + \mathcal{F}_{vol}(. ,u,Du) w\, dS
	\end{align*}
	is continuous: $\mathcal{F}_{vol}(. ,u,Du)\vert_{\Gamma}$ and $ |\Gamma| $ are bounded and thus $$  \int_{\Gamma} | \mathcal{F}_{vol}(. ,u, Du) w|\,  dS \leq C \Norm{w}{C^{1}(\Gamma,\R{3})}.$$ Moreover,  the adjoint state $ p $ is independent of $ w $ and thus 
	\begin{align*}
	\int_{\Gamma} &\left\langle (f+\kappa g +Dg \, \vec{n})w +\Div_{\Gamma}(w\sigma_{\Gamma}(u))  ,p \right\rangle \, dS \\
	&\leq  \, | \Gamma | \Norm{(f+\kappa g +Dg \, \vec{n})w +\Div_{\Gamma}(w\sigma_{\Gamma}(u))}{\infty} \Norm{p}{L^2(\Gamma,\R{3})}  \\
	&\leq  \, C \Norm{p}{H^1(\Omega,\R{3})}  \Norm{w}{C^{1}(\Gamma,\R{3})}  
	\end{align*}
	according to \eqref{Eq: <f,u> L^1_est Sur} and the trace Theorem \ref{App: Trace Operator}. The constant $ C $ depends on $ f,\, g,\, \Gamma,\, u $ and $ p $. This implies that $  dJ(\Omega)[V]: C^{1}_{0}(\Oext,\R{3}) \to \R{}$  is continuous since
	the trace operator $ \mathbf{T_{\Gamma}} : C^{1}_0(\Oext,\R{3}) \to C^{1}(\Gamma,\R{3})  $ is continuous. Consider also the proof of Theorem \ref{Transf_shape_opt:Thm:Hadamard_Thm}. This already implies shape differentiability on $ C^{k+1} $-domains, since $ C^{1}_{0}(\Oext,\R{n})' \subset C^{k+1}_{0}(\Oext,\R{n})' $. The Hadamard surface representation of $ dJ(\Omega)[V] $ is thus given by
	\begin{align*}
	dJ(\Omega)[V] &= \mathcal{G}(\Gamma)( V_{\vec{n}})\\
	&= \int_{\Gamma_{N}} \left\langle(f+\kappa g +Dg \, \vec{n})V_{\vec{n}} +\Div_{\Gamma}(V_{\vec{n}}\sigma_{\Gamma}(u))  ,p \right\rangle + \mathcal{F}_{vol}(. ,u, Du)  V_{\vec{n}} \, dS.
	\end{align*} 
	
	\noindent b)1.1) The function
	$ \tfrac{\partial \mathcal{F}_{vol}}{\partial z_2}$ $-\Div\left(\tfrac{\partial \mathcal{F}_{vol}}{\partial z_3}\right) $ is an element of $ L^{q}(\Omega,\R{3})$ and $\tfrac{\partial \mathcal{F}_{vol}}{\partial z_3}^{\top} \vec{n} \in W^{1-\nicefrac{1}{q},q}(\Gamma,\R{3})$ for some $ \nicefrac{6}{5} \leq q < \infty $. Then, Theorem \ref{Linear_Elasticity:Thm:Weak_Reg_LinEl} implies that there is a unique solution $ p\in  H^1_{D}(\Omega,\R{3}) \cap W^{2,q}(\Omega,\R{3}) $  of \eqref{Shape_Grad_LinEl:Eq: Adjoint Eq Elasticity Vol u, sigma(u)}.
	The vector field  $ u'\in  C^{1,\phi}(\overline{\Omega},\R{3}) \subset H^1(\Omega,\R{3}) $ is a weak solution of \eqref{Ex_Shape_Deriv_LinEl:Eq: PDE u'} and regular enough to apply integration by parts and we obtain 
	\begin{align*}
	\int_{\Gamma_{N}}&\left\langle (f+\kappa g +Dg \, \vec{n})V_{\vec{n}} +\Div_{\Gamma}(V_{\vec{n}}\sigma_{\Gamma}(u)), p \right\rangle \, dS \\
	&= \int_{\Omega} \left\langle\frac{\partial \mathcal{F}_{vol}}{\partial z_2}(. ,u, Du),u' \right\rangle + \frac{\partial \mathcal{F}_{vol}}{\partial z_3}(. ,u, Du):Du'\, dx
	\\
	&=\int_{\Omega} \left\langle\frac{\partial \mathcal{F}_{vol}}{\partial z_2}(.,u, Du) -\Div\left(\frac{\partial \mathcal{F}_{vol}}{\partial z_3}(. ,u, Du)^{\top}\right) ,u' \right\rangle\, dx \\
	&~~~ +    \int_{\Gamma_N} \left\langle\frac{\partial \mathcal{F}_{vol}}{\partial z_3}(. ,u, Du)\vec{n},u' \right\rangle \, dS.
	\end{align*}
	by \eqref{Shape_Grad_LinEl:Eq: Shape derivative volume - replace by BC u'} 
	and $ \sigma(p):\varepsilon(u') =  \sigma(u'):\varepsilon(p)  $. This implies the assertion.
	
	\noindent b)1.2) This is due to the Sobolev Embedding Theorem.
	Since $ V_{\vec{n}}=0 $ on $ \Gamma_D $ the $ L^2 $-Hadamard representation of $ dJ(\Omega)[V] $ is given by \eqref{Shape_Grad_LinEl:Eg:Haramard Representation Lvol}.
	
	\noindent b)2) If  $ \mathcal{F}_{vol} \in C^{l+1,\psi}(\R{3} \times \R{3} \times \R{3 \times 3}) $, then 
	$ \tfrac{\partial \mathcal{F}_{vol}}{\partial z_2}(.,u,Du)$ $-\Div\big(\tfrac{\partial \mathcal{F}_{vol}}{\partial z_3}(.,u,Du)^{\top}\big)$ $ \in  C^{l-1,\psi}(\overline{ \Omega},\R{3}) \cap C^{k-2,\phi}(\overline{ \Omega},\R{3})  $ and $\tfrac{\partial \mathcal{F}_{vol}}{\partial z_3}(.,u,Du) \vec{n} \in C^{l,\psi}(\Gamma,\R{3}) $ $\cap C^{k-1,\phi}(\Gamma,\R{3})$. Then Theorem \ref{Linear_Elasticity:Thm:LinEl_Classical_Sol}  implies the assertion.
\end{proof}

Now we derive the regularity classification for the density $ G(\Gamma) = \mathcal{F}_{vol} (.,u,Du) +  \langle f+\kappa g +Dg \, \vec{n} ,p \rangle  - D_{\Gamma}p:\sigma_{\Gamma}(u) $. The regularity of $ G(\Gamma)$ is determined by the term with the lowest regularity appearing in the formula. \\[1em]
\noindent The following table shows the Hölder exponents and thus the regularities of the leading terms appearing in $ G(\Gamma) $ for different regularities of $ \mathcal{F}_{vol} $.  We make a interval-based decision between the cases, where the intervals and are given in the first row. The regularities for terms $ p $, $ f+ \kappa g + Dg\, \vec{n} $,$ \ldots $ can then be found in the column belonging to the respective case. The regularity of the density $ G(\Gamma) $ is the minimum over these regularities (column wise).   
\begin{table}[htb]
	\centering
	\begin{small}
		\begin{tabular}{|c|c||c|c|c|} \hline
			$ \mathcal{ F}_{vol} ~/~ l+\psi \in  $					& $ \{1\} $			    & $ (1,k-1 + \phi] $ 	  & $ (k-1+\phi,\infty) $    \\ \hline\hline
			$ p $ 						& $ 1+\phi  $       & $ l+1+\psi $    & $ k+\phi $ 	 \\ \hline \hline
			$  f+\kappa g + Dg \,\vec{n}$ 
			& $ k-1 $           & $ k-1 $		  & $ k-1 $		 \\ \hline
			$ D_{\Gamma}p $				& $ 0+\phi $        & $ l+\psi $ 	  & $ k-1+\phi$  \\ \hline
			$ D_{\Gamma}u $				& $k-1+\phi $        & $k-1+\phi $  		  & $k-1+\phi $      \\ \hline
			$ \mathcal{ F}_{vol} $					& $1$               & $l+\psi $       & $ k-1+\psi $\\ \hline\hline
			$ G(\Gamma) $				& $ 0+\phi $        & $l+\psi $		  & $ k-1 $		 \\ \hline
		\end{tabular} \\
	\end{small}
	\caption{Regularities of the $ L^2 $- Hadamard representation $ G(\Gamma) $ for functionals of \\ the type $ J_{vol} $.}
	\label{Tab: Reg Grad Fvol}
\end{table}
\newpage

This shows that in the best case $G(\Gamma) \in C^{k-1}(\Gamma) $ if $ \Omega $ is of class $ C^{k+1} $.

\begin{rem}
	One could also expect that 	
	\begin{align*}
	\mathcal{G}(\Gamma):& H^{1}(\Gamma) \to  \R{} ,\,w \mapsto  \int_{\Gamma} \left\langle (f+\kappa g +Dg \, \vec{n})w +\Div_{\Gamma}(w\sigma_{\Gamma}(u))   ,p \right\rangle + \mathcal{F}_{vol}(Du) w \, dS
	\end{align*}
	is continuous, but unfortunately, equation (2.124) in \cite{SokZol92} and the remark below are not applicable here, since the tangential divergence of $ w \sigma_{\Gamma}(u) $ is only defined if $ w \in H^{1}(\Gamma) $ if $ (w\sigma_{\Gamma}(u))\vec{n} =0 $ on $ \Gamma $ and $ p \in H^1(\Gamma,\R{3}) $. The latter is not the case here.   
\end{rem}

\begin{rem}
	The case $$ J_{vol}(\Omega)=\int_{\Omega} \mathcal{F}_{vol}(., u, \sigma(u)) \, dx $$ can be treated analogously to the case of $  J_{vol}(\Omega)=\int_{\Omega} \mathcal{F}_{vol}(., u, Du) \, dx  $. Then the adjoint equation reads
	
	\begin{align*}
	\left. 
	\begin{array}{r c l l}
	\Div( \sigma(p)) & = & \frac{\partial \mathcal{F}_{sur}}{\partial z_{2}} -\left[\lambda  \nabla \tr(\frac{\partial \mathcal{F}_{sur}}{\partial z_{3}})I +\mu\Div(\frac{\partial \mathcal{F}_{sur}}{\partial z_{3}}+\frac{\partial \mathcal{F}_{sur}}{\partial z_{3}}^{\top})\right] &\text{ in } \Omega \\
	p & = & 0  &\text{ on } \Gamma_{D} \\
	\sigma(p) \, \vec{n}& = & \left[\lambda \tr(\frac{\partial \mathcal{F}_{vol}}{\partial z_3})I + \mu\left( \frac{\partial \mathcal{F}_{vol}}{\partial z_3}+\frac{\partial \mathcal{F}_{vol}}{\partial z_3}^{\top}\right)\right] \vec{n}  &\text{ on }\Gamma_{N}.
	\end{array} 
	\right.
	\end{align*} 
	or 
	\[ \int_{\Omega} \sigma(p):\varepsilon(\vartheta) \, dx = \int_{\Omega} \frac{\partial \mathcal{F}_{sur}}{\partial z_3}(.,u,Du):\sigma(\vartheta) \, dx~~~\forall \vartheta \in H^1_{D}(\Omega,\R{3}) \]
	in weak formulation. In the formula for the gradient $ \mathcal{F}_{vol}(.,u,Du) $ has to be replaced by $  \mathcal{F}_{vol}(.,u,\sigma(u)) $ 
\end{rem}

\begin{thm}
	Let $ \Omega $ be of class $C^{3} $, $ V \in C^{3}_{0}(\Oext,\R{3}) $, $ f \in C^{1,\phi}(\overline{\Oext},\R{3}) $ and $ g \in C^{2,\phi}(\overline{\Oext},\R{3}) $. If $ m\geq 3 $, then $ J^{cer} $ is shape differentiable and the $ L^{2}(\Gamma) $-shape gradient is given by
	\[ G^{cer}(\Gamma) = \int_{\mathbb{S}^{2}}\left(\frac{\sigma(u)_{n}^{+}}{\sigma_{c}}\right)^{m}\, dS_{\mathbb{S}^{2}} + \langle f+\kappa g + Dg\vec{n},p\rangle -D_{\Gamma}p:\sigma_{\Gamma}(u) \in C^{1}(\Gamma) \]
	where $ u \in C^{2,\phi}(\overline{ \Omega},\R{3}) $ is the unique solution of  \eqref{Reliability:Eq:LinEl} and $ p \in C^{2,\phi}(\overline{ \Omega},\R{3})  $ is the unique solution of \eqref{Shape_Grad_LinEl:Eq: Adjoint Eq Elasticity Vol u, sigma(u)}.
\end{thm}

\begin{proof}
	Apply Lemma \ref{Shape_Grad_LinEl:Prop: Differentiability Ceramic} and Theorem \ref{Shape_Grad_LinEl:Thm: Shape Gradient Vol Functional u, sigma(u)}.
\end{proof}

\vspace*{1em}

\section{$ L^2 $-gradient regularity for local surface cost functionals}\label{Shape_Grad_LinEl:Sec:Surface Functional Gradient}

We now continue with the derivation and regularity classification of $ L^2 $-shape gradient for shape functionals of the surface types
\begin{align*}
J_{1,\, sur}(\Omega)&:=\int_{\Omega} \mathcal{F}_{sur}(.,u)\, dx,\\
\intertext{ and }
J_{2,\,sur}(\Omega)&:=\int_{\Omega} \mathcal{F}_{sur}(.,u,\sigma(u))\, dx
\end{align*}
w.r.t. linear elasticity constraints.

\begin{prop}\label{Shape_Grad_LinEl:Thm: Shape Gradient Sur1}
	Let $ k\geq 2 $, $ \Omega \in \mathcal{O}_{k+1}^{b} $, $ \Omega_{t}=\T{t}(\Omega), \, t \in I_V  $ for some admissible vector field $V \in C^{k}_{0}(\Oext,\R{3})$. Suppose that  $ f\in C^{k-1,\phi}(\overline{\Oext},\R{3}) $ and $ g\in C^{k,\phi}(\overline{\Oext},\R{3})$ for some $ \phi\in (0,1) $. Let $ u=u(\Omega) \in C^{k,\phi}(\O,\R{3}) $ be the unique solution of \eqref{Reliability:Eq:LinEl} and $ \mathcal{F}_{sur} \in C^{l,\psi}(\R{3}) $, $ l+\psi\geq 1,\, \psi\in [0,1] $ . Then a weak adjoint equation to $ J_{1,sur}(\Omega) $
	and the disjoint displacement-traction problem of linear elasticity \eqref{Reliability:Eq:LinEl} as state equation is given by
	\begin{align}\label{Shape_Grad_LinEl:Eq: weak adjoint Eq Elasticity Fsur u}
	\int_{\Omega} \sigma(\vartheta):\varepsilon(p) \, dx 
	&= \int_{\Gamma_N} \left\langle \tfrac{\partial \mathcal{ F}_{sur}}{\partial z_2}(.,u), \vartheta \right\rangle \, dS~~~  \forall \, \vartheta \in  H^1_{D}.
	\end{align}
	The strong form is given by
	\begin{equation}\label{Shape_Grad_LinEl:Eq: Adjoint Eq Elasticity Fsur u}
	\left. 
	\begin{array}{c c l l}
	\Div( \sigma(p)) & = & 0  &\text{ in } \Omega \\
	p & = & 0  &\text{ on } \Gamma_{D} \tag{AS-1}\\
	\sigma(p) \, \vec{n}& = & \tfrac{\partial \mathcal{ F}_{sur}}{\partial z_2}(.,u) & \text{ on }\Gamma_{N}  \\
	\end{array}. 
	\right.
	\end{equation} 
	\begin{itemize}
		\item[a)]  There exists a unique weak solution $ p \in H^1_{D}(\Omega,\R{3}) $ of the adjoint equation \eqref{Shape_Grad_LinEl:Eq: Adjoint Eq Elasticity Fsur u} and the Hadamard representation of the shape derivative is given by
		\begin{align*}
		\mathcal{G}(\Gamma): ~& C^{1}(\Gamma) \to  \R{} \\
		&w \mapsto  \int_{\Gamma_{N}} \left[\left\langle \tfrac{\partial \mathcal{ F}_{sur}}{\partial z_2}(.,u),Du\, \vec{n} \right\rangle +\langle \tfrac{\partial \mathcal{ F}_{sur}}{\partial z_1}(.,u),\vec{n}\rangle+ \kappa \mathcal{F}_{sur} (.,u)  \right]w  \,dS \\
		&~~~~~~~~+ \int_{\Gamma_N} \langle f+\kappa g +Dg\, \vec{n},p\rangle w +  \langle \Div_{\Gamma}(w\sigma_{\Gamma}(u)), p \rangle \, dS.
		\end{align*}
	\end{itemize}
	\begin{itemize}
		\item[b)]
		\begin{itemize}
			\item[1.1)] 
			If $~\tfrac{\partial \mathcal{ F}_{sur}}{\partial z_2}(.,u) \in W^{1-\nicefrac{1}{q},q}(\Gamma,\R{3}) $, $ \infty> q \geq \frac{4}{3}  $, there exists a unique weak solution $ p \in H^1_{D}(\Omega,\R{3}) \cap W^{2,q}(\Omega,\R{3}) $ of the adjoint equation \eqref{Shape_Grad_LinEl:Eq: Adjoint Eq Elasticity Fsur u}.
			\item[1.2)]  If additionally $ q \geq 4 $  then the (weak) adjoint state  $ p $ has a representation in $ C^{1}(\overline{\Omega},\R{3}) $. If even $ 0< \phi < 1 -\frac{3}{q}  $ then $ p  \in C^{1,\phi}(\overline{\Omega},\R{3}) $.
			The shape derivative is given by $dJ_{vol}(\Omega)[V] =   \langle G(\Gamma) \vec{n}, V \rangle_{L^{2}(\Gamma,\R{3})}$
			where 
			\[G(\Gamma)= \left\langle f+\kappa g +Dg\, \vec{n},p\right\rangle - D_{\Gamma}p :\sigma_{\Gamma}(u)+ \langle \tfrac{\partial \mathcal{ F}_{sur}}{\partial z_2},Du \vec{n} \rangle +\langle \tfrac{\partial \mathcal{ F}_{sur}}{\partial z_1},\vec{n} \rangle + \kappa\mathcal{F}_{sur} (.,u).\]
			\item[2)]	If $ \mathcal{F}_{sur} \in C^{l,\psi}(\R{3}),\, l+\psi>2,\, \psi\in (0,1) $ there exists a unique strong solution $ p $ of \eqref{Shape_Grad_LinEl:Eq: Adjoint Eq Elasticity Fsur u} which is an element of $ C^{l-1,\psi}(\overline{\Omega},\R{3}) \cap C^{k,\phi}(\overline{\Omega},\R{3}) $.
		\end{itemize}
	\end{itemize}
\end{prop}

\begin{proof}
	a) Lemma \ref{Transf_shape_opt:Lem: Product and Chain rule for shape derivatives} and Lemma \ref{Transf_shape_opt:Lem: d/dt int J(T_t,u_t,Du_t) in shape derivative form} we obtain
	\begin{align*}
	dJ_{1,\,sur}(\Omega)[V]=
	&\,  \int_{\Gamma_N}\left[\left\langle \tfrac{\partial \mathcal{ F}_{sur}}{\partial z_2}(.,u),Du\, \vec{n} \right\rangle+ \langle \tfrac{\partial \mathcal{ F}_{sur}}{\partial z_1}(.,u),\vec{n} \rangle + \kappa \mathcal{F}_{sur} (u)\right]V_{\vec{n}} \, dS  \\
	&+\int_{\Gamma_N} \left\langle \tfrac{\partial \mathcal{ F}_{sur}}{\partial z_2}(.,u)  ,u' \right\rangle \, dS\,.
	\end{align*} 
	Suppose that $ \mathcal{F}_{sur}\in C^{1}(\R{3}) $ and $ u\in C^{k,\phi} $, $ k\geq 2 $. Then $ \tfrac{\partial \mathcal{ F}_{sur}}{\partial z_2}(.,u)$  is an element of $C^{0}(\overline{\Omega},\R{3})$ and thus contained in $ L^{4/3}(\Gamma_{N},\R{3}) $. Hence, by the Lax-Milgram Theorem, there exists a unique  solution $ p\in  H^1_{D}(\Omega,\R{3})  $  of equation \eqref{Shape_Grad_LinEl:Eq: weak adjoint Eq Elasticity Fsur u}.
	Moreover, the local shape derivative $ u' \in C^{1,\phi} $ is a weak solution of   \eqref{Ex_Shape_Deriv_LinEl:Eq: PDE u'} and thus
	\begin{align*}
	\int_{\Gamma_N} \hspace*{-2mm}\left\langle \tfrac{\partial \mathcal{ F}_{sur}}{\partial z_2}(.,u) ,u' \right\rangle \, dx  
	& = \hspace*{-1mm}  \int_{\Omega} \hspace*{-1mm}\sigma(u'):\varepsilon(p) \, dx 
	= \hspace*{-1mm} \int_{\Gamma_{N}} \hspace*{-3mm}\left\langle (f+\kappa g +Dg\, \vec{n})V_{\vec{n}} +\Div_{\Gamma}(V_{\vec{n}}\sigma_{\Gamma}(u)), p \right\rangle dS.
	\end{align*}
	This leads to 
	\begin{align*}\label{Shape_Grad_LinEl:Eq: Hadamard Represetation J_sur u'}
	dJ_{1,\,sur}(\Omega)[V]
	&=\int_{\Gamma_{N}}\langle \Div_{\Gamma}(V_{\vec{n}}\sigma_{\Gamma}(u)), p \rangle + \langle f+\kappa g +Dg\, \vec{n},p\rangle V_{\vec{n}}\, dS \\
	& +\int_{\Gamma_N} \left[\left\langle \tfrac{\partial \mathcal{ F}_{sur}}{\partial z_2}(.,u),Du\, \vec{n} \right\rangle + \langle \tfrac{\partial \mathcal{ F}_{sur}}{\partial z_1}(.,u),\vec{n} \rangle + \kappa \mathcal{F}_{vol} (u)\right]V_{\vec{n}} \, dS .
	\end{align*} 
	Analogously to the argumentation in Proposition \ref{Shape_Grad_LinEl:Thm: Shape Gradient Vol Functional u, sigma(u)} a) we conclude that a) holds.
	
	\noindent b)1.1) \& 1.2) The proofs are analogous to those of b)1.1) \& 1.2) of Proposition \ref{Shape_Grad_LinEl:Thm: Shape Gradient Vol Functional u, sigma(u)}.
	
	\noindent 2) If $ 2<l+\psi $, then $ \tfrac{\partial \mathcal{ F}_{sur}}{\partial z_2}(.,u) \in C^{l-1,\psi} \cap C^{k,\phi} $ and the unique adjoint state $p$ is contained in $ C^{l,\psi}(\overline{\Omega},\R{3})\cap C^{k,\phi}(\overline{\Omega},\R{3})  $ since the regularity of the boundary is a natural restriction for the regularity of $ p $. We can thus again integrate by parts and obtain the demanded formula for $ G(\Gamma) $.
\end{proof}

\noindent The following table shows the Hölder exponents and thus the regularities of the leading terms appearing in
\[ G(\Gamma)= \left\langle f+\kappa g +Dg\, \vec{n},p\right\rangle - D_{\Gamma}p :\sigma_{\Gamma}(u)+ \langle \tfrac{\partial \mathcal{ F}_{sur}}{\partial z_2}(.,u),Du\, \vec{n} \rangle +\langle \tfrac{\partial \mathcal{ F}_{sur}}{\partial z_1}(.,u),\vec{n} \rangle + \kappa\mathcal{F}_{sur} (u) \]
for different regularities of $ \mathcal{F}_{vol} $ and has to be read in the same manner as table \ref{Tab: Reg Grad Fvol}.

\begin{table}[htb]
	\centering
	\begin{small}
		\begin{tabular}{|c||c|c|c|}
			\hline
			$ l+\psi $					& $ 2 $			    & $ (2,k+\phi] $ 	  & $ (k+\phi,\infty) $    \\ \hline\hline
			$ p $ 						& $ 1+\phi $       & $ l+\psi $    & $ k+\phi $ 	\\ \hline \hline
			$  f+\kappa g + Dg\,\vec{n} $ 
			& $ k-1 $           & $ k-1 $		  & $ k-1 $		 \\ \hline
			$ D_{\Gamma}p $				& $ 0+\phi $        & $ l-1+\psi $ 	  & $ k-1+\phi$  \\ \hline
			$ D_{\Gamma}u $				& $k-1+\phi $              & $k-1+\phi $  		  & $k-1+\phi $       \\ \hline
			$ \tfrac{\partial \mathcal{ F}_{sur}}{\partial z_2}(.,u) $					& $1$               & $l-1+\psi $       & $ k-1+\psi $ \\ \hline\hline
			$ G(\Gamma) $				& $ 0+\phi $        & $l-1+\psi $		  & $ k-1 $		 \\ \hline
		\end{tabular} \\
	\end{small}
	\caption{Regularities of the $ L^2 $- Hadamard representation $ G(\Gamma) $ for functionals of \\ the type $ J_{1,sur} $}
	\label{Tab: Reg Grad F1sur}
\end{table}
\bigskip
\noindent Now we investigate shape functionals of the strain driven surface type
\begin{equation} \label{Shape_Grad_LinEl:Eg:J2sur}
J_{2,\,sur}(\Omega):=\int_{\Omega} \mathcal{F}_{sur}(.,u,\sigma(u))\, dx.
\end{equation}
To abbreviate the calculations we introduce the following notation:
\begin{align}
\begin{array}{lll}
I_{\vec{n}}:= \vec{n}\vec{n}^{\top} & M_{\vec{n}}:= MI_{\vec{n}} & 
\\
I_{\Gamma}:=I-\vec{n}\vec{n}^{\top}=I-I_{\vec{n}} ~~~~  &   M_{\Gamma}:=M I_{\Gamma}  &~~~~ ~ _{\Gamma}M:= I_{\Gamma}M.
\end{array}
\end{align} 
Note that 
\begin{align*}
D_{\Gamma}v&=Dv(I-\vec{n}\vec{n}^{\top})=Dv_{\Gamma}\\
\Div_{\Gamma}(v)&=\Div(v)-\langle Dv\,\vec{n},\vec{n} \rangle =\tr(Dv-Dv_{\vec{n}})=\tr(D_{\Gamma}v) \\
\sigma_{\Gamma}(v)&=\sigma(v)(I-\vec{n}\vec{n}^{\top})=\sigma(u)_{\Gamma}.
\end{align*}

\newpage
\begin{prop}\label{Transf_shape_opt:Prop: Tangential Stokes Sigma}
	Let $ \Omega \subset \R{n}  $  be a domain of class $ C^2 $,  $ v \in C^{1}(\Gamma,\R{m}) $ a vector field and $ M \in C^1(\Gamma,\R{m \times m}) $ a matrix field. Then
	\begin{align*}
	&\int_{\Gamma} \tr(M\sigma_{\Gamma}(v))  + \left\langle 
	\tilde{\lambda} \nabla_{\Gamma}\tr\left(M_{\Gamma}\right) +\mu \Div_{\Gamma}\left( \left[M + M^{\top}\right]_{\Gamma} \right),v   \right\rangle \, dS \\
	&= \int_{\Gamma} \kappa \left\langle \left( \tilde{\lambda} \tr\left(M_{\Gamma}\right) I + \mu~ \leftidx{_\Gamma}{\Big(M+M^{\top}\Big)} \right) \vec{n},v\right\rangle +  \left\langle \sigma(v)\vec{n},\left(\tfrac{\lambda}{\lambda + 2 \mu}\tr(M_{\Gamma}) \mathrm{I}+\,\leftidx{_\Gamma}{M}\right)\vec{n} \right\rangle  \, dS 
	\end{align*}
	where $ \tilde{\lambda} = \left(\lambda - \frac{\lambda^2}{\lambda + 2 \mu}\right)$.
\end{prop}

\bigskip

\noindent To prove the Proposition we establish the following calculation rules:

\begin{lem}\label{Transf_shape_opt:Lem:Prepartations}
	Let $ \Omega \subset \R{3} $ be a bounded domain with boundary $ \Gamma $ of class $ C^1 $, $\vec{n} $ the unity outward normal vector field and $ v \in C^{1}(\overline{\Omega},\R{3}) $. Then the following holds on $ \Gamma $:
	\begin{align*}
	\begin{array}{l l}
	i)& \displaystyle Dv_{\vec{n}}= \frac{1}{\mu}\left[\sigma(v)_{\vec{n}} -\lambda\Div(v)I_{\vec{n}}\right] -(Dv^{\top})_{\vec{n}}, 
	\\[1ex]
	ii)& \displaystyle\tr(Dv_{\vec{n}})=\frac{1}{\lambda+2\mu}\left(\left\langle \sigma(v)\vec{n},\vec{n} \right\rangle - \lambda \Div_{\Gamma}(v)\right),
	\\[1ex] 
	iii)&\displaystyle \tr\left((Dv_{\vec{n}})^{\top}\, \leftidx{_\Gamma}{M}\right) = \frac{1}{\mu} \left\langle \sigma(v)\vec{n},\,\leftidx{_\Gamma}{M}\vec{n} \right\rangle - \tr\left(M_{\vec{n}}D_{\Gamma}v\right),
	\\[1ex]
	iv) & \displaystyle\sigma_{\Gamma}(v) = \lambda \Div_{\Gamma}(v)I_{\Gamma} + \mu (D_{\Gamma}v + D_{\Gamma}v^{\top} I_{\Gamma}) +  \lambda \tr\left(Dv_{\vec{n}}\right)I_{\Gamma} +   \mu(Dv_{\vec{n}})^{\top} I_{\Gamma}, 
	\\[1ex]
	v)&	\displaystyle\tr(M\sigma_{\Gamma}(v)) =
	\tilde{\lambda}\tr\left(M_{\Gamma}\right) \Div_{\Gamma}(v) + \mu \tr\left([M+M^{\top}]_{\Gamma} D_{\Gamma}v\right) \\[1ex]
	&\displaystyle \hspace{2.2cm}+ \left\langle \sigma(v)\vec{n}, \left(\frac{\lambda}{\lambda+2\mu}\tr\left(M_{\Gamma}\right)I + \,\leftidx{_\Gamma}{M} \right)\vec{n}\right\rangle.
	\end{array}
	\end{align*}
\end{lem}

Since the proof is very extensive and consists only of basic algebraic computations but is not very instructive and disturbs the reading flow, it is outsourced and can be found in the appendix, consider Lemma \ref{App:Transf_shape_opt:Lem:Prepartations}. Thus we continue with the proof of  Proposition \ref{Transf_shape_opt:Prop: Tangential Stokes Sigma}.

\begin{proof}
	We apply Lemma \ref{Transf_shape_opt:Lem:Prepartations} and Theorems \ref{Transf_shape_opt:Thm: Tangential Stokes Vector} and \ref{Transf_shape_opt:Cor: Tangential Stokes Matrix}.
	\begin{align*}
	\int_{\Gamma}  \tilde{\lambda} &\tr(M_{\Gamma})\Div_{\Gamma}(v)  
	+ \mu\tr\left([M  +  M^{\top}]_{\Gamma} D_{\Gamma}v\right)\, dS\\
	&= \int_{\Gamma}\tilde{\lambda} \left[\kappa \left\langle \tr\left( M_{\Gamma}\right)v,\vec{n}\right\rangle - \skp{\nabla_{\Gamma}\tr\left( M_{\Gamma}\right)}{v} \right]\, dS \\
	&~~~+ \int_{\Gamma} \mu\left[\kappa\skp{\left[M+M^{\top}\right]_{\Gamma} \, v}{\vec{n}} - \skp{\Div_{\Gamma}\left(\left[M+M^{\top}\right]_{\Gamma}\right)}{v}\right] \, dS\\[1ex]
	&= \tilde{\lambda} \int_{\Gamma}  \kappa \left\langle \tr\left( M_{\Gamma}\right)\vec{n},v\right\rangle - \skp{\nabla_{\Gamma}\tr\left( M_{\Gamma}\right)}{v} \, dS \\
	&~~~+ \mu \int_{\Gamma} \kappa\skp{\,\leftidx{_\Gamma}{\left[M+M^{\top}\right]} \vec{n}}{v} - \skp{\Div_{\Gamma}\left(\left[M+M^{\top}\right]_{\Gamma}\right)}{v} \, dS
	\\[1ex]
	&= \int_{\Gamma} \kappa \left\langle \left( \tilde{\lambda} \tr\left(M_{\Gamma}\right) I + \mu \leftidx{_\Gamma}{\left[M+M^{\top}\right]} \right) \vec{n},v\right\rangle \, dS\\
	&~~~-\int_{\Gamma}  \left\langle \tilde{\lambda} \nabla_{\Gamma}\tr\left(M_{\Gamma}\right) +\mu \Div_{\Gamma}\left( \left[M + M^{\top}\right]_{\Gamma} \right),v   \right\rangle \, dS
	\end{align*}
	This implies
	\begin{align*}
	\int_{\Gamma}\tr(M\sigma_{\Gamma}(v))\, dS
	&= \int_{\Gamma} \kappa \left\langle \left( \tilde{\lambda} \tr\left(M_{\Gamma}\right) I + \mu \leftidx{_\Gamma}{\left[M+M^{\top}\right]} \right) \vec{n},v\right\rangle \, dS\\
	&~~~-\int_{\Gamma}  \left\langle \tilde{\lambda} \nabla_{\Gamma}\tr\left(M_{\Gamma}\right) +\mu \Div_{\Gamma}\left( \left[M + M^{\top}\right]_{\Gamma} \right),v   \right\rangle \, dS \\
	&~~~+	\int_{\Gamma}  \left\langle \sigma(v)\vec{n},\left(\tfrac{\lambda}{\lambda + 2 \mu}\tr(M_{\Gamma}) \mathrm{I}+\,\leftidx{_\Gamma}{M}\right)\vec{n} \right\rangle  \, dS.
	\end{align*}
\end{proof}

\begin{thm}\label{Shape_Grad_LinEl:Thm: Shape Gradient Sur2} Let $ k\geq 2 $, $ \Omega \in \mathcal{O}_{k+1}^{b} $, $ \Omega_{t}=\T{t}(\Omega), \, t \in I_V  $ for some admissible vector field $V \in  \mathcal{V}_{k+1}^{ad}(\Oext)$. Suppose that  $ f\in C^{k-1,\phi}(\overline{\Oext},\R{3}) $ and $ g\in C^{k,\phi}(\overline{\Oext},\R{3})$ for some $ \phi\in (0,1) $. Let $ u=u(\Omega) \in C^{k,\phi}(\Omega,\R{3}) $ be the unique solution of \eqref{Reliability:Eq:LinEl}. Suppose that $ \mathcal{F}_{sur} \in C^{l,\psi}(\R{3}) $, $ l+\psi\geq 1,\, \psi\in [0,1] $ and let $ J_{2,sur}$ be given by \eqref{Shape_Grad_LinEl:Eg:J2sur}. \\
	Then the shape derivative of $ J_{2,sur}$ exists.
	
	\noindent If $ \mathcal{F}_{sur} $ is two times differentiable in $ z_3 $ then a weak adjoint equation to $ J_{2,sur}(\Omega) $ and the disjoint displacement traction problem of linear elasticity \eqref{Reliability:Eq:LinEl} is given by
	\begin{align}\label{Shape_Grad_LinEl:Eq: weak adjoint Eq Elasticity Fsur SigmaU}
	\int_{\Omega} \sigma(\vartheta):\varepsilon(p) \, dx 
	&= \int_{\Gamma_N} \left\langle h, \vartheta \right\rangle\, dS ~~~~~\forall \vartheta \in H^{1}_{D}(\Omega,\R{3})
	\end{align} 
	with
	\begin{equation*}
	\begin{split}
	h:= &\tfrac{\partial \mathcal{F}_{sur}}{\partial z_2}(.,u,\sigma(u))  
	+\kappa  \left[ \tilde{\lambda} \tr\left(M_{\Gamma}\right) I + \mu \,\leftidx{_\Gamma}{\big[M+M^{\top}\big]} \right] \vec{n}- \tilde{\lambda} \nabla_{\Gamma}\tr\left(M_{\Gamma}\right) \\
	&-\mu \, \Div_{\Gamma}\left( \big[M + M^{\top}\big]_{\Gamma} \right),
	\end{split}
	\end{equation*} 
	where
	$M = M(.,u,\sigma(u)):=\tfrac{\partial \mathcal{F}_{sur}}{\partial z_3}(.,u,\sigma(u)) $ and $ \tilde{\lambda} = \left(\lambda - \tfrac{\lambda^2}{\lambda + 2 \mu}\right).$
	The strong formulation reads
	\begin{equation}\label{Shape_Grad_LinEl:Eq: Adjoint Eq Elasticity Fsur u sigma} \tag{AS-2}
	\left. 
	\begin{array}{r c l l}
	\Div( \sigma(p)) & = & 0  &\text{ in } \Omega \\
	p & = & 0  &\text{ on } \Gamma_{D} \\
	\sigma(p) \, \vec{n}& = & h & \text{ on }\Gamma_{N}.  \\
	\end{array} 
	\right.
	\end{equation}
	
	\begin{itemize}
		\item[a)] Supposed that $ \mathcal{F}_{sur} $ is regular enough such that 
		$ h \in L^{4/3}(\Gamma,\R{3}) $. Then there exists a unique weak solution $ p \in H^1_{D}(\Omega,\R{3}) $ of  \eqref{Shape_Grad_LinEl:Eq: Adjoint Eq Elasticity Fsur u sigma} and the surface representation of the shape derivative is an element of $ C^{-1}(\Gamma,\R{3}) $.
		
		\noindent It is given by
		\begin{align*}
		\mathcal{G}(\Gamma): C^{1}(\Gamma,\R{3}) \to & \R{}\\
		w \mapsto & \int_{\Gamma_{N}} \hspace*{-2mm}\left(\tfrac{\partial \mathcal{ F}_{sur}}{\partial z_1}\vec{n}+\kappa \mathcal{F}_{sur} + \left\langle\tfrac{\partial \mathcal{F}_{sur}}{\partial z_2} , Du\, \vec{n} \right\rangle  + M : D(\sigma(u))[\vec{n}]\right)  w \, dS 
		\\
		& + \hspace*{-1mm} \int_{\Gamma_N} \hspace*{-2mm} \left\langle  f+\kappa g +Dg\, \vec{n} \, , \, \left(\tfrac{\lambda}{\lambda + 2 \mu}\tr(\,\leftidx{_\Gamma}{M}) \mathrm{I}+\,\leftidx{_\Gamma}{M} + M^{\top}\right)\vec{n} +p \right\rangle w  \, dS 
		\\
		& +\hspace*{-1mm} \int_{\Gamma_N} \hspace*{-2mm} \left\langle \Div_{\Gamma}(w\sigma_{\Gamma}(u))   \, , \, \left(\tfrac{\lambda}{\lambda + 2 \mu}\tr(\,\leftidx{_\Gamma}{M}) \mathrm{I}+\leftidx{_\Gamma}{M} + M^{\top}\right)\vec{n} +p \right\rangle  \, dS.
		\end{align*}
		\item[b)] If $ u \in C^{k+1}(\overline{\Omega},\R{3}) $ and $ \mathcal{F} \in C^{l+\psi}(\R{3}\times \R{3} \times \R{3 \times 3}) $, $ l+\psi> 3,\, 0< \psi<1$ then $ p \in H^1_{D}(\Omega,\R{3}) \cap W^{2,q}(\Omega,\R{3}) $ and $ p $ has a representation in $ C^{1,\phi}(\overline{\Omega},\R{3}) $. Moreover,
		\begin{align}\label{Shape_Grad_LinEl:Eq: Shape Gradient Fsur u sigma}
		G(\Gamma)
		&=   \tfrac{\partial \mathcal{ F}_{sur}}{\partial z_1}\vec{n}+\kappa \mathcal{F}_{sur} + \left\langle\tfrac{\partial \mathcal{F}_{sur}}{\partial z_2} , Du\,\vec{n}  \right\rangle  + M : D(\sigma(u))[\vec{n}]  
		\\ \nonumber
		&~~~+\left\langle  f+\kappa g +Dg\, \vec{n} \, , \, \left(\tfrac{\lambda}{\lambda + 2 \mu}\tr(\leftidx{_\Gamma}{M}) \mathrm{I}+\,\leftidx{_\Gamma}{M} + M^{\top}\right)\vec{n} +p \right\rangle
		\\
		&~~~ - \sigma_{\Gamma}(u): D_{\Gamma}\left[\left(\tfrac{\lambda}{\lambda + 2 \mu}\tr(\leftidx{_\Gamma}{M}) \mathrm{I}+\,\leftidx{_\Gamma}{M} + M^{\top}\right)\vec{n} +p\right]\,.\nonumber
		\end{align} 
	\end{itemize}
\end{thm}

\bigskip

\begin{rem}
	Note that the dependence of $ \mathcal{ F}_{sur} $ on $ (.,u,\sigma(u)) $ is neglected in the notation.
\end{rem}

\begin{proof}
	The shape derivative of $ J_{2,sur} $ is given by
	\begin{align*}
	dJ_{2,sur}(\Omega)[V] =& \int_{\Gamma_{N}} \left(\tfrac{\partial \mathcal{ F}_{sur}}{\partial z_1}\vec{n}+\kappa \mathcal{F}_{sur} + \left\langle\tfrac{\partial \mathcal{F}_{sur}}{\partial z_2} , Du\, \vec{n}  \right\rangle  + \tfrac{\partial \mathcal{F}_{sur}}{\partial z_3} : D(\sigma(u))[\vec{n}]\right)  V_{\vec{n}} \, dS 
	\\
	&+ \int_{\Gamma_{N}} \left\langle\tfrac{\partial \mathcal{F}_{sur}}{\partial z_2}, u'  \right\rangle + \tfrac{\partial \mathcal{F}_{sur}}{\partial z_3} :\sigma(u')  \, dS. 
	\end{align*}
	We split
	\begin{align*}
	& \int_{\Gamma_{N}}    \tfrac{\partial \mathcal{F}_{sur}}{\partial z_3} :\sigma(u') \, dS =  \int_{\Gamma_{N}}    M :\sigma(u') \, dS 
	=\mathcal{I}_{\Gamma} + \mathcal{I}_{\vec{n}}
	\end{align*}
	into two parts, where   
	\begin{equation}
	\mathcal{I}_{\Gamma} =\int_{\Gamma_{N}}  M:\sigma_{\Gamma}(u')  \, dS  ~~~~~ \text{ and } ~~~~~
	\mathcal{I}_{\vec{n}} = \int_{\Gamma_{N}} M:(\sigma(u')\vec{n}\vec{n}^{\top})\, dS.
	\end{equation}
	Using the Neumann condition on $ \Gamma_N $, we rewrite the integral $ \mathcal{I}_{\vec{n}} $ in the desired way:
	\begin{align*}
	\mathcal{I}_{\vec{n}}&
	=\int_{\Gamma_N} \left\langle \sigma(u')\vec{n}\, , \,M^{\top}\vec{n} \right\rangle \, dS
	=\int_{\Gamma_N} \left\langle  \left\lbrace f+\kappa g +Dg\, \vec{n}\right\rbrace  V_{\vec{n}} + \Div_{\Gamma}(V_{\vec{n}}\sigma_{\Gamma}(u))   \, , \,M^{\top}\vec{n} \right\rangle \, dS.
	\end{align*}
	Now we rephrase $ \mathcal{I}_{\Gamma} $ by  integrating $  M:\sigma_{\Gamma}(u') $ by parts on $ \Gamma_N $ using  Proposition \ref{Transf_shape_opt:Prop: Tangential Stokes Sigma}.
	We obtain 
	\begin{align*}
	&\int_{\Gamma_N} M:\sigma_{\Gamma}(u')  \, dS
	\\
	&= \int_{\Gamma_{N}} \hspace*{-2mm}\left\langle \kappa  \left[ \tilde{\lambda} \tr\left(\,\leftidx{_\Gamma}{M}\right) \mathrm{I}+ \mu \,\leftidx{_\Gamma}{\big[M+M^{\top}\big]} \right] \vec{n}- \tilde{\lambda} \nabla_{\Gamma}\tr\left(M_{\Gamma}\right) -\mu \Div_{\Gamma}\left( \big[M + M^{\top}\big]_{\Gamma} \right),u'   \right\rangle \, dS \\
	&~~~+	\int_{\Gamma_{N}}  \left\langle \sigma(u')\vec{n},\left(\tfrac{\lambda}{\lambda + 2 \mu}\tr(\,\leftidx{_\Gamma}{M}) \mathrm{I}+\,\leftidx{_\Gamma}{M}\right)\vec{n} \right\rangle  dS
	\end{align*}
	and thus
	\begin{align}\label{Shape_Grad_LinEl:Eq:int M:sigma}
	&\int_{\Gamma_{N}}   M :\sigma(u') \, dS \nonumber  \\
	&= \int_{\Gamma_N}\hspace*{-2mm} \left\langle \sigma(u')\vec{n}  \, , \, \left(\tfrac{\lambda}{\lambda + 2 \mu}\tr(\,\leftidx{_\Gamma}{M}) \mathrm{I}+\,\leftidx{_\Gamma}{M} + M^{\top}\right)\vec{n} \right\rangle  \, dS
	\\
	&~~~+ \int_{\Gamma_{N}} \hspace*{-2mm}\left\langle \kappa  \left[ \tilde{\lambda} \tr\left(M_{\Gamma}\right) \mathrm{I} + \mu \,\leftidx{_\Gamma}{\big[M+M^{\top}\big] }\right] \vec{n}- \tilde{\lambda} \nabla_{\Gamma}\tr\left(M_{\Gamma}\right) -\mu \Div_{\Gamma}\left( \big[M + M^{\top}\big]_{\Gamma} \right),u'   \right\rangle  dS\, . \nonumber 
	\end{align}
	The Eulerian derivative of $ J_{2,sur} $ can now be rewritten in the desired form:
	\begin{align*}
	&dJ_{2,sur}(\Omega)[V] \\
	&= \int_{\Gamma_{N}} \hspace*{-1.5mm}\left[\tfrac{\partial \mathcal{ F}_{sur}}{\partial z_1}\vec{n}+\kappa \mathcal{F}_{sur} + \left\langle\tfrac{\partial \mathcal{F}_{sur}}{\partial z_2} , \tfrac{\partial u }{\partial \vec{n}}  \right\rangle  + M \hspace*{-.5mm}:\hspace*{-.5mm} D(\sigma(u))[\vec{n}]\right] \hspace*{-.5mm} V_{\vec{n}} + \left\langle\tfrac{\partial \mathcal{F}_{sur}}{\partial z_2}, u'  \right\rangle +  M \hspace*{-.5mm}:\hspace*{-.5mm}\sigma(u') \, dS 
	\\[1ex]
	&
	= \int_{\Gamma_{N}}\hspace*{-1.5mm} \left[\tfrac{\partial \mathcal{ F}_{sur}}{\partial z_1}\vec{n}+\kappa \mathcal{F}_{sur} + \left\langle\tfrac{\partial \mathcal{F}_{sur}}{\partial z_2} , \tfrac{\partial u }{\partial \vec{n}}  \right\rangle  + M \hspace*{-.5mm} :\hspace*{-.5mm} D(\sigma(u))[\vec{n}]\right] \hspace*{-0.5mm} V_{\vec{n}} \, dS 
	\\
	&~~~+ \int_{\Gamma_{N}} \hspace*{-1.5mm} \left\langle \sigma(u')\vec{n}  \, , \, \left(\tfrac{\lambda}{\lambda + 2 \mu}\tr(\leftidx{_\Gamma}{M}) \mathrm{I}+\,\leftidx{_\Gamma}{M} + M^{\top}\right)\vec{n} \right\rangle + \left\langle\tfrac{\partial \mathcal{F}_{sur}}{\partial z_2}, u'  \right\rangle    \, dS \\ 
	&~~~+ \int_{\Gamma_{N}} \hspace*{-2mm}\left\langle \kappa \hspace*{-.5mm} \left[ \tilde{\lambda} \tr\left(\leftidx{_\Gamma}{M}\right) \mathrm{I} + \mu \,\leftidx{_\Gamma}{\big[M+M^{\top}\big] }\right] \vec{n}- \tilde{\lambda} \nabla_{\Gamma}\tr\left(M_{\Gamma}\right) -\mu \Div_{\Gamma}\left( \big[M + M^{\top}\big]_{\Gamma} \right),u'   \right\rangle  dS 
	\\
	&
	= \int_{\Gamma_{N}} \left[\tfrac{\partial \mathcal{ F}_{sur}}{\partial z_1}\vec{n}+\kappa \mathcal{F}_{sur} + \left\langle\tfrac{\partial \mathcal{F}_{sur}}{\partial z_2} , \tfrac{\partial u }{\partial \vec{n}}  \right\rangle  + M : D(\sigma(u))[\vec{n}]\right]  V_{\vec{n}} \, dS \\
	&~~~+ \int_{\Gamma_{N}} \left\langle \sigma(u')\vec{n}  \, , \, \left(\tfrac{\lambda}{\lambda + 2 \mu}\tr(\,\leftidx{_\Gamma}{M}) \mathrm{I}+\,\leftidx{_\Gamma}{M} + M^{\top}\right)\vec{n} \right\rangle + \left\langle h, u'  \right\rangle    \, dS.
	\end{align*}
	Again, Lemma \ref{Parameter_Dep_PDE:Thm: Lax-Milgram} implies that there is a unique weak solution $ p \in H^{1}_{D}  $ of  \eqref{Shape_Grad_LinEl:Eq: weak adjoint Eq Elasticity Fsur SigmaU}. Then, we use the relation $ \sigma(u')\vec{n}= (f+ \kappa g + Dg\, \vec{n})V_{\vec{n}} + \Div_{\Gamma}(V_{\vec{n}}\sigma_{\Gamma}(u)) $ to deduce
	\begin{align*}
	dJ_{2,sur}(\Omega)[V]
	&= \int_{\Gamma_{N}} \left(\tfrac{\partial \mathcal{ F}_{sur}}{\partial z_1}\vec{n}+\kappa \mathcal{F}_{sur} + \left\langle\tfrac{\partial \mathcal{F}_{sur}}{\partial z_2} , \tfrac{\partial u }{\partial \vec{n}}  \right\rangle  + M : D(\sigma(u))[\vec{n}]\right)  V_{\vec{n}} \, dS 
	\\
	&~~~ + \int_{\Gamma_N} \left\langle  f+\kappa g +Dg\, \vec{n} \, , \, \left(\tfrac{\lambda}{\lambda + 2 \mu}\tr(\,\leftidx{_\Gamma}{M}) \mathrm{I}+\,\leftidx{_\Gamma}{M} + M^{\top}\right)\vec{n} \right\rangle V_{\vec{n}}  \, dS 
	\\
	&~~~ + \int_{\Gamma_N} \left\langle \Div_{\Gamma}(V_{\vec{n}}\sigma_{\Gamma}(u))   \, , \, \left(\tfrac{\lambda}{\lambda + 2 \mu}\tr(\,\leftidx{_\Gamma}{M}) \mathrm{I}+\,\leftidx{_\Gamma}{M} + M^{\top}\right)\vec{n} \right\rangle  \, dS + \int_{\Gamma_{N}} \hspace*{-2mm}  
	\langle h,u' \rangle\, dS.
	\end{align*}
	Since $ p $ and $ u' $, respectively, satisfy 
	\begin{align*}
	\int_{\Omega} \sigma(\vartheta):\varepsilon(p) \, dx  
	&= \int_{\Gamma_N} \left\langle h, \vartheta \right\rangle\, dS,
	\\
	\int_{\Omega} \sigma(u'):\varepsilon(z) \, dx &= \int_{\Gamma_{N}}\left\langle (f+\kappa g +Dg\, \vec{n})V_{\vec{n}} +\Div_{\Gamma}(V_{\vec{n}}\sigma_{\Gamma}(u)), z \right\rangle \, dS,
	\end{align*}
	for all $ \vartheta,\, z \in  H^1_{D}(\Omega,\R{3})$ and we derive
	\begin{align*}
	\int_{\Gamma_N} \left\langle h,u' \right\rangle \, dS  = \int_{\Omega} \sigma(p): \varepsilon(u') \, dx 
	= \int_{\Gamma_{N}}\left\langle (f+\kappa g +Dg\, \vec{n})V_{\vec{n}} +\Div_{\Gamma}(V_{\vec{n}}\sigma_{\Gamma}(u)), p \right\rangle \, dS\,.
	\end{align*}
	And thus for $ p \in H^1_{D}(\Omega,\R{3}) $
	\begin{align*}
	dJ_{2,sur}(\Omega)[V]
	&= \int_{\Gamma_{N}} \left(\tfrac{\partial \mathcal{ F}_{sur}}{\partial z_1}\vec{n}+\kappa \mathcal{F}_{sur} + \left\langle \tfrac{\partial \mathcal{F}_{sur}}{\partial z_2} , \tfrac{\partial u }{\partial \vec{n}}  \right\rangle  + M : D(\sigma(u))[\vec{n}]\right)  V_{\vec{n}}\, dS 
	\\
	&~~~ + \int_{\Gamma_N} \left\langle  f+\kappa g +Dg\, \vec{n} \, , \, \left(\tfrac{\lambda}{\lambda + 2 \mu}\tr(\,\leftidx{_\Gamma}{M}) \mathrm{I}+\,\leftidx{_\Gamma}{M} + M^{\top}\right)\vec{n} +p \right\rangle V_{\vec{n}}  \, dS 
	\\
	&~~~ + \int_{\Gamma_N} \left\langle \Div_{\Gamma}(V_{\vec{n}}\sigma_{\Gamma}(u))   \, , \, \left(\tfrac{\lambda}{\lambda + 2 \mu}\tr(\,\leftidx{_\Gamma}{M}) \mathrm{I}+\,\leftidx{_\Gamma}{M} + M^{\top}\right)\vec{n} +p \right\rangle  \, dS\,.
	\end{align*}
	Then analogous arguments to those used in the proof of  Proposition \ref{Shape_Grad_LinEl:Thm: Shape Gradient Vol Functional u, sigma(u)} conclude the proof of a).\\
	
	\noindent b) Supposed that $ u \in C^{k+1}(\Omega,\R{3}) $ and $ \mathcal{F}_{sur} \in C^{l+\psi}(\R{3} \times \R{3} \times \R{3 \times 3}),\, l+\psi>3 $, then $ h $ is at least an element of $ C^{1}(\Gamma,\R{3}) \subset W^{1-\nicefrac{1}{q},q}(\Gamma,\R{3})$ for any $ 1 \leq q \leq \infty $. Since $ q \geq 4,\, 0<\phi<1-\frac{3}{q} $ can be chosen the solution $ p $ is an element of $ W^{2,q} \hookrightarrow C^{1,\phi} $. Thus the assertion again follows by integration by parts. 
\end{proof}

\begin{rem}
	The Assumption $  u \in C^{k+1} $ can be achieved by enhancing the regularity of the boundary of $ \Omega $ from $ C^{k+1} $ to $ C^{k+1,\varphi} $ for arbitrary $ \varphi>0 $ since the regularity of $ f $ and $ g $ is already high enough.
\end{rem}

\noindent The following table shows the Hölder exponents and thus the regularities of the terms appearing in $ G(\Gamma) $ for different regularities of $ \mathcal{F}_{2,sur} $. The regularity of the gradient is the minimum over these regularities in the respective case, since the regularity of $ G(\Gamma)$ is determined by the term with the lowest regularity appearing in the formula. 

Let $ k\geq 2 $, $ \Omega \in \mathcal{O}_{k+2}^{b} $, $V \in  \mathcal{V}_{k+2}^{ad}(\Oext)$,  $ f\in C^{k-1,\phi}(\overline{\Oext},\R{3}) $ and $ g\in C^{k,\phi}(\overline{\Oext},\R{3})$ for some $ \phi\in (0,1) $. Let $ u=u(\Omega) \in C^{k+1,\phi}(\Omega,\R{3}) $ be the unique solution of \eqref{Reliability:Eq:LinEl} and $ \mathcal{F}_{sur} \in C^{l,\psi}(\R{3} \times \R{3} \times \R{3 \times 3}) $. Then the regularity of $ G(\Gamma) $ can be classified as the following tabular shows.

\begin{table}[htb]
	\centering
	\begin{small}
		\begin{tabular}{|c||c|c|c|} \hline
			$ \mathcal{ F}_{sur} ~/~ l+\psi $					& $ 3 $			    & $ (3,k+1+\phi) $ 	  & $ [k+1+\phi,\infty) $    \\ \hline\hline
			$ p $ 						& $ 1+\phi $       & $ l-1+\phi $    & $ k+\phi $ 	\\ \hline \hline
			$  f+\kappa g + Dg\, \vec{n} $ 
			& $ k-1+\phi $           & $ k-1+\phi $ 		  & $ k-1+\phi $ 		 \\ \hline
			$ D_{\Gamma}p $				& $ 0+\phi $        & $ l-2+\psi $ 	  & $ k-1+\phi$  \\ \hline
			$ D_{\Gamma}u $				& $k+\phi $              & $k+\phi $  		  & $k+\phi $       \\ \hline
			$ D_{\Gamma}\left(\frac{\partial \mathcal{F}_{sur}}{\partial z_3}^{\top}\vec{n}\right) $					& $1$               & $l-2+\psi $       & $ k-1+\phi $ \\ \hline\hline
			$ G(\Gamma) $				& $ 0+\phi $        & $l-2+\psi $		  & $ k-1+\phi $		 \\ \hline
		\end{tabular} \\
	\end{small}
	\caption{Regularities of the $ L^2 $- Hadamard representation $ G(\Gamma) $ for functionals of \\ the type $ J_{2,sur} $}
	\label{Tab: Reg Grad F2sur}
\end{table}
\noindent Note that the gradient takes a maximal regularity of $C^{k-1+\phi}$ (depending on the regularity of $ \mathcal{F}_{sur} $) which is more than two regularities less than the boundary regularity.

\noindent Even if  $ f\in C^{k,\phi}(\overline{\Oext},\R{3}) $ and $ g\in C^{k+1,\phi}(\overline{\Oext},\R{3})$ for some $ \phi\in (0,1) $ then $ u=u(\Omega) \in C^{k+1,\phi}(\Omega,\R{3}),\, p=p(\Omega) \in C^{k,\phi}(\Omega,\R{3})  $. In this case the regularity of $ D_{\Gamma}p $ bounds the regularity of $ G(\Gamma) $ from below. We discuss an approach on how this behavior can potentially be avoided in Section \ref{outlook: Bound. and Sol. Reg. }.

\begin{rem}
	The weak formulation of equation \eqref{Shape_Grad_LinEl:Eq: Adjoint Eq Elasticity Fsur u sigma} is formally equivalent to 
	\begin{align*}
	&\int_{\Omega} \sigma(\vartheta):\varepsilon(p) \, dx \\
	&=  \int_{\Gamma_N} \frac{\partial \mathcal{F}_{sur}}{\partial z_3}(.,u,\sigma(u)):\sigma(\vartheta) + \left\langle \frac{\partial \mathcal{F}_{sur}}{\partial z_2}(.,u,\sigma(u)),\vartheta \right \rangle \, dS \\
	&+ \int_{\Gamma_N} \left\langle(f+\kappa g + Dg \,\vec{n})V_{\vec{n}} + \Div_{\Gamma}(V_{\vec{n}}\sigma_{\Gamma}(u)) \, , \, \left(\tfrac{\lambda}{\lambda + 2 \mu}\tr(M_{\Gamma}) \mathrm{I}+\,\leftidx{_\Gamma}{M} + M^{\top}\right)\vec{n} \right\rangle \, dS
	\end{align*}
	since 
	\begin{align*}
	&\int_{\Gamma_N} \frac{\partial \mathcal{F}_{sur}}{\partial z_3}(.,u,\sigma(u)):\sigma(\vartheta) + \left\langle \frac{\partial \mathcal{F}_{sur}}{\partial z_2}(.,u,\sigma(u)),\vartheta \right \rangle \, dS \\
	&= \int_{\Gamma_N} \hspace*{-2mm}\langle h,\vartheta \rangle + \left\langle (f+\kappa g + Dg\, \vec{n})V_{\vec{n}} + \Div_{\Gamma}(V_{\vec{n}}\sigma_{\Gamma}(u)), \left[\tfrac{\lambda}{\lambda + 2 \mu}\tr(M_{\Gamma}) \mathrm{I}+\,\leftidx{_\Gamma}{M} + M^{\top}\right] \hspace*{-1mm} \vec{n} \right\rangle dS
	\end{align*}
	holds according to equation \eqref{Shape_Grad_LinEl:Eq:int M:sigma} and the definition of the vector field $ h $.
\end{rem}

\begin{thm}
	Let $ \Omega $ be of class $C^{4} $, $ V \in C^{4}_{0}(\Oext,\R{3}) $, $ f \in C^{2,\phi}(\overline{\Oext},\R{3}) $ and $ g \in C^{3,\phi}(\overline{\Oext},\R{3}) $. If the material constants $ E,\, K,\, \hat{n},\, b,\, c,\, \sigma_{f},\,\epsilon_{f}$ are given as in 
	Lemma \ref{Shape_Grad_LinEl:Prop: Differentiability LCF}.
	Then $ J^{\mathrm{lcf}} $ is shape differentiable and the $ L^{2}(\Gamma) $-shape gradient, given by
	\begin{equation}
	\begin{split}
	G^{\mathrm{lcf}}(\Gamma)
	&= M : D(\sigma(u))[\vec{n}]  - \sigma_{\Gamma}(u): D_{\Gamma}\left[\left(\tfrac{\lambda}{\lambda + 2 \mu}\tr(\leftidx{_\Gamma}{M}) \mathrm{I}+\,\leftidx{_\Gamma}{M} + M^{\top}\right)\vec{n} +p\right]
	\\
	&~~~+\left\langle  f+\kappa g +Dg\, \vec{n} \, , \, \left(\tfrac{\lambda}{\lambda + 2 \mu}\tr(\leftidx{_\Gamma}{M}) \mathrm{I}+\,\leftidx{_\Gamma}{M} + M^{\top}\right)\vec{n} +p \right\rangle,
	\end{split}
	\end{equation} 
	is an element of $ C^{1,\phi}(\Omega,\R{3}) $ where $ u \in C^{3,\phi}(\overline{ \Omega},\R{3}) $ is the unique solution of  \eqref{Reliability:Eq:LinEl} and $ p \in  C^{2,\phi}(\overline{ \Omega},\R{3})  $ is the unique solution of \eqref{Shape_Grad_LinEl:Eq: Adjoint Eq Elasticity Vol u, sigma(u)}. Moreover $ M = M(\sigma(u)) := \frac{\partial \mathcal{ F}^{\mathrm{lcf}}}{\partial \sigma}(\sigma(u)) $.
\end{thm}

\begin{proof}
	Apply Lemma \ref{Shape_Grad_LinEl:Prop: Differentiability LCF} and Theorem \ref{Shape_Grad_LinEl:Thm: Shape Gradient Sur2} and consider Table \ref{Tab: Reg Grad F2sur}.
\end{proof}

Note that in this case the derivative $ \frac{\partial \mathcal{ F}^{\mathrm{lcf}}}{\partial \sigma} $ can only be computed numerically e.g. by a newton scheme. The function $ \mathcal{ F}^{\mathrm{lcf}} =CMB^{-1} \circ RO \circ SD^{-1} \circ VM \circ TF $ is certainly differentiable but $ CMB^{-1}  $ and $ SD^{-1} $ has no representation by elementary functions and thus two nonlinear equations have to be solved.


\chapter{Perspective on Further Research}\label{Outlook}

\section{A glance at reduced regularity requirements}\label{outlook: Reduces Requirements}

As the explanations in Section \ref{Reliability:Sec:Reg_Req} already indicate $ C^{2,\phi} $-solutions are not an optimal choice in view of the minimal regularity that is required for the construction of the Shape functional $ \mathcal{F}^{\mathrm{lcf}} $ - here $ W^{2,p}(\Omega,\R{3}) $ with $ p $ large enough would be sufficient. Moreover, the theorem on the existence of the shape derivative 
requires only $ C^{1} $ - material derivatives. 
Thereof the question arises, if the regularity of $ u $ can be reduced without loosing either the property of well definedness of the functional or the shape differentiability.  

Let us suppose that $ \Omega $ is a domain of class $ C^2 $ instead of $ C^{2,\phi} $ (\footnote{Note that in this case the regularity can not be reduced to $ C^{1,\phi} $ since Theorem  \ref{Linear_Elasticity:Thm:Weak_Reg_LinEl} requires at least $ C^{2} $}), $ V \in \Vad{2}{\Omega^{ext}} $, $ f \in C^{1}(\overline{\Omega^{ext}},\R{3}) $ and $ g \in C^{2}(\overline{\Omega^{ext}},\R{3}) $. With $ V $ we again associate the family $ T_t=T_t[V],\, t \in I= I_{V}$ of $ C^{2} $-transformations obtained from the ODE \eqref{Transf_shape_opt:Eq: ODE} or the family $ \Phi_t[V],\, t \in I=(-\epsilon,\epsilon) $ obtained from \eqref{Shape_Grad_LinEl: Gradient Flow Equation} and the family $ (\Omega_{t})_{t \in I} $ of $ C^{2} $-shapes. 

Under these assumptions Theorem \ref{Ex_Shape_Deriv_LinEl:Thm: existence of strong H^1 material derivatives} clearly holds true and $ t \to u^{t} = u_t \circ T_t \in H^{1}(\Omega,\R{3}) $ is continuously differentiable where the solution $ u_t $ of equation \eqref{Ex_Shape_Deriv_LinEl:Eq: LinEl Omega_t} on $ \Omega_t $ is an element of $ H^1(\Omega,\R{3}) $. 

Furthermore, let $  f \in W^{1,p}(\Omega,\R{3}) $ and $ W \in W^{2,p}(\Omega,\R{3}) $ or rather $ g\vert_{\Gamma} \in W^{2-\nicefrac{1}{p},p}(\Gamma,\R{3}) $ for any $ 1 \leq p < \infty $. Then, since $ L^p \subset W^{1,p}  $, \eqref{Ex_Shape_Deriv_LinEl:Eq: LinEl Omega_t} admits a solution $  u_t \in W^{2,p}(\Omega_t,\R{3}) $ for any $ t \in I $, consider Theorem \ref{Linear_Elasticity:Thm:Weak_Reg_LinEl}. This is exactly the regularity which is needed to assure that the functional $ J^{\mathrm{lcf}}(\Omega_t,\sigma(u_t)) $ is defined for any $ \Omega_t $. 

\noindent Additionally, this solution satisfies, see the original work \cite{Agm64} or equation \eqref{Linear_Elasticity:Eq:Schauder_Estimate_u_Sobolev}
\[ \Norm{u_t}{W^{2,p}(\Omega,\R{3})} \leq C(\Omega_t)\left(\Norm{f}{W^{1,p}(\Omega_t,\R{3})}) + \Norm{g}{W^{1 - \nicefrac{1}{p},p}(\Omega_t,\R{3})} + \Norm{u_t}{L^1(\Omega_t,\R{3}) }\right) \]
and the right hand sides of the PDE which determine the material derivative $ \dot{u} $, i.e.
\enlargethispage{\baselineskip}
\begin{align} 
\left. 
\begin{array}{rcll}
-\Div( \sigma(q_t)) &=&f_{V}+f_{u_t}   &\text{ in } \Omega_t \\
q_t &=& 0  &\text{ on } \Gamma_{D,t}  \\
\se(q_t) \vec{n}_t &=&g_{V}-G_{u_t}\vec{n}_t &\text{ on }\Gamma_{N,t}
\end{array} 
\right.
\end{align}
satisfy $ f_V + f_{u_t} \in L^{p}(\Omega_t,\R{3})$ and $ g_{V}-G_{u_t}\vec{n}_t \in W^{1-\nicefrac{1}{p},p}(\Gamma_{t},\R{3}) $  what implies $ q_t \in W^{2,p}(\Omega_{t},\R{3}) $ and 
\[ \Norm{q_t}{W^{2,p}(\Omega,\R{3})} \leq C(\Omega_t)\left[\Norm{ f_V + f_{u_t} }{W^{1,p}(\Omega_t,\R{3})} + \Norm{g_{V}-G_{u_t}\vec{n}_t }{W^{1 - \nicefrac{1}{p},p}(\Omega_t,\R{3})} + \Norm{q_t}{L^1(\Omega_t,\R{3}) }\right]\hspace*{-1mm}. \]

In this case we have $ u^t,\, q^t \in W^{2,p}(\Omega,\R{3})$  such that $ t \mapsto u^{t} $ is continuously differentiable w.r.t. the strong topology on $ H^1(\Omega,\R{3}) $ with $ \dot{u}^t =q^{t} $. 
Moreover, the Rellich-Kondrachov Theorem \ref{App: Rellich Kondrachov} yields that $ W^{2,p}(\Omega_t,\R{3}) \hookrightarrow C^{1,\phi}(\overline{\Omega}_t,\R{3}) $ is a continuous and compact embedding for any $ p\geq 4 $ such that $ u_t,\, q_t \in  C^{1,\phi}(\Omega_t,\R{3}) $ for any $ 0\leq \phi \leq 1 - \frac{3}{p}   $. 
To apply again Lemma \ref{Lem: Topo-Lemma} we thus  define the following chain of topologies $$ W^{2,p}(\Omega_t,\R{3}) \subset  C^{1,\phi}(\overline{ \Omega}_t,\R{3}) \subset H^{1}(\Omega,\R{3})\, .$$

\noindent \textbf{What is left to show?}\\
We have to prove that $ u^{t} $ and $ q^{t} $ are uniformly bounded in $ W^{2,p} $. The central point in the proof of this statement will be showing that the constant $ C(\Omega_t) $ can be chosen uniformly regarding the parameter $ t \in (-\epsilon,\epsilon) $. This will require a detailed investigation of the construction of $ C(\Omega_t) $. But again the constant $ C(\Omega_t) $ depends mainly on $ \Omega_t $ through the hemisphere transformations which involve the transformations $ T_t $, compare also Lemma \ref{Ex_Shape_Deriv_LinEl:Lem: uniform hemisphere property Omega_t} and Proposition \ref{Ex_Shape_Deriv_LinEl:Prop: properies of solutions u_t of DTP on Omega_t}.

In this case $ u' \in C^{0,\varphi}(\overline{\Omega},\R{3}) $ is no longer differentiable, thus the shape derivative in the local shape derivative form is no longer defined and the shape gradient can not be derived as explained in Chapter \ref{Shape_Grad_LinEl}. Nevertheless, the functional $ \mathcal{J}^{\mathrm{lcf}}(\Omega,\sigma(u(\Omega))) $ is still shape differentiable and the shape derivative can be calculated in material derivative form using Lemma \ref{Transf_shape_opt:Lem: d/dt int(T_t,u_t,Du_t) in material derivative form}. 
Also the adjoint method can be applied, but usually the functional $ dJ(\Omega)[V] $ decomposes no longer into a product of $ V $ and a function $ G $, see also the following Section \ref{Shape_Grad_LinEl:Sec: Distributed Shape Derivative}.  

\vspace{2\parindent}

\section{Towards shape flows}\label{Shape_Grad_LinEl:Sec: Distributed Shape Derivative}

As illustrated above, the $ L^2 $-surface representation of the shape derivative  is often not available even if the functional is shape differentiable. Even if it can be deduced, the gradient representation $ G(\Gamma) $ is too irregular to maintain the domain regularity of $ \Omega $, consider the tables \ref{Tab: Reg Grad Fvol} - \ref{Tab: Reg Grad F2sur}, i.e. if the shape is of class $ C^{k},\, k \geq 3 $ then $ G(\Gamma) $ takes a maximal regularity of $ C^{k-2} $, sometimes even $ C^{k-3 +\phi} $. In this case it is impossible to derive a flow along a descent direction according to equation \eqref{Shape_Grad_LinEl: Gradient Flow Equation}. Thus at least the following approaches seem to be suitable.
\newpage
\begin{itemize}
	\item[1.]It is possible to look for descent directions defined on $ \Gamma $ w.r.t. other scalar products instead of the $ L^2(\Gamma) $-scalar product. In this context usually a partial differential equation has to be solved on the surface of $ \Omega $ to  obtain a descent direction $ W(\Gamma) $. In this way the regularity of the descent direction can be enhanced. 
	\item[2.] If this surface representation is not available, it is still possible to derive the weak (volume/distributed) shape derivative $ dJ(\Omega)[V] $ and to define descent directions according to the approach proposed in \cite{SturmLaurain2016distributed,Sturm2015shape} also for $V\in C^{0,1}_{0}(\Oext,\R{3})$.
	\item[3.] In the case of $ C^{\infty} $-domains, the set of admissible shapes can be described as the manifold 
	$$ B_{e} := Diff^\infty(\mathbb{S}^{2},\R{3})/ Diff^\infty(\mathbb{S}^{2},\mathbb{S}^{2}) $$
	
	with tangential space $ \mathcal{T}_{B_{e}}  \cong \{ h\,| \, h= v \vec{n},\, v \in C^{\infty}(\mathbb{S}^2,\R{}) \} $. \cite{Schulz2014riemannian,Michor2003ShapeManifold} This means if the surface representation of a shape functional $ J $ is available such that \[ dJ(\Omega)[V] = \int_{\Gamma} \langle G(\Gamma)\vec{n},V \rangle \, dS = \int_{\Gamma} G(\Gamma) V_{\vec{n}} \, dS =dJ(\Gamma)[V_{\vec{n}}] \] with $ G(\Gamma) \in C^{\infty}(\mathbb{S}^2,\R{})  $, then $ G(\Gamma)\vec{n} $ is an element of the tangent space and the  \textit{shape gradient} w.r.t. to the inner product
	\[ m_{L^2}:\mathcal{T}_{B_{e}} \times \mathcal{T}_{B_{e}} \to \R{},\, (W,V) \mapsto \int_{\Gamma} wv \, dS\]
	where $ V=v\vec{n} $ and $ W = w \vec{n} $. In this context also other metrics can be considered, as there are for example the $ H^1(\Gamma) $-metric or the Steklov-Poincaré metric, consider \cite{Michor2007Metrics,Michor2007Metrics,Schulz2015PdeShapeManifolds,Schulz2016Steklov,Schulz2016metricsComparison}. Meanwhile, also other shape spaces with diffeological structure \cite{Iglesias2013diffeology} are considered, see \cite{Welker2017optimization}.
\end{itemize}

We will give a short outlook on these approaches and their similarities, advantages and disadvantages in the context of $ C^{k} $-shapes here and start with the weak/distributed shape derivative formulation.

In the last years this representation gained more popularity in the shape optimization community since it has computational advantages and can be derived under weaker assumptions, see Theorem \ref{Transf_shape_opt:Thm:Hadamard_Thm}, and consider \cite{SturmLaurain2016distributed,Schmidt2018weak} also.

As mentioned above, the approach to derive the adjoint equation for the distributed shape derivative is exactly the same as it is for the Hadamard representation: 
The terms that contain the material derivative $ \dot{u} $ are sat to the right hand side of the adjoint equation. And since these terms even have the same structure also the adjoint equations remain exactly the same.

The shape derivative in material derivative form has the following structure:
\begin{align*}
dJ(\Omega)[V] = 
&  \int_{\Omega} \left\langle\frac{\partial \mathcal{F}_{vol}}{\partial z_2} (.,u,\sigma(u)) ,\dot{u}\right\rangle+ \frac{\partial \mathcal{F}_{vol}}{\partial z_3} (.,u,\sigma(u) : \sigma(\dot{u})\, dx \\
&+  \int_{\Gamma}  \left\langle\frac{\partial \mathcal{F}_{sur}}{\partial z_2} (.,u,Du) ,\dot{u} \right\rangle + \frac{\partial \mathcal{F}_{sur}}{\partial z_3} (.,u,\sigma(u) : \sigma(\dot{u}) \, dS. \nonumber \\
&+ \text{"terms containing $ V $ or derivatives thereof".}  \\
=&\,  \int_{\Omega} \left\langle\frac{\partial \mathcal{F}_{vol}}{\partial z_2} (.,u,\sigma(u)) ,u' \right\rangle + \frac{\partial \mathcal{F}_{vol}}{\partial z_3} (.,u,\sigma(u)):\sigma(u')\, dx
\\
&+  \int_{\Gamma}  \left\langle\frac{\partial \mathcal{F}_{sur}}{\partial z_2} (.,u,\sigma(u)), u'   \right\rangle + \frac{\partial \mathcal{F}_{sur}}{\partial z_3} (.,u,\sigma(u)) :\sigma(u')  \, dS\\
&+ \text{ "terms with $ V_{\vec{n}} $"\, .}
\end{align*}

The only formal difference is the fact that the terms involving $ u' $ before were substituted by the weak formulation of the right hand side of \eqref{Ex_Shape_Deriv_LinEl:Eq: PDE for u' g(Gamma), f(Omega)}, i.e.
\[ \int_{\Gamma_N}\langle f + \kappa g + Dg \vec{n},p \rangle V_{\vec{n}} + \langle\Div_{\Gamma}(V_{\vec{n}} \sigma_{\Gamma}(u)),p \rangle \, dS\] and now the terms containing $ \dot{u} $ are replaced by the right hand side of \eqref{Ex_Shape_Deriv_LinEl:Eq: Weak equation for dot_u_t} 
\[ \int_{\Omega} \langle f_V,p \rangle  + \tr\left((DV\sigma(u) + \dot{\sigma}(u) + \Div(V)\sigma(u))Dv \right) \, dx + \int_{\Gamma_N}   \langle g_V,p \rangle \,dS .  \]
\noindent Exemplarily, we carry out the calculations in the case of the volume functional $  J_{vol}(\Omega) $, see equation \ref{Shape_Grad_LinEl:Eq:Jvol}, formally:
\begin{align*}
dJ_{vol}(\Omega)[V]
=&\, \int_{\Omega}
\Div(V) \mathcal{F}_{vol}(.,u,Du) \, dx  \nonumber \\
&+ \int_{\Omega} \left\langle \frac{\partial \mathcal{F}_{vol}}{\partial z_1} (.,u,Du ),V \right\rangle -  \frac{\partial \mathcal{F}_{vol}}{\partial z_3} (.,u,Du) : (DuDV) \, dx \nonumber  \\
&+  \int_{\Omega} \left\langle\frac{\partial \mathcal{F}_{vol}}{\partial z_2} (.,u,Du) ,\dot{u} \right\rangle + \frac{\partial \mathcal{F}_{vol}}{\partial z_3} (.,u,Du) : D\dot{u} \, dx\, .
\end{align*}
Thus the equation \eqref{Shape_Grad_LinEl:Eq: Adjoint Eq Elasticity Vol u, sigma(u)} leads to
\begin{align*} 
\int_{\Omega}  \sigma(w):\varepsilon(p) \, dx=& \int_{\Omega} \left\langle\frac{\partial \mathcal{F}_{vol}}{\partial z_2}(.,u,Du),w \right\rangle + \frac{\partial \mathcal{F}_{vol}}{\partial z_3}(.,u,Du):Dw\, dx  \, \forall w \in H^1_{D} 
\intertext{ and }
\int_{\Omega}  \sigma(\dot{u}):\varepsilon(z) \, dx =& \int_{\Omega} \langle f_V,z \rangle  + \tr\left((DV\sigma(u) + \dot{\sigma}(u) + \Div(V)\sigma(u))Dz \right) \, dx \\
&+ \int_{\Gamma_N}   \langle g_V,z \rangle \,dS  \, \forall z \in H^{1}_{\Gamma_D}.
\end{align*}
Since we have shown that both equations have unique solutions in $ H^1_{D} $ we obtain 
\begin{align*}
&\int_{\Omega} \left\langle\frac{\partial \mathcal{F}_{vol}}{\partial z_2}(.,u,Du),\dot{u} \right\rangle + \frac{\partial \mathcal{F}_{vol}}{\partial z_3}(.,u,Du):D\dot{u}\, dx \\
&= \int_{\Omega} \langle f_V,z \rangle  + \tr\left((DV\sigma(u) + \dot{\sigma}(u) + \Div(V)\sigma(u))Dp \right) \, dx + \int_{\Gamma_N}   \langle g_V,p \rangle \,dS
\end{align*}
and the distributed shape derivative becomes 
\begin{align*}
dJ_{vol}(\Omega)[V]
=&\, \int_{\Omega}
\Div(V) \mathcal{F}_{vol}(.,u,Du) \, dx  \nonumber \\
&+ \int_{\Omega} \left\langle \frac{\partial \mathcal{F}_{vol}}{\partial z_1} (.,u,Du ),V \right\rangle -  \frac{\partial \mathcal{F}_{vol}}{\partial z_3} (.,u,Du) : (DuDV) \, dx \nonumber  \\
&+ \int_{\Omega} \langle f_V,p \rangle  + \tr\left((DV\sigma(u) + \dot{\sigma}(u) + \Div(V)\sigma(u))Dp \right) \, dx + \int_{\Gamma_N} \hspace*{-2mm}   \langle g_V,p \rangle \,dS \, .
\end{align*}
\indent Obviously, this formulation does not contain curvature terms or second order derivatives of $ u $ and can be computed easily in a numerical scheme. At first sight it seems to have the disadvantage to provide no decent direction, but this problem can be overcome, see \cite{SturmLaurain2016distributed,Schmidt2018weak}.  

\noindent The distributed shape derivatives for $ J_{1,sur} $ and $ J_{2,sur} $ can be derived analogously - For $ J_{1,sur} $ we use \eqref{Shape_Grad_LinEl:Eq: Adjoint Eq Elasticity Fsur u} and for $ J_{2,sur} $  \eqref{Shape_Grad_LinEl:Eq: Adjoint Eq Elasticity Fsur u sigma}. In the case of $ J_{1,sur} $ this leads to
\begin{align*}
dJ_{1,sur}(\Omega)[V]
=&\, \int_{\Gamma}
\Div_{\Gamma}(V) \mathcal{F}_{sur}(.,u,Du) + \left\langle \frac{\partial \mathcal{F}_{vol}}{\partial z_1} (.,u,Du ),V \right\rangle dS \nonumber  \\
&+ \int_{\Omega} \langle f_V,p \rangle  + \tr\left((DV\sigma(u) + \dot{\sigma}(u) + \Div(V)\sigma(u))Dp \right) \, dx + \int_{\Gamma_N} \hspace*{-2mm}   \langle g_V,p \rangle \,dS
\end{align*}
Suppose that $ V=0 $ in a neighborhood $ U \subset \R{3} $ of $ \Gamma_D $. Then also $ D\dot{u}=0 $ on $ \Gamma_D $. Hence (again with $ M=M(u)=\tfrac{\partial \mathcal{F}_{sur}}{\partial z_3} (.,u,Du) $)
\enlargethispage{\baselineskip}
\begin{align*}
dJ_{2,sur}(\Omega)[V]
=&\, \int_{\Gamma_N}
\Div_{\Gamma}(V) \mathcal{F}_{sur}(.,u,Du) + \left\langle \frac{\partial \mathcal{F}_{sur}}{\partial z_1} (.,u,Du ),V \right\rangle -  M : (DuDV)^{\sigma} \, dS \nonumber  \\
&+  \int_{\Gamma_N} \left\langle\frac{\partial \mathcal{F}_{sur}}{\partial z_2} (.,u,Du) ,\dot{u} \right\rangle + M : \sigma(\dot{u}) \, dS 
\\
=&\, \int_{\Gamma_N}
\Div_{\Gamma}(V) \mathcal{F}_{sur}(.,u,Du) + \left\langle \frac{\partial \mathcal{F}_{sur}}{\partial z_1} (.,u,Du ),V \right\rangle \, dS\\
&-\int_{\Gamma_N}  M : (DuDV)^{\sigma} + \left\langle \sigma(\dot{u})\vec{n}, (...)\vec{n} \right\rangle + \langle h,\dot{u} \rangle \, dS 
\\
=&\, \int_{\Gamma_N}
\Div_{\Gamma}(V) \mathcal{F}_{sur}(.,u,Du) + \left\langle \frac{\partial \mathcal{F}_{sur}}{\partial z_1} (.,u,Du ),V \right\rangle \, dS\\
&-\int_{\Gamma_N}  M : (DuDV)^{\sigma} +\left\langle g_{V} + G_{u}\vec{n}, \left(\tfrac{\lambda}{\lambda + 2 \mu}\tr(M_{\Gamma}) \mathrm{I}+\,\leftidx{_\Gamma}{M} + M^{\top}\right)\vec{n} \right\rangle \, dS \\
& +\int_{\Omega} \langle f_V,p \rangle  + \tr\left((DV\sigma(u) + \dot{\sigma}(u) + \Div(V)\sigma(u))Dp \right) \, dx + \int_{\Gamma_N} \hspace*{-2mm}   \langle g_V,p \rangle \,dS.
\end{align*}

We showed in Chapter \ref{Ex_Shape_Deriv_LinEl} that the shape derivative $ dJ(\Omega) $ exists for vector fields $ V \in C^{3}_{0}(\Oext,\R{3}) $ and $ C^{3} $-shapes $ \Omega \in \mathcal{O}_{3}^{b} $. Since the adjoint equation remains the same, the adjoint state $ p $ and the solution $ u $ are elements of $ C^{2,\phi}(\overline{\Omega},\R{3}) $ if we assume that $ \mathcal{ F}_{vol} $ is regular enough as it is the case for the Ceramic and the LCF functional, see Lemma \ref{Shape_Grad_LinEl:Prop: Differentiability Ceramic} and Lemma \ref{Shape_Grad_LinEl:Prop: Differentiability LCF}. Then shape differentiability can be shown analogously to Theorem \ref{Shape_Grad_LinEl:Thm: Shape Gradient Vol Functional u, sigma(u)}.

We now follow the approach proposed in \cite{SturmLaurain2016distributed} and define a decent direction as a solution of 
\begin{align}\label{outlook:Eq: Desc Acc. Laurain/Sturm}
\int_{\Omega} \varepsilon(W):\varepsilon(V)\, dx = dJ(\Omega)[V]~~~ \forall V \in H^{1}_{D}(\Omega,\R{3}).
\end{align}
Thus we have to assure that such a solution exists and therefore we investigate  
\begin{align*}
dJ_{vol}(\Omega)[V]
=&\, \int_{\Omega}
\Div(V) \mathcal{F}_{vol}(.,u,Du) \, dx  \nonumber \\
&+ \int_{\Omega} \left\langle \frac{\partial \mathcal{F}_{vol}}{\partial z_1} (.,u,Du ),V \right\rangle -  \frac{\partial \mathcal{F}_{vol}}{\partial z_3} (.,u,Du) : (DuDV) \, dx \nonumber  \\
&+ \int_{\Omega} \langle f_V,p \rangle  + \tr\left((DV\sigma(u) + \dot{\sigma}(u) + \Div(V)\sigma(u))Dp \right) \, dx + \int_{\Gamma_N} \hspace*{-2mm}   \langle g_V,p \rangle \,dS.
\end{align*}
where 
\begin{align}
f_{V} = Df\, V + f\Div(V) \text{ and } g_{V} = Dg\, V + f\Div_{\Gamma}(V) .
\end{align}
Due to the properties of $f,\,g,\, u,\, p $ and $ \mathcal{F}_{vol} $ the mapping $ dJ(\Omega):H^1_{D}(\Omega,\R{3})  \to \R{}$ can be judged to be continuous. Thus a unique solution $ W =W(\Omega) \in H^{1}_{D}(\Omega,\R{3}) $ can be found by the Lax-Milgram theorem.

However, a regularity of $ H^1 $ is clearly not enough to maintain the $ C^{k} $-domain regularity in a descent along $ W =W(\Omega_t)$, where the according volume-flow equation can be derived as 
\begin{align}\label{Outlook:Eq:Displacement Flow}
\frac{d}{dt} \Phi_t = - W(\Omega_t) \circ \Phi_t \text{ on } \Omega ~~~~\text{ where } \Omega_t = \Phi_t(\Omega_t)
\end{align}
for the initial shape $ \Omega \in \mathcal{O}_{k}^{b} $. Thus we have to switch over to strong solutions again.

Deriving the strong formulation of \eqref{outlook:Eq: Desc Acc. Laurain/Sturm} means separating $ V $ on the right hand side of Equation \eqref{outlook:Eq: Desc Acc. Laurain/Sturm} by partial integration. But since the surface formulation of $ dJ(\Omega)[V] $ can be derived exactly in this way we end up with
\begin{align}
\int_{\Omega} \varepsilon(W):\varepsilon(V)\, dx = dJ(\Gamma)[V_{\vec{n}}] =  \int_{\Gamma} \langle G(\Gamma) \vec{n},V \rangle \, dS ~~~\forall V \in H^{1}_{D}(\Omega,\R{3})
\end{align}
where 
$G(\Gamma)=  \mathcal{F}_{vol} (.,u,Du) +  \langle f+\kappa g +Dg \, \vec{n} ,p \rangle  - D_{\Gamma}p:\sigma_{\Gamma}(u) \text{ on } \Gamma,$
what then corresponds to the Steklov-Poincaré metric, consider \cite{Schulz2016Steklov} \footnote{Note that we pretend to be in the setup of the manifold $ B_e $ here - what truly not the case - as it is also not the case in the context of discretized domains, as they appear in FE discretizations.}. Since we want to keep the Dirichlet boundary fixed, we consider again only vector fields with $ \langle V,\vec{n} \rangle = 0 $ on $ \Gamma_D $. In strong form this equation thus reads  
\begin{equation}\label{Outlook:Eq:Strong Displacement PDE}
\left. 
\begin{array}{r c l l}
\Div( \varepsilon(W)) & = & 0  &\text{ in } \Omega \\
W & = & 0  &\text{ on } \Gamma_{D} \\
\sigma(W) \vec{n}& = & G(\Gamma)\vec{n} & \text{ on }\Gamma_{N}.  \\
\end{array} 
\right.
\end{equation}

As stated in Table \ref{Tab: Reg Grad Fvol}, $ G(\Gamma)$ is an element of $ C^{1}(\Gamma) \subset W^{1 - \nicefrac{1}{p},p}(\Omega,\R{3})$ and $ \vec{n} $ is an element of $ C^{2}(\Gamma) $. Thus we find $ W=W(\Omega) \in W^{2,p}(\Omega,\R{3}) \hookrightarrow C^{1,\phi}(\overline{\Omega},\R{3}) $ by Theorem \ref{Linear_Elasticity:Thm:LinEl_Classical_Sol} what is still leads to insufficient regularity. Hence, only a twice application of the same approach would leed to $ C^{2,\phi} $-descent directions. Presumed that $ g=0 $, the curvature term, which bounds the regularity from above by $ C^{1} $, vanishes and $ G(\Gamma)\vec{n} \in C^{2}(\Gamma) $. In this case, we derive that the solution $ W $ is an element of $ C^{2,\phi}(\overline{\Omega},\R{3}) $ what is close to regular enough. We will come back to this problem in Section \ref{outlook: Bound. and Sol. Reg. }.

Howsoever, we conclude that the descent directions obtained from this approach provide more regularity than pure $ L^2 $-descent directions, but still too less regularity to maintain the domain regularity in displacement flows or algorithms. 

Taking the $ H^1 $-scalar product, as it is proposed in \cite{Schulz2016metricsComparison}, i.e. solving the variational formulation

\[ \int_{\Gamma}  wv +c \nabla_{\Gamma}w \nabla_{\Gamma}v  \, dS =  \int_{\Gamma} \langle G(\Gamma),v \rangle \, dS ~~~\forall v \in H^{1}(\Gamma) \]
for $ c >0 $, leads to the same problematic.
In the space $ B_e $ this corresponds to the metric 
\[ m_{H^1}: \mathcal{T}_{B_{e}} \times \mathcal{T}_{B_{e}} \to \R{},\, (W,V) \mapsto \langle (id -c\Delta_{\Gamma})w, v\rangle_{L^2(\Gamma)}.\]
where $ V=v \vec{n} $ and $ W = v \vec{n} $. This means solving the PDE
\[  (id-c\Delta_{\Gamma})w =G(\Gamma) \text{ on } \Gamma.\]
In $ B_e $ this ansatz is suitable, since if $ \Omega $ is of class $ C^{\infty} $ and $ G(\Gamma) \in C^{\infty}(\Gamma) $ then also the solution $ w =w(\Gamma)$
and the descent direction $ W(\Gamma) = -w(\Gamma)\vec{n} $ are of this regularity class, consider also \cite[Prop. 1.9]{Taylor2013partial}. But in the framework of $ C^{k} $-shapes, this approach fails. 
Even if the solution, i.e. the $ H^1 $ gradient $ w $, is an element of $ C^{k}(\Gamma) $, then $ W(\Gamma) =  - w(\Gamma)\vec{n} $ is only contained in $ C^{k-1}(\Gamma,\R{3}) $. Thus we propose a related but nevertheless different approach here:

The idea is to solve the equation
\begin{align} \label{Outlook:Eq: Weak Helmhholtz System}
\int_{\Gamma}   \langle W, V \rangle +cD_{\Gamma}W : D_{\Gamma}V    \, dS =  \int_{\Gamma} \langle G(\Gamma)\vec{n},V \rangle \, dS ~~~\forall V \in H^{1}(\Gamma,\R{3}) 
\end{align}
for $ c>0 $ on the boundary. The associated PDE system in strong form reads
\begin{align}\label{Outlook:Eq: Strong Helmhholtz System}
(id -  c\Delta_{\Gamma})W_i = G(\Gamma)\vec{n}_i \text{ on } \Gamma,\, i=1,\,2,\,3.
\end{align}
This equation has an additional regularization effect on the outward normal $ \vec{n} $. Moreover, the constant $ c $ can be seen as an additional smoothing parameter. Suppose that $ V,\, W \in H^{1}(\Gamma,\R{3}) $, then \[ \int_{\Gamma} c D_{\Gamma} W : D_{\Gamma}V +  \langle W,V \rangle \, dS \to \int_{\Gamma} \langle W,V \rangle \, dS \text{ if } c \to 0.\] In case of $ c=0 $, solving \eqref{Outlook:Eq: Weak Helmhholtz System} corresponds to finding the classical $ L^2 $-descent direction. 

\noindent Thus there are some tasks left to do to complete the proposed approach: 
\begin{itemize}
	\item[1.] Show that \ref{Outlook:Eq: Strong Helmhholtz System} has classical solutions in $ C^{k,\phi}(\Gamma)$ if $ G(\Gamma)\vec{n} \in C^{k-2,\phi}(\Gamma),\, k\geq 2 $. The regularity theory presented in Chapter \ref{PDE_Systems} adapted to differentiable manifolds should lead to these results.
	\item[2.] Close the gap between the domain regularity of $ C^{k+1} $ and the $ C^{k,\phi} $-regularity of the PDE solutions, consider Section \ref{outlook: Bound. and Sol. Reg. }  and
	\item[3.] show the existence of flows $ \Phi_{t},\, t \in [0,\epsilon) $ along descent directions $ W(\Gamma_t) \in C^{k,\phi} $ or $ W(\Omega_t) \in C^{k,\phi} $ according to \eqref{Shape_Grad_LinEl: Gradient Flow Equation} or \eqref{Outlook:Eq:Displacement Flow}.
\end{itemize}

\vspace*{3em}
\section{Consistency of domain and solution regularity} \label{outlook: Bound. and Sol. Reg. }

The speed method, as it is described in \cite{SokZol92}, seems not to be the ideal tool to derive $ C^{k,\phi} $-Hölder material derivatives and gradients for $ k\geq 2 $ and $ 0<\phi<1 $ due to the existing regularity theory for linear elliptic PDE systems.

Let us illustrate this observation by some small examples.
We regard the case of a domain $ \Omega $ with a $ C^3$ boundary and we suppose that $ f \in C^{1}(\overline{ \Omega},\R{3}) $ and $ g^{2}(\Gamma,\R{3}) $. In this case the speed method demands a $ C^{3} $ admissible vector field such that we receive $ C^{3} $ transformations $ T_t $. Unfortunately there seems to be a gap in the regularity theory for PDE systems which in this case only provides existence results which assure that $ u \in C^{2,\phi}(\overline{ \Omega},\R{3}) $, $ 0\leq \phi <1 $ (consider Theorem \ref{Linear_Elasticity:Thm:LinEl_Classical_Sol}) although $ \Omega $ has a $ C^{3} $ boundary and $ f $ and $ g $ are regular enough. 

The same can be observed in the case of surface shape gradients, see Table \ref{Tab: Reg Grad F2sur}, where $ G(\Gamma) $ reaches at most a regularity of $ C^{k-1+\phi} $ when the initial domain $ \Omega $ ist of class $ C^{k+2} $. Thus even if a solution $ W $ of \eqref{Outlook:Eq: Strong Helmhholtz System} belongs to $ C^{k+1,\phi} $ it still pertains not enough regularity to maintain the $ C^{k+2} $-smoothness of $ \Omega $. This phenomenon also appears when the approach proposed in \cite{SturmLaurain2016distributed} is used, see the explanations below \eqref{Outlook:Eq:Strong Displacement PDE}.
It is due to the regularity theory for elliptic PDE. An explanation for this phenomenon are the Sobolev embeddings which guarantee only  
\[  W^{k,p}(\Omega) \hookrightarrow C^{k-1,\phi}(\overline{\Omega}) ~~~~~ \text{ for }0\leq \phi \leq 1-\nicefrac{n}{p} \]
if $ mp>n>(m-1)p $ or \[  W^{k,p}(\Omega) \hookrightarrow C^{k-1,\phi}(\overline{\Omega}) ~~~~~ \text{ for }0\leq \phi <1 \] in the case of $ (m-1)p=n $. Only the condition ($ n=m-1 $ and $ p=1 $)  - which is not of interest here - implies 
\[  W^{j+m,p}(\Omega) \hookrightarrow C^{j,\phi}(\overline{\Omega}) ~~~~~ \text{ for }0\leq \phi \leq 1.\]

\noindent Better assumptions would thus be: A domain $ \Omega $ of  class $ C^{k,\phi} $, $ g \in C^{k-1,\phi}(\Omega,\R{3}) $, $ f \in C^{k-2,\phi}(\overline{ \Omega},\R{3}) $ and a vector field $ V \in \Vad{k,\phi}{\Omega^{ext}} $ respectively $ V \in C^{k,\phi}_{0}(\Oext,\R{3}) $. In this case we achieve the same regularities for the PDE solutions as before, i.e. $ u \in C^{k,\phi} $ but the gap between domain regularity and solution regularity exists no longer, since then $ u \in C^{k,\phi} $ for a domain of class $ C^{k,\phi} $, consider Theorem \ref{Linear_Elasticity:Thm:LinEl_Classical_Sol}. This also remains true for the adjoint equations and the shape gradients  or descent directions. 

\begin{itemize}
	\item[1.] In case of the speed method it would thus be left to show that the transformations $T_t[V]$ inherit the $ C^{k,\phi} $ regularity of vector fields $ V \in C^{k,\phi} $ with $ \langle V,\vec{n}^{ext} \rangle =0 \text{ on } \Gamma^{ext} $ and still maintain their properties. 
	
	\item[2.]  Another approach that could be considered is the \textit{perturbation of identity method}, see for example \cite{DelfZol11,SokZol92}.  Therein, suitable deformation mappings $ t \to \Phi_t[V] $  are derived by setting $ \Psi_t[V]:= id + t V $ for some vector field $ V$. Since $ id$ is smooth $ \Psi_t $ takes the regularity of $ V $ and remains a transformation as long as $ |t| $ is close$  $ to $ 0 $.
	
	
	Obviously, $ \Psi_t[V] $ is the first order Taylor expansion of $ T_t[V] $ at $ t_0=0 $, see Lemma  \ref{Diff_Banach_Space:Lem:Linear_Approx}.
	Therefore, it seems to be possible to transfer all the proofs presented in this work to the perturbation of identity method with $ \Omega $ of class $ C^{k,\phi} $ and $ V \in C^{k,\phi} $ since the crucial properties of the transformations $ T_t[V] $ are  maintained. Let us mention here for example the first derivatives by $ t $ and $ x $ at $ t=0 $. In this setting, the regularity of the solutions $ u \in C^{k,\phi}$ becomes consistent with the smoothness of the domain. 
\end{itemize}

Nevertheless, we have to take good care since the first order derivatives at $ t\neq 0 $ are different and also the higher order derivatives differ. Further, the question arrises if the perturbation of identity method is suitable for the examination of shape flows into descent directions or if it is necessary to transfer the results from the speed method to $ C^{k,\phi} $-domains and vector fields. Answering this question is a task for the future.

\chapter{Conclusions}

We close this thesis with a summary of the main results and a compilation of possible directions for further research.

This thesis provides a framework for shape optimization problems with elliptic PDE constraints in classical function spaces and extends the existing results on shape differentiability to a class of shape functionals which are $ H^1 $-ill-defined. 

The two reliability shape functionals $ J^{\mathrm{lcf}} $ and $ J^{\mathrm{cer}} $, introduced in Chapter \ref{Reliability}, belong to this class of functionals and shall be minimized with respect to linear elasticity constraints. They measure the failure rates of metal devices under cyclic loading or ceramic components under tensile loading and serve as a motivation and as a common theme during this work. 
We illustrated that both of them posses the mentioned property, in the sense, that they are only defined if the solution $ u $ of the linear elasticity equation \eqref{Reliability:Eq:LinEl} is an element of a Sobolev spaces of higher order or a Hölder space of differentiable functions. Therefore, it was unavoidable to investigate the resulting reliability shape optimization problems under consideration of regularity theory for elliptic PDE.

It was an open question how the existence of material derivatives with respect to topologies on these spaces can be proved. We answered this question in this thesis in the following way.\\
First, in Chapter \ref{Parameter_Dep_PDE} we presented a novel functional analytic concept for sensitivity analysis of solutions of parameter dependent linear variational equations. This framework has also other applications outside of shape optimization and covers many linear elliptic PDE and PDE systems. The essential components of this framework are Theorem \ref{Parameter_Dep_PDE:Thm: Existence Deriv. strong Topology} and Lemma \ref{Parameter_Dep_PDE:Lem: Crit. for strong cont.}.\\
Then, in Chapter \ref{Ex_Shape_Deriv_LinEl}, we showed the functionality of this concept on the example of linear elasticity equation with perturbed  $ C^{k} $-domains $ \Omega_t $. The central outcome is the result on the existence of material derivatives in Hölder spaces, Theorem \ref{Ex_Shape_Deriv_LinEl:Thm: q^t C^3,phi material derivative of u_t}, which is derived under application of uniform Schauder estimates and compact embeddings. 
This result also implies the existence of shape derivatives of the general class of local cost functionals of first order (Definition \ref{Transf_shape_opt:Defn: Local Cost Functionals}), in particular for the reliability functionals under consideration, see Section \ref{Ex_Shape_Deriv_LinEl:Eq: Shape Deriv Jvol Jsur}.

Furthermore, we provided regularity theory for the associated  adjoint equations and calculated the $ L^2 $-shape gradients in Hölder spaces in Chapter \ref{Shape_Grad_LinEl}, presented a regularity theory for these Hadamard shape gradients, and showed in a mathematically rigorous way, that their smoothness is insufficient to pertain the shape regularity along shape flows in the framework of $ C^{k} $-shapes.

In the future, we can reduce the regularity assumptions for some of our results, as discussed in Section \ref{outlook: Reduces Requirements} and show the existence of material derivatives e.g. w.r.t. Sobolev topologies. Furthermore, showing the existence of shape flows in suitable spaces, as explained in Section \ref{Shape_Grad_LinEl:Sec: Distributed Shape Derivative} and Section \ref{outlook: Bound. and Sol. Reg. }, would be a significant enhancement of the theory of shape optimization.\\
From the numerical point of view, a good approximation of the boundary of the domain is essential, to fulfill the regularity assumptions of our theory. Therefore a special numerical implementation is required and could be attained by application of curved finite elements or isogeometric analysis.

All in all, we presented a comprehensive approach for regularity theory in shape optimization.


\appendix
\chapter{Appendix} \label{App}
\section{Topology and measures} \label{App: Sec:TopoMeasures}
\begin{defn}[Topological Space \& Hausdorf-Space]
	Let $ X $ be a set and $ \mathscr{T} $ a family of subsets of $ X $. 
	\begin{itemize}
		\item[i)~~] $(X,\mathscr{T}) $ is a topological Space if 
		\begin{itemize}
			\item[1)] $ X \in \mathscr{T}$ and $ \emptyset \in \mathscr{T} $,
			\item[2)] if $ O_1,\, O_2 \in  \mathscr{T} $ then $ O_1 \cap O_2 \in  \mathscr{T} $,
			\item[3)] if $  (O_i)_{i \in \N{}}  \subset \mathscr{T} $ then $  \bigcup_{n =1}^{\infty} O_i \in  \mathscr{T} $.
		\end{itemize}
		The elements $ O \in \mathscr{T} $ are called \textit{open sets}.
		\item[ii)~] A set  $ V \subset X $ is called neighborhood of $ A \subset X $ if there is $ O \in \mathscr{T} $ such that $  A \subset O \subset V $. The set $ \mathcal{U}(A)$ of neighborhoods of $  A $ is called a \textit{neighborhood system} of $ A $.
		\item[iii)] A topological space $ (X,\mathscr{T}) $ is called a Hausdorff space if for any two points $ x,\,y \in X $ there are $ U_x \in \mathcal{U}(x) $ and $ U_y \in \mathcal{U}(y) $ such that $ U_x \cap U_y=\emptyset $.
	\end{itemize}
\end{defn}

\begin{defn}
	Let $ (X,\mathscr{T}) $ be a topological space. 
	\begin{itemize}
		\item[i)~~] A family $ \mathcal{U} \subset X $ is a \textit{cover} of $ X $ if $ X = \bigcup_{U \in \mathcal{U}} U $.
		\item[ii)~] The cover is called open if any $ U \in \mathcal{U}$ is open in $ X $, i.e. $ \mathcal{U} \subset \mathscr{T} $.
		\item[iii)] A \textit{subcover} is a family $ \mathcal{V} \subset \mathcal{U} $ such that $ X = \bigcup_{U \in \mathcal{V}} U $.
	\end{itemize}
\end{defn}

\begin{defn}
	Let $ (X,\mathscr{T}) $ be a topological space and let $ S \subset X $.
	\begin{itemize}
		\item[i)~]  $ (X,\mathscr{T}) $  is called \text{compact} if any open cover has a finite subcover.
		\item[ii)] The topology  $ \mathscr{T}_{S}:=\{S \cap O \vert O \in \mathscr{T}\} $ is called the \textit{subspace topology on $ S $} and $ (S,\mathscr{T}_{S}) $  is a topological subspace. The set $ S $ is said to be \textit{compact} if it is compact w.r.t. $ \mathscr{T}_{S} $.
	\end{itemize} 
\end{defn}

\begin{defn} 
	\begin{itemize}
		\item[i)~] A \textit{basis} of the topological space $ (X,\mathscr{T}) $ is a subset $ \mathscr{B} \subset \mathscr{T} $ such that for any $ O \in \mathscr{T} $ there is a index set $  \mathcal{I} $ and a family $ B_{i},\, i\in \mathcal{I}, B_{i} \in \mathscr{B} $ with $ O =\bigcup_{i \in \mathcal{I}} B_{i} $.
		\item[ii)] $ (X,\mathscr{T}) $ is said to be \textit{second countable} if $ \mathscr{B} $ is countable.
	\end{itemize}
\end{defn}

\begin{defn}
	The topological space $ (X,\mathscr{T}) $ is called locally compact if for any $ x \in X $ holds: If $ U_x \in \mathcal{U}(x) $ is an arbitraty neighbourhood of $ x $ then there is a comact set $ x \in K$ such that $ K \subset U_x $ .
\end{defn}

\begin{defn}
	Let $ (X,\mathscr{T}) $ be a Hausdorff space and let $ \mathcal{A} $ be a $ \sigma $-algebra on $ X $ that contains $ \mathscr{T} $.
	\begin{itemize}
		\item[a)] A measure $\gamma  $ on the measurable space $ (X,\mathcal{A})$ is called inner regular, if \[ \gamma(A)=\sup_{ K \subset\subset A} \gamma(K) ~~~~~\text{ for any } A \in \mathcal{A}. \]
		\item[b)] A measure $\gamma  $ on the measurable space $ (X,\mathcal{A})$ is called locally finite, if for every $ x \in X $ there exists a neighbourhood $ U_{x} \in \mathcal{U}(x) $ containing $ x $ such that $ \vert  \gamma(U_{x}) \vert < \infty $.
	\end{itemize}
\end{defn}

\begin{thm}[Lebesgue's Dominated Convergence Theorem] \cite[Sec. 8.6. Thm. 2 ]{Cheney} \label{App:Lebesgue}
	Suppose that $ (Z,\mathcal{Z},\mu) $ is a measurable space and let $ (f_{j})_{j \in \N{}} $ be a sequence of integrable functions converging pointwise $ \mu $-almost everywhere in $ Z $ to $ f \in L^1 (Z,\mathcal{Z},\mu) $. If there exists a function $ \mathfrak{m} \in L^1 (Z,\mathcal{Z},\mu)  $ such that $ |f_{j}(x)| \leq \mathfrak{m}(x) $ for every $ j $ and $ x \in \Omega $, then 
	\[ \lim_{j \to \infty} \int_{Z} f_{j}(z) \, d\mu(z) = \int_{Z} \lim_{j \to \infty} f_{j}(z) \, d\mu(z) = \int_{z} f(z)\, d\mu(z). \] 
\end{thm}

\vspace*{2em}

\section{Analysis}

%

\begin{defn}[Sobolev Extension Operators] \cite[Definition 5.17]{AdamsFournier} \label{App: Extension operators}
	Let $ \Omega  $ be a domain in $ \R{n} $. For given $ m $ and $ p $ a linear extension operator $ E:W^{m,p}(\Omega) \to W^{m,p}(\R{n}) $ is called
	\begin{itemize}
		\item[i)] \textit{simple
			$ (m,p) $-extension operator} if there exists a constant $ K(m,p) $ such that  for every $ u \in W^{m,p}(\Omega) $ 
		\begin{itemize}
			\item[a)] $ Eu(x)=u(x) $ a.e. in $ \Omega $ and
			\item[b)] $ \Norm{Eu}{W^{m,p}(\R{n})} \leq K \Norm{u}{W^{m,p}(\Omega)} $.
		\end{itemize}
		\item[ii)] \textit{strong $ m $-extension operator} for $ \Omega $ if additionally
		\begin{itemize}
			\item[a)] $ E $ maps functions defined a.e. on $ \Omega $ to functions defined a.e. on $ \R{n} $
			\item[b)] For every $ 1 \leq p < \infty $ and any integer $ 0\leq k \leq m $ $ E\vert_{W^{k,p}(\Omega)}$ is a simple $ (k,p) $-extension operator for $ \Omega $.                  
		\end{itemize}
		\item[iii)] \textit{total extension operator} for $ \Omega $ is $ E $ is a strong $ m $-extension operator for every $ m  \in \N{}$. 
	\end{itemize}
\end{defn}

\begin{thm}[The Stein Extension Theorem]\cite[Theorem 5.24]{AdamsFournier}\label{App: Stein Extension Thm}
	Let $ \Omega $ in $ \R{n} $ be a bounded domain with (local) Lipschitz boundary \footnote{Note the remark below of in \cite[Definition 4.8]{AdamsFournier}}. Then there exists a total extension operator $ E $ for $ \Omega $. 
\end{thm}

\begin{lem}[Extension Lemma]\label{App:Lem: Hölder Extension Lemma}\cite[6.37]{GilbTrud}
	Let $ \Omega $ be a $ C^{k,\phi} $-domain in $ \R{n} $ with $ k \geq 1 $ and let $ \Omega' $ be a open set containing $ \overline{\Omega} $. Suppose that $ u  \in C^{k,\phi}(\overline{\Omega}) $. Then there exists a function $ w \in C^{k,\phi}(\Omega') $ with compact support such that $ w=u $ in $ \Omega $ and $ \Norm{w}{C^{k,\phi}(\Omega')} \leq C \Norm{u}{C^{k,\phi}(\Omega)} $ where $ C $ depends on $ k,\, \Omega,\, \Omega' $.
\end{lem}

We reduce the Sobolev Embedding Theorem and the Rellich-Kondrachov-Theorem to the case when $ \Omega $ is bounded. In this case \cite[Theorem  4.2]{AdamsFournier} (we refer also to the remarks below) states the following:
\begin{thm}[Sobolev Embedding Theorem] \label{App: Sobolev embedding}
	Let $ \Omega $ be a bounded domain in $ \R{n} $ satisfying a cone condition and let $ j\geq 0 $, $ m \geq 1 $ be integers and $ 1 \leq p < \infty $
	
	\noindent 	\textbf{PART I} 
	\begin{itemize}
		\item[i)] If either $ mp>n $ or $ m=n $ and $ p=1 $, then \[ W^{j+m,p}(\Omega) \hookrightarrow C^{j}_{b}(\Omega) ~~~~~\text{ and } ~~~~~ W^{j+m,p}(\Omega) \hookrightarrow W^{j,q}(\Omega) \text{ for } 1 \leq q \leq \infty.   \]
		\item[ii)] If $ mp=n$ then $$ W^{j+m,p}(\Omega) \hookrightarrow W^{j,q}(\Omega) \text{ for } 1 \leq q < \infty. $$
		\item[iii)] If $mp<n $ or $ p=1 $ then $$ W^{j+m,p}(\Omega) \hookrightarrow W^{j,q}(\Omega) \text{ for } 1 \leq q \leq p^{\ast}=\nicefrac{np}{(n-mp)}. $$ 
	\end{itemize}
	The embedding constants for the embeddings above depend only on $ n,\,m,\, p,\, q,\, j $ and the cone $ C $ in the cone condition. 
	
	\noindent \textbf{PART II}  Suppose that $ \Omega $ has a Lipschitz boundary satisfies a strong local Lipschitz condition then in PART I i) $ C^{j}_{b}(\Omega) $ can be replaced by $ C^{j}(\overline{\Omega}) $ and the embedding can be further refined as follows:
	\begin{itemize}
		\item[i)] If $ mp>n>(m-1)p $ then, \[  W^{j+m,p}(\Omega) \hookrightarrow C^{j,\phi}(\overline{\Omega}) ~~~~~ \text{ for }0\leq \phi \leq m-\nicefrac{n}{p}.\]
		\item[ii)] If $ (m-1)p=n $  then \[  W^{j+m,p}(\Omega) \hookrightarrow C^{j,\phi}(\overline{\Omega}) ~~~~~ \text{ for }0\leq \phi <1 .\]
		If furthermore $ n=m-1 $ and $ p=1 $ then 
		\[  W^{j+m,p}(\Omega) \hookrightarrow C^{j,\phi}(\overline{\Omega}) ~~~~~ \text{ for }0\leq \phi \leq 1 .\]
	\end{itemize}
\end{thm}

\begin{thm}[Rellich-Kondrachov Theorem] \cite[Theorem 6.3]{AdamsFournier} \label{App: Rellich Kondrachov}
	Let $ \Omega $ be a bounded domains in $ \R{n} $ satisfying a cone condition and let $ j\geq 0 $, $ m \geq 1 $ be integers and $ 1 \leq p < \infty $.
	
	\noindent \textbf{PART I}:
	\begin{itemize}
		\item[i)] If $ mp\leq n $  then the embeddings 
		\[ W^{j+m,p}(\Omega) \hookrightarrow W^{j,q}(\Omega) ~~~~~ \text{ for }1 \leq q < \nicefrac{np}{n-mp} ,\, mp<n\] 
		and 
		\[ W^{j+m,p}(\Omega) \hookrightarrow W^{j,q}(\Omega) ~~~~~ \text{ for }1 \leq q < \infty,\, mp=n\] 
		are compact.
		\item[ii)] If $ mp>n $ then the embeddings 
		\[ W^{j+m,p}(\Omega) \hookrightarrow C^{j}_{b}(\Omega)~~~~~ \text{ and } ~~~~~ W^{j+m,p}(\Omega) \hookrightarrow W^{j,q}(\Omega), \, 1 \leq q < \infty \] are compact.
	\end{itemize}	
	
	\noindent \textbf{PART II}:
	\begin{itemize}
		\item[i)] If $ \Omega $ has a Lipschitz boundary then 
		\begin{eqnarray}
		W^{j+m,p}(\Omega) &\hookrightarrow& C^{j}(\overline{\Omega}) ~~~ \text{ for } mp>n \\
		W^{j+m,p}(\Omega) & \hookrightarrow& C^{j,\phi}(\overline{\Omega}) ~~~ \text{ for } mp>n\geq (m-1)p,\, 0< \phi < m-\nicefrac{n}{p}.
		\end{eqnarray} 
	\end{itemize}
\end{thm}

\begin{thm}[Properties of the Trace Operator] \label{App: Trace Operator} \cite[Theorem 5.36]{AdamsFournier}
	Let $ \Omega  $ be a domain in $ \R{n} $ and $ m \geq 1 $ an integer.
	Let $ 1 \leq p <\infty $ and $ \Omega $ satisfy have a $ C^{m} $- boundary
	\begin{itemize}
		\item[i)] $ mp<n $ and $ p\leq q \leq p^{\#}:= (n-1)p/(n-mp) $ then \[ \mathbf{T_{\Gamma}} \in \mathcal{L}(W^{m,p}(\Omega), L^{q}(\Gamma)) ~~~~~\text{ for } \frac{1}{p^{\#}} = \frac{1}{p} -\frac{p-1}{(n-1)p}. \]
		\item[ii)] If $ mp=n$  then \[ \mathbf{T_{\Gamma}} \in \mathcal{L}(W^{m,p}(\Omega), L^{q}(\Gamma)) ~~~~~\text{ for } 1 \leq q < \infty . \]
		\item[iii)] If $ mp>n$  then $ \mathbf{T_{\Gamma}} \in \mathcal{L}(W^{m,p}(\Omega), C^{0}(\Gamma))$ by the Sobolev embedding Theorem.  
	\end{itemize}	
\end{thm}
\vspace*{1ex}

\begin{defn}[Besov Spaces]\cite[Definition 7.32 and Theorem 7.16]{AdamsFournier} \label{App: Besov Spaces}
	Let $\Omega  $ ba a domain in $ \R{n} $ and $ 1 \leq q \leq \infty $ an integer. For $ 0<s<\infty $ let $ m \in \N{} $ be the smallest integer larger than $ s $. Then we define the Besov space $ B^{s,p,q} (\Omega)$ to be the intermediate space between $ L^{p}(\Omega) $ and $ W^{m,p}(\Omega
	) $, i.e.
	\[ B^{s,p,q}(\Omega): = \{ u \in L^{p}(\Omega) + W^{m,p}(\Omega)\, \vert t \mapsto t^{-\nicefrac{s}{m}} K(t,u) \in L^{q}((0,\infty),\tfrac{dt}{t})\} \]
	where $ \frac{dt}{t} $ is the Haar measure, see \cite{diestel2014joys}, and \[ K(t,u) := \inf\{\Norm{u_0}{L^{p}(\Omega)} + t \Norm{u_1}{W^{m,p}(\Omega)} ,\, u = u_0 + u_1 \in L^{p}(\Omega) + W^{m,p}(\Omega)\}.\]
\end{defn}

\begin{thm}\cite[Theorem 7.29]{AdamsFournier} \label{App: Besov Trace Theorem}
	If $ 1<p<\infty $ the following conditions on a measurable function $ u $ on $ \R{n-1} $, $ n \geq 1 $ are equivalent:
	\begin{itemize}
		\item[(a)] There is an a function $ U \in W^{m,p}(\R{n}) $ such that $ u $ is the trace of $ U $ .
		\item[(b)] $ u \in B^{m-\nicefrac{1}{p},p,p}(\R{n-1}):=  W^{m-\nicefrac{1}{p},p}(\R{n-1}) $ .
	\end{itemize}
\end{thm}

The previous Theorem combined with Theorem \ref{App: Stein Extension Thm} justifies the following definition:

\begin{defn}[The Trace Spaces] \cite{Agm64} \label{App: Trace Space}
	Let $ \Omega  $ be a domain in $ \R{n} $ with Lipschitz boundary and $ m \geq 1 $ an integer. Then
	$ W^{m-1/p,p}(\Gamma)=\mathbf{T_{\Gamma}}(W^{m,p}(\Omega)) =\{\mathbf{T_{\Gamma}}(u) \vert u \in W^{m,p}(\Omega) \} $ where $ \mathbf{T_{\Gamma}} $ is the trace operator on $ W^{k,p}(\Omega) $.
	
	This class of functions is normed by $ \Norm{g}{W^{m - \nicefrac{1}{p},p}(\Omega)}= \inf\{ \Norm{\tilde{g}}{W^{m,p}(\Omega)} \, | \, g = \mathbf{T_{\Gamma}}(\tilde{g}) \} $
\end{defn}

\begin{lem}
	Let $ u,v \in H^{1}(\Omega,\R{n}) $ for some domain $ \Omega \subset \R{n} $. Then
	\begin{flalign}
	i) &\int_{\Omega} |\tr(D_{u}^{\top} D_v)| dx \leq \Norm{u}{H^{1}(\Omega,\R{n})} \Norm{v}{H^{1}(\Omega,\R{n})} \label{App: Norm_H1_trDuTDv}\\
	ii) & \int_{\Omega} |\tr(D_{u}D_{v})| dx \leq \Norm{u}{H^{1}(\Omega,\R{n})} \Norm{v}{H^{1}(\Omega,\R{n})} \label{App: Norm_H1_trDuDv}\\
	iii) &\int_{\Omega} |\tr(\varepsilon(u)\varepsilon(v))| dx \leq \Norm{u}{H^{1}(\Omega,\R{n})}\Norm{v}{H^{1}(\Omega,\R{n})} \label{App: Norm_H1_tr_epsuepsv}\\
	iv) &\int_{\Omega} |\Div(u)\Div(v)| dx \leq n \Norm{u}{H^{1}(\Omega,\R{n})}\Norm{v}{H^{1}(\Omega,\R{n})}.\label{App: Norm_H1_Divu_Divv}
	\end{flalign}
\end{lem}

\begin{proof}
	\vspace*{-1ex}
	\begin{eqnarray*} 
		i) ~~~\int_{\Omega} |\tr(D_{u}^{\top} D_v)| dx 
		& \leq &\int_{\Omega}  \sum_{i,k=1}^{n}\left| (D_{u})_{k,i} (D_{v})_{k,i} \right | \, dx 
		\\
		& \underset{C.S.}{\leq} & \int_{\Omega}  \left(\sum_{i,k=1}^{n}\left| (D_{u})_{k,i} \right|^2\right)^{\nicefrac{1}{2}}  \left(\sum_{i,k=1}^{n}\left | (D_{v})_{k,i} \right |^2 \right)^{\nicefrac{1}{2}}\, dx \\
		& \underset{C.S.}{\leq} & \left( \int_{\Omega}  \sum_{i,k=1}^{n}\left| (D_{u})_{k,i} \right|^2\, dx\right)^{\nicefrac{1}{2}} \left(\int_{\Omega} \sum_{i,k=1}^{n}\left | (D_{v})_{k,i} \right |^2 \, dx \right)^{\nicefrac{1}{2}}\\
		& \leq & \Norm{u}{H^{1}(\Omega,\R{n})} \Norm{v}{H^{1}(\Omega,\R{n})}
	\end{eqnarray*} 
	$ ii) $ Analogously to $ i) $. \\[1ex]
	$ iii) $ Follows directly from 
	$$ \tr(\varepsilon(u)\varepsilon(v))  =\frac{1}{4}\tr((D_{u}+D_{u}^{\top})(D_{v}+D_{v}^{\top}))=\frac{1}{2}\tr(D_{u}D_{v}) + \frac{1}{2} \tr(D_{u}^{\top}D_{v}). $$
	\begin{eqnarray*} 
		iv) \int_{\Omega} |\Div(u)\Div(v)| \, dx 
		& \leq    &
		\int _{\Omega} \sum_{i=1}^{n}\left|(D_{u})_{i,i}\right| \sum_{j=1}^{n}\left|(D_{v})_{j,j}\right| \, dx 
		\\
		&\underset{C.S.}{\leq} &
		\int_{\Omega} n^{\nicefrac{1}{2}}  \left(\sum_{i=1}^{n} \left|(D_{u})_{i,i}\right|^2\right)^{\nicefrac{1}{2}} n^{\nicefrac{1}{2}} \left(\sum_{j=1}^{n} \left|(D_{v})_{j,j}\right|^2\right)^{\nicefrac{1}{2}} \,dx 
		\\
		& \leq &n \left(\int_{\Omega}  \sum_{i,j=1}^{n} \left|(D_{u})_{i,j}\right|^2\,dx\right)^{\nicefrac{1}{2}} \left(\int_{\Omega}\sum_{i,j=1}^{n} \left|(D_{v})_{i,j}\right|^2\,dx\right)^{\nicefrac{1}{2}}  \\[1ex]
		& \leq & n \Norm{u}{H^{1}(\Omega,\R{n})} \Norm{v}{H^{1}(\Omega,\R{n})} 
	\end{eqnarray*}
\end{proof}

\begin{lem}
	Let $ \Omega \subset \R{n} $ be a bounded domain of class $ C^1 $ and $ \Gamma \subset \partial \Omega $ a connected subset of $ \partial \Omega $ with $ |\Gamma|>0 $. 
	\begin{itemize}
		\item[i)] 
		Let $ u,v \in L^2(\Omega,\R{n}) $ then 
		\begin{equation}\label{Eq: Absch Skp u,v L1}
		\Norm{\langle u,v \rangle}{L^1(\Omega)}=\int_{\Omega} |\langle u,v \rangle| \, dx \leq \Norm{u}{L^2(\Omega,\R{n})}\Norm{v}{L^2(\Omega,\R{n})}. 
		\end{equation}
		\item[ii)]
		Let $ f\in C(\overline{\Omega},\R{n}) $ and $ u \in L^2(\Omega,\R{n}) $ then
		\begin{equation}\label{Eq: <f,u> L^1_est Vol}
		\Norm{\skp{f}{u}}{L^{1}(\Omega)} \leq \Norm{f}{C(\Omega,\R{n})}\sqrt{ \vert \Omega\vert} \Norm{u}{L^2(\Omega,\R{n})}.
		\end{equation}
		If $ u \in L^2(\Gamma,\R{n}) $, then
		\begin{equation}\label{Eq: <f,u> L^1_est Sur}
		\Norm{\skp{f}{u}}{L^{1}(\Gamma)} \leq \Norm{f}{C(\Gamma,\R{n})} \sqrt{\vert \partial\Omega\vert} \Norm{u}{L^2(\Gamma,\R{n})}.
		\end{equation}
		\item[iii)] Let $ f\in C(\overline{\Omega},\R{n}) $ and $ u \in L^2(\Omega) $, then
		\begin{equation}\label{Eq: <f,u> L^2_est Vol}
		\begin{split}
		\Norm{\langle f,u\rangle}{L^2(\Omega)}\leq\Norm{f}{C(\Omega,\R{n})} \Norm{u}{L^2(\Omega,\R{n})} . 
		\end{split}
		\end{equation}
		\item[iv)]
		For $ u,v \in H^1(\Omega,\R{n}) $ it holds
		\begin{equation}\label{Eq: Absch skp Du,Dv L1} 
		\Bigg\Vert \sum_{k,l=1}^{n} \Norm{\nabla u_k}{2} \Norm{\nabla v_l}{2} \Bigg\Vert_{L^1(\Omega)} \leq n\Norm{u}{H^1(\Omega,\R{n})} \Norm{v}{H^1(\Omega,\R{n})}.
		\end{equation}
	\end{itemize}
\end{lem}

\begin{proof}
	i) The Euclidean scalar product on $ \R{n} $ is given by $ |<u,v>|=|\sum_{i=1}^{n} u_{i}v_{i}| $ which is again an element in $ L^2(\Omega) $ as well as $ \Norm{u}{2}, \Norm{v}{2} \in L^2(\Omega) $. These functions exist almost everywhere (except for sets of measure zero) and from Cauchy-Schwarz inequality on $ \R{n} $ and $ L^p(\Omega) $ follows 
	\begin{equation*}
	\Norm{\langle u,v \rangle}{L^1(\Omega)}=\int_{\Omega} |\langle u,v \rangle| \, dx \leq \int_{\Omega}\Norm{u}{2}\Norm{v}{2} \, dx \leq \Norm{u}{L^2(\Omega,\R{n})}\Norm{v}{L^2(\Omega,\R{n})}. 
	\end{equation*}
	ii) For any $ x \in \Omega $ $ \Norm{f(x)}{2} \leq \Norm{f}{C(\Omega,\R{n})}$ and thus
	\begin{equation}
	\Norm{\langle f,u \rangle}{L^1(\Omega)}=\int_{\Omega} |\langle f,u \rangle| \, dx \leq \int_{\Omega}\Norm{f}{2}\Norm{u}{2} \, dx \leq \Norm{f}{\C{}(\Omega,\R{n})} \int_{\Omega}\Norm{u}{2} \, dx
	\end{equation}
	where 
	\[ \int_{\Omega}\Norm{u}{2} \, dx \leq \Norm{1}{L^2(\Omega)} \Norm{\Norm{u}{2}}{L^2(\Omega)}  =\sqrt{\vert \Omega\vert} \Norm{u}{L^2(\Omega,\R{n})}.\]
	The second statement follows analogously with $ \Norm{f(x)}{2} \leq \Norm{f}{C(\Omega,\R{n})} $ for any $ x \in \Gamma $.\\[1ex]
	iii) Cauchy-Schwarz inequality implies
	\begin{align*}
	\Norm{\langle f,u \rangle}{L^2(\Omega)}^2 
	&=\int_{\Omega} |\langle f,u \rangle|^2 \, dx \leq \int_{\Omega}\Norm{f}{}^2\Norm{u}{}^2 \, dx  
	= \Norm{f}{C(\Omega,\R{n})}^2 \Norm{u}{L^2(\Omega,\R{n})}^2. 
	\end{align*}
	iv) We apply Cauchy-Schwarz inequalty:
	\begin{equation}
	\begin{split}
	\Norm{\sum_{k,l=1}^{n} \Norm{\nabla u_k}{2} \Norm{\nabla v_l}{2}}{L^1(\Omega)} 
	&\leq \Norm{\sum_{k=1}^{n} \Norm{\nabla u_k}{2}}{L^2(\Omega)}\Norm{\sum_{l=1}^{n} \Norm{\nabla v_l}{2}}{L^2(\Omega)} \\
	&\leq n\Norm{u}{H^1(\Omega,\R{n})} \Norm{v}{H^1(\Omega,\R{n})}
	\end{split}
	\end{equation}
	since $$ \Norm{\sum_{k=1}^{n} \Norm{\nabla u_k}{2}}{L^2(\Omega)}^2 = \int_{\Omega} \left(\sum_{k=1}^{n} \Norm{\nabla u_k}{2}\right)^2 \, dx  \underset{C.S.}{\leq} n \int_{\Omega} \sum_{k=1}^{n} \Norm{\nabla u_k}{2}^2 \,dx = n\Norm{u}{H^1(\Omega,\R{n})}^2. $$
\end{proof}

\vspace*{-1em}

\begin{lem}\cite[Therem 3.41]{AdamsFournier}\label{Lem: H1_Estimate_u}
	Let $ T \in C^{1}(\overline{\Omega},\overline{ \Omega}' ) $ be a $ C^1 $-diffeomorphism and $ \Omega, \Omega'\subset \R{n} $ bounded domains of class $ C^1 $. Then there are constants $ C_1(T),\, C_2(T), C_3(T)>0 $ such that 
	\begin{align}
	\begin{array}{ll}
	\Norm{u \circ T^{-1}}{L^2(T(\Omega),\R{3})}& \leq C_1 \Norm{u}{L^2(\Omega,\R{3})},\,u \in L^2(\Omega,\R{n})\\
	\Norm{u \circ T^{-1}}{H^{1}(\O{t},\R{3})}&\leq C_2 \Norm{u}{H^{1}(\Omega,\R{3})},\, u \in H^1(\Omega,\R{n})\\
	\Norm{u \circ T^{-1}}{L^2(\Gamma',\R{3})}&\leq C_3 \Norm{u}{L^2(\Gamma,\R{3})},\,u \in L^2(\Gamma,\R{n}).
	\end{array}
	\end{align}
	These inequalities also hold for other $ 1\leq p < \infty $.
\end{lem}

\begin{proof}
	The existence of such constants $ C_1,\, C_2 $ and $ C_3 $ depending on $ \Omega' $ follows from \cite[Theorem 3.41]{AdamsFournier} but since we will need a detailed description and independence from the domain we will carry out the calculation in detail. 
	
	\noindent Set $ \gamma=|\det(DT)| $ and $ \omega=\Norm{\mathcal{M}(T)}{2}$ where $\mathcal{M}(T)=\gamma\left((DT)^{-1}\right)^{\top}\vec{n} $ with outward normal vector field $ \vec{n} $ on $ \Gamma $. For any $ u \in H^1(\Omega,\R{3}) $ the squared $ H^1(\Omega,\R{3}) $-norm is given by
	$$ 
	\Norm{u}{H^{1}(\Omega;\R{3})}^2 
	= \Norm{u}{L^{2}(\Omega,\R{3})}^2 + \sum_{i=1}^{3}\Norm{\nabla u_i}{L^{2}(\Omega,\R{n})}^2
	= \Norm{u}{L^{2}(\Omega,\R{3})}^2 + \sum_{i,j=1}^{3} \Norm{\tfrac{\partial u_{i}}{\partial x_{j}}}{L^2(\Omega)}^2.
	$$
	Let $ u \in L^2(\Omega) $. Then change of coordinates for $ L^p$-spaces to $ u \circ T^{-1} $  leads to
	\begin{align*}
	\Norm{u \circ T^{-1}}{L^{2}(\Omega')}^2 
	=\int_{\Omega'}  |(u \circ T^{-1})|^2 \,dx 
	= \int_{\Omega}  |u|^2 |\gamma| \, dx 
	\leq \Norm{\gamma}{\infty,\Omega} \Norm{u}{L^{2}(\Omega)}^2
	=C_1(T)\Norm{u}{L^{2}(\Omega)}^2.
	\end{align*}
	Now let $ u \in H^1(\Omega,\R{3}) $. In the same manner, additionally using $D(u \circ T^{-1}) \circ T=Du\,( D(T^{-1}) \circ T)  $, we obtain  
	\begin{align*}
	\Norm{D(u \circ T^{-1})_{i,j}}{L^{2}(\Omega')}^2 
	&\leq \Norm{\gamma}{\infty} \Norm{(Du \cdot D(T)^{-1} \circ T)_{i,j}}{L^2(\Omega)}^2 = \Norm{\gamma}{\infty} \Norm{(Du \cdot (DT)^{-1})_{i,j}}{L^2(\Omega)}^2
	\end{align*}
	where
	\begin{align*} 
	\Norm{(Du \cdot (DT)^{-1})_{i,j}}{L^2(\Omega)}^2 
	&= \Norm{\left\langle \nabla u_i, \tfrac{\partial T^{-1}}{\partial x_{j}} \circ T \right\rangle}{L^2(\Omega)}^2  \underset{\eqref{Eq: <f,u> L^2_est Vol}}{\leq} \Norm{\tfrac{\partial T^{-1}}{\partial x_{j}} \circ T}{C(\Omega,\R{3})}^2 \Norm{\nabla u_i}{L^2(\Omega,\R{3})}^2\\
	&=  \Norm{\tfrac{\partial T^{-1}}{\partial x_{j}}}{C(\Omega',\R{3})}^2 	  \Norm{\nabla u_i}{L^2(\Omega,\R{3})}^2.
	\end{align*} 
	According to the last equality, summing up over all $ i,j $ gives 
	\begin{align*}
	\sum_{i,j=1}^{3}\Norm{\left(D(u \circ T^{-1})_{i,j}\right) }{L^{2}(\Omega',\R{})}^2 
	\leq & \Norm{\gamma}{\infty} \sum_{i,j=1}^{3}  \Norm{\tfrac{\partial T^{-1}}{\partial x_{j}}}{C(\Omega',\R{3})}^2 \Norm{\nabla u_i}{L^2(\Omega,\R{3})}^2 \\
	\leq &  \Norm{\gamma}{\infty} \Hnorm{T^{-1}}{C^{1}(\Omega',\R{3})}^2  \Norm{u}{H^{1}(\Omega,\R{3})}^2.
	\end{align*}
	Assembling these estimates yields 
	\begin{align*}
	\Norm{u \circ T^{-1}}{H^{1}(\Omega';\R{3})}^2 
	\leq &\Norm{\gamma}{\infty} \Norm{u}{L^{2}(\Omega,\R{3})}^2 
	+\Hnorm{T^{-1}}{C^{1}(\Omega',\R{3})}^2  \Norm{\gamma}{\infty}   \Norm{u}{H^{1}(\Omega,\R{3})}^2  \Norm{u}{H^{1}(\Omega,\R{3})}^2 
	\\[1ex]
	\leq & \Norm{\gamma}{\infty} \left(1+\Hnorm{T^{-1}}{C^{1}(\Omega',\R{3})}^2\right) \Norm{u}{H^{1}(\Omega,\R{3})}^2 .
	\end{align*}
	Since $ T $ is a diffeomorphism on the compact set $ \Omega $ the term 
	$$ \left(\Norm{\gamma}{\infty} \left(1+ \Hnorm{T^{-1}}{C^{1}(\Omega',\R{3})}^2\right)\right)^{\nicefrac{1}{2}} $$ 
	can be bounded from above by $ C_2(T) $ and thus
	$\Norm{u \circ T^{-1}}{H^{1}(\Omega';\R{n})}  \leq C_2(T)  \Norm{u}{H^{1}(\Omega,\R{3})}.$
	Analogously we obtain
	\begin{align*}
	\Norm{u \circ T^{-1}}{L^{2}(\Gamma')}^2 
	=\int_{\Gamma'} |(u \circ T^{-1})|^2 \,dS 
	= \int_{\Gamma}  |u|^2 |\omega| \, dx 
	\leq \Norm{\omega}{\infty,\overline{\Gamma}} \Norm{u}{L^{2}(\Gamma)}^2
	= C_3\Norm{u}{L^{2}(\Gamma)}^2 .
	\end{align*}
\end{proof}

\noindent \textbf{Integration by Parts in the Volume}\\[1em]
The divergence theorem \cite{Heuser_Analysis2} for a vector field $ a \in C^{l}(D,\R{n}) $ and a scalar function $ \zeta \in C^{l}(D,\R{n}) $ with $ l\geq 1 $ on a domain  $ D \subset \R{n} $ with piece wise $ C^{1} $-boundary reads
\begin{align*}
\int_{D} \skp{a }{\nabla \zeta} \, dx = -\int_{D} \zeta\Div(a)+\int_{\partial D}  \skp{\zeta a}{\vec{n}} \, dS
\end{align*}
with $ \vec{n} $ denoting the outward normal vector field on $ \partial D $.

By $ A_{.,k} $ we denote the $ k $-th sparse of a matrix $ A \in \R{n \times n} $ and by $ A_{k,.} $ the $ k $-th row. With $ \tr(A(x) Dz(x))= \sum_{k=1}^{n} \skp{A_{.,k}(x)}{\nabla z_{k}(x)} $  we thus obtain for any matrix valued $ C^{k} $ function $ A $ on $ \R{n\times n} $ and any vector field $ z \in C^{k}(D,\R{n}) $
\begin{align}\label{App: Eq: divergence  matrix valued functions}
\int_{D} \tr(A Dz) \, dx 
&=  \int_{D} \sum_{k=1}^{n} \skp{A_{.,k}}{\nabla z_{k}} \, dx = - \int_{D} \sum_{k=1}^{n} z_k \Div(A_{.,k})\, dx +  \int_{\partial D} \sum_{k=1}^{n}  \skp{A_{.,k}z_k}{\vec{n}} \nonumber\\
&=  - \int_{D} \skp{\Div(A)}{z} \, dx +  \int_{\partial D} \skp{Az}{\vec{n}} \, dS\\
&=\int_{D} \skp{-\Div(A)}{z} \, dx +  \int_{\partial D} \skp{A^{\top} \vec{n}}{z} \, dS \nonumber
\end{align}
In the last step we applied 
$ \Div(A)= \begin{pmatrix} \Div(A_{.,1}), & \cdots, & \Div(A_{.,n}) \end{pmatrix}^{\top} $.
\vspace{1em}

\noindent \textbf{Integration by Parts on the Boundary} \\[1ex]
The divergence theorem \cite{SokZol92} for a vector field $ a \in C^{l}(D,\R{n}) $ and a scalar function $ \zeta \in C^{l}(D,\R{n}) $ with $ l\geq 1 $ on the boundary of a domain  $ D \subset \R{n} $ with piece wise $ C^{1} $-boundary reads
\begin{align*}
\int_{\partial D} \skp{a }{\nabla_{\partial D} \zeta} \, dx = \int_{\partial D} - \zeta\Div_{\partial D}(a)+\kappa  \zeta\skp{ a}{\vec{n}} \, dS
\end{align*}
with $ \vec{n} $ denoting the outward normal vector field on $ \partial D $ and where $ \kappa=\Div_{\Gamma}\vec{n} $ denotes the mean curvature of $ \partial D $. With $ \Div_{\partial D}(A)= \begin{pmatrix} \Div_{\partial D}(A_{.,1}), & \cdots, & \Div_{\partial D}(A_{.,n}) \end{pmatrix}^{\top} $ we obtain (analogously to \ref{App: Eq: divergence  matrix valued functions})
\begin{align}\label{Anh:Eq: div_thm_matrix_ vector_fields_boundary}
\int_{\partial D} \tr(A D_{\partial D}z) \, dx =   \int_{\partial D} \skp{-\Div_{\partial D}(A)}{z} + \kappa \skp{A^{\top}\vec{n}}{z} \, dS.
\end{align}

\begin{lem}\label{App:Transf_shape_opt:Lem:Prepartations}
	Let $ \Omega \subset \R{3} $ be a bounded domain with boundary $ \Gamma $ of class $ C^1 $, $\vec{n} $ the unity outward normal vector field and $ v \in C^{1}(\overline{\Omega},\R{3}) $. Then the following holds on $ \Gamma $:
	\begin{align*}
	\begin{array}{l l}
	i)& Dv_{\vec{n}}= \frac{1}{\mu}[\sigma(v)_{\vec{n}} -\lambda\Div(v)I_{\vec{n}}] -(Dv^{\top})_{\vec{n}}, 
	\\[1ex]
	ii)&\tr(Dv_{\vec{n}})=\frac{1}{\lambda+2\mu}(\langle \sigma(v)\vec{n},\vec{n} \rangle - \lambda \Div_{\Gamma}(v)),
	\\[1ex] 
	iii)& \tr\left((Dv_{\vec{n}})^{\top}\, \leftidx{_\Gamma}{M}\right) = \frac{1}{\mu} \left\langle \sigma(v)\vec{n},\,\leftidx{_\Gamma}{M}\vec{n} \right\rangle - \tr\left(M_{\vec{n}}D_{\Gamma}v\right),
	\\[1ex]
	iv) & \sigma_{\Gamma}(v) = \lambda \Div_{\Gamma}(v)I_{\Gamma} + \mu (D_{\Gamma}v + D_{\Gamma}v^{\top} I_{\Gamma}) +  \lambda \tr\left(Dv_{\vec{n}}\right)I_{\Gamma} +   \mu(Dv_{\vec{n}})^{\top} I_{\Gamma}, 
	\\[1ex]
	v)&	\tr(M\sigma_{\Gamma}(v)) =
	\left(\lambda - \frac{\lambda^2}{\lambda+2\mu}\right)\tr\left(M_{\Gamma}\right) \Div_{\Gamma}(v) + \mu \tr\left([M+M^{\top}]_{\Gamma} D_{\Gamma}v\right) \\[1ex]
	&\hspace{2.2cm}+ \left\langle \sigma(v)\vec{n}, \left(\frac{\lambda}{\lambda+2\mu}\tr\left(M_{\Gamma}\right)I + \,\leftidx{_\Gamma}{M} \right)\vec{n}\right\rangle.
	\end{array}
	\end{align*}
\end{lem}

\begin{proof}
	\begin{itemize}
		\item[i)] 
		$ \sigma(v)\vec{n}\vec{n}^{\top}
		= (\lambda\Div(v)I+\mu(Dv+Dv^{\top}))\vec{n}\vec{n}^{\top}
		=\lambda\Div(v)I_{\vec{n}} + \mu (Dv_{\vec{n}} + (Dv^{\top})_{\vec{n}}).
		$ 
		Solving for $Dv_{\vec{n}}$ proofs the assertion.
		\item[ii)] Since $ \langle \sigma(v)\vec{n},\vec{n} \rangle  =  \tr(\sigma(v)\vec{n}\vec{n}^{\top}) =\tr(\sigma(v)_{\vec{n}}) $ and 
		\[ \tr((Dv^{\top})_{\vec{n}})=\tr(Dv^{\top}\vec{n}\vec{n}^{\top}) = \tr(\vec{n}\vec{n}^{\top}Dv)=\tr(Dv\,\vec{n}\vec{n}^{\top})=\tr(Dv_{\vec{n}})\] the proof of i) implies
		\begin{align*}
		\langle \sigma(v)\vec{n},\vec{n} \rangle 
		&=\lambda \Div_{\Gamma}(v)\tr(I_{\vec{n}})+ \lambda\tr(Dv_{\vec{n}})\tr(I_{\vec{n}}) + \mu\tr(Dv_{\vec{n}}) + \mu \tr(Dv_{\vec{n}}) \\
		&= \lambda \Div_{\Gamma}(v) + (\lambda+2\mu)\tr(Dv_{\vec{n}}).
		\end{align*}
		\item[iii)] With $ I_{\vec{n}}I_{\Gamma}=\vec{n}\vec{n}^{\top} (I -\vec{n}\vec{n}^{\top} ) = \vec{n}\vec{n}^{\top} -\vec{n}\vec{n}^{\top} =0 $ and $ \tr(ab^{\top}) =\left\langle a , b\right \rangle \, \forall a,\, b \in \R{n} $ we derive
		\begin{align*}
		\tr\left[(Dv_{\vec{n}})^{\top}\,\leftidx{_\Gamma}{M}\right]
		&= \frac{1}{\mu}\tr\left[(\sigma(v)_{\vec{n}})^{\top}\,\leftidx{_\Gamma}{M}\right]  
		- \frac{\lambda}{\mu}\, \tr(Dv)\tr\left[(I_{\vec{n}})^{\top}\,\leftidx{_\Gamma}{M}\right] 
		- \tr\left[( Dv^{\top})_{\vec{n}})^{\top}\,\leftidx{_\Gamma}{M}\right]
		\\
		&=\frac{1}{\mu}\tr\left[\,\leftidx{_\Gamma}{M} \vec{n} \,\vec{n}^{\top} \,\sigma(v)\right]   - \frac{\lambda}{\mu} \tr(Dv)\tr\left(I_{\vec{n}}I_{\Gamma}M\right)-\tr\left(I_{\vec{n}} DvI_{\Gamma}\,M\right)
		\\
		&=\frac{1}{\mu}\left\langle\sigma(v)\vec{n},\,\leftidx{_\Gamma}{M}\vec{n}\right\rangle  
		- \tr\left(M_{\vec{n}} D_{\Gamma}v\right)
		\end{align*}
		\item[iv)] Since $ \sigma_{\Gamma}(v)  = \sigma(v)-\sigma(v)_{\vec{n}}  $ we obtain
		\begin{align*}
		\sigma_{\Gamma}(v)  
		&\underset{}{=} \lambda \Div(v)I+\mu (Dv+Dv^{\top}) - [\lambda\Div(v)I_{\vec{n}} + 
		\mu (Dv_{\vec{n}} + (Dv^{\top})_{\vec{n}})]
		\\
		&= \lambda \Div(v)(I-I_{\vec{n}})+ \mu(Dv-Dv_{\vec{n}}+Dv^{\top}-Dv^{\top}_{\vec{n}}) 
		\\
		&\underset{}{=}  \lambda \Div(v)I_{\Gamma}+ \mu(D_{\Gamma}v+Dv^{\top}I_{\Gamma}) 
		\\
		&\underset{}{=} \lambda [\Div_{\Gamma}(v) + \tr(Dv_{\vec{n}})]I_{\Gamma}+ \mu(D_{\Gamma}v+[D_{\Gamma}v + Dv_{\vec{n}}]^{\top}I_{\Gamma}) 
		\\
		&\underset{}{=} \lambda\Div_{\Gamma}(v)I_{\Gamma} + \lambda\tr(Dv_{\vec{n}})I_{\Gamma}+ \mu(D_{\Gamma}v+D_{\Gamma}v^{\top}I_{\Gamma}) +\mu (Dv_{\vec{n}})^{\top}I_{\Gamma}) 
		\end{align*}
		\item[v)] We have $ \tr(D_{\Gamma}v M)=\tr(M D_{\Gamma}v) $  and $ \tr((D_{\Gamma}v)^{\top}\,\leftidx{_\Gamma}{M})=\tr((\,\leftidx{_\Gamma}{M})^{\top}D_{\Gamma}v)$. \begin{align*}
		&\tr(\sigma_{\Gamma}(v)M) \\
		&=\tr\left(\lambda\Div_{\Gamma}(v)\,\leftidx{_\Gamma}{M} +\lambda \tr(Dv_{\vec{n}})\,\leftidx{_\Gamma}{M} + \mu D_{\Gamma}v M +\mu D_{\Gamma}v^{\top}\,\leftidx{_\Gamma}{M}  + \mu (Dv_{\vec{n}})^{\top} \,\leftidx{_\Gamma}{M}\right)\\
		&=\lambda\Div_{\Gamma}(v)\tr\left(\,\leftidx{_\Gamma}{M}\right) + \mu \tr\left(\left[M+(\,\leftidx{_\Gamma}{M})^{\top}\right] D_{\Gamma}v\right)+ \mu \tr\left((Dv_{\vec{n}})^{\top} \,\leftidx{_\Gamma}{M}\right)\\
		&~~~+\lambda \tr\left(Dv_{\vec{n}}\right)\tr\left(\,\leftidx{_\Gamma}{M}\right),
		\end{align*}
		Using ii) and iii) we get 
		\begin{align*}
		\tr(M\sigma_{\Gamma}(v)) 
		=& \lambda\Div_{\Gamma}(v)\tr\left(\,\leftidx{_\Gamma}{M}\right)  + \mu \tr\left(\left[M+(\,\leftidx{_\Gamma}{M})^{\top}\right] D_{\Gamma}v\right)  + \left\langle \sigma(v)\vec{n},\,\leftidx{_\Gamma}{M}\vec{n} \right\rangle
		\\
		& +  \frac{\lambda}{\lambda+2\mu} \left[\langle \sigma(u)\vec{n},\vec{n} \rangle - \lambda\Div_{\Gamma}(v)\right]\tr\left(\,\leftidx{_\Gamma}{M}\right)
		-\mu\tr\left(M_{\vec{n}}D_{\Gamma}(v)\right)
		\\
		=&\left(\lambda-\frac{\lambda^2}{\lambda+2\mu}\right)\Div_{\Gamma}(v)\tr(\,\leftidx{_\Gamma}{M})+\mu\tr\left(\left[M+(\,\leftidx{_\Gamma}{M})^{\top}-M_{\vec{n}}\right]D_{\Gamma}v\right)
		\\
		&+\frac{\lambda}{\lambda+2\mu}\langle \sigma(v)\vec{n},\vec{n}\rangle\tr\left(\,\leftidx{_\Gamma}{M}\right) + \left\langle \sigma(v)\vec{n}, \,\leftidx{_\Gamma}{M}\vec{n}\right\rangle
		\\
		=&\left(\lambda-\frac{\lambda^2}{\lambda+2\mu}\right)\Div_{\Gamma}(v)\tr(\,\leftidx{_\Gamma}{M})
		+\mu\tr\left(\left[M+M^{\top}\right]_{\Gamma}D_{\Gamma}v\right)
		\\
		&+\left\langle \sigma(v)\vec{n},\frac{\lambda}{\lambda+2\mu}\left[\tr\left(\,\leftidx{_\Gamma}{M}\right)I+\,\leftidx{_\Gamma}{M}\right]\vec{n}\right\rangle 
		\end{align*}
	\end{itemize}
\end{proof}

\backmatter

\bibliographystyle{plain}
\bibliography{quellen}

\end{document}